\documentclass[12pt]{report}
\author{Will Johnson}
\title{Combinatorial Game Theory, Well-Tempered Scoring Games, and a Knot Game}
\usepackage{amsfonts}
\usepackage{amsmath}
\usepackage{amssymb}
\usepackage{amsthm}
\usepackage{hyperref}
\usepackage{graphicx}
\usepackage{float}
\usepackage{synttree}

\DeclareMathOperator{\start}{start}
\DeclareMathOperator{\outcome}{o^\#}
\DeclareMathOperator{\loutcome}{L}
\DeclareMathOperator{\routcome}{R}
\DeclareMathOperator{\lfout}{Lf}
\DeclareMathOperator{\rfout}{Rf}

\DeclareMathOperator{\defi}{def}
\DeclareMathOperator{\val}{val}

\newtheorem{theorem}{Theorem}[section] % numbered like the section
\newtheorem{lemma}[theorem]{Lemma} % numbered like the theorems
\newtheorem{proposition}[theorem]{Proposition}
\newtheorem{corollary}[theorem]{Corollary}
\newtheorem{definition}[theorem]{Definition}
\newtheorem{claim}[theorem]{Claim}

\newtheorem{exercise}[theorem]{Exercise}
\newtheorem{fact}[theorem]{Fact}

\begin{document}
\maketitle
%\begin{abstract}
%Abstract goes here
%\end{abstract}
\tableofcontents
\chapter{To Knot or Not to Knot}\label{chap:zero}
Here is a picture of a knot:
\begin{figure}[H]
\begin{center}
\includegraphics[width=0.7in]
					{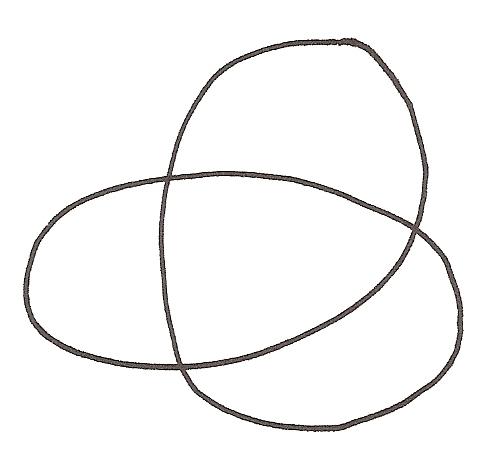}
\caption{A Knot Shadow}
\label{startshadow}
\end{center}
\end{figure}
%\begin{equation} \text{ trefoil shadow }\label{startshadow}\end{equation}

Unfortunately, the picture doesn't show which strand is on top and which strand is below, at each intersection.
So the knot in question could be any one of the following eight possibilities.
%\begin{equation} \text{ picture of all eight possibilities }\label{octagonal}\end{equation}
\begin{figure}[H]
\begin{center}
\includegraphics[width=2in]
					{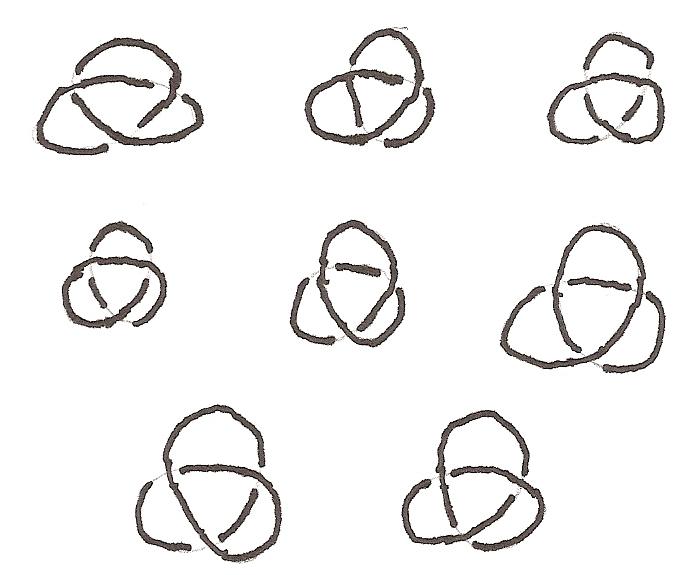}
\caption{Resolutions of Figure~\ref{startshadow}}
\label{octagonal}
\end{center}
\end{figure}
Ursula and King Lear decide to play a game with Figure~\ref{startshadow}.  They take turns
alternately \emph{resolving} a crossing, by choosing which strand is on top.  If Ursula goes first, she could
move as follows:
\begin{figure}[H]
\begin{center}
\includegraphics[width=3in]
					{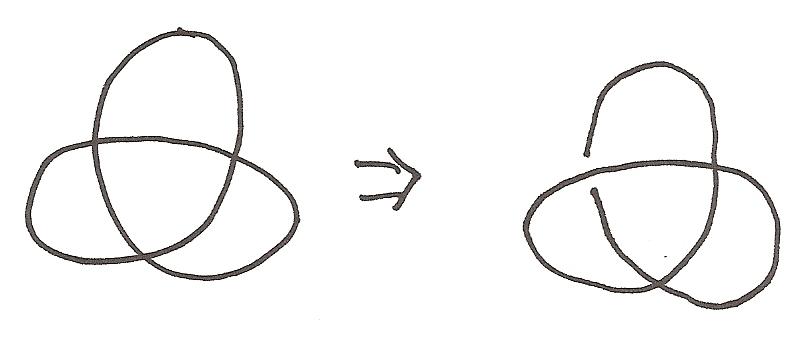}
%\caption{Projections corresponding to Figure~\ref{startshadow}}
%\label{firstmove}
\end{center}
\end{figure}
%\begin{equation} \text{ First move }\label{firstmove}\end{equation}
King Lear might then respond with
\begin{figure}[H]
\begin{center}
\includegraphics[width=3in]
					{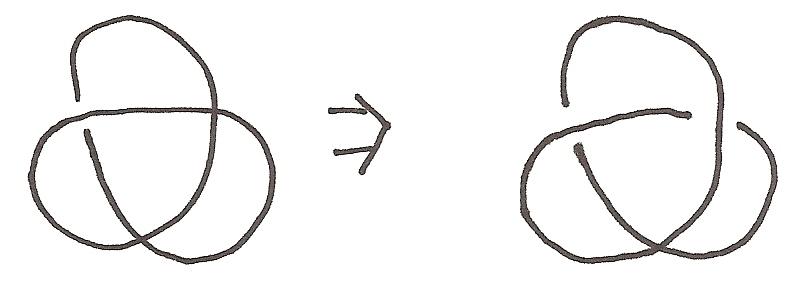}
%\caption{Projections corresponding to Figure~\ref{startshadow}}
%\label{secondmove}
\end{center}
\end{figure}
%\begin{equation} \text{ Second move }\label{secondmove}\end{equation}
For the third and final move, Ursula might then choose to move to
\begin{figure}[H]
\begin{center}
\includegraphics[width=3in]
					{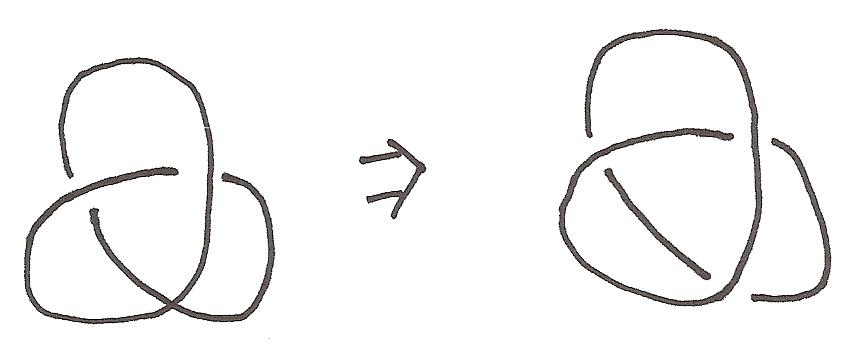}
\caption{The final move}
\label{thirdmove}
\end{center}
\end{figure}
%\begin{equation} \text{ Final move }\label{thirdmove}\end{equation}
Now the knot is completely identified.  In fact, this knot can be untied as follows, so mathematically
it is the \emph{unknot}:
\begin{figure}[H]
\begin{center}
\includegraphics[width=2.5in]
					{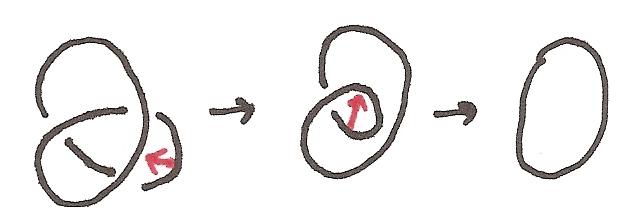}
%\caption{The final move}
%\label{thirdmove}
\end{center}
\end{figure}
Because the final knot was the \textbf{u}nknot, \textbf{U}rsula
is the winner - had it been truly \textbf{k}notted, \textbf{K}ing Lear would be the winner.

A picture of a knot like the ones in Figures \ref{octagonal} and \ref{thirdmove} is called a \emph{knot diagram}
or \emph{knot projection} in the field of mathematics known as Knot Theory.  The generalization in which
some crossings are \emph{unresolved} is called a \emph{pseudodiagram} - every diagram we have just seen is an example.
A pseudodiagram in which \emph{all} crossing are unresolved is called a \emph{knot shadow}.  While knot diagrams
are standard tools of knot theory, pseudodiagrams are a recent innovation by Ryo Hanaki for the sake of mathematically
modelling electron microscope images of DNA in which the elevation of the strands is unclear, like the
following\footnote{Image taken from http://www.tiem.utk.edu/bioed/webmodules/dnaknotfig4.jpg on July 6, 2011.}:
\begin{figure}[H]
\begin{center}
\includegraphics[width=2.5in]
					{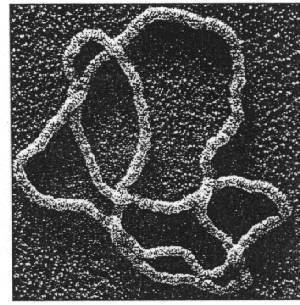}
\caption{Electron Microscope image of DNA}
%\label{thirdmove}
\end{center}
\end{figure}
Once upon a time, a group of students in a Research Experience for Undergraduates (REU) at Williams College in 2009 were studying
properties of knot pseudodiagrams, specifically the \emph{knotting number} and \emph{trivializing number},
which are the smallest number of crossings which one can resolve to ensure that the resulting pseudodiagram corresponds to a knotted
knot, or an unknot, respectively.  One of the undergraduates\footnote{Oliver Pechenik, according to \url{http://www.math.washington.edu/~reu/papers/current/allison/UWMathClub.pdf}} had the idea of turning this process into a game
between two players, one trying to create an unknot and one trying to create a knot, and thus was born \emph{To Knot or Not to Knot} (TKONTK),
the game described above.

In addition to their paper on knotting and trivialization numbers, the students in the REU wrote an additional
Shakespearean-themed paper \emph{A Midsummer Knot's Dream} on \textsc{To Knot or Not to Knot} and a couple of other knot games, with names like
``Much Ado about Knotting.''  In their analysis of TKONTK specifically, they considered
starting positions of the following sort:
\begin{figure}[H]
\begin{center}
\includegraphics[width=3in]
					{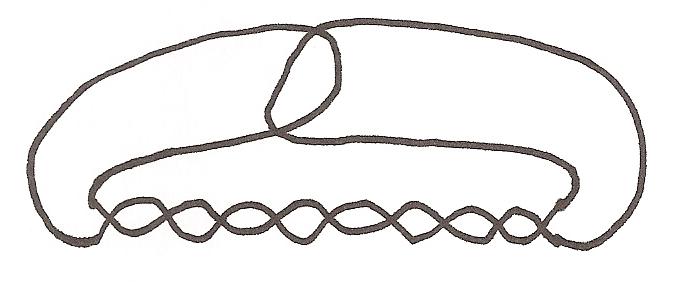}
%\caption{The final move}
%\label{thirdmove}
\end{center}
\end{figure}
For these positions, they determined which player wins under perfect play:
\begin{itemize}
\item If the number of crossings is odd, then Ursula wins, no matter who goes first.
\item If the number of crossings is even, then whoever goes second wins.
\end{itemize}
They also showed that on a certain large class of shadows, the second player wins.

\section{Some facts from Knot Theory}
In order to analyze TKONTK, or even to play it, we need a way to tell whether a given knot diagram
corresponds to the unknot or not.  Unfortunately this problem is very non-trivial, and while algorithms
exist to answer this question, they are very complicated.

One fundamental fact
in knot theory is that two knot diagrams correspond to the same knot if and only if they can be obtained
one from the other via a sequence of \emph{Reidemeister moves}, in addition to mere distortions (isotopies) of the plane
in which the knot diagram is drawn.  The three types of Reidemeister moves are
\begin{enumerate}
\item Adding or removoing a twist in the middle of a straight strand.
\item Moving one strand over another.
\item Moving a strand over a crossing.
\end{enumerate}
These are best explained by a diagram:
\begin{figure}[H]
\begin{center}
\includegraphics[width=4in]
					{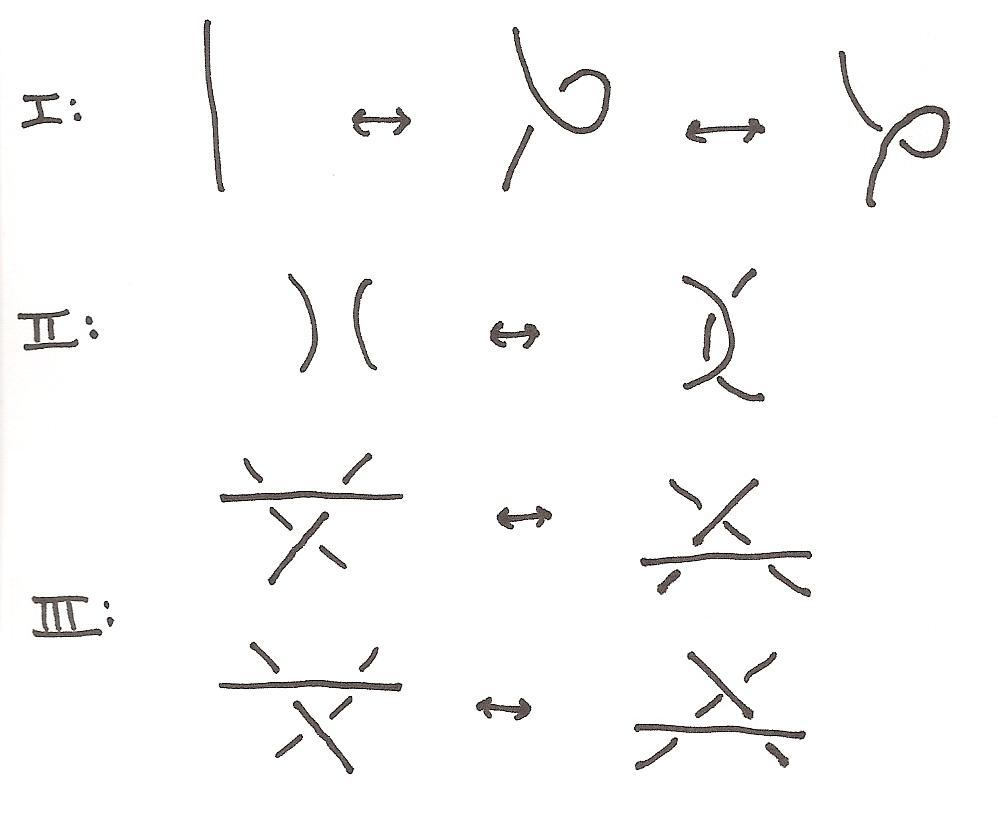}
\caption{The three Reidemeister Moves}
%\label{thirdmove}
\end{center}
\end{figure}
Given this fact, one way to classify knots is by finding properties of knot diagrams which are invariant under the
Reidemeister moves.  A number of surprising \emph{knot invariants} have been found, but none are known to be \emph{complete invariants},
which exactly determine whether two knots are equivalent.

Although this situation may seem bleak, there are certain families of knots in which we can test for unknottedness easily.
One such family is the family of \emph{alternating knots}.  These are knots with the property that if you walk along them, you alternately
are on the top or the bottom strand at each successive crossing.  Thus the knot weaves under and over itself perfectly.  Here are some
examples:
\begin{figure}[H]
\begin{center}
\includegraphics[width=2in]
					{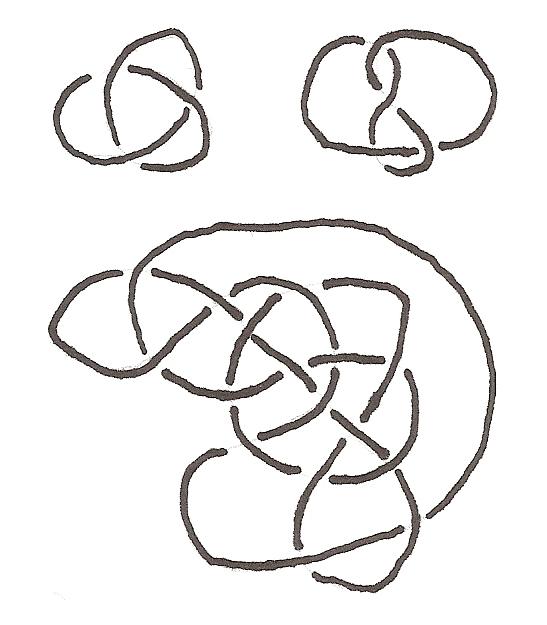}
%\caption{The three Reidemeister Moves}
%\label{thirdmove}
\end{center}
\end{figure}
The rule for telling whether an alternating knot is the unknot is simple: color the regions between the strands black and white
in alternation, and connect the black regions into a graph.  Then the knot is the unknot if and only if the graph can be
reduced to a tree by removing self-loops.
For instance,
\begin{figure}[H]
\begin{center}
\includegraphics[width=2in]
					{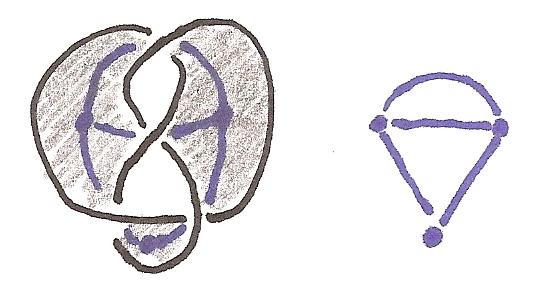}
%\caption{The three Reidemeister Moves}
%\label{thirdmove}
\end{center}
\end{figure}
is not the unknot, while
\begin{figure}[H]
\begin{center}
\includegraphics[width=2in]
					{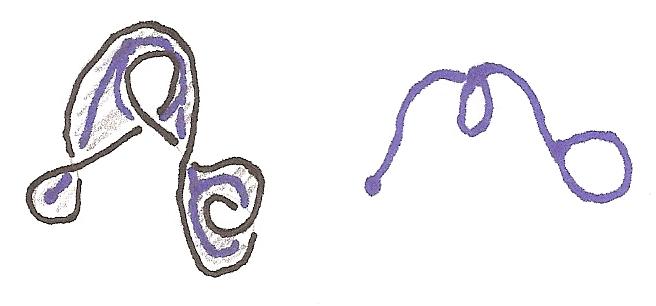}
%\caption{The three Reidemeister Moves}
%\label{thirdmove}
\end{center}
\end{figure}
is.

Now it turns out that any knot shadow can be turned into an alternating knot - but in only two ways.  The players are unlikely
to produce one of these two resolutions, so this test for unknottedness is not useful for the game.

Another family of knots, however, works out perfectly for TKONTK.  These are the \emph{rational knots}, defined in terms of the \emph{rational tangles}.
A \emph{tangle} is like a knot with four loose ends, and two strands.  Here are some examples:
\begin{figure}[H]
\begin{center}
\includegraphics[width=3in]
					{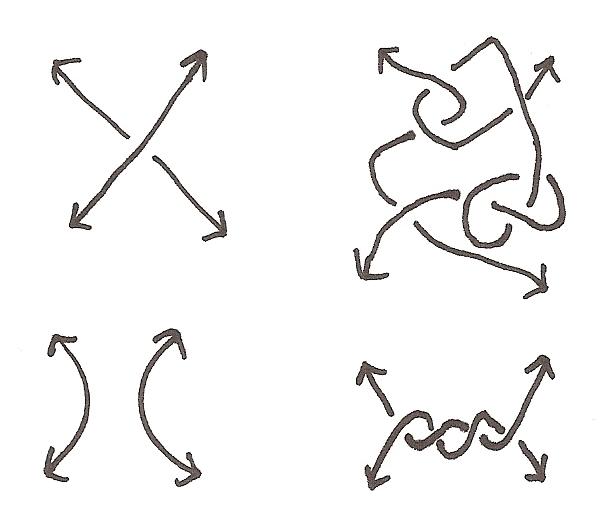}
%\caption{The three Reidemeister Moves}
%\label{thirdmove}
\end{center}
\end{figure}
The four loose ends should be though of as going off to infinity, since they can't be pulled in to unknot the tangle.  We consider
two tangles to be equivalent if you can get from one to the other via Reidemeister moves.

A \emph{rational} tangle is one built up from the following two
\begin{figure}[H]
\begin{center}
\includegraphics[width=3in]
					{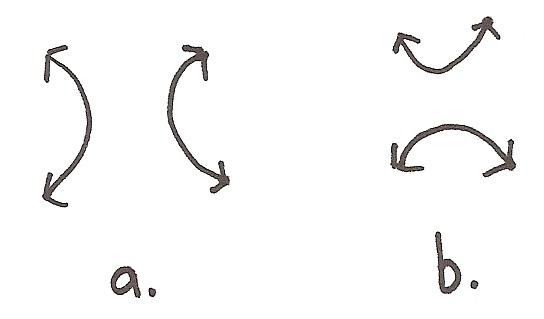}
%\caption{The three Reidemeister Moves}
%\label{thirdmove}
\end{center}
\end{figure}
via the following operations:
\begin{figure}[H]
\begin{center}
\includegraphics[width=4in]
					{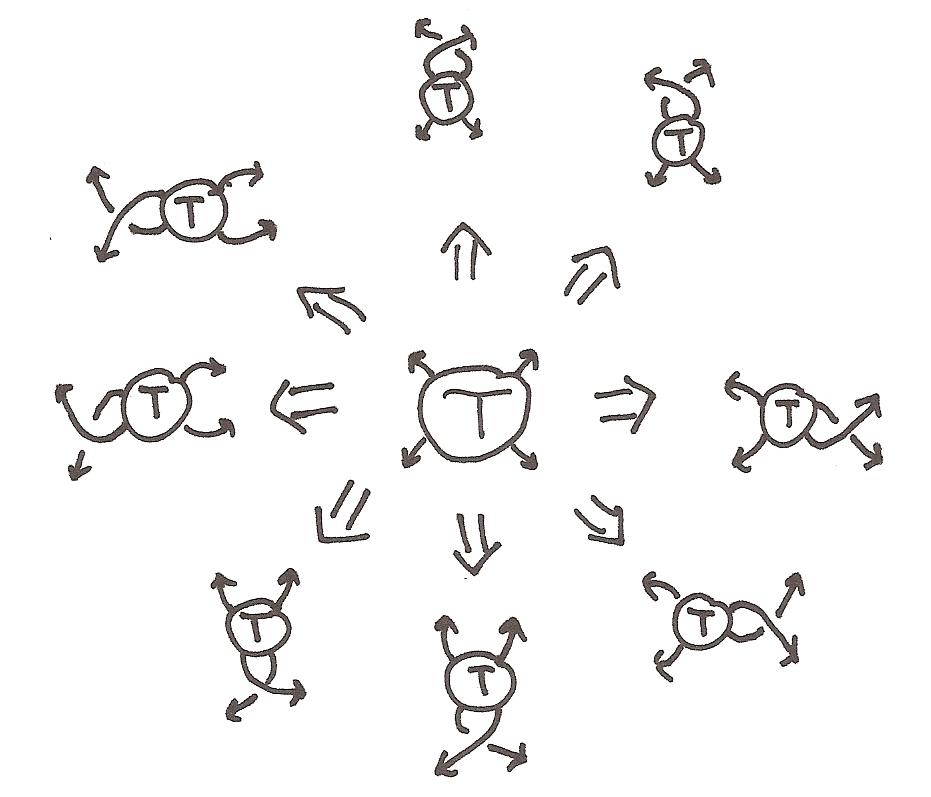}
%\caption{The three Reidemeister Moves}
%\label{thirdmove}
\end{center}
\end{figure}
Now it can easily be seen by induction that if $T$ is a rational tangle, then $T$ is invariant under $180^{\textrm{o}}$ rotations
about the $x$, $y$, or $z$ axes:
\begin{figure}[H]
\begin{center}
\includegraphics[width=2.5in]
					{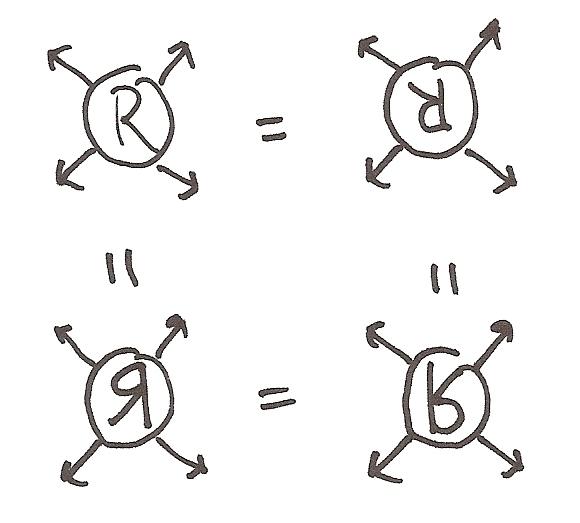}
%\caption{The three Reidemeister Moves}
%\label{thirdmove}
\end{center}
\end{figure}
Because of this, we have the following equivalences,
\begin{figure}[H]
\begin{center}
\includegraphics[width=3in]
					{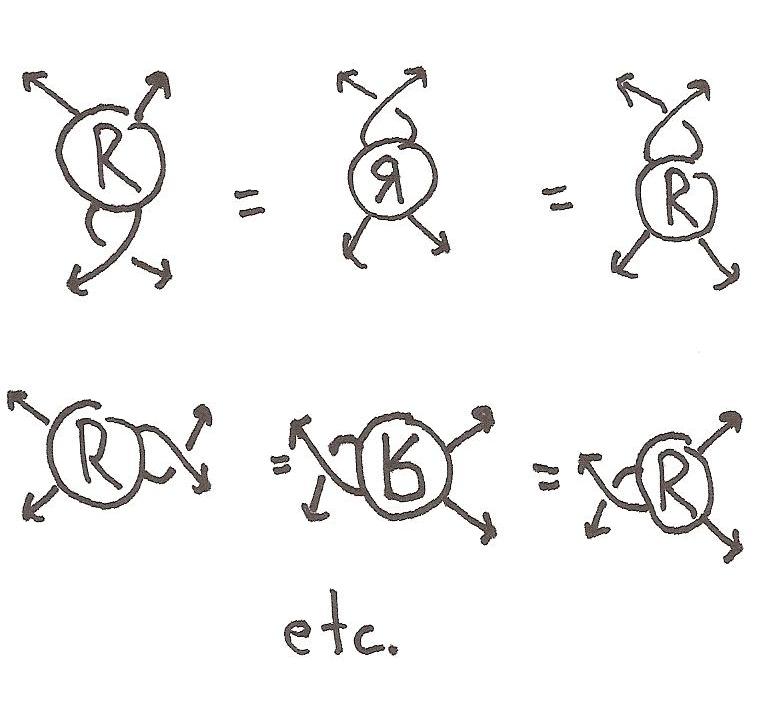}
%\caption{The three Reidemeister Moves}
%\label{thirdmove}
\end{center}
\end{figure}
In other words, adding a twist to the bottom or top of a rational tangle has the same effect, and so does adding a twist on
the right or the left.  So we can actually build up all rational tangles via the following smaller set of operations:
\begin{figure}[H]
\begin{center}
\includegraphics[width=3in]
					{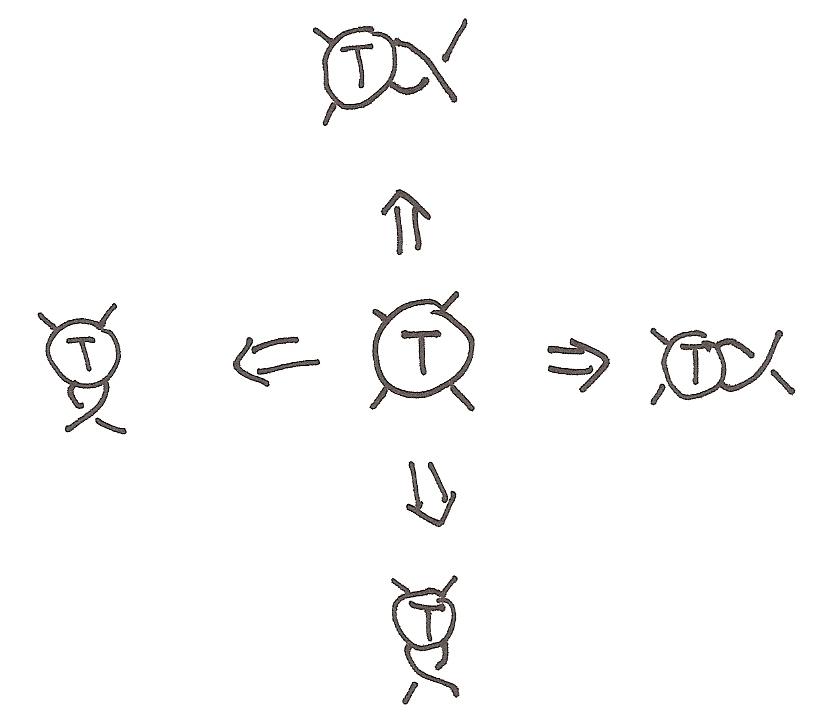}
%\caption{The three Reidemeister Moves}
%\label{thirdmove}
\end{center}
\end{figure}

John Conway found a way to assign a rational number (or $\infty$) to each rational tangle, so that the tangle is determined
up to equivalence by its number.  Specifically, the initial tangles
\begin{figure}[H]
\begin{center}
\includegraphics[width=2in]
					{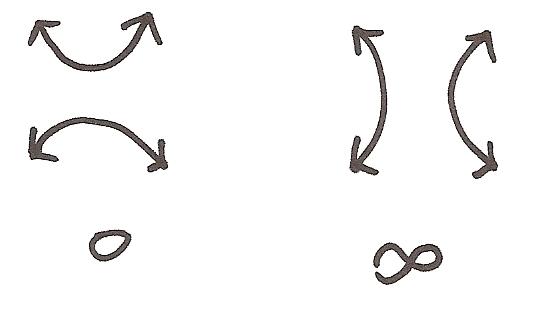}
%\caption{The three Reidemeister Moves}
%\label{thirdmove}
\end{center}
\end{figure}
have values $0$ and $\infty = 1/0$.  If a tangle $t$ has value $\frac{p}{q}$, then adding a twist on the
left or right changes the value to $\frac{p+q}{q}$ if the twist is left-handed, or $\frac{p-q}{q}$ if the twist is right handed.
Adding a twist on the top or bottom changes the value to $\frac{p}{q-p}$ if the twist is left-handed, or
$\frac{p}{q+p}$ if the twist is right-handed.
\begin{figure}[H]
\begin{center}
\includegraphics[width=4in]
					{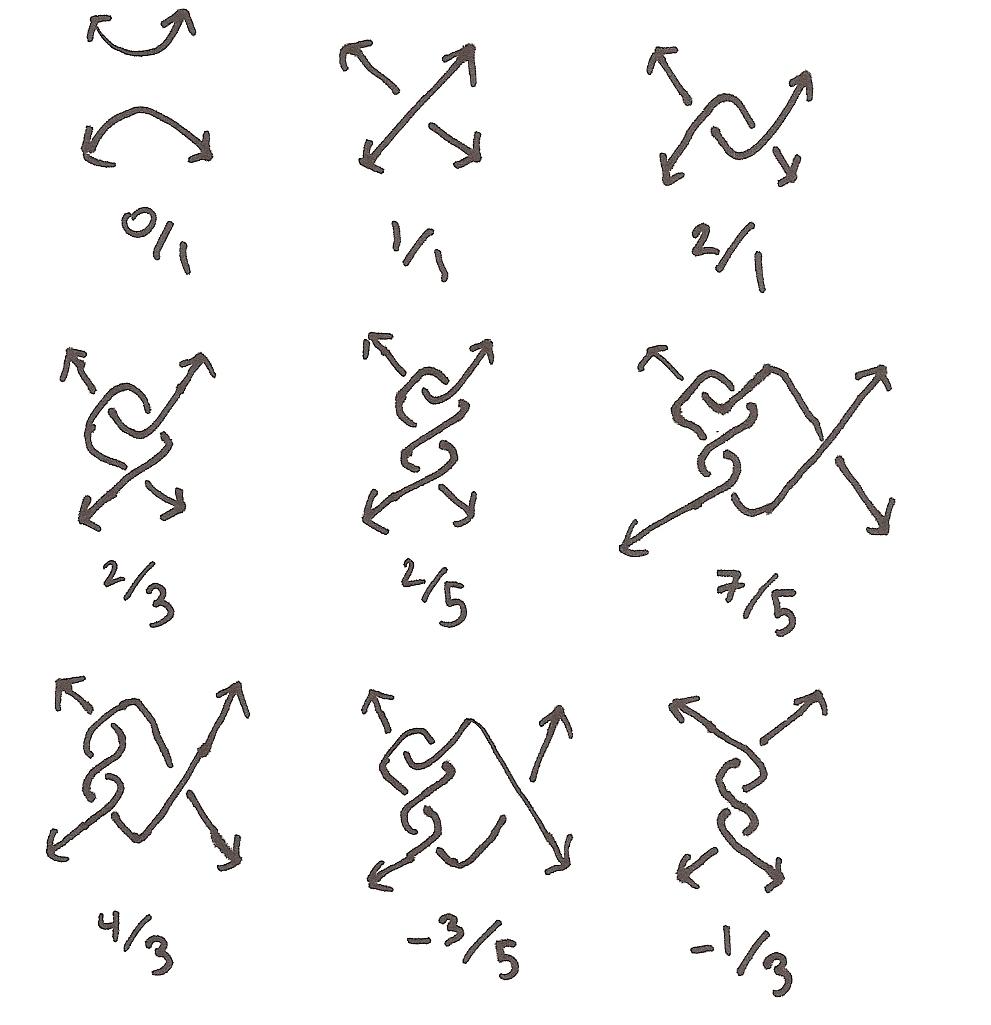}
\caption{Sample rational tangles}
\label{sample-rational}
\end{center}
\end{figure}
Reflecting a tangle over the $45^{\circ}$ diagonal plane corresponds to taking the reciprocal:
\begin{figure}[H]
\begin{center}
\includegraphics[width=3in]
					{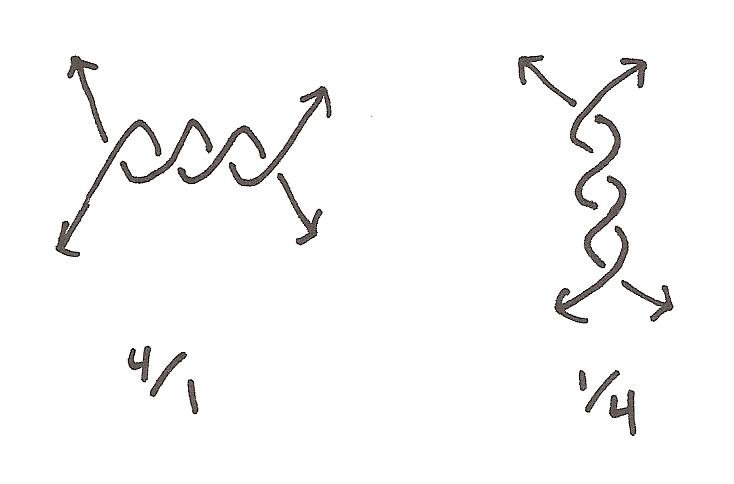}
\caption{Reflection over the diagonal plane corresponds to taking the reciprocal.  Note which strands are on top
in each diagram.}
\label{reciprocal-example}
\end{center}
\end{figure}
Using these rules, it's easy to see that a general rational tangle, built up by adding twists on the bottom or right side,
has its value determined by a \emph{continued fraction}.  For instance, the following rational tangle
\begin{figure}[H]
\begin{center}
\includegraphics[width=3in]
					{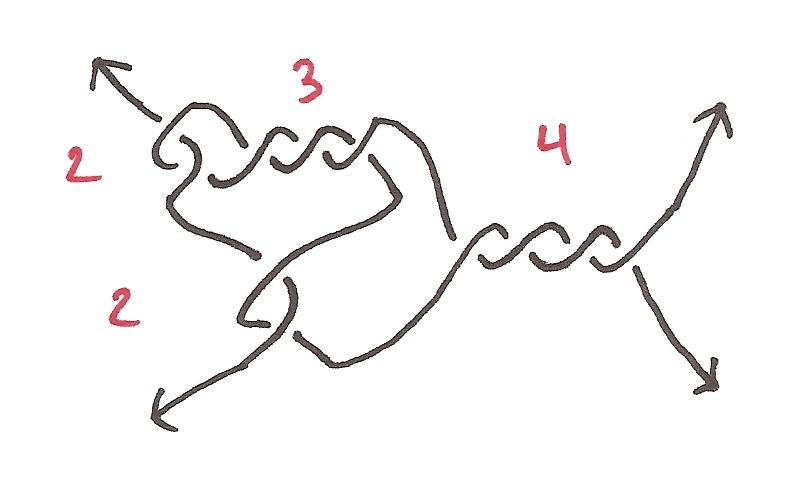}
%\caption{}
%\label{reciprocal-example}
\end{center}
\end{figure}
has value
\[ 4 + \frac{1}{2 + \frac{1}{3 + \frac{1}{2}}} = \frac{71}{64}.\]
Now a basic fact about continued fractions is that if $n_1, \ldots, n_k$ are positive integers, then the continued fraction
\[ n_1 + \frac{1}{n_2 + \frac{1}{\ddots + \frac{1}{n_k}}}\]
almost encodes the sequence $(n_1,\ldots,n_k)$.  So this discussion of continued fractions might sound like an elaborate
way of saying that rational tangles are determined by the sequence of moves used to construct them.

But our continued fractions can include negative numbers.  For instance, the following tangle
\begin{figure}[H]
\begin{center}
\includegraphics[width=3in]
					{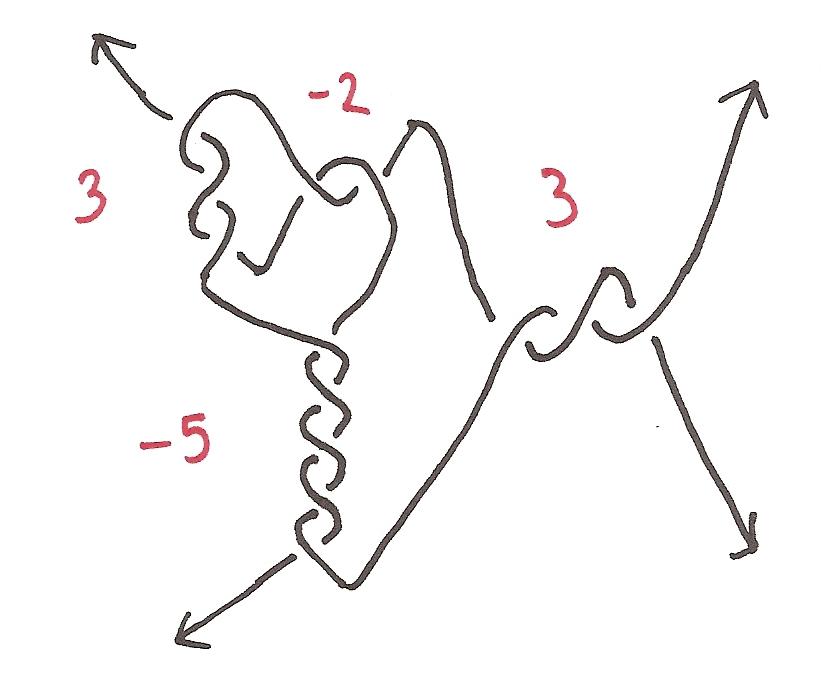}
%\caption{}
%\label{reciprocal-example}
\end{center}
\end{figure}
has continued fraction
\[ 3 + \frac{1}{-5 + \frac{1}{-2 + \frac{1}{3}}} =  79/28
= 2 + \frac{1}{1 + \frac{1}{4 + \frac{1}{1 + \frac{1}{1 + \frac{1}{2}}}}},\]
so that we have the following nontrivial equivalence of tangles:
\begin{figure}[H]
\begin{center}
\includegraphics[width=5in]
					{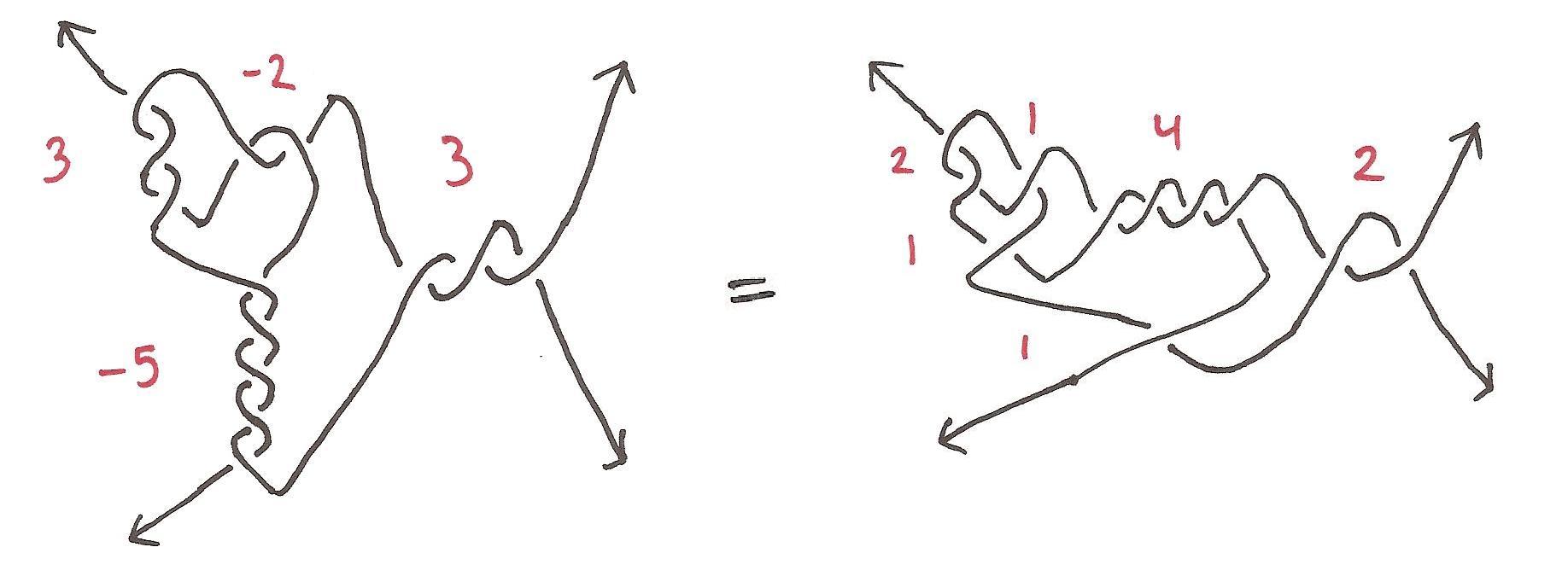}
%\caption{}
%\label{reciprocal-example}
\end{center}
\end{figure}

Given a rational tangle, its \emph{numerator closure} is obtained by
connecting the two strands on top and connecting the two strands on bottom, while the \emph{denominator closure} is obtained
by joining the two strands on the left, and joining the two strands on the right:
\begin{figure}[H]
\begin{center}
\includegraphics[width=3in]
					{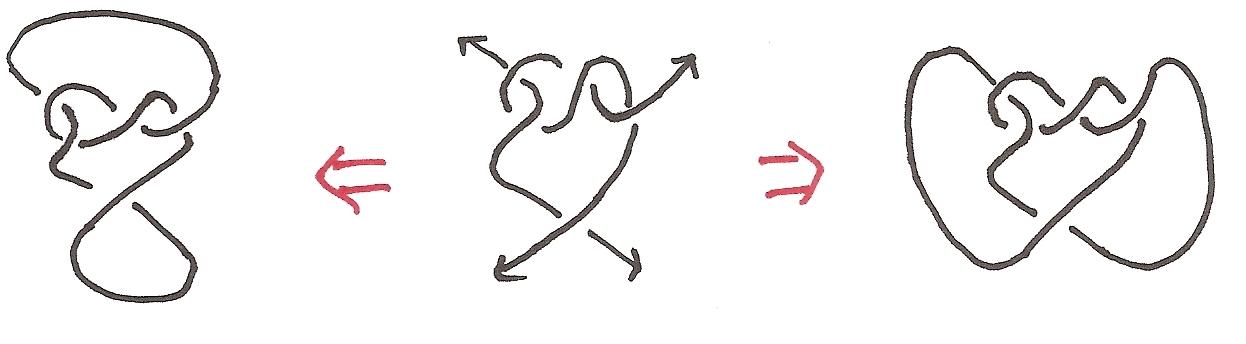}
\caption{the numerator closure (left) and denominator closure (right) of a tangle.}
\label{infinity-zero-closure}
\end{center}
\end{figure}
In some cases, the result ends up consisting of two disconnected strands, making it a \emph{link} rather
than a \emph{knot}:
\begin{figure}[H]
\begin{center}
\includegraphics[width=3in]
					{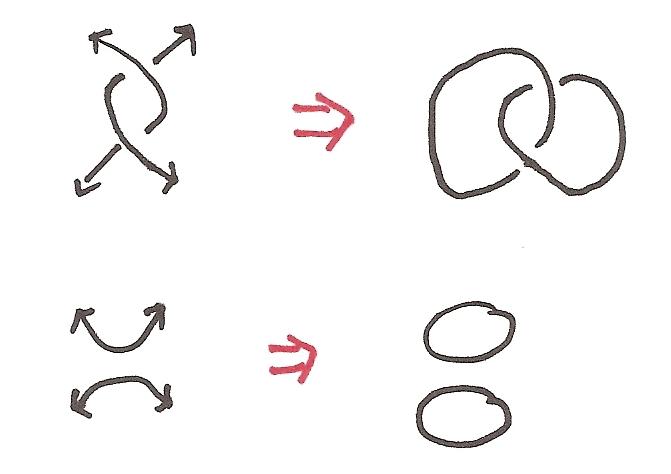}
%\caption{}
%\label{reciprocal-example}
\end{center}
\end{figure}
As a general rule, one can show that the numerator closure is a knot as long as the numerator of $p/q$
is odd, while the denominator closure is a knot as long as the denominator of $p/q$ is odd.

Even better, it turns out that the \emph{numerator closure} is an unknot exactly if the value $p/q$ is
the reciprocal of an integer, and the \emph{denominator closure} is an unknot exactly if the value $p/q$ is an integer.
% apricot
%For the interested reader, we offer proofs of these claims (modulo the fact that Reidemeister moves
%connect all equivalent knots) in the appendix. TODO not really.

The upshot of all this is that if we play TKONTK on a ``rational shadow,'' like the following:
\begin{figure}[H]
\begin{center}
\includegraphics[width=3in]
					{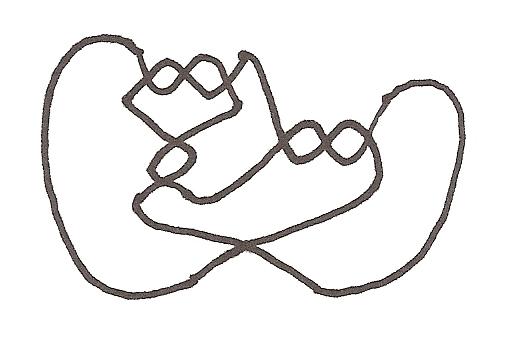}
%\caption{}
%\label{reciprocal-example}
\end{center}
\end{figure}
then at the game's end the final knot will be rational, and we can check who wins by means of continued fractions.

The twist knots considered in \emph{A Midsummer Knot's Dream} are instances of this, since they are the denominator closures of the
following rational tangle-shadows:
\begin{figure}[H]
\begin{center}
\includegraphics[width=4.5in]
					{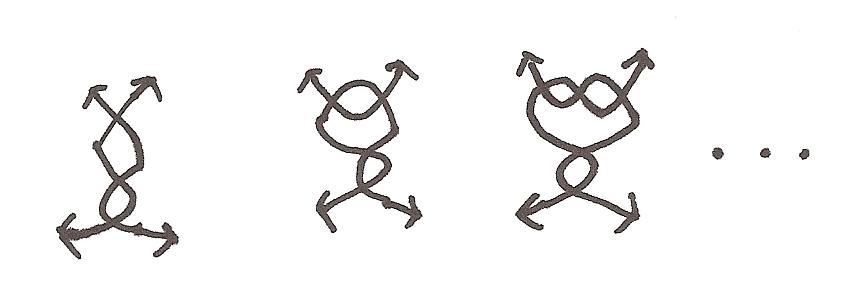}
%\caption{}
%\label{reciprocal-example}
\end{center}
\end{figure}

\section{Sums of Knots}
Now that we have a basic set of analyzable positions to work with, we can quickly extend them by the operation
of the \emph{connected sum} of two knots.

Here are two knots $K_1$ and $K_2$:
\begin{figure}[H]
\begin{center}
\includegraphics[width=2.5in]
					{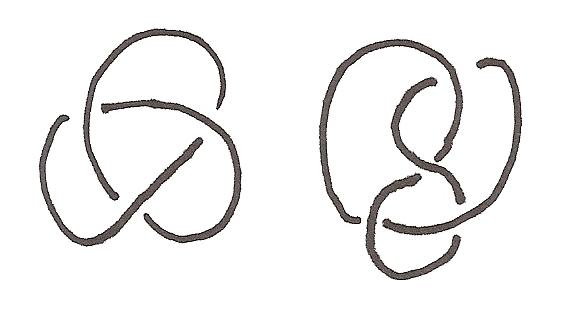}
%\caption{}
%\label{reciprocal-example}
\end{center}
\end{figure}
and here is their connected sum $K_1\#K_2$
\begin{figure}[H]
\begin{center}
\includegraphics[width=3in]
					{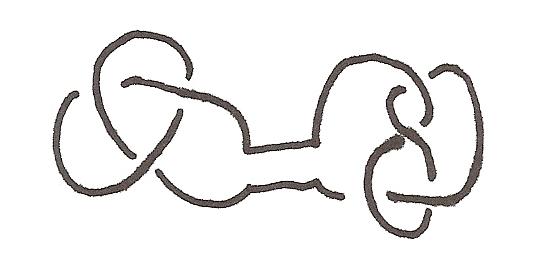}
%\caption{}
%\label{reciprocal-example}
\end{center}
\end{figure}
This sum may look arbitrary, because it appears to depend on the places where we chose to attach the two knots.
However, we can move one knot along the other to change this, as shown in the following picture:
\begin{figure}[H]
\begin{center}
\includegraphics[width=5in]
					{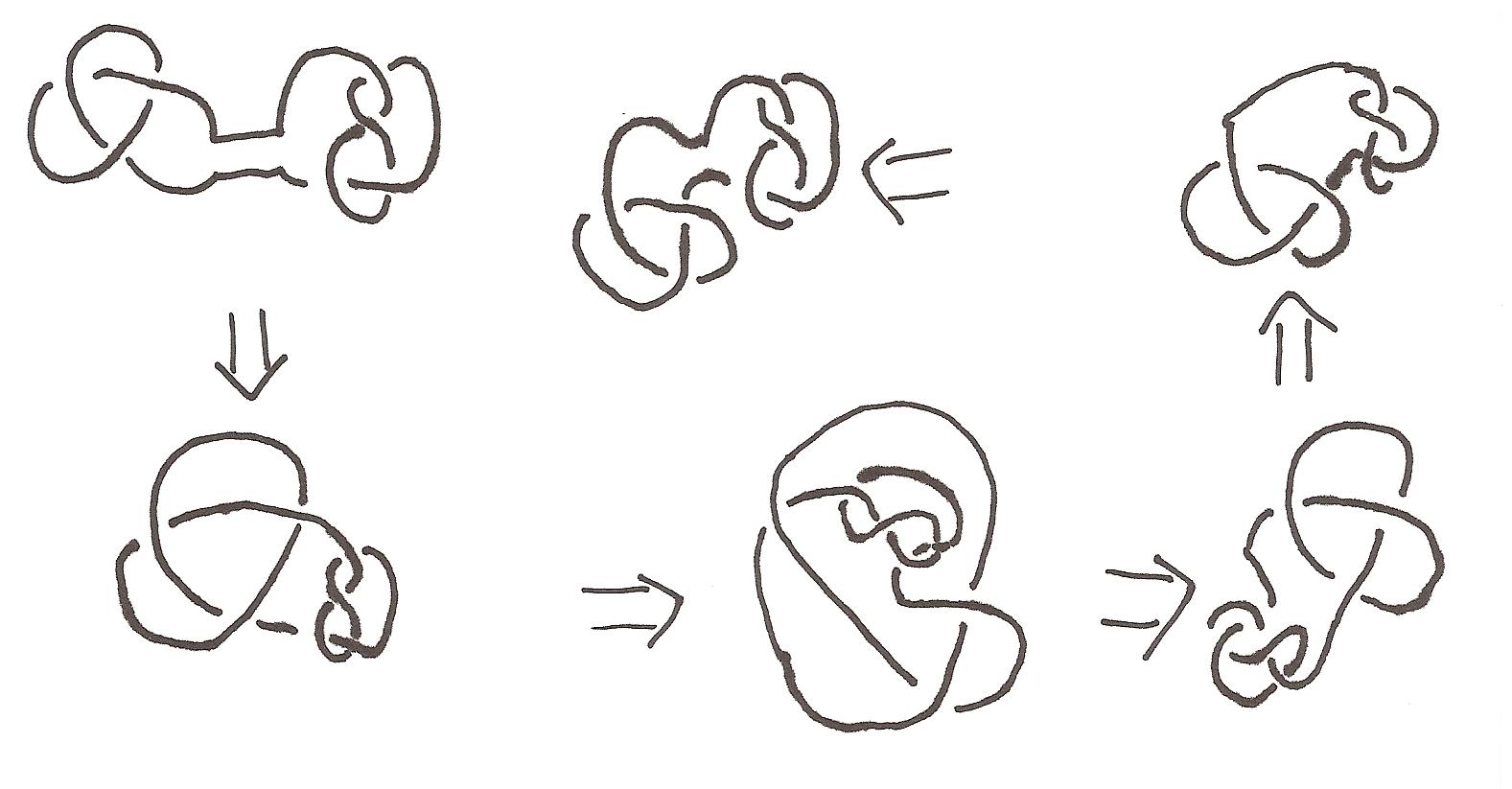}
%\caption{}
%\label{reciprocal-example}
\end{center}
\end{figure}
So the place where we choose to join the two knots doesn't matter.\footnote{Technically, the definition is still
ambiguous, unless we specify an orientation to each knot.  When adding two ``noninvertible'' knots, where the choice of orientation matters, there are two non-equivalent
ways of forming the connected sum.  We ignore these technicalities, since our main interest is in Fact~\ref{firstfact}.}

%On the other hand, there is actually one other bit of ambiguity in the definition.  Consider the following
%knot:
%\begin{figure}[H]
%\begin{center}
%\includegraphics[width=3.5in]
%					{pretzel.jpg}
%%\caption{}
%%\label{reciprocal-example}
%\end{center}
%\end{figure}
%It turns out that the connected sum of this knot can be formed in two different ways, which turn out to yield
%non-equivalent knots: 
%\begin{figure}[H]
%\begin{center}
%\includegraphics[width=4in]
%					{pretzel-sums.jpg}
%%\caption{}
%%\label{reciprocal-example}
%\end{center}
%\end{figure}

Our main interest is in the following fact:
\begin{fact}\label{firstfact}
If $K_1$ and $K_2$ are knots, then $K_1\#K_2$ is an unknot if and only if both $K_1$ and $K_2$
are unknots.
\end{fact}
In other words, two non-unknots can never be added and somehow cancel each other.  There is actually 
an interesting theory here, with knots decomposing uniquely as sums of ``prime knots.'' For more information,
and proofs of \ref{firstfact}, I refer the reader to Colin Adams' \emph{The Knot Book}.

Because of this fact, we can play \textsc{To Knot or Not to Knot} on \emph{sums of rational shadows}, like the following
\begin{figure}[H]
\begin{center}
\includegraphics[width=3in]
					{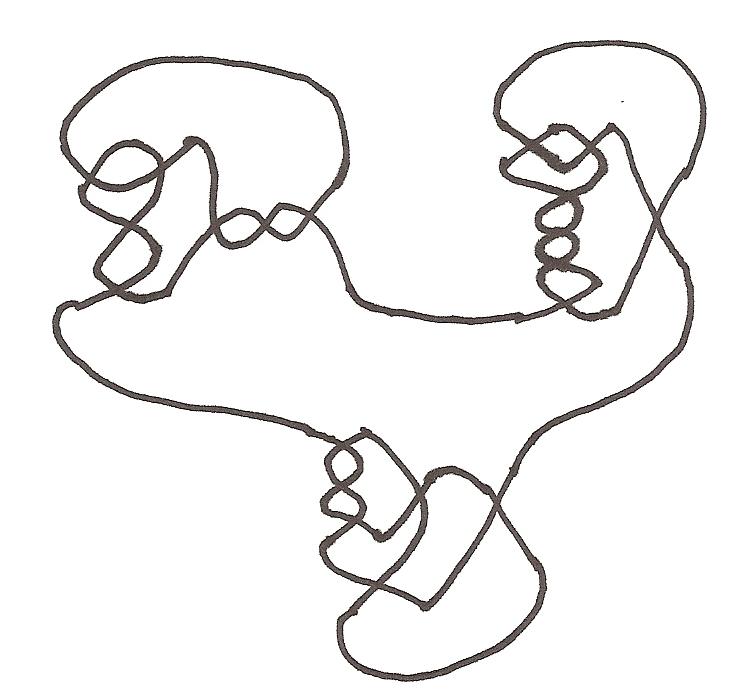}
%\caption{}
%\label{reciprocal-example}
\end{center}
\end{figure}
and actually tell which player wins at the end.  In fact, the winner will be King Lear as long as
he wins in any of the summands, while Ursula needs to win in every summand.
\begin{figure}[H]
\begin{center}
\includegraphics[width=4.5in]
					{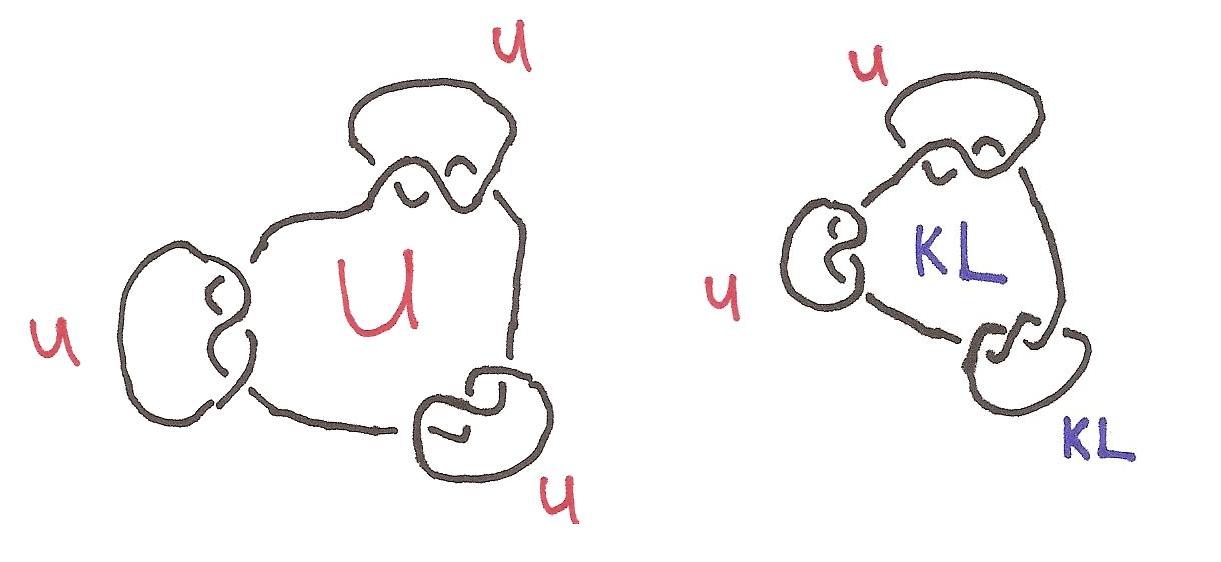}
\caption{On the left, Ursula has won every subgame, so she wins the connected sum.  On the right, King Lear
has won only one subgame, but this is still enough to make the overall figure knotted, so he wins the connected sum.}
\label{knot-sum-logical-or-example}
\end{center}
\end{figure}
Indeed, this holds even when the summands are not rational, though it is harder to tell who wins in that case.

When TKONTK is played on a connected sum of knot shadows, each summand acts as a fully independent game.  There is no interaction
between the components, except that at the end we pool together the results from each component to see who wins (in an asymmetric
way which favors King Lear).  We can visualize each component as a black box, whose output gets fed into a logical OR gate
to decide the final winner:
\begin{figure}[H]
\begin{center}
\includegraphics[width=4.5in]
					{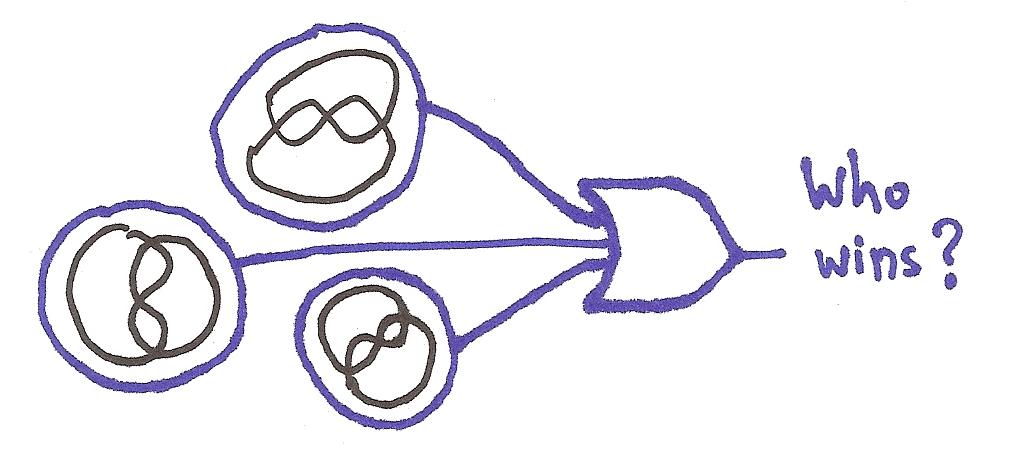}
%\caption{}
%\label{knot-sum-logical-or-example}
\end{center}
\end{figure}

The way in which we can add positions of \textsc{To Knot or Not to Knot} together, or decompose positions as sums of multiple non-interacting
smaller positions, is highly reminiscent of the branch of recreational mathematics known as \emph{combinatorial game theory.}
Perhaps it can be applied to \textsc{To Knot or Not to Knot}?
\begin{quote}\end{quote}
The rest of this work is an attempt to do so.  We begin with an overview
of combinatorial game theory, and then move on to the modifications to the theory that we need to analyze TKONTK.
We proceed by an extremely roundabout route, which may perhaps give better insight into the origins of the final theory.

For completeness we include all the basic proofs of combinatorial game theory, though many of them can be found in
John Conway's book \emph{On Numbers and Games}, and Guy, Berlekamp, and Conway's book \emph{Winning Ways.}  However
\emph{ONAG} is somewhat spotty in terms of content, not covering Norton multiplication or many of the other interesting results
of \emph{Winning Ways}, while \emph{Winning Ways} in turn is generally lacking in proofs.  Moreover, the proofs of basic combinatorial game theory are the basis
for our later proofs of new results, so they are worth understanding.

\part{Combinatorial Game Theory}
\chapter{Introduction}
\section{Combinatorial Game Theory in general}
Combinatorial Game Theory (CGT) is the study of \emph{combinatorial games}.  In the losse sense,
these are \emph{two-player discrete deterministic games of perfect information}:
\begin{itemize}
\item There must be only two players.  This rules out games like Bridge or Risk.
\item The game must be discrete, like Checkers or Bridge, rather than continuous, like Soccer or Fencing.
\item There must be no chance involved, ruling out Poker, Risk, and Candyland.  Instead, the game
must be deterministic.
\item At every stage of the
game, both players have perfect information on the state of the game.
This rules out Stratego and Battleship.  Also, there can be no
simultaneous decisions, as in Rock-Paper-Scissors.  Players must take turns.
\item The game must be zero-sum, in the sense of classical game theory.  One player wins
and the other loses, or the players receive scores that add to zero.
This rules out games like Chicken and Prisoner's Dilemma.
\end{itemize}
While these criteria rule out most popular games, they include Chess, Checkers, Go, Tic-Tac-Toe, Connect Four,
and other abstract strategy games.

By restricting to combinatorial games, CGT distances itself from the classical game theory
developed by von Neuman, Morgenstern, Nash, and others.  Games studied in classical game theory
often model real-world problems like geopolitics, market economics, auctions, criminal justice,
and warfare.  This makes classical game theory a much more practical and empirical subject
that focuses on imperfect information, political coalitions, and various sorts of strategic equilibria.
Classical game theory starts begins its analyses by enumerating strategies for all players.  In
the case of combinatorial games, there are usually 
too many strategies too list, rendering the techniques of classical game theory somewhat
useless.

Given a combinatorial game, we can ask the question: who wins if both players play perfectly?  The answer is called the
\emph{outcome (under perfect play)} of the game.  The underlying goal of combinatorial game theory
is to \emph{solve} various games by determining their outcomes.  Usually
we also want a strategy that the winning player can use to ensure victory.

As a simple example, consider the following game: Alice and Bob sit
on either side of a pile of beans, and alternately take turns removing 1 or 2 beans
from the pile, until the pile is empty.  Whoever removes the last bean wins.

If the players start with 37 beans, and Alice goes first, then she can guarantee
that she wins by always ending her turn in a configuration where the number of beans remaining
is a multiple of three.  This is possible on her first turn because she can remove one bean.
On subsequent turns, she moves in response to Bob, taking one bean if he took two,
and vice versa.  So every two turns, the number of beans remaining decreases by three.
Alice will make the final move to a pile of zero beans, so she is guaranteed the victory.
Because Alice has a perfect winning strategy, Bob has no useful strategies
at all, and so all his strategies are ``optimal,'' because all are equally bad.

On the other hand, if there had been 36 beans originally, and Alice
had played first, then Bob would win by the same strategy, taking one or two beans
in response to Alice taking two or one beans, respectively.  Now Bob will always end his turn
with the number of beans being a multiple of three, so he will be the one to move to the position
with no beans.

The general solution is as follows:
\begin{itemize}
\item If there are $3n + 1$
or $3n + 2$ beans
on the table, then the next player to move will win
under perfect play.
\item If there are $3n$ beans on the table, then the next player
to move will lose under perfect play.
\end{itemize}
Given this solution, Alice or Bob can consider each potential move, and choose the one which
results in the optimal outcome.  In this case, the optimal move is to always move to a multiple of three.
The players can play perfectly as long as they are able to tell the outcome of an arbitrary position
under consideration.

As a general principle, we can say that
\begin{quote}
In a combinatorial game, knowing the outcome (under perfect play) of every position
allows one to play perfectly.
\end{quote}
This works because the players can look ahead one move and choose the move with
the best outcome.  Because of this, the focus of CGT is to determine the outcome
(under perfect play) of positions in arbitrary games.  Henceforth, we assume
that the players are playing perfectly, so that the ``outcome'' always refers
to the outcome under perfect play, and ``Ted wins'' means that Ted has
a strategy guaranteeing a win.

Most games do not admit such simple solutions as the bean-counting game.
As an example of the
complexities that can arise, consider \emph{Wythoff's Game}
%(which may
%also be a traditional Chinese game). apricot
In this game, there
are two piles of beans, and the two players (Alice and Bob) alternately take
turns removing beans.  In this game, a player can remove any number of
beans (more than zero) on her turn, but if she removes beans from both piles,
then she must remove the same number from each pile.  So if it is Alice's turn,
and the two piles have sizes 2 and 1, she can make the following moves: remove
one or two beans from the first pile, remove one bean from the second pile,
or remove one bean from each pile.  Using $(a,b)$ to represent a state
with $a$ beans in one pile and $b$ beans in the other, the legal moves
are to states of the form $(a-k,b)$ where $0 < k \le a$, $(a - k, b - k)$,
where $0 < k \le \min(a,b)$, and $(a,b-k)$, where $0 < k \le b$.  As before,
the winner is the player who removes the last bean.

Equivalently, there is a lone Chess queen on a board, and the players
take turns moving her south, west, or southwest.  The player who moves
her into the bottom left corner is the winner.  Now $(a,b)$ is the queen's grid
coordinates, with the origin in the bottom left corner.

Wythoff showed that the following positions are the ones you should move
to under optimal play - they are the positions for which the \emph{next} player
to move will lose:
\[ \left(\lfloor n \phi \rfloor, \lfloor n \phi^2 \rfloor\right)\]
\textrm{ and }
\[ \left(\lfloor n \phi^2 \rfloor, \lfloor n \phi \rfloor \right)\]
where $\phi$ is the golden ratio $\frac{1 + \sqrt{5}}{2}$ and $n = 0,1,2,\ldots$.
(As an aside, the two sequences $a_n = \left\lfloor n \phi \right\rfloor$ and $b_n = \left\lfloor n \phi^2 \right\rfloor$:
\[ \{a_n\}_{n = 0}^\infty = \{0,1,3,4,6,8,9,\ldots\}\]
\[ \{b_n\}_{n = 0}^\infty = \{0,2,5,7,10,13,15,\ldots\}\]
are examples of \emph{Beatty sequences}, and have several interesting properties.  For example,
$b_n = n + a_n$ for every $n$, and each positive integer occurs in exactly one of the two
sequences.  These facts play a role in the proof of the solution of Wythoff's game.)

Much of combinatorial game theory consists of results of this sort - independent analyses of isolated games.  Consequently, CGT has a tendency to lack overall
coherence. The closest
thing to a unifying framework within CGT is what I will call \emph{Additive Combinatorial Game Theory}\footnote{I thought I heard this name
once but now I can't find it anywhere.  I'll use it anyways. The correct name for this subject may be Conway's combinatorial game theory, or partizan theory,
but these seem to specifically refer to the study of disjunctive sums of partizan games.},
by which I mean the theory begun and extended by Sprague, Grundy, Milnor, Guy, Smith,
Conway, Berlekamp, Norton, and others.  Additive CGT will be
the focus of most of this thesis.\footnote{Computational Complexity Theory
has also been used to prove many negative results.  If we assume the standard conjectures
of computational complexity theory (like P $\ne$ NP), then it is impossible to efficiently evaluate positions of generalized versions
of Gomoku, Hex, Chess, Checkers, Go, Philosopher's Football, Dots-and-Boxes, Hackenbush, and many other games.
Many puzzles are also known to be
intractable if P$\ne$NP.  This subfield of combinatorial game theory is called
\emph{algorithmic} combinatorial game theory.  In a sense it provides another theoretical framework
for CGT.  We will not discuss it further, however.}

%Computational complexity theory has also been the basis of many negative results
%Additive CGT will be the focus of the rest of this thesis, but we mention here
%that there is another theoretical framework which many results about combinatorial games
%build off of, namely Computational Complexity.  The subfield of Algorithmic Combinatorial
%Game Theory consists of applying the techniques of Computational Complexity to the problem
%of solving various games, often producing negative results.  For example,
%we can say without qualification that no efficient (polynomial time) algorithm
%exists for playing (generalized) Chess, Checkers, or Go efficiently,
%so in some sense, attempting to analyze these games is a hopeless endeavour.  And
%assuming the standard conjecture of Computational Complexity, that $P \ne NP$,
%many other games (Hex, Gomoku, Dots and Boxes, Philosopher's Football) and many puzzles
%are also known to be intractable.  Thus Algorithmic CGT provides a complementary negative
%approach to the positive goals of Additive CGT.  On rare occasions, the two theories
%can be used together, as in the proof that Hackenbush is intractable.

\subsection{Bibliography}
The most famous books on CGT are John Conway's \emph{On Numbers and Games},
Conway, Guy, and Berlekamp's four-volume \emph{Winning Ways For your Mathematical Plays}
(referred to as \emph{Winning Ways}), and three collections of articles published by the Mathematical Sciences Research Institute: \emph{Games of No Chance},
\emph{More Games of No Chance}, and \emph{Games of No Chance 3}.  There are also over a thousand
articles in other books and journals, many of which are listed in the bibliographies of the \emph{Games of No Chance} books.

\emph{Winning Ways} is an encyclopedic work: the first
volume covers the core theory of additive CGT, the second covers ways of bending the rules,
and the third and fourth volumes apply these theories to various games and puzzles.
Conway's \emph{ONAG} focuses more closely on the Surreal Numbers (an ordered field
extending the real numbers to also include all the transfinite ordinals),
for the first half of the book, and then considers additive CGT
in the second half.  Due to its earlier publication, the second half
of \emph{ONAG} is generally superseded by the first two volumes of \emph{Winning Ways},
though it tends to give more precise proofs.  The \emph{Games of No Chance} books
are anthologies of papers on diverse topics in the field.

Additionally, there are at least two books applying these theories to specific games:
Berlekamp's \emph{The Dots and Boxes Game: Sophisticated Child's Play} and
Wolfe and Berlekamp's \emph{Mathematical Go: Chilling Gets the Last Point.}
These books focus on Dots-and-Boxes and Go, respectively.
% I just verified these book names.

% apricot History.

\section{Additive CGT specifically}\label{sec:examples}
We begin by introducing a handful of example combinatorial games.

The first is \emph{Nim}, in which there are several piles of counters (as in Figure~\ref{nim-example}),
and players take turns alternately removing pieces until none remain.
A move consists of removing one or more pieces from a signle pile.
The player to remove the last piece wins.  There is no set starting
position.  %As noted above, the goal of CGT is usually to analyze and evaluate every position,
%so there is no real purpose in fixing a starting position.

\begin{figure}[h]
\begin{center}
\includegraphics[width=2in]
					{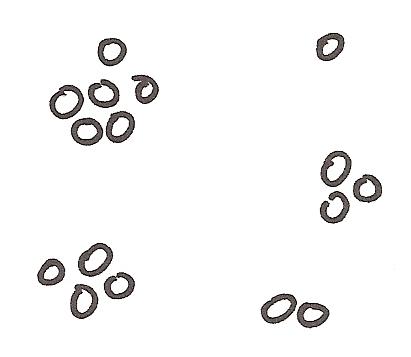}
\caption{A Nim position containing piles of size 6, 1, 3, 2, and 4.}
\label{nim-example}
\end{center}
\end{figure}

A game of \emph{Hackenbush} consists of a drawing made of red and blue % apricot green
edges
connected to each other, and to the ``ground,'' a dotted line at the edge of the world.  See Figure~\ref{hackenbush-example}
for an example.  Roughly, a Hackenbush poition is a graph whose edge have been colored red and blue.

\begin{figure}[H]
\begin{center}
\includegraphics[width=4in]
					{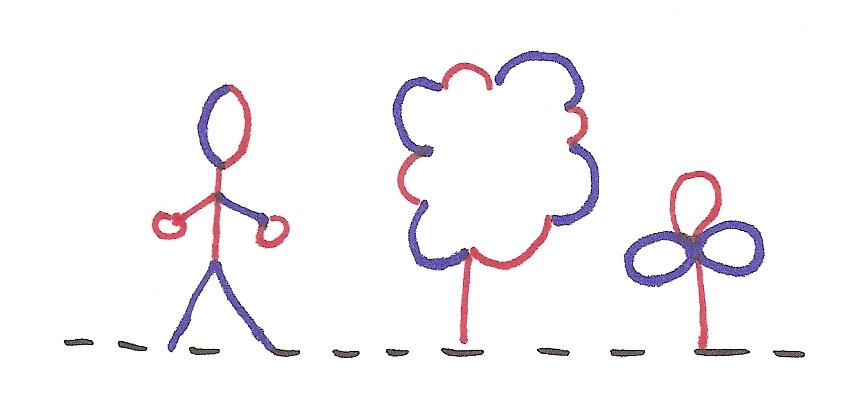}
\caption{A Hackenbush position.}
\label{hackenbush-example}
\end{center}
\end{figure}

On each turn, the current player chooses one edge of his own color, and erases it.  In the process, other
edge may become disconnected from the ground.  \emph{These edges are also erased.}  If the current player
is unable to move, then he loses.  Again, there is no set starting position.
%
%Two players, Red and Blue, take turns erasing edges.  On your turn, you can erase\footnote{This game is
%ideal for chalkboards} one edge
%of your own color. % apricot green
%  But after erasing an edge, some of the remaining
%edges may become disconnected from the ground.  \emph{These edges are also deleted.}
%A small game of Hackenbush is demonstrated in Figure~\ref{hackenbush-moves}.
%Play continues until some player is unable to move, because there are no more edges remaining,
%or all are her opponent's color.  The first player unable to move \emph{loses}.

\begin{figure}[H]
\begin{center}
\includegraphics[width=4in]
					{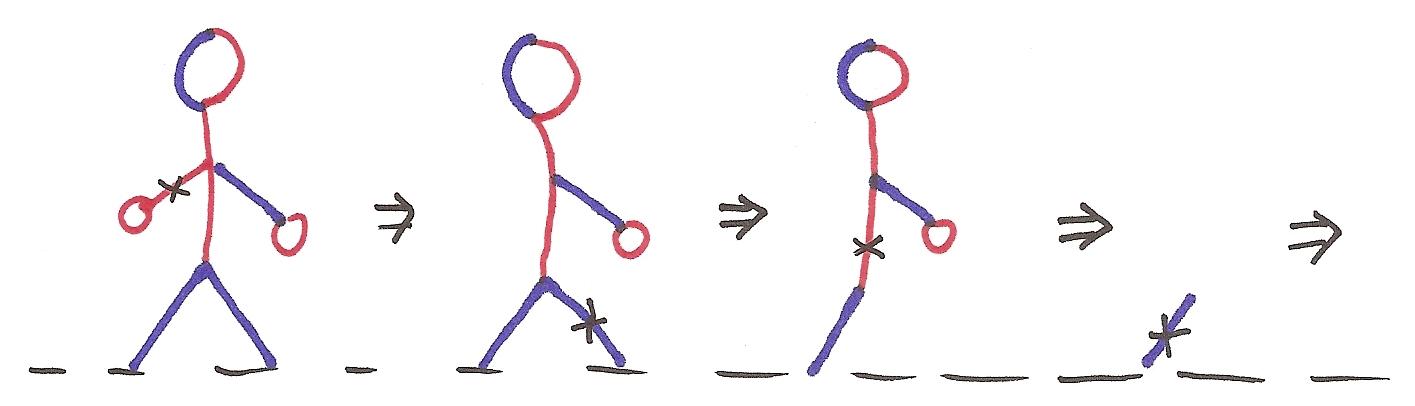}
\caption{Four successive moves in a game of Hackenbush.  Red goes first, and Blue makes the final move of this game.
Whenever an edge is deleted, all the edges that become disconnected from the ground disappear at the same time.}
\label{hackenbush-moves}
\end{center}
\end{figure}

% apricot: reverse this and make it be the explanation of rgb hackenbush.
%The variant in which there are no green edges is called (red-blue) \emph{Hackenbush}.
%Again, there is no set starting configuration.

In \emph{Domineering}, invented by G\"oran Andersson, a game begins with an empty chessboard.
Two players, named Horizontal and Vertical, place dominoes on the board, as in Figure~\ref{domineering-example}.  Each domino
takes up two directly adjacent squares.  Horizontal's dominoes must be aligned
horizontally (East-West), while Vertical's must be aligned vertically (North-South).
Dominoes are not allowed to overlap, so eventually the board fills up.  The first player unable to move
on his turn loses.  %This may not be because the board is completely full,
%because there can be isolated empty spaces, or empty spaces arranged in the wrong configurations.

\begin{figure}[h]
\begin{center}
\includegraphics[width=4in]
					{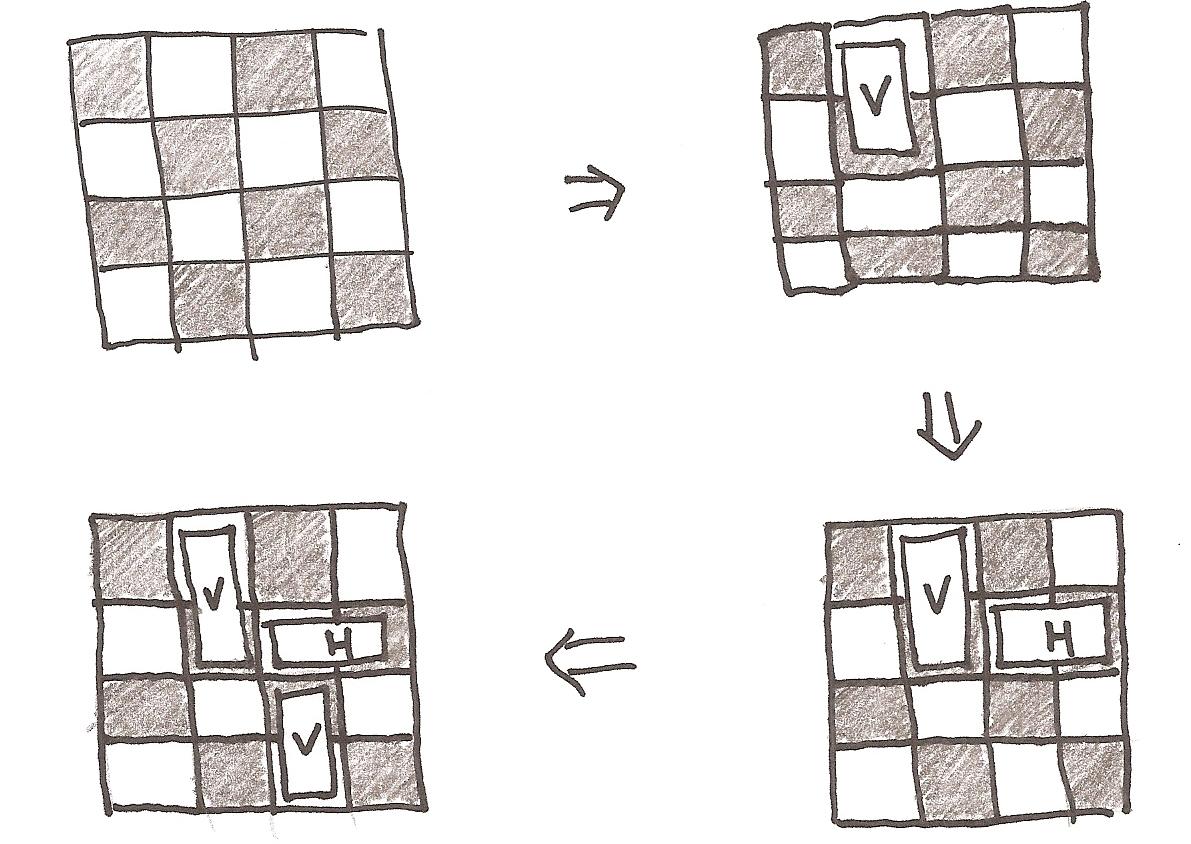}
\caption{Three moves in a game of Domineering.  The board starts empty, and Vertical goes first.}
\label{domineering-example}
\end{center}
\end{figure}

A pencil-and-paper
variant of this game is played on a square grid of dots.  Horizontal draws connects adjacent dots
with horizontal lines, and vertical connects adjacent dots with vertical lines. No dot may have more
than one line out of it.  The reader can easily check that this is equivalent to placing dominoes on a grid of squares.

\emph{Clobber} is another game played on a square grid, covered with White and Black checkers.
Two players, White and Black, alternately move until someone is unable to, and that player loses.
A move consists of moving a piece of your own color onto an immediately adjacent piece of your opponent's
color, which gets removed.  The game of \emph{Konane} (actually an ancient Hawaiian gambling game), is played by the same
rules, except that a move consists of jumping over an opponent's piece and removing it, rather than
moving onto it, as in Figure~\ref{clobber-konane-moves}.  In both games, the board starts out with
the pieces in an alternating checkerboard pattern, except that in Konane two adjacent pieces are removed
from the middle, to provide room for the initial jumps.

\begin{figure}[H]
\begin{center}
\includegraphics[width=4in]
					{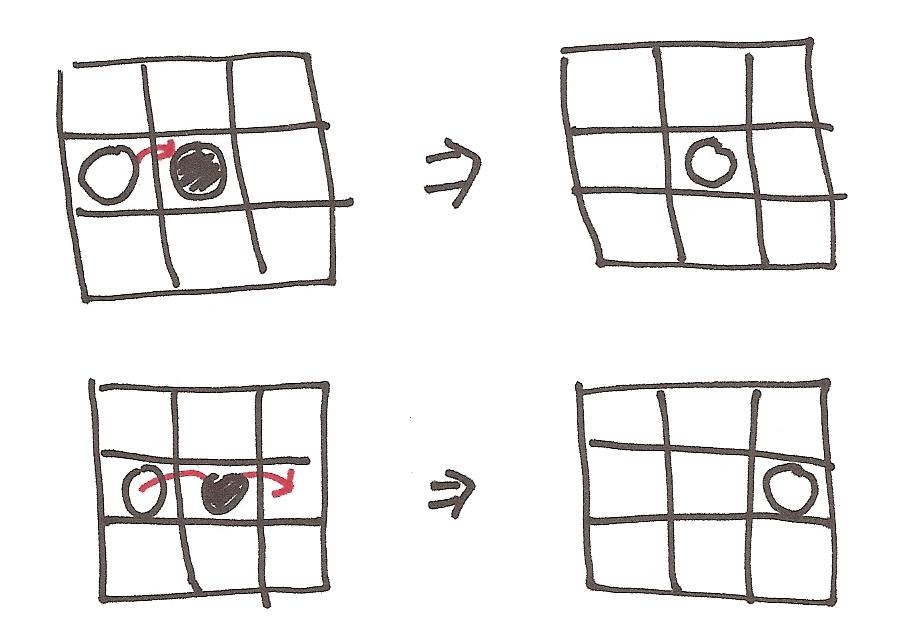}
\caption{Example moves in Clobber and Konane.  In Clobber (top), the capturing piece displaces
the captured piece.  In Konane (bottom), the capturing piece instead jumps over to the captured piece,
to an empty space on the opposite side.  Only vertical and horizontal moves are allowed in both games,
not diagonal.}
\label{clobber-konane-moves}
\end{center}
\end{figure}
%[Pictures would help for all these descriptions.]

The games just described have the following properties in common, in addition to being combinatorial games:
\begin{itemize}
\item \emph{A player loses when and only when he is unable to move.}  This is called the \textbf{normal play convention.}
\item \emph{The games cannot go on forever, and eventually one player wins.}  In every one of our example games,
the number of pieces or parts remaining on the board decreases over time (or in the case of Domineering,
the number of empty spaces decreases.)  Since all these games are finite, this means that
a game can never loop back to a previous position.  These games are all \textbf{loopfree.}
\item \emph{Each game has a tendency to break apart into independent subcomponents.}  This is less obvious but the motivation
for additive CGT.
In the Hackenbush position of Figure~\ref{hackenbush-example}, the stick person, the tree, and the flower each functions
as a completely independent subgame.  In effect, three games of Hackenbush are being played in parallel.

Similarly, in Domineering, as the board begins to fill up,
the remaining empty spaces (which are all that matter from a strategic point of view) will be disconnected
into separate clusters, as in Figure~\ref{domineering-sum}. Each cluster might as well be on a separate board.  So again, we find that
the two players are essentially playing several games in parallel.

In Clobber and Konane, as the pieces disappear they begin to fragment into clusters, as in Figure~\ref{clobber-sundering}.
In Clobber,
once two groups of checkers are disconnected they have no future way of interacting with each other.  So
in Figure~\ref{clobber-subdivision}, each of the red circled regions is an independent subgame.
In Konane, pieces can jump into empty space, so it is possible for groups of checkers to reconnect,
but once there is sufficient separation, it is often possible to prove that such connection is impossible.
Thus Konane splits into independent subgames, like Clobber.

\begin{figure}[p]
\begin{center}
\includegraphics[width=2.5in]
					{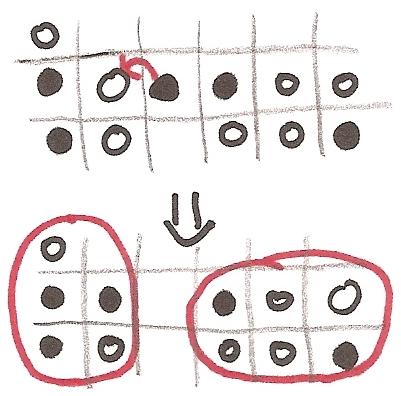}
\caption{Subidivision of Clobber positions: Black's move breaks up the position into a sum of two smaller positions.}
\label{clobber-sundering}
\end{center}
\end{figure}
\begin{figure}[p]
\begin{center}
\includegraphics[width=2.5in]
					{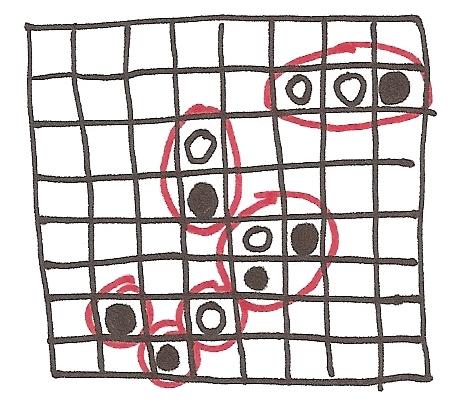}
\caption{A position of Clobber that decomposes as a sum of independent positions.  Each circled area
functions independently from the others.}
\label{clobber-subdivision}
\end{center}
\end{figure}

In Nim, something more subtle happens: each pile is an independent game.  As an isolated position, an individual
pile is not interesting because whoever goes first takes the whole pile and wins.  In combination, however,
nontrivial things occur.

In all these cases, we end up with positions that are \textbf{sums} of other positions.
  In some sense, additive combinatorial game theory is the study of the nontrivial behavior
of sums of game.
\end{itemize}

The core theory of additive CGT, the theory of partizan games,
focuses on \emph{loopfree} combinatorial games played by the  \emph{normal play rule}.
There is no requirement for the games under consideration to decompose as sums,
but unless this occurs, the theory has no a priori reason to be useful.
Very few real games (Chess, Checkers, Go, Hex) meet these requirements,
so Additive CGT has a tendency to focus on obscure games that nobody plays.  Of course,
this is to be expected, since once a game is solved, it loses its appeal as a playable game.

In many cases, however, a game which does not fit these criteria can be analyzed or partially analyzed
by clever applications of the core theory.  For example, Dots-and-Boxes and Go have both been studied
using techniques from the theory of partizan games.  In other cases, the standard rules can be bent,
to yield modified or new theories.  This is the focus of Part 2 of \emph{Winning Ways} and Chapter 14 of \emph{ONAG},
as well as Part II of this thesis.

\section{Counting moves in Hackenbush}
Consider the following Hackenbush position:
\begin{figure}[H]
\begin{center}
\includegraphics[width=1.5in]
					{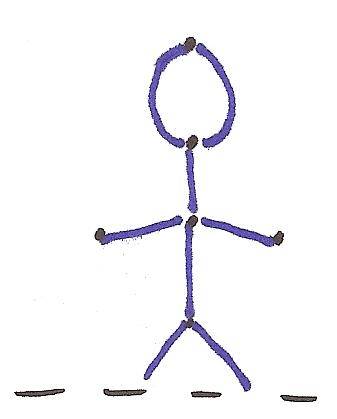}
%\caption{}
%\label{clobber-sundering}
\end{center}
\end{figure}
Since there are only blue edges present, Red has no available moves, so as soon as his turn comes around,
he loses.  On the other hand, Blue has at least one move available, so she will win no matter what.
To make things more interesting, lets give Red some edges:
\begin{figure}[H]
\begin{center}
\includegraphics[width=3in]
					{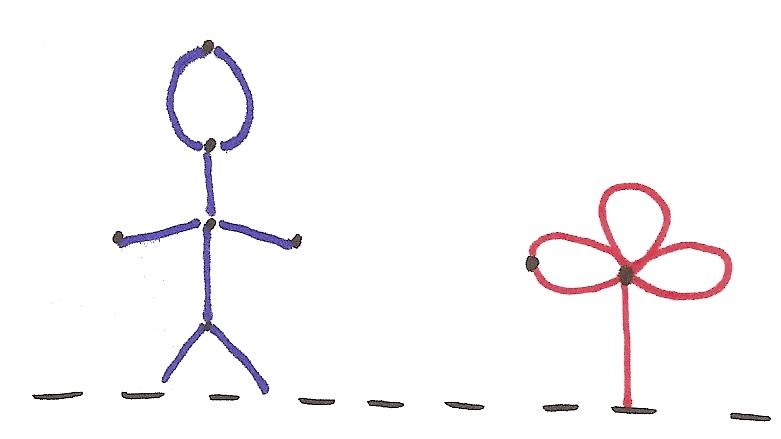}
%\caption{}
%\label{clobber-sundering}
\end{center}
\end{figure}
Now there are 5 red edges and 8 blue edges.  If Red plays wisely, moving on the petals of the flower rather than the stem, he will be able to move 5 times.
However Blue can similarly move  8 times, so Red will run out of moves first and lose, no matter which player moves first.  So again,
Blue is guaranteed a win.

This suggests that we balance the position by giving both players equal numbers of edges:
\begin{figure}[H]
\begin{center}
\includegraphics[width=4in]
					{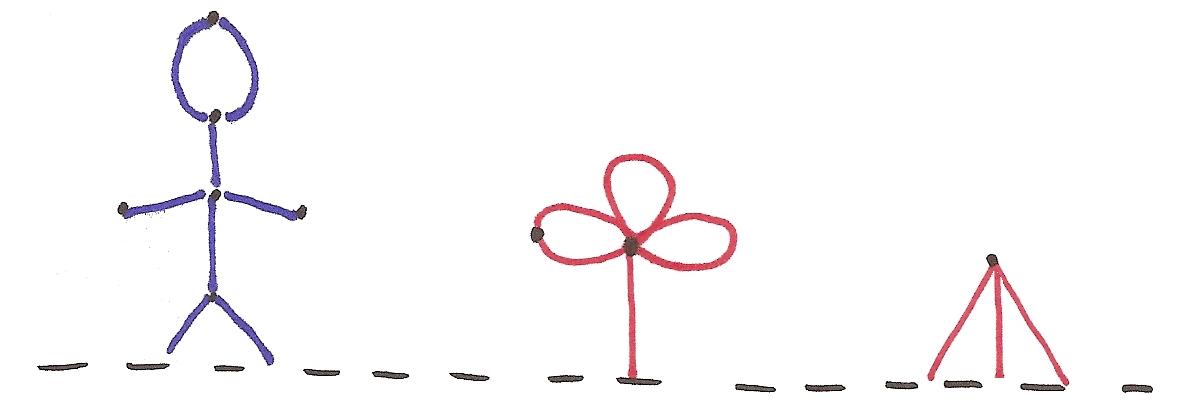}
%\caption{}
%\label{clobber-sundering}
\end{center}
\end{figure}
Now Blue and Red can each move exactly 8 times.  If Blue goes first, then she will run out of moves
first, and therefore lose, but conversely if Red goes first he will lose.  So whoever goes second wins.

In general, if we have a position like
\begin{figure}[H]
\begin{center}
\includegraphics[width=5in]
					{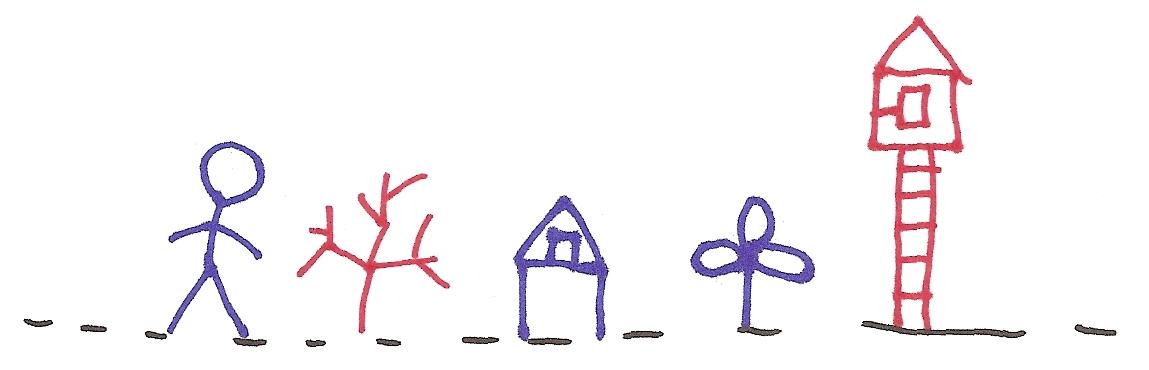}
%\caption{}
%\label{clobber-sundering}
\end{center}
\end{figure}
which is a sum of non-interacting red and blue components, then the player with the greater number of edges will win.
In the case of a tie, whoever goes second wins.  The players are simply seeing how long they can last before
running out of moves.

But what happens if red and blue edges are mixed together, like so?
\begin{figure}[H]
\begin{center}
\includegraphics[width=3.5in]
					{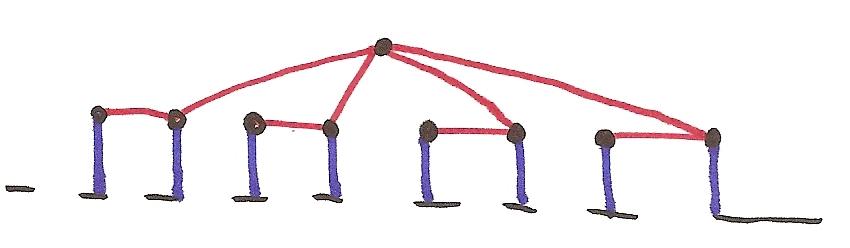}
%\caption{}
%\label{clobber-sundering}
\end{center}
\end{figure}
We claim that Blue can win in this position.  Once all the blue edges are gone,
the red edges must all vanish, because they are disconnected from the ground.  So if Red is able to move at any point,
there must be blue edges remaining in play, and Blue can move too.  Since no move by Red can eliminate blue edges, it follows that after any move by Red,
blue edges will remain and Blue cannot possibly lose.  This demonstrates that the simple rule of counting edges is no longer valid, since both 
players have eight edges but Blue has the advantage.

Let's consider a simpler position:
\begin{figure}[H]
\begin{center}
\includegraphics[width=0.5in]
					{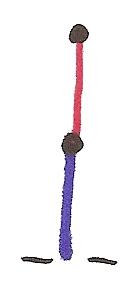}
\caption{How many moves is this worth?}
\label{one-half}
\end{center}
\end{figure}
Now Red and Blue each have 1 edge, but for similar reasons to the previous picture, Blue wins.  How much of an advantage
does Blue have?  Let's add one red edge, giving a 1-move advantage to Red:
\begin{figure}[H]
\begin{center}
\includegraphics[width=1in]
					{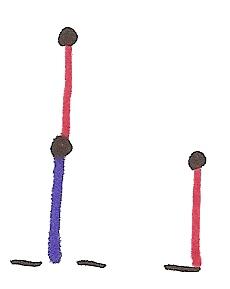}
\caption{Figure~\ref{one-half} plus a red edge.}
\label{one-half-minus-one}
\end{center}
\end{figure}
Now Red is guaranteed a win!  If he moves first, he can move to the following position:
\begin{figure}[H]
\begin{center}
\includegraphics[width=1in]
					{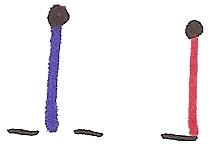}
\caption{A red edge and a blue edge.  This position is balanced, so whoever goes next loses.  This
is a good position to move to.}
\label{one-minus-one}
\end{center}
\end{figure}
which causes the next player (Blue) to lose, and if he moves second, he simply ignores the extra red edge on the left
and treats Figure~\ref{one-half-minus-one} as Figure~\ref{one-minus-one}.

So although Figure~\ref{one-half} is advantageous for Blue, the advantage is worth less than 1 move.  Perhaps Figure~\ref{one-half} is worth half a move for Blue?
We can check this by adding two copies of Figure~\ref{one-half} to a single red edge:
\begin{figure}[H]
\begin{center}
\includegraphics[width=1.3in]
					{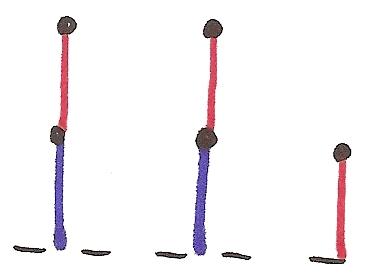}
%\caption{}
%\label{}
\end{center}
\end{figure}
You can easily check that this position is now a balanced second-player win, just like Figure~\ref{one-minus-one}.
So two copies of Figure~\ref{one-half} are worth the same as one red edge, and Figure~\ref{one-half} is worth half a red edge.

In the same way, we can show for
\begin{figure}[H]
\begin{center}
\includegraphics[width=2in]
					{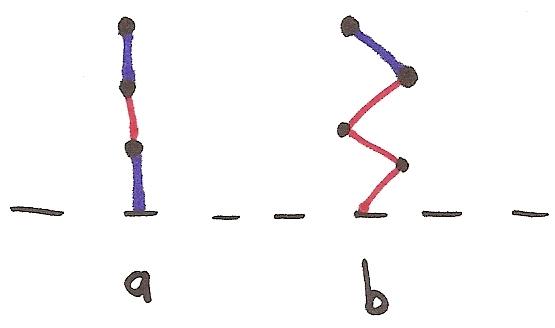}
%\caption{}
%\label{}
\end{center}
\end{figure}
that (a) is worth $3/4$ of a move for Blue, and (b) is worth $2.5$ moves for Red, because the following two positions turn out to be balanced:
\begin{figure}[H]
\begin{center}
\includegraphics[width=3.5in]
					{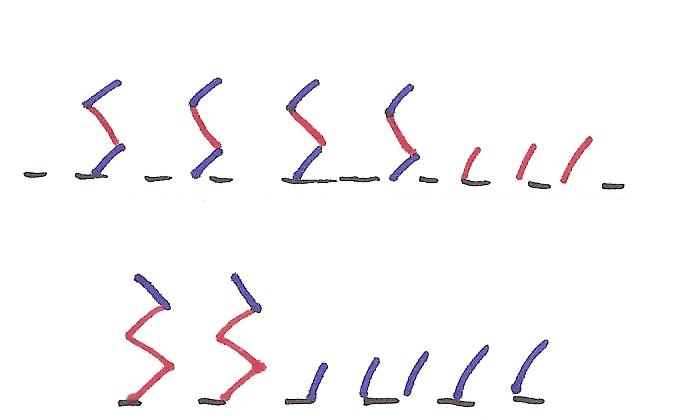}
%\caption{}
%\label{}
\end{center}
\end{figure}
We can combine these values, to see that (a) and (b) together
are worth $2.5 - 3/4 = 7/4$ moves for Red.

The reader is probably wondering why any of these operations are legitimate.  Additive CGT shows that we can assign a rational number
to each Hackenbush position, measuring the advantage of that position to Blue.  The sign of the number
determines the outcome:
\begin{itemize}
\item If positive, then Blue will win no matter who goes first.
\item If negative, then Red will win no matter who goes first.
\item If zero, then whoever goes second will win.
\end{itemize}
And the number assigned to the sum of two positions
is the sum of the numbers assigned to each position.  With games other than Hackenbush, we can assign
values to positions, but the values will no longer be numbers.  Instead they will live in a partially ordered abelian
group called \textbf{Pg}.  The structure of \textbf{Pg} is somewhat complicated, and is one of the focuses
of CGT.

\chapter{Games}
\section{Nonstandard Definitions}
% apricot turn it into something simpler, like game graphs or something
% apricot safe and semi-safe sets
An obvious way to mathematically model a combinatorial game is as a set of positions
with relations to specify how each player can move.  This is not the conventional way of defining
games in combinatorial game theory, but we will use it at first because it is more intuitive in some ways:
\begin{definition}
A \emph{game graph} is a set $S$ of \emph{positions}, a designated \emph{starting position} $\start(S) \in S$,
and two relations $\stackrel{L}{\to}$ and $\stackrel{R}{\to}$
 on $S$.  For any $x \in S$, the $y \in S$
such that $x \stackrel{L}{\to} y$ are called the \emph{left options} of $x$,
and the $y \in S$ such that $x \stackrel{R}{\to} y$ are called the \emph{right options} of $x$.
\end{definition}
For typographical reasons that will become clear in the next section, the two players in additive CGT
are almost always named Left and Right.\footnote{As a rule, B\textbf{l}ue, B\textbf{l}ack, and Vertica\textbf{l} are \textbf{L}eft, while \textbf{R}ed, White,
and Ho\textbf{r}izontal are \textbf{R}ight.  This tells which player is which in our sample partizan games.}
The two relations $\stackrel{L}{\to}$ and $\stackrel{R}{\to}$ are interpreted as follows: $x \stackrel{L}{\to} y$ means that
Left can move from position $x$ to position $y$, and $x \stackrel{R}{\to} y$ means that Right can
move from position $x$ to position $y$.  So if the current position is $x$,
Left can move to any of the left options of $x$, and Right can move to any of the right options of $x$.
We use the shorthand $x \to y$ to denote $x \stackrel{L}{\to} y$ or $x \stackrel{R}{\to} y$.
% apricot TODO: a diagram perhaps?

The game starts out in the position $s_0$.  We intentionally do not specify who will move first.
There is no need for a game graph to specify which player wins at the game's end, because we are using the \emph{normal play} rule:
the first player unable to move loses.

But wait - why should the game ever come to an end?  We need to add an additional condition:
there should be no infinite sequences of play
\[ a_1 \stackrel{L}{\to} a_2 \stackrel{R}{\to} a_3 \stackrel{L}{\to} a_4 \stackrel{R}{\to} \cdots. \]
\begin{definition}
A game graph is \emph{well-founded} or \emph{loopfree} if there are no infinite sequences
of positions $a_1, a_2, \ldots$ such that
\[ a_1 \to a_2 \to a_3 \to \cdots.\]
We also say that the game graph satisfies the \emph{ending condition}.
\end{definition}
This property might seem like overkill: not only does it rule out
\[ a_1 \stackrel{L}{\to} a_2 \stackrel{R}{\to} a_3 \stackrel{L}{\to} a_4 \stackrel{R}{\to} \cdots\]
and
\[ a_1 \stackrel{R}{\to} a_2 \stackrel{L}{\to} a_3 \stackrel{R}{\to} a_4 \stackrel{L}{\to} \cdots\]
but also sequences of play in which the players aren't taking turns correctly, like
\[ a_1 \stackrel{R}{\to} a_2 \stackrel{R}{\to} a_3 \stackrel{R}{\to} a_4 \stackrel{L}{\to} a_5 \stackrel{R}{\to} a_6 \stackrel{L}{\to} \cdots.\]

The ending condition is actually necessary, however, when we play sums of games.
When games are played in parallel,
there is no guarantee that within each component the players will alternate.  If Left and Right
are playing a game $A + B$, Left might move repeatedly in $A$ while Right moved repeatedly in $B$.
Without the full ending condition, the sum of two well-founded games might not be well-founded.
If this is not convincing, the reader can take this claim on faith, and also verify that all of the games
described above are well-founded in this stronger sense.

The terminology ``loopfree'' refers to the fact that, when there are only finitely many positions,
being loopfree is the same as having no cycles $x_1 \to x_2 \to \cdots \to x_n \to x_1$,
because any infinite series would necessarily repeat itself.  In the infinite case,
the term loopfree might not be strictly accurate.

A key fact of well-foundedness, which will be fundamental in everything that follows,
is that it gives us an induction principle
\begin{theorem}
Let $S$ be the set of positions in a well-founded game graph, and let $P$ some subset of $S$.
Suppose that $P$ has the following property: if $x \in S$ and every left and right option
of $x$ is in $P$, then $x \in P$.  Then $P = S$.
\end{theorem}
\begin{proof}
Let $P' = S \setminus P$.  Then by assumption, for every $x \in P'$, there is some
$y \in P'$ such that $x \to y$.  Suppose for the sake of contradiction that
$P'$ is nonempty.  Take $x_1 \in P'$, and find $x_2 \in P'$ such that $x_1 \to x_2.$
Then find $x_3 \in P'$ such that $x_2 \to x_3$.  Repeating this indefinitely
we get an infinite sequence
\[ x_1 \to x_2 \to \cdots\]
contradicting the assumption that our game graph is well-founded.
\end{proof}
To see the similarity with induction, suppose that the set of positions is
$\{1,2,\ldots,n\}$, and $x \to y$ iff $x > y$.  Then this is nothing
but strong induction.

As a first application of this result, we show that in a well-founded game graph,
somebody has a winning strategy.  More precisely, every position in a well-founded game graph can be put into
one of four \emph{outcome classes}:
\begin{itemize}
\item Positions that are wins\footnote{Under optimal play by both players, as usual} for Left, no matter which player moves next.
\item Positions that are wins for Right, no matter which player moves next.
\item Positions that are wins for whichever player moves next.
\item Positions that are wins for whichever player doesn't move next (the previous player).
\end{itemize}
These four possible outcomes are abbreviated as $L$, $R$, $1$ and $2$.
\begin{theorem}
Let $S$ be the set of positions of a well-founded game graph.  Then every position in $S$ falls into
one of the four outcome classes.
\end{theorem}
\begin{proof}
%Given a position, we can ask the question ``who wins if Left goes first?''.  We can also ask the question
%``who wins if Right goes first?''.  Assuming that both questions have an answer (i.e., in both cases
%one of the players actually has a winning strategy), the answers will combine to give the outcome class.\footnote{TODO consider doing a table.}%
%%as shown in Figure~\ref{outcome-classes-diagram}.
%%
%%\begin{figure}[htb]
%%\begin{center}
%%\includegraphics[width=3.5in]
%%					{outcome-classes.jpg}
%%\caption{}
%%\label{outcome-classes-diagram}
%%\end{center}
%%\end{figure}

Let $L_1$ be the set of positions that are wins for Left when she goes first, $R_1$ be the
set of positions that are wins for Right when he goes first, $L_2$ be the set of positions
that are wins for Left when she goes second, and $R_2$ be the set of positions that are wins for Right when
he goes second.  

The reader can easily verify that a position $x$ is in
\begin{itemize}
\item $L_1$ iff some left option is in $L_2$.
\item $R_1$ iff some right option is in $R_2$.
\item $L_2$ iff every right option is in $L_1$.
\item $R_2$ iff every left option is in $R_1$.
\end{itemize}
These rules are slightly subtle, since they implicitly contain the normal play convention, in the case where
$x$ has no options.

If Left goes first from a given position $x$, we want to show that either Left or Right has a winning strategy,
or in other words that $x \in L_1$ or $x \in R_2$.  Similarly, we want to show that every position is in either $R_1$ or $L_2$.
Let $P$ be the set of positions for which $x$ is in exactly one of $L_1$ and $R_2$ and in exactly one of $R_1$ and $L_2$.
By the induction principle, it suffices to show that when all options of $x$ are in $P$,
then $x$ is in $P$.  So suppose all options of $x$ are in $P$.  Then the following
are equivalent:
\begin{itemize}
\item $x \in L_1$
\item some option of $x$ is in $L_2$
\item some option of $x$ is not in $R_1$
\item not every option of $x$ is in $R_1$
\item $x$ is not in $R_2$.
\end{itemize}
Here the equivalence of the second and third line follows from the inductive hypothesis, and
the rest follows from the reader's exercise.  So $x$ is in exactly one of $L_1$ and $R_2$.
A similar argument shows that $x$ is in exactly one of $R_1$ and $L_2$.  So by induction
every position is in $P$.

So every position is in one of $L_1$ and $R_2$, and one of $R_1$ and $L_2$.  This yields four possibilities,
which are the four outcome classes:
\begin{itemize}
\item $L_1 \cap R_1 = 1$.
\item $L_1 \cap L_2 = L$.
\item $R_2 \cap R_1 = R$.
\item $R_2 \cap L_2 = 2$.
\end{itemize}
\end{proof}

\begin{definition}
The \emph{outcome} of a game is the outcome class (1, 2, L, or R) of its starting position.
\end{definition}

Now that we have a theoretical handle on perfect play, we turn towards sums of games.
\begin{definition}
If $S_1$ and $S_2$ are game graphs, we define the \emph{sum} $S_1 + S_2$ to be a game graph
with positions $S_1 \times S_2$ and starting position $\start(S_1 + S_2) = (\start(S_1),\start(S_2))$.
The new $\stackrel{L}{\to}$ relation is defined by
\[ (x,y) \stackrel{L}{\to} (x',y')\]
if $x = x'$ and $y \stackrel{L}{\to} y'$, or $x \stackrel{L}{\to} x'$ and $y = y'$.  The new $\stackrel{R}{\to}$ is defined similarly.
\end{definition}
This definition generalizes in an obvious way to sums of three or more games.  This operation
is essentially associative and commutative, and has as its identity the \emph{zero game}, in which there
is a single position from which neither player can move.

In all of our example games, the sum of two positions can easily be constructed.
In Nim, we simply place the two positions side by side.  In fact this is literally what we do
in each of the games in question.  In Clobber, one needs to make sure that the two positions
aren't touching, and in Konane, the two positions need to be kept a sufficient distance apart.
In Domineering, the focus is on the empty squares, so one needs to ``add'' the gaps together,
again making sure to keep them separated.  And as noted above, such composite sums occur
naturally in the course of each of these games.

\begin{figure}[htb]
\begin{center}
\includegraphics[width=4in]
					{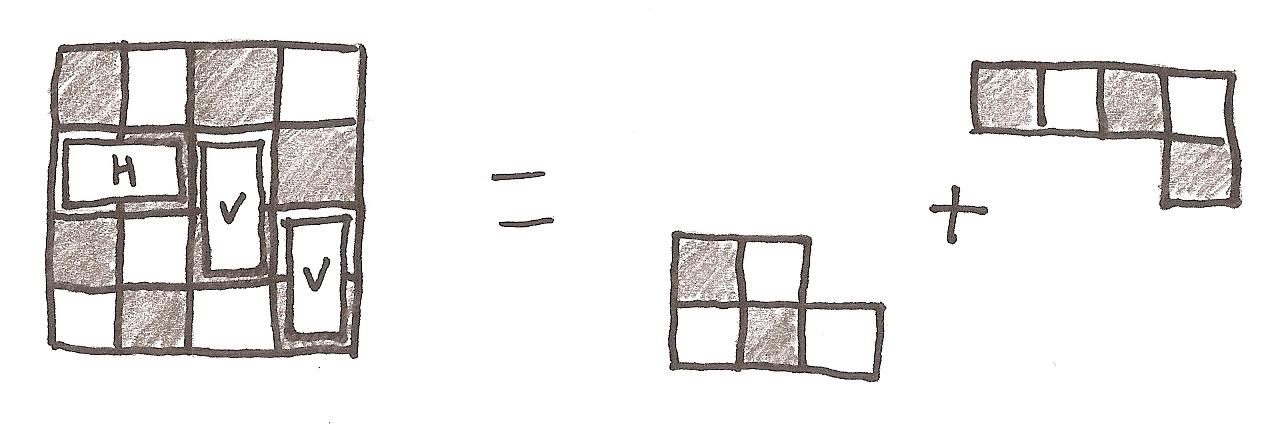}
\caption{The Domineering position on the left decomposes as a sum of the two positions on the right.}
\label{domineering-sum}
\end{center}
\end{figure}

Another major operation that can be performed on games: is \emph{negation}:
\begin{definition}
If $S$ is a game graph, the \emph{negation} $-S$ has the same set of positions,
and the same starting position, but $\stackrel{L}{\to}$ and $\stackrel{R}{\to}$ are interchanged.
\end{definition}
Living up to its name, this operation will turn out to actually produce additive inverses,
modulo Section~\ref{sec:equivalence}.  This operation is easily exhibited in our example games (see Figure~\ref{negation-guises} for examples):
\begin{itemize}
\item In Hackenbush, negation reverses the color of all the red and blue edges.
\item In Domineering, negation corresponds to reflecting the board over a 45 degree line.
\item In Clobber and Konane, it corresponds to changing the color of every piece.
\item Negation has no effect on Nim-positions.  This works because $\stackrel{L}{\to}$
and $\stackrel{R}{\to}$ are the same in any position of Nim.
\end{itemize}
So in general, negation interchanges the roles of the two players.

\begin{figure}[h]
\begin{center}
\includegraphics[width=4.5in]
					{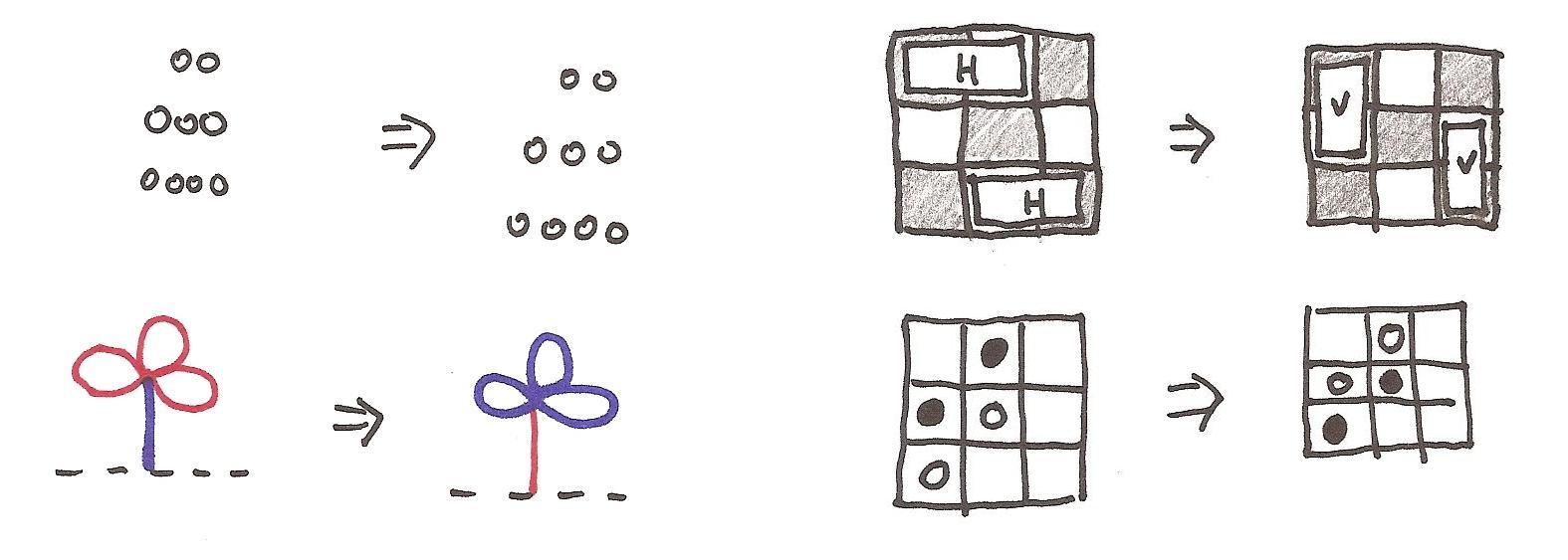}
\caption{Negation in its various guises.}
\label{negation-guises}
\end{center}
\end{figure}

We also define \emph{subtraction} of games by letting
\[ G - H = G + (-H).\]

% apricot: prove the basic results using game graphs / game graphs

\section{The conventional formalism}
While there are no glaring problems with ``game graphs,'' a different convention is used in literature.  We merely included it here
because it is slightly more intuitive than the actual definition we are going to give
later in this section.  And even this defintion will be lacking one last
clarification, namely Section~\ref{sec:equivalence}.

To motivate the standard formalism, we turn to an analogous situation
in set theory: well-ordered sets.

A well-ordered set is a set $S$ with a relation $>$ having the following properties:
\begin{itemize}
\item Irreflexitivity: $a > a$ is never true.
\item Transitivity: if $a > b$ and $b > c$ then $a > c$.
\item Totality: for every $a, b \in S$, either $a > b$, $a = b$, or $a < b$.
\item Well-orderedness: there are no infinite descending chains
\[ x_1 > x_2 > x_3 > \cdots \]
\end{itemize}
These structures are very rigid, and there is a certain
canonical list of well-ordered sets called the \emph{von Neumann ordinals}.  A von Neumann ordinal
is rather opaquely defined as a set $S$ with the property that $S$ and all its members are transitive.  Here
we say that a set is \emph{transitive} if it contains all members of its members.

Given a von Neumann ordinal $S$, we can define a well-ordering on $S$ be letting $x > y$ mean $y \in x$.
Moreover each well-ordered set is isomorphic to a unique von Neuman ordinal.
The von Neumann ordinals themselves are well-ordered by $\in$, and the first few are
\[ 0 = \{\} = \emptyset \]
\[ 1 = \{\{\}\} = \{0\} \]
\[ 2 = \{\{\},\{\{\}\}\} = \{0,1\}\]
\[ 3 = \{\{\},\{\{\}\},\{\{\},\{\{\}\}\}\} = \{0,1,2\}\]
\[ \vdots\]
In general, each von Neumann ordinal is the set of preceding ordinals - for instance,
the first infinite ordinal number is $\omega = \{0,1,2,\ldots\}$.

In some sense, the point of (von Neumann) ordinal numbers is to provide canonical instances
of each isomorphism class of well-ordered sets.  Well-ordered sets are rarely considered
in their own right, because the theory immediately reduces to the theory of
ordinal numbers.  Something similar will happen with games - each isomorphism class
of game graphs will be represented by a single \emph{game}.  This will be made
possible through the magic of well-foundedness.

Analogous to our operations on game graphs, there are certain ways one can combine well-ordered sets.
For instance, if $S$ and $T$ are well-ordered sets, then one can produce (two!) well-orderings
of the disjoint union $S \coprod T$, by putting all the elements of $S$ before (or after) $T$.
And similarly, we can give $S \times T$ a lexicographical ordering, letting $(s_1, t_1) < (s_2, t_2)$ if
$t_1 < t_2$ or ($s_1 < s_2$ and $t_1 = t_2$).  This also turns out to be a well-ordering.

These operations give rise to the following recursively-defined operations on ordinal numbers,
which don't appear entirely related:
\begin{itemize}
\item $\alpha + \beta$ is defined to be $\alpha$ if $\beta = 0$,
the successor of $\alpha + \beta'$ if $\beta$ is the successor of $\beta'$,
and the supremum of $\{\alpha + \beta'\,:\,\beta' < \beta\}$ if $\beta$ is a limit ordinal.
\item $\alpha \beta$ is defined to be $0$ if $\beta = 0$,
defined to be $\alpha \beta' + \alpha$ if $\beta$ is the successor of $\beta'$,
and defined to be the supremum of $\{\alpha \beta'\,:\, \beta' < \beta\}$ if $\beta$ is a limit ordinal.
\end{itemize}
% Note that I checked these definitions (with wikipedia, but that's ok)

In what follows, we will give recursive definitions of ``games,'' and also of their outcomes, sums, and negatives.
These definitions might seem strange, so we invite the reader to check
that they actually come out to the right things, and agree with the definitions given
in the last section.

The following definition is apparently due to John Conway:
\begin{definition}
A \emph{(partizan) game} is an ordered pair $(L,R)$ where $L$ and $R$ are sets of games.
If $L = \{A,B,\ldots\}$ and $R = \{C,D,\ldots\}$, then we write $(L,R)$ as
\[ \{A,B,\ldots|C,D,\ldots\}.\]
The elements of $L$ are called the
\emph{left options} of this game, and the elements of $R$ are called its
\emph{right options}.  The \emph{positions} of a game $G$ are $G$ and all the positions
of the options of $G$.

Following standard conventions in the literature, we will always denote
direct equality between partizan games with $\equiv$, and refer to this relation
as \emph{identity}.\footnote{The reason for this will be explained in Section~\ref{sec:equivalence}.}
\end{definition}
Not only does the definition of ``game'' appear unrelated to combinatorial games, it also seems to be missing
a recursive base case.

The trick is to begin with the empty set, which gives us the following game
\[ 0 \equiv (\emptyset,\emptyset) \equiv \{|\}.\]
Once we have one game, we can make three more:
\[ 1 \equiv \{0|\}\]
\[ -1 \equiv \{|0\}\]
\[ * \equiv \{0|0\}\]
The reason for the numerical names like $0$ and $1$ will become clear later.

In order to avoid a proliferation of brackets, we use $||$ to indicate a higher level of nesting:
\[ \{w,x||y|z\} \equiv \{w,x|\{y|z\}\}\]
\[ \{a|b||c|d\} \equiv \{\{a|b\}|\{c|d\}\}\]

The interpretation of $(L,R)$ is a position whose left options are the elements of $L$
and right options are the elements of $R$.  In particular, this shows us how to associate game graphs with games:
\begin{theorem}
Let $S$ be a well-founded game graph.  Then there is a unique function $f$ assigning a game to
each position of $S$ such that for every $x \in S$, $f(x) \equiv (L,R)$, where
\[ L = \{f(y):x \stackrel{L}{\to} y\}\]
\[R = \{f(y):x \stackrel{R}{\to} y\}\]
In other words, for every position $x$, the left and right options of $f(x)$ are the images
of the left and right options of $x$ under $f$.

Moreover, if we take a partizan game $G$, we can make a game graph $S$ by letting $S$ be the set
of all positions of $G$, $\start(S) = G$, and letting
\[ (L,R) \stackrel{L}{\to} x \iff x \in L\]
\[ (L,R) \stackrel{R}{\to} x \iff x \in R\]
Then the map $f$ sends each element of $S$ to itself.
\end{theorem}
This theorem is a bit like the Mostowski collapse lemma of set theory, and the proof is similar.  Since we will
make no formal use of game graphs, we omit the proof, which mainly consists of set theoretic technicalities.

As an example, for Hackenbush positions we have
\begin{figure}[H]
\begin{center}
\includegraphics[width=5in]
					{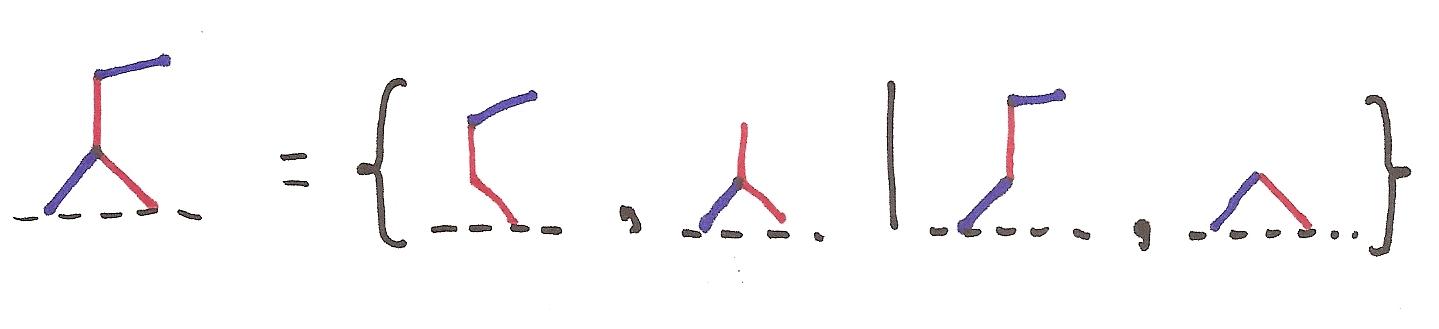}
%\caption{}
%\label{domineering-sum}
\end{center}
\end{figure}
where
\begin{figure}[H]
\begin{center}
\includegraphics[width=4in]
					{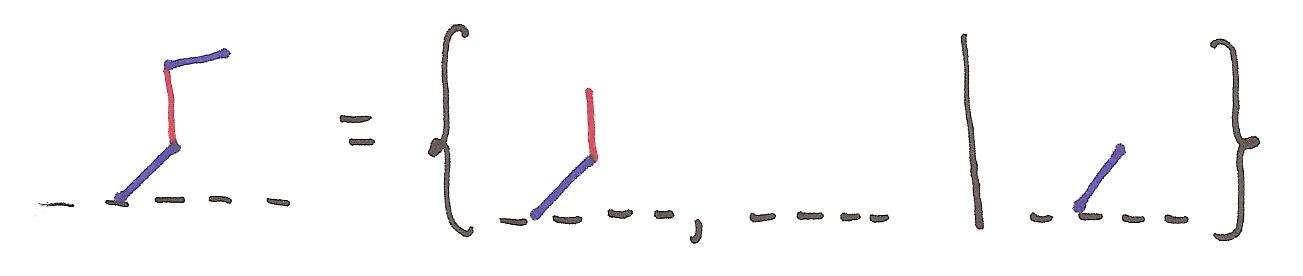}
%\caption{}
%\label{domineering-sum}
\end{center}
\end{figure}
where
\begin{figure}[H]
\begin{center}
\includegraphics[width=2in]
					{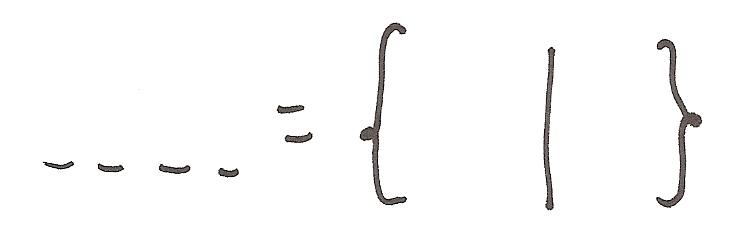}
%\caption{}
%\label{domineering-sum}
\end{center}
\end{figure}
and so on.  Also see Figure~\ref{day-1-examples} for examples of $0$, $1$, $-1$, and $*$
in their various guises in our sample games.

The terms ``game'' and ``position'' are used interchangeably\footnote{Except for the technical sense in which one
game can be a ``position'' of another game.} in the literature, identifying a game with its starting position.
This plays into the philosophy
of evaluating every position and assuming the players are smart enough to look ahead one move.
Then we can focus on outcomes rather than strategies.

%\begin{tabular}{l|l|l}
%blah & blah & blah
%\end{tabular}
\begin{figure}[htb]
\begin{center}
\includegraphics[width=4.5in]
					{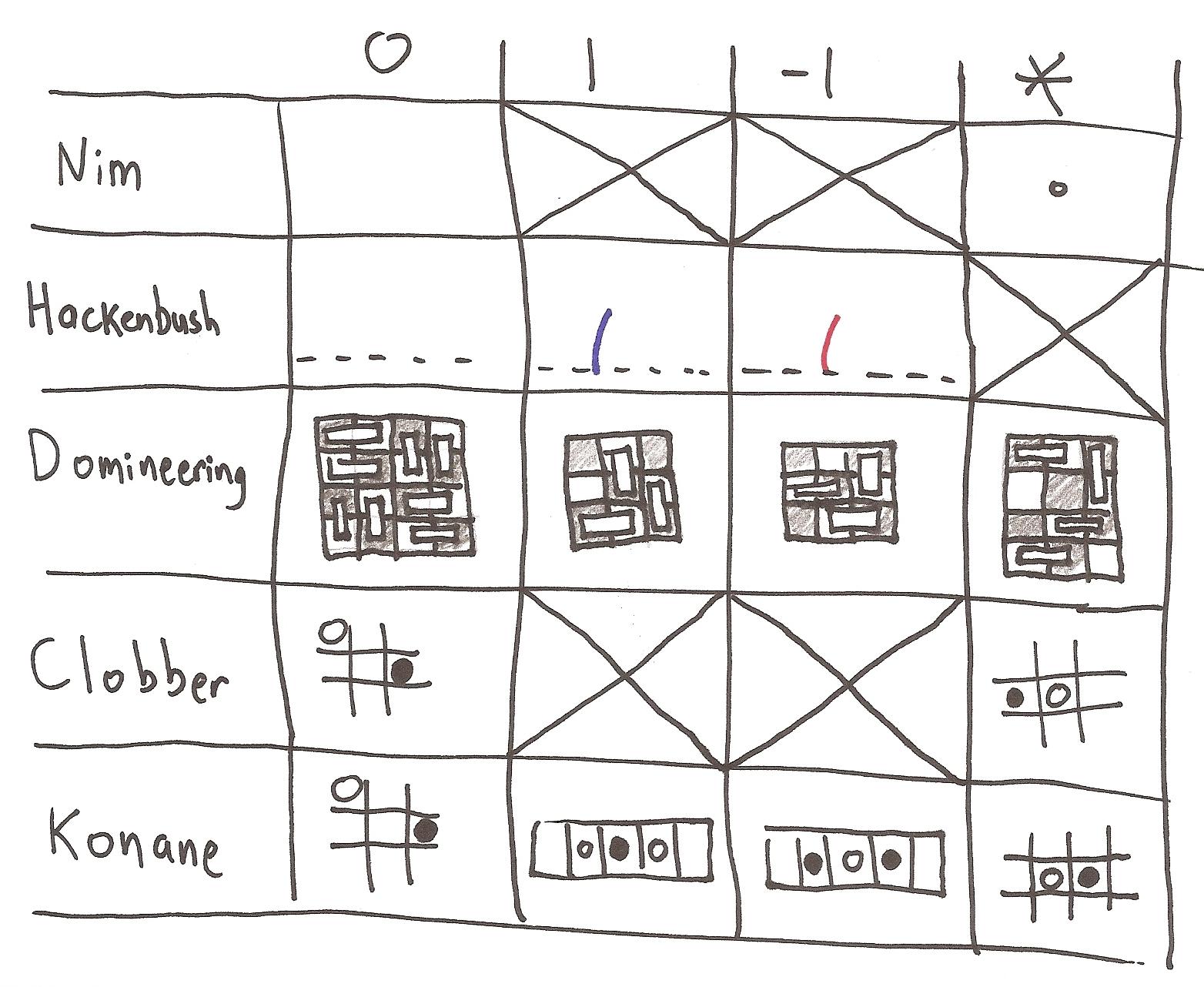}
\caption{The games $0$, $1$, $-1$, and $*$ in their various guises in our sample games.
Some cases do not occur - for instance $*$ cannot occur in Hackenbush, and $1$ cannot occur in Clobber.
We will see why in Sections~\ref{sec:hack-num} and \ref{sec:small-games}.}
\label{day-1-examples}
\end{center}
\end{figure}

Another way to view what's going on is to consider $\{\cdot|\cdot\}$ as an extra operator
for combining games, one that construct a new game with specified left options
and specified right options.

We next define the ``outcome'' of a game, but change notation, to match the standard conventions in the field:
\begin{itemize}
\item $G \ge 0$ means that Left wins when Right goes first.
\item $G \lhd 0$ means that Right wins when Right goes first.
\item $G \le 0$ means that Right wins when Left goes first.
\item $G \rhd 0$ means that Left wins when Right goes first.
\end{itemize}
The $\rhd$ and $\lhd$ are read as ``greater than or fuzzy with'' and ``less than or fuzzy with.''

These are defined recursively and opaquely as:
\begin{definition}
If $G$ is a game, then
\begin{itemize}
\item $G \ge 0$ iff every right option $G^R$ satisfies
$G^L \rhd 0$.
\item $G \le 0$ iff every left option $G^L$ satisfies
$G^R \lhd 0$.
\item $G \rhd 0$ iff some left option $G^L$ satisfies $G^L \ge 0$.
\item $G \lhd 0$ iff some right option $G^R$ satisfies $G^R \le 0$.
\end{itemize}
\end{definition}
One can easily check that exactly one of $G \ge 0$ and $G \lhd 0$ is true,
and exactly one of $G \le 0$ and $G \rhd 0$ is true.

We then define the four outcome classes as follows:
\begin{itemize}
\item $G > 0$ iff Left wins no matter who goes first, i.e., $G \ge 0$ and $G \rhd 0$.
\item $G < 0$ iff Right wins no matter who goes first, i.e., $G \le 0$ and $G \lhd 0$.
\item $G = 0$ iff the second player wins, i.e., $G \le 0$ and $G \ge 0$.
\item $G || 0$ (read $G$ is incomparable or fuzzy with zero) iff the first player wins, i.e., $G \rhd 0$ and $G \lhd 0$.
\end{itemize}
Here is a diagram summarizing the four cases:
\begin{figure}[H]
\begin{center}
\includegraphics[width=3in]
					{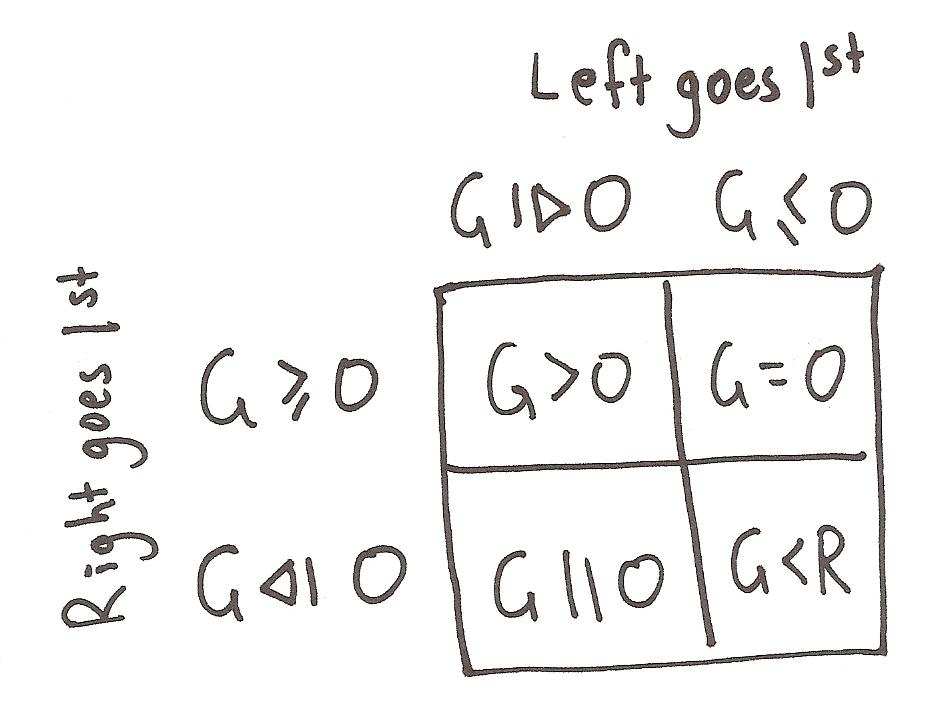}
%\caption{}
%\label{domineering-sum}
\end{center}
\end{figure}
The use of relational symbols like $>$ and $<$ will be justified in the next section.

If you're motivated, you can check that these definitions agree with our definitions
for well-founded game graphs.

As an example, the four games we have defined so far fall into the four classes:
\[ 0 = 0\]
\[ 1 > 0\]
\[ -1 < 0\]
\[ * || 0\]

Next, we define negation:
\begin{definition}
If $G \equiv \{A,B,\ldots|C,D,\ldots\}$ is a game, then its negation $-G$ is recursively
defined as
\[ -G \equiv \{-C,-D, \ldots | -A, -B,\ldots\}\]
\end{definition}
Again, this agrees with the definition for game graphs.  As an example,
we note the negations of the games defined so far
\[ -0 \equiv 0\]
\[ -1 \equiv -1\]
\[ -(-1) \equiv 1\]
\[ -* \equiv *\]
In particular, the notation $-1$ remains legitimate.

Next, we define addition:
\begin{definition}
If $G \equiv \{A,B, \ldots |C,D,\ldots\}$ and
$H \equiv \{E,F,\ldots |X,Y,\ldots\}$ are games, then the \emph{sum} $G + H$ is recursively
defined as
\[ G + H = \{G + E, G + F, \ldots, A + H, B + H, \ldots |\]\[ G + X, G + Y, \ldots, H + C, H + D, \ldots\}\]
\end{definition}
This definition agrees with our definition for game graphs, though it may not be very obvious.  As before, we define subtraction
by
\[ G - H = G + (-H).\]

The usual shorthand for these definitions is
\[ -G \equiv \{-G^R|-G^L\}\]
\[ G + H \equiv \{G^L + H, G + H^L|G^R + H, G + H^R\}\]
\[ G - H \equiv \{G^L - H, G - H^R| G^R - H, G - H^L\}\]
Here $G^L$ and $G^R$ stand for ``generic'' left and right options of $G$, and represent variables ranging over
all left and right options of $G$.  We will make use of this compact and useful notation, which seems to be due to Conway, Guy, and Berlekamp.

We close this section with a list of basic identities satisfied
by the operations defined so far:
\begin{lemma}\label{basic}
If $G$, $H$, and $K$ are games, then
\[ G + H \equiv H + G \]
\[ (G + H) + K \equiv G + (H + K)\]
\[ -(-G) \equiv G \]
\[ G + 0 \equiv G\]
\[ -(G + H) \equiv (-G) + (-H)\]
\[ -0 \equiv 0\]
\end{lemma}
\begin{proof}
All of these are intuitively obvious if you interpret them within the context of Hackenbush, Domineering, or more
abstractly game graphs.  But the rigorous proofs work by induction.  For instance,
to prove $G + H \equiv H + G$, we proceed by joint induction on $G$ and $H$.  Then we have
\[ G + H \equiv \{G^L + H, G + H^L | G^R + H, G + H^R\} \]\[
\equiv \{H + G^L, H^L + G | H + G^R, H^R + G\} \equiv H + G,\]
where the outer identities follow by definition, and the inner one follows by the inductive hypothesis.
These inductive proofs need no base case, because the recursive definition of ``game'' had no base case.
\end{proof}

On the other hand, $G - G \not \equiv 0$ for almost all games $G$.  For instance, we have
\[ 1 - 1 \equiv 1 + (-1) \equiv \{1^L + (-1), 1 + (-1)^L|1^R + (-1), 1 + (-1)^R\} \]\[
\equiv \{0 + (-1) | 1 + 0\} \equiv \{-1|1\}\]
Here there are no $1^R$ or $(-1)^L$, since $1$ has no right options and $-1$ has no left options.

\begin{definition}
A \emph{short game} is a partizan game with finitely many positions.
\end{definition}
\emph{We will assume henceforth that all our games are short.}  Many of the results hold for
general partizan games, but a handful do not, and we have no interest in infinite games.

\section{Relations on Games}\label{sec:equivalence}
So far, we have done nothing but give complicated definitions of simple concepts.
In this section, we begin to look at how our operations for combining
games interact with their outcomes.

Above, we defined $G \ge 0$ to mean that $G$ is a win for Left, when Right moves first.
Similarly, $G \rhd 0$ means that $G$ is a win for Left when \emph{Left} moves first.
From Left's point of view, the positions $\ge 0$ are the good positions to 
move to, and the positions $\rhd 0$ are the ones that Left would like to receive from
her opponent.  In terms of options,
\begin{itemize}
\item $G$ is $ \ge 0$ iff every one of its right option $G^R$ is $\rhd 0$
\item $G$ is $ \rhd 0$ iff at least one of its left option $G^L$ is $\ge 0$.
\end{itemize}

One basic fact about outcomes of sums is that if $G \ge 0$ and
$H \ge 0$, then $G + H \ge 0$.  That is, if Left can win both $G$
and $H$ as the second player, then she can also win $G + H$ as the second player.
She proceeds by combining her strategy in each summand.  Whenever
Right moves in $G$ she replies in $G$, and whenever Right moves in $H$
she replies in $H$.  Such responses are always possible because
of the assumption that $G \ge 0$ and $H \ge 0$.

Similarly, if $G \rhd 0$ and $H \ge 0$, then $G + H \rhd 0$.  Here
left plays first in $G$, moving to a position $G^L \ge 0$,
and then notes that $G^L + H \ge 0$.

Properly speaking, we prove both statements together by induction:
\begin{itemize}
\item If $G, H \ge 0$, then every right option of $G + H$ is of the form
$G^R + H$ or $G + H^R$ by definition.  Since $G$ and $H$ are $\ge 0$,
$G^R$ or $H^R$ will be $\rhd 0$, and so every right option of $G + H$
is the sum of a game $\ge 0$ and a game $\rhd 0$.  By induction,
such a sum will be $\rhd 0$.  So every right option of $G + H$ is $\rhd 0$,
and therefore $G + H \ge 0$.
\item If $G \rhd 0$ and $H \ge 0$, then $G$ has a left option
$G^L \ge 0$.  Then $G^L + H$ is the sum of two games $\ge 0$.  So by induction
$G^L + H \ge 0$.  But it is a left option of $G + H$, so $G + H \rhd 0$.
\end{itemize}

Now one can easily see that
\begin{equation} G \le 0 \iff -G \ge 0 \label{gle}\end{equation}
and
\begin{equation} G \lhd 0 \iff G \rhd 0.\label{elg}\end{equation}
Using these, it similarly follows that if $G \le 0$ and $H \le 0$, then $G + H \le 0$, among other things.

Another result about outcomes is that $G + (-G) \ge 0$.
This is shown using what \emph{Winning Ways }
calls the Tweedledum and Tweedledee Strategy\footnote{The diagram in \emph{Winning Ways} actually looks
like Tweedledum and Tweedledee.}. % I checked the name, this is correct
\begin{figure}[H]
\begin{center}
\includegraphics[width=4in]
					{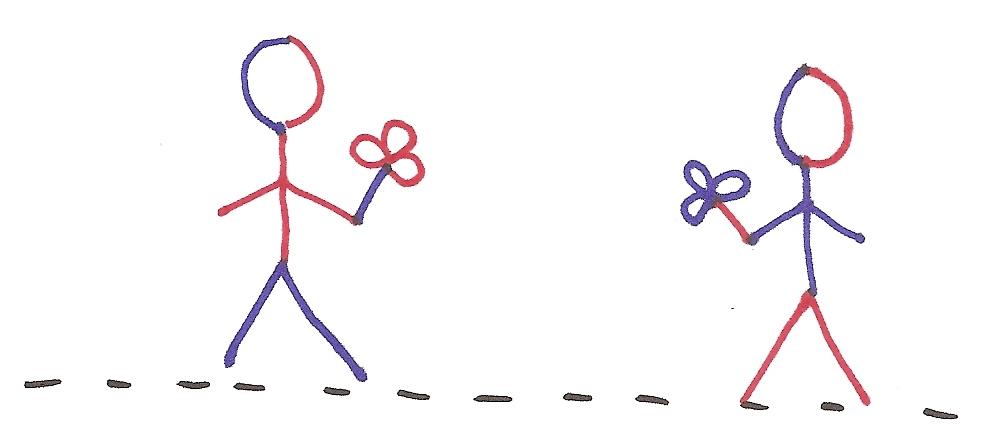}
%\caption{}
%\label{domineering-sum}
\end{center}
\end{figure}
Here we see the sum of a Hackenbush position and its negative.
If Right moves first, then Left can win as follows: whenever Right
moves in the first summand, Left makes the corresponding move in the second
summand, and vice versa.  So if Right initially moves
to $G^R + (-G)$, then Left moves to $G^R + (-(G^R))$, which is possible
because $-(G^R)$ is a left option of $-G$.  On the other hand, if
Right initially moves to $G + (-(G^L))$, then Left responds by moving
to $G^L + (-(G^L))$.  Either way, Left can always move to 
a position of the form $H + (-H)$, for some position $H$ of $G$.

More precisely, we prove the following facts 
\begin{itemize}
\item $G + (-G) \ge 0$
\item $G^R + (-G) \rhd 0$ if $G^R$ is a right option of $G$.
\item $G + (-(G^L)) \rhd 0$ if $G^L$ is a left option of $G$.
\end{itemize}
together jointly by induction:
\begin{itemize}
\item For any game $G$, every Right option of $G + (-G)$ is of
the form $G^R + (-G)$ or $G + (-(G^L))$, where $G^R$ ranges
over right options of $G$ and $G^L$ ranges over left options of $G$.
This follows from the definitions of addition and negation.
By induction all of these options are $\rhd 0$, and so
$G + (-G) \ge 0$.
\item If $G^R$ is a right option of $G$, then $-(G^R)$ is a left option
of $-G$, so $G^R + (-(G^R))$ is a left option
of $G^R + (-G)$.  By induction $G^R + (-(G^R)) \ge 0$,
so $G^R + (-G) \rhd 0$.
\item If $G^L$ is a left option of $G$, then
$G^L + (-(G^L))$ is a left option of $G + (-(G^L))$, and by induction
$G^L + (-(G^L)) \ge 0$.  So $G + (-(G^L)) \rhd 0$.
\end{itemize}

We summarize our results
in the following lemma.
\begin{lemma}\label{outcomesums}
Let $G$ and $H$ be games.
\begin{description}
\item[(a)] If $G \ge 0$ and $H \ge 0$, then $G + H \ge 0$.
\item[(b)] If $G \rhd 0$ and $H \ge 0$, then $G + H \rhd 0$.
\item[(c)] If $G \le 0$ and $H \le 0$, then $G + H \le 0$.
\item[(d)] If $G \lhd 0$ and $H \le 0$, then $G + H \lhd 0$.
\item[(e)] $G + (-G) \le 0$ and $G + (-G) \ge 0$, i.e., $G + (-G) = 0$.
\item[(f)] If $G^L$ is a left option of $G$, then $G^L + (-G) \lhd 0$.
\item[(g)] If $G^R$ is a right option of $G$, then $G^R + (-G) \rhd 0$.
\end{description}
\end{lemma}

These results allow us to say something about \emph{zero games} (not to be confused with the zero game $0$).
\begin{definition}
A game $G$ is a \emph{zero game} if $G = 0$.
\end{definition}
Namely, zero games have no effect on outcomes:
\begin{corollary}
If $H = 0$, then $G + H$ has the same outcome as $G$
for every game $H$.
\end{corollary}
\begin{proof}
Since $H \ge 0$ and $H \le 0$, we have by
part (a) of Lemma~\ref{outcomesums}
\[ G \ge 0 \Rightarrow G + H \ge 0.\]
By part (b)
\[ G \rhd 0 \Rightarrow G + H \rhd 0.\]
By part (c)
\[ G \le 0 \Rightarrow G + H \le 0.\]
By part (d)
\[ G \lhd 0 \Rightarrow G + H \lhd 0.\]
So in every case, $G + H$ has whatever outcome
$G$ has.
\end{proof}

This in some sense justifies the use of the terminology
$H = 0$, since this implies that $G + H$
and $G + 0 \equiv G$ always have the same outcome.

We can generalize this sort of equivalence:
\begin{definition}
If $G$ and $H$ are games, we write $G = H$ (read $G$ equals $H$)
to mean $G - H = 0$.  Similarly, if $\Box$ is any of $||$ $<$, $>$, $\ge$, $\le$, $\lhd$,
or $\rhd$, then we use $G \Box H$ to denote $G - H \Box 0$.
\end{definition}
Note that since $G + (-0) \equiv G$, this notation
does not conflict with our notation for outcomes.
The interpretation of $\ge$ is that $G \ge H$ if
$G$ is at least as good as $H$, from Left's point of view,
or that $H$ is better than $G$, from Right's point of view.
Similarly, $G = H$ should mean that $G$ and $H$
are strategically equivalent.  Further
results will justify these intuitions.

These relations have the properties that one would hope for.
It's clear that $G = H$ iff $G \le H$ and $G \ge H$,
or that $G \ge H$ iff $G > H$ or $G = H$.  Also,
$G \rhd H$ iff $G \not \le H$.  Somewhat less obviously,
\[ G \le H \iff G + (-H) \le 0 \iff -(G + (-H)) \equiv H + (-G) \ge 0 \iff H \ge G\]
and
\[ G \lhd H \iff G + (-H) \lhd 0 \iff -(G + (-H)) \equiv H + (-G) \rhd 0 \iff H \rhd G\]
using equations (\ref{gle}-\ref{elg}).  So we see that
$G = H$ iff $H = G$, i.e., $=$ is symmetric.

Moreover, part (e) of Lemma~\ref{outcomesums} shows that $G = G$, so that =
is reflexive.  In fact,
\begin{lemma}
The relations $=$, $\ge$, $\le$, $>$, and $<$ are transitive.
And if $G \le H$ and $H \lhd K$, then $G \lhd K$.  Similarly
if $G \lhd H$ and $H \le K$, then $G \lhd K$.
\end{lemma}
\begin{proof}
We first show that $\ge$ is transitive.  If $G \ge H$ and
$H \ge K$, then by definition $G + (-H) \ge 0$ and
$H + (-K) \ge 0$.  By part (a) of Lemma~\ref{outcomesums},
\[ (G + (-K)) + (H + (-H)) \equiv (G + (-H)) + (H + (-K)) \ge 0.\]
But by part (e), $H + (-H)$ is a zero game, so we can (by the Corollary),
remove it without effecting the outcome.  Therefore $G + (-K) \ge 0$,
i.e., $G \ge K$.  So $\ge$ is transitive.  Therefore so are
$\le$ and $=$.

Now if $G \le H$ and $H \lhd K$, suppose for the sake of contradiction
that $G \lhd K$ is false.  Then $K \le G \le H$, so $K \le H$, contradicting $H \lhd K$.
A similar argument shows that
if $G \lhd H$ and $H \le K$, then $G \lhd K$.

Finally, suppose that $G < H$ and $H < K$.  Then
$G \le H$ and $H \le K$, so $G \le K$.  But also,
$G \lhd H$ and $H \le K$, so $G \lhd K$.  Together these
imply that $G < K$.  A similar argument shows that $>$ is transitive.
\end{proof}

So we have just shown that $\ge$ is a preorder (a reflexive
and transitive relation), with $=$ as its associated equivalence
relation (i.e., $x = y$ iff $x \ge y$ and $y \ge x$).  So
$\ge$ induces a partial order on the quotient of games modulo $=$.
Because the outcome depends only on a game's comparison to 0,
it follows that if $G = H$ then $G$ and $H$ have the same outcome.

We use $\mathbf{Pg}$ to denote the class of all partizan games, modulo $=$. % apricot: set theory technicalities
This class is also sometimes denoted with \textbf{Ug} (for unimpartial games)
in older books.  We will use $\mathcal{G}$ to denote the class of all short games modulo $=$.  The only article I have seen
which explicitly names the group of short games is David Moews' article \emph{The Abstract Structure of the Group
of Games} in \emph{More Games of No Chance}, which uses the notation \textbf{ShUg}.  This notation is outdated,
however, as it is based on the older \textbf{Ug} rather than \textbf{Pg}.  Both \textbf{ShUg} and \textbf{ShPg} are notationally
ugly, and scattered notation in several other articles suggests we use $\mathcal{G}$ instead.

From now on, we use ``game'' to refer to an element of $\mathbf{Pg}$.

When we need to speak
of our old notion of game, we talk of the ``form'' of a game, as opposed to
its ``value,'' which is the corresponding representative in \textbf{Pg}.
We abuse notation and use $\{A,B,\ldots|C,D,\ldots\}$ to refer to both the form and the corresponding value.

But after making these identifications, can we still use our operations on games,
like sums and negation?
An analogous question arises in the construction of the rationals from the integers.
Usually one defines a rational number to be a pair $\frac{x}{y}$, where
$x, y \in \mathbb{Z}$, $y \ne 0$.  But we identify $\frac{x}{y} = \frac{x'}{y'}$
if $xy' = x'y$.  Now, given a definition like
\[ \frac{x}{y} + \frac{a}{b} = \frac{xb + ay}{yb},\]
we have to verify that the right hand side does not depend on the form we choose to represent
the summands on the left hand side.  Specifically,
we need to show that if $\frac{x}{y} = \frac{x'}{y'}$ and $\frac{a}{b} = \frac{a'}{b'}$,
then
\[ \frac{xb + ay}{yb} = \frac{x'b' + a'y'}{y'b'}.\]
This indeed holds, because 
\[ (xb + ay)(y'b') = (xy')(bb') + (ab')(yy') = (x'y)(bb') + (a'b)(yy') = (x'b' + a'y')(yb).\]

Similarly, we need to show for games that if $G = G'$ and $H = H'$, then $G + H = G' + H'$.
In fact, we have
\begin{theorem}\label{congruences}
\begin{description}
\item[(a)] If $G \ge G'$ and $H \ge H'$, then $G + H \ge G' + H'$.  In particular
if $G = G'$ and $H = H'$, then $G + H = G' + H'$.
\item[(b)] If $G \ge G'$, then $-G' \ge -G$.  In particular
if $G = G'$, then $-G = -G'$.
\item[(c)] If $A \ge A'$, $B \ge B'$,\ldots, then
\[ \{A,B,\ldots|C,D,\ldots\} \ge \{A',B',\ldots|C',D',\ldots\}\]
In particular if $A = A'$, $B = B'$, and so on, then
\[ \{A,B,\ldots|C,D,\ldots\} = \{A',B',\ldots|C',D',\ldots\}\]
\end{description}
\end{theorem}
What this theorem is saying is that whenever we combine games using one of our operations,
the final value depends only on the \emph{values} of the operands, not on their forms.
\begin{proof}
\begin{description}
\item[(a)]
Suppose first that $H = H'$.  Then we need to show that if $G' \ge G$
then $G + H \ge G' + H$, which is straightforward:
\[ G + H \ge G' + H \iff (G + H) + (-(G' + H)) \ge 0 \iff \]\[ (G + (-G')) + (H + (-H)) \ge 0 \iff G + (-G)' \ge 0 \iff G \ge G'.\]
Now in the general case, if $G \ge G'$ and $H \ge H'$ we have
\[ G + H \ge G' + H \equiv H + G' \ge H' + G' \equiv G' + H'.\]
So $G + H \ge G' + H'$.  And if $G = G'$ and $H = H'$, then $G' \ge G$ and $H' \ge H$ so by
what we have just shown,
$G' + H' \ge G + H$.  Taken together, $G' + H' = G + H$.
\item[(b)] Note that
\[ G \ge G' \iff G + (-G') \ge 0 \iff (-G') + (-(-G)) \ge 0 \iff -G' \ge -G.\]
So in particular if $G = G'$, then $G \ge G'$ and $G' \ge G$ so
$-G' \ge -G$ and $-G \ge -G'$.  Thus $-G = -G'$. 
\item[(c)]
We defer the proof of this part until after the proof of Theorem~\ref{betwixt}.
\end{description}
\end{proof}

%We henceforth adopt the notation $G - H$ for $G + (-H)$.  Actually it would
%have made sense to do that a long time ago.
% now done way above. and ycleped subtraction.

Next we relate the partial order to options:
\begin{theorem}\label{betwixt}
If $G$ is a game, then $G^L \lhd G \lhd G^R$ for every left
option $G^L$ and every right option $G^R$.

If $G$ and $H$ are games, then $G \le H$ unless
and only unless there is a right option $H^R$ of $H$
such that $H^R \le G$, or there is
a left option $G^L$ of $G$ such that $H \le G^L$.
\end{theorem}
\begin{proof}
Note that $G^L \lhd G$ iff $G^L - G \lhd 0$, which is part (f)
of Lemma~\ref{outcomesums}.  The proof that $G \lhd G^R$ similarly uses
part (g).

For the second claim, note first that $G \le H$ iff $G - H \le 0$,
which occurs iff every left option of $G - H$ is not $\ge 0$.

But the left options of $G - H$ are of the forms $G^L - H$ and
$G - H^R$, so $G \le H$ iff no $G^L$ or $H^R$ satisfy $G^L \ge H$ or $G \ge H^R$.
\end{proof}

Now we prove part (c) of Theorem~\ref{congruences}:
\begin{proof}
Suppose $A' \ge A$, $B' \ge B$, and so on.
Let
\[ G = \{A,B,\ldots|C,D,\ldots\}\]
and
\[ G' = \{A',B',\ldots|C',D',\ldots\}.\]
Then $G \le G'$ as long as there is no $(G')^R \le G$,
and no $G^L \ge G'$.  That is, we need to check that
\[ C' \not \le G \]
\[ D' \not \le G\]
\[ \vdots\]
\[ G' \not \le A \]
\[ G' \not \le B \]
\[ \vdots\]
But actually these are clear: if $C' \le G$ then because $C \le C'$ we would
have $C \le G$, contradicting $G \lhd G$ by the previous theorem.
Similarly if $G' \le A$, then since $A \le A'$, we would have
$G' \le A'$, rather than $A' \lhd G'$.

The same argument shows that if $A = A'$, $B = B'$, and so on,
then $G' \le G$, so that $G' = G$ in this particular case.
\end{proof}
Using this, we can make substitutions in expressions.  For instance,
if we know that $G = G'$, then we can conclude that
\[ -\{25|13, (* + G)\} = -\{25|13, (* + G')\}\]

\begin{definition}
A \emph{partially-ordered abelian group} is an abelian group $G$
with a partial order $\le$ such that
\[ x \le y \implies x + z \le y + z\]
for every $z$.
\end{definition}
All the expected algebraic
facts hold for partially-ordered abelian groups.  For instance,
\[ x \le y \iff x + z \le y + z\]
(the $\Leftarrow$ direction follows by negating $z$),
and
\[ x \le y \iff -y \le -x\]
and the elements $\ge 0$ are closed under addition, and so on.

With this notion we summarize all the results so far:
\begin{theorem}\label{everything}
The class $\mathcal{G}$ of (short) games modulo equality is a partially ordered abelian group, with addition
given by addition of games, identity given by the game $0 = \{|\}$, and additive
inverses given by negation.  The outcome of a game $G$ is determined by its comparison
to zero:
\begin{itemize}
\item If $G = 0$, then $G$ is a win for whichever player moves second.
\item If $G || 0$, then $G$ is a win for whichever player moves first.
\item If $G > 0$, then $G$ is a win for Left either way.
\item If $G < 0$, then $G$ is a win for Right either way.
\end{itemize}
Also, if $A,B,C, \ldots \in \mathcal{G}$, then we can meaningfully talk about
\[ \{A,B,\ldots|C,D,\ldots\} \in \mathcal{G}\]
This gives a well-defined map
\[ \mathcal{P}_f(\mathcal{G}) \times \mathcal{P}_f(\mathcal{G}) \to \mathcal{G}\]
where $\mathcal{P}_f(S)$ is the set of all finite subsets of $S$.
Moreover, if $G = \{G^L|G^R\}$ and $H = \{H^L|H^R\}$, then $G \le H$ unlesss
$H^R \le G$ for some $H^R$, or $H \le G^L$ for some $G^L$.  Also,
$G^L \lhd G \lhd G^R$ for every $G^L$ and $G^R$.
\end{theorem}

\section{Simplifying Games}
Now that we have an equivalence relation on games, we seek a canonical representative
of each class.  We first show that removing a left (or right) option of a game doesn't help Left (or Right).
\begin{theorem}
If $G = \{A,B,\ldots|C,D,\ldots\}$, then
\[ G' = \{B,\ldots|C,D,\ldots\} \le G.\]
Similarly,
\[ G \le \{A,B,\ldots|D,\ldots\}.\]
\end{theorem}
\begin{proof}  We use Theorem~\ref{betwixt}.
To see $G' \le G$, it suffices to show that $G$ is not $\le$ any left option of $G'$ (which is obvious,
since every left option of $G'$ is a left option of $G$), and that $G'$ is not $\ge$ any right 
option of $G$, which is again obvious since every right option of $G$ is a right option of $G'$.

The other claim is proven similarly.
\end{proof}

On the other hand, sometimes options can be added/removed without affecting the value:
\begin{theorem}(Gift-horse principle)
If $G = \{A,B,\ldots|C,D,\ldots\}$, and $X \lhd G$, then
\[ G = \{X,A,B,\ldots|C,D,\ldots\}.\]
Similarly if $Y \rhd G$, then
\[ G = \{A,B,\ldots|C,D,\ldots,Y\}.\]
\end{theorem}
\begin{proof}
We prove the first claim because the other is similar.
From the previous theorem we already know that $G' = \{X,A,B,\ldots|C,D,\ldots\}$ is $\ge G$.
So it remains to show that $G' \le G$.  To see this, it suffices by Theorem~\ref{betwixt} to show that
\begin{itemize}
\item $G$ is not $\le$ any left option of $G'$: obvious since every left option of $G'$ is a left option
of $G$, except for $X$, but $G \not\le X$ by assumption.
\item $G'$ is not $\ge$ any right option of $G$: obvious since every right option of $G$ is a right
option of $G'$.
\end{itemize}
\end{proof}
%Such an option is called a Gift-Horse for some reason.

\begin{definition}
Let $G$ be a (form of a) game.  Then a left
option $G^L$ is \emph{dominated} if there is some other left option
$(G^L)'$ such that $G^L \le (G^L)'$.  Similarly, a
right option $G^R$ is \emph{dominated} if there is some
other right option $(G^R)'$ such that $G^R \ge (G^R)'$.
\end{definition}
That is, an option is dominated when its player has a better alternative.
The point of dominated options is that they are useless and can be removed:
\begin{theorem}(dominated moves).  If $A \le B$, then
\[ \{A,B,\ldots|C,D,\ldots\} = \{B,\ldots|C,D,\ldots\}.\]
Similarly, if $D \le C$ then
\[ \{A,B,\ldots|C,D,\ldots\} = \{A,B,\ldots|D,\ldots\}.\]
\end{theorem}
\begin{proof}
We prove the first claim (the other follows by symmetry).
\[ A \le B \lhd \{B,\ldots|C,D,\ldots\}.\]
So therefore $A \lhd \{B,\ldots|C,D,\ldots\}$ and we are done by the gift-horse
principle.
\end{proof}

\begin{definition}
If $G$ is a game, $G^L$ is a left option of $G$, and $G^{LR}$ is a right 
option of $G^L$ such that $G^{LR} \le G$, then we say that $G^L$
is a \emph{reversible option}, which is \emph{reversed through} % this is the only terminology for this that I found (check terminology)
its option $G^{LR}$.

Similarly, if $G^R$ is a right option, having a left option $G^{RL}$ with
$G^{RL} \ge G$, then $G^R$ is also a reversible option, reversed through
$G^{RL}$.
\end{definition}
A move from $G$ to $H$ is reversible when the opponent can ``undo'' it with
a subsequent move.  It turns out that a player
might as well always make such a reversing move.
\begin{theorem}(reversible moves)
If $G = \{A,B,\ldots|C,D,\ldots\}$ is a game, and $A$ is a reversible
left option, reversed through $A^R$, then
\[ G = \{X,Y,Z,\ldots,B,\ldots|C,D,\ldots\}\]
where $X,Y,Z,\ldots$ are the left options of $A^R$.

Similarly, if $C$ is a reversible move, reversed through $C^L$, then
\[ G = \{A,B,\ldots|D,\ldots,X,Y,Z,\ldots\}\]
where $X,Y,Z,\ldots$ are the right options of $C^L$.
\end{theorem}
\begin{proof}
We prove the first claim because the other follows by symmetry.

Let $G'$ be the game \[\{X,Y,Z,\ldots,B,\ldots|C,D,\ldots\}.\]
We need to show $G' \le G$ and $G \le G$'.

First of all, $G'$ will be $\le G$ unless $G'$ is $\ge$ a right option of $G$
(impossible, since all right options of $G$ are right options of $G'$),
or $G$ is $\le$ a right option of $G'$.  Clearly $G$ cannot be
$\le B,\ldots$ because those are already right options of $G$.
So suppose that $G$ is $\le$ a right option of $A^R$, say $X$.
Then
\[ G \le X \lhd A^R \le G,\]
so that $G \lhd G$, an impossibility.  Thus $G' \le G$.

Second, $G$ will be $\le G'$ unless $G$ is $\ge$ a right option
of $G'$ (impossible, because every right option of $G'$ is
a right option of $G$), or $G'$ is $\le$ a left option of
$G$.  Now every left option of $G$ aside from $A$ is
a left option of $G'$ already, so it remains to show that
$G' \not\le A$.

This follows if we show that $A^R \le G'$.
Now $G'$ cannot be $\le$ any left option of $A^R$, because
every left option of $A^R$ is also a left option of $G'$.  So
it remains to show that $A^R$ is not $\ge$ any right option
of $G'$.  But if $A^R$ was $\ge$ say $C$, then
\[ A^R \ge C \rhd G \ge A^R,\]
so that $A^R \rhd A^R$, a contradiction.
\end{proof}
The game $\{X,Y,Z,\ldots,B,\ldots|C,D,\ldots\}$ is called the game
obtained by \emph{bypassing the reversible move $A$}. % WW terminology

%\begin{definition}
%A \emph{short game} is a game with finitely many subpositions.
%\end{definition}
%An equally good name would be ``finite game.''
%By K\"onig's Lemma or something similar, this is the same as
%having only finitely many options in every subposition.
%All the games that we'll be looking at are short games.
%The class of short games is countable, and denoted
%sometimes by \textbf{ShUg} or I suppose \textbf{ShPg} or something?

The key result is that for short games, there is a canonical
representative in each equivalence class:
\begin{definition}
A game $G$ is in \emph{canonical form} if every position of $G$ has no dominated or reversible moves.
\end{definition}
\begin{theorem}
If $G$ is a game, there is a unique canonical form equal to $G$,
and it is the unique smallest game equivalent to $G$, measuring size
by the number of edges in the game tree of $G$.\footnote{The number of edges in the game tree of $G$
can be defined recursively as the number of options of $G$ plus the sum of the number of edges in the game trees
of each option of $G$.  So $0$ has no edges, $1$ and $-1$ have one each, and $*$ has two.  A game like $\{*,1|-1\}$ then
has $2 + 1 + 1$ plus three, or seven total.  It's canonical form is $\{1|-1\}$ which has only four.}
\end{theorem}
\begin{proof}
Existence: if $G$ has some dominated moves, remove them.  If it has
reversible moves, bypass them.  These operations may introduce new dominated
and reversible moves, so continue doing this.  Do the same thing in all
subpositions.  The process cannot go on forever
because removing a dominated move or bypassing a reversible move always
strictly decreases the total number of edges in the game tree.  So at least one canonical
form exists.

Uniqueness: Suppose that $G$ and $H$ are two equal games in
canonical form.  Then because $G - H = 0$, we know that
every right option of $G - H$ is $\rhd 0$.  In particular,
for every right option $G^R$ of $G$, $G^R - H \rhd 0$,
so there is a left option of $G^R - H$ which is $\ge 0$.
This option will either be of the form $G^{RL} - H$
or $G^R - H^R$ (because of the minus sign).  But since
$G$ is assumed to be in canonical form, it has no reversible
moves, so $G^{RL} \not\ge G = H$.  Therefore $G^{RL} - H \not \ge 0$.
So there must be some $H^R$ such that $G^R \ge H^R$.

In other words, if $G$ and $H$ are both in canonical form,
and if they equal each other, then for every right option
$G^R$ of $G$, there is a right option $H^R$ of
$H$ such that $G^R \ge H^R$.  Of course we can apply the
same logic in the other direction, to $H^R$, and produce
another right option $(G^R)'$ of $G$, such that
\[ G^R \ge H^R \ge (G^R)'.\]
But since $G$ has no dominated moves, we must have
$G^R = (G^R)'$, and so $G^R = H^R$.  In fact, by induction,
we even have $G^R \equiv H^R$.

So every right
option of $G$ occurs as a right option of $H$.
Of course the same thing holds in the other direction,
so the set of right options of $G$ and $H$ must be equal.
Similarly the set of left options will be equal too.  Therefore
$G \equiv H$.

Minimality: If $G = H$, then $G$ and $H$
can both be reduced to canonical form, and by uniqueness
the canonical forms must be identical.  So if $H$ is of minimal size
in its equivalence class, then it cannot be made any smaller and
must equal the canonical form.  So any game of
minimal size in its equivalence class is identical to the unique canonical form.
\end{proof}

\section{Some Examples}
% apricot priority: put in examples of domineering, clobber positions

Let's demonstrate some of these ideas with the simplest four games:
\[ 0 \equiv \{|\}\]
\[ 1 \equiv \{0|\}\]
\[ -1  \equiv \{|0\}\]
\[ * \equiv \{0|0\}.\]
Each of these games is already in canonical form, because there can
be no dominated moves (as no game has more than two options on either side),
nor reversible moves (because every option is $0$, and $0$ has no options itself).

Let's try adding some games together:
\[ 1 + 1 \equiv \{1^L + 1, 1 + 1^L|1^R + 1, 1 + 1^R\} \equiv \{0+1|\} \equiv \{1|\}\]
This game is called $2$, and is again in canonical form, because the move to $1$
is not reversible (as $1$ has no right option!).

On the other hand, sometimes games become simpler when added together.  We already know
that $G - G = 0$ for any $G$, and here is an example:
\[ 1 + (-1) \equiv \{1^L + (-1), 1 + (-1)^L|1^R + (-1), 1 + (-1)^R\}
 \equiv \{-1|1\}\]
since no $1^R$ or $(-1)^L$ exist, and the only $1^L$ or $(-1)^R$ is $0$.
Now $\{-1|1\}$ is \emph{not} in canonical form, because the moves are reversible.
If Left moves to $-1$, then Right can reply with a move to $0$, which is
$\le \{-1|1\}$ (since we know $\{-1|1\}$ actually \emph{is} zero).
Similarly, the right option to $1$ is also reversible.  This yields
\[ \{-1|1\} = \{0^L|1\} \equiv \{|1\} = \cdots =  \{|\} = 0.\]

Likewise, $*$ is its own negative, and indeed we have
\[ * + * \equiv \{*^L + *|*^R + *\} \equiv \{*|*\}\]
which reduces to $\{|\}$ because $*$ on either side is reversed by a move to $0$.

For an example of a dominated move that appears, consider $1 + *$
\[ 1 + * \equiv \{1^L + *, 1 + *^L|1^R + *, 1 + *^R\} \equiv \{0 + *, 1 + 0| 1 + 0\} \equiv\{*,1|1\}.\]
Now it is easy to show that $* = \{0|0\} \le \{0|\} = 1$, (in fact, this follows because
$1$ is obtained from $*$ by deleting a right option), so $*$ is dominated and we actually have
\[ 1 + * = \{1|1\}.\]
Note that $1 + * \ge 0$, even though $* \not \ge 0$.  We will see that $*$ is an ``infinitesimal''
or ``small'' game which is insignificant in comparison to any number, such as $1$.

\chapter{Surreal Numbers}
\section{Surreal Numbers}
One of the more surprising parts of CGT is the manifestation of the numbers.

\begin{definition}
A \emph{(surreal) number} is a partizan game $x$ such that every option of $x$
is a number, and every left option of $x$ is $\lhd$ every right
option of $x$.
\end{definition}
Note that this is a recursive definition.  Unfortunately, it is not compatible
with equality: by the definition just
given, $\{*,1|\}$ is not a surreal number (since $*$ is not), but $\{1|\}$ is, even though
$\{1|\} = \{*,1|\}$.  But we can at least say that
if $G$ is a short game that is a surreal number, then its canonical form
is also a surreal number.  In general, we consider a game to be a surreal number
if it has some form which is a surreal number.

Some simple examples of surreal numbers are
\[ 0 = \{|\}\]
\[ 1 = \{0|\}\]
\[ -1 = \{|0\}\]
\[ \frac{1}{2} = \{0|1\}\]
\[ 2 = 1 + 1 = \{1|\}.\]
Explanation for these names will appear quickly.  But first, we prove
some basic facts about surreal numbers.

\begin{theorem}\label{twolessthans}
If $x$ is a surreal number, then $x^L < x < x^R$ for every left option
$x^L$ and every right option $x^R$.
\end{theorem}
\begin{proof}
We already know that $x^L \lhd x$.  Suppose for the sake of contradiction
that $x^L \not \le x$.  Then either
$x^L$ is $\ge$ some $x^R$ (directly contradicting the definition of surreal number),
or $x \le x^{LL}$ for some left option $x^{LL}$ of $x^L$.  Now by induction,
we can assume that $x^{LL} < x^L$, so it would follow that
$x \le x^{LL} < x^L$, and so $x \le x^L$, contradicting the fact that
$x^L \lhd x$.  Therefore our assumption was false, and $x^L \le x$.
Thus $x^L < x$.  A similar argument shows that $x < x^R$.
\end{proof}

\begin{theorem}
Surreal numbers are closed under negation and addition.
\end{theorem}
\begin{proof}
Let $x$ and $y$ be surreal numbers.  Then the left options of
$x + y$ are of the form $x^L + y$ and $x + y^L$.  By the previous theorem,
these are less than $x + y$.  By induction they are surreal numbers.
By similar arguments, the right options of $x + y$ are greater than
$x + y$ and are also surreal numbers.  Therefore every left option
of $x + y$ is less than every right option of $x + y$, and every option
is a surreal number.  So $x + y$ is a surreal number.

Similarly, if $x$ is a surreal number, we can assume inductively that
$-x^L$ and $-x^R$ are surreal numbers for every $x^L$ and $x^R$.  Then
since $x^L \lhd x^R$ for every $x^L$ and $x^R$, we have
$-x^L \lhd -x^R$ for every $x^L$ and $x^R$.  So $-x = \{-x^L|-x^R\}$
is a surreal number.
\end{proof}

\begin{theorem}
Surreal numbers are \emph{totally} ordered by $\ge$.  That is, two surreal numbers
are never incomparable.
\end{theorem}
\begin{proof}
If $x$ and $y$ are surreal numbers, then by the previous theorem $x - y$ is also
a surreal number.  So it suffices to show that no surreal number is fuzzy (incomparable to zero).
Let $x$ be a minimal counterexample.  If any left option $x^L$ of $x$ is
$\ge 0$, then $0 \le x^L < x$, contradicting fuzziness of $x$.  So every
left option of $x$ is $\lhd 0$.  That means that if Left moves first in $x$,
she can only move to losing positions.  By the same argument, if Right 
moves first in $x$, then he loses too.  So no matter who goes first,
they lose.  Thus $x$ is a zero game, not a fuzzy game.
\end{proof}

John Conway defined a way to multiply\footnote{
His definition is
\[ xy = \{x^Ly + xy^L - x^Ly^L, x^Ry + xy^R - x^Ry^R|x^Ly + xy^R - x^Ly^R, x^Ry + xy^L - x^Ry^L\}\]
Here $x^L$, $x^R$, $y^L$, and $y^R$ have their usual meanings, though within
an expression like $x^Ly + xy^L - x^Ly^L$, the two $x^L$'s should be the same.} surreal numbers,
making the class \textbf{No} of all surreal numbers into a totally 
ordered real-closed field which turns out to contain all the real numbers
and transfinite ordinals.  We refer the interested reader to Conway's book \emph{On Numbers and Games}.

\section{Short Surreal Numbers}
If we just restrict to short games, the short surreal numbers end up being in correspondence with the \emph{dyadic rationals}
 - rational
numbers of the form $i/2^j$ for $i, j \in \mathbb{Z}$.  We now work to
show this, and to give the rule for determining which numbers are which.

We have already shown that the (short) surreal numbers form a totally ordered
abelian group.  In particular, if $x$ is any nonzero surreal number, then the integral
multiples of $x$ form a group isomorphic to $\mathbb{Z}$ with its usual order, because
$x$ cannot be incomparable to zero.

\begin{lemma}\label{proto-simplicity}
Let $G \equiv \{G^L|G^R\}$ be a game, and $\mathcal{S}$ be the class of all surreal numbers
$x$ such that $G^L \lhd x \lhd G^R$ for every left option $G^L$ and every
right option $G^R$.  Then if $\mathcal{S}$ is nonempty, $G$ equals a surreal number,
and there is a surreal number $y \in \mathcal{S}$ none of whose options are in $\mathcal{S}$.
This $y$ is unique up to equality, and in fact equals $G$.
\end{lemma}
So roughly speaking, $\{G^L|G^R\}$ is always the simplest surreal number
between $G^L$ and $G^R$, unless there is no such number, in which case $G$ is not a surreal number.
\begin{proof}
Suppose that $\mathcal{S}$ is nonempty.  It is impossible for every element of $\mathcal{S}$ to have an option in $\mathcal{S}$,
or else $\mathcal{S}$ would be empty by induction.
Let $y$ be some element of $\mathcal{S}$, none of whose options are in $\mathcal{S}$.
Then it suffices to show that $y = G$, for then $G$ will equal a surreal number,
and $y$ will be unique.

By symmetry, we need only show that $y \le G$.  By Theorem~\ref{betwixt}, this will be true unless
$y \ge G^R$ for some right option $G^R$ (but this can't happen because
$y \in \mathcal{S}$), or $G \le y^L$ for some left option $y^L$ of $y$.  Suppose then
that $G \le y^L$ for some $y^L$.
By choice of $y$, $y^L \notin \mathcal{S}$.

So for any $G^L$, we have $G^L \lhd G \le y^L$.  But also, for any
$G^R$, we have $y^L \le y \lhd G^R$, so that $y^L \lhd G^R$ for any $G^R$.
So $y^L \in \mathcal{S}$, a contradiction.
\end{proof}

\begin{lemma}\label{dyadic-rationals-to-surreals}
Define the following infinite sequence:
\[ a_0 \equiv 1 \equiv \{0|\}\]
and
\[ a_{n+1} \equiv \{0|a_n\}\]
for $n \ge 0$.  Then every term in this sequence is a positive
surreal number, and $a_{n+1} + a_{n+1} = a_n$ for $n \ge 0$.
Thus we can embed the dyadic rational numbers into the surreal numbers
by sending $i/2^j$ to
\[ \underbrace{a_j + a_j + \cdots + a_j}_{\textrm{$i$ times}}\]
where the sum of $0$ terms is $0$.
This map is an order-preserving homomorphism.
\end{lemma}
\begin{proof}
Note that if $x$ is a positive surreal number, then $\{0|x\}$ is clearly
a surreal number, and it is positive by Theorem~\ref{twolessthans}.  So
every term in this sequence is a positive surreal number, because $1$ is.

For the second claim, proceed by induction.  Note that
\[ a_{n+1} + a_{n+1} \equiv \{a_{n+1} | a_{n+1} + a_n\}.\]
Now by Theorem~\ref{twolessthans}, $0 < a_{n+1} < a_n$, so that
\[ a_{n+1} < a_n < a_{n+1} + a_n,\]
or more specifically
\[ a_{n+1} \lhd a_n \lhd a_{n+1} + a_n.\]
By Lemma~\ref{proto-simplicity}, it will follow that $a_{n+1} + a_{n+1} = a_n$ as long
as no option $x^*$ of $a_n$ satisfies
\begin{equation} a_{n+1} \lhd x^* \lhd a_{n+1} + a_n.\label{goal1}\end{equation}
Now $a_n$ has the option $0$, which fails the left side of (\ref{goal1}) because
$a_{n+1}$ is positive, and the only other option of $a_n$ is $a_{n-1}$, which
only occurs in the case $n > 1$.  By induction, we know that
\[ a_n + a_n = a_{n-1}.\]
Since $a_{n+1} < a_n$, it's clear that
\[ a_{n+1} + a_n < a_n + a_n = a_{n-1},\]
so that $a_{n-1} \lhd a_{n+1} + a_n$ is false.  So no option
of $a_n$ satisfies (\ref{goal1}), but $a_n$ does.  Therefore
by Lemma~\ref{proto-simplicity}, $\{a_{n+1}|a_{n+1} + a_n\} = a_n$.

The remaining claim follows easily and is left as an exercise to the reader.
\end{proof}

\begin{definition}
A surreal number is \emph{dyadic} if it occurs in the range of the map
from dyadic rationals to surreal numbers in the previous theorem.
\end{definition}
Note that these are closed under addition, because the map from
the theorem is a homomorphism.

\begin{theorem}
Every short surreal number is dyadic.
\end{theorem}
\begin{proof}
Let's say that a game $G$ is \emph{all-dyadic} if every position of $G$ (including $G$)
equals ($=$) a dyadic surreal number.  (This depends on the form of $G$, not just its value.)

We first claim that all-dyadic games are closed under addition.  This follows easily
by induction and the fact that the values of dyadic surreal numbers are closed under addition.
Specifically, if $x = \{x^L|x^R\}$ and $y = \{y^L|y^R\}$ are all-dyadic, then $x^L$, $x^R$, $y^L$,
and $y^R$ are all-dyadic, so by induction $x + y^L$, $x^L + y$, $x + y^R$, $x^R + y$ are all
all-dyadic.  Therefore $x + y$ is, since it equals a dyadic surreal number itself, because $x$ and $y$ do.

Similarly, all-dyadic games are closed under negation, and therefore subtraction.

%We also claim that the canonical form of an all-dyadic game is all-dyadic.  This is clear because
%the set of positions of a game only gets smaller as we remove dominated moves and bypass reversible moves.
%If all the original positions had the property of equaling dyadic surreal numbers, then so would
%the positions of the new form.

Next, we claim that every dyadic surreal number has an all-dyadic form.  The dyadic
surreal numbers are sums of the $a_n$ games of Lemma~\ref{dyadic-rationals-to-surreals}, and by construction,
the $a_n$ are all-dyadic in form.  So since all-dyadic games are closed under addition and subtraction,
every dyadic surreal number has an all-dyadic form.

We can also show that if $G$ is a game, and there is some all-dyadic surreal number $x$
such that $G^L \lhd x \lhd G^R$ for every $G^L$ and $G^R$, then $G$ equals
a dyadic surreal number.  The proof is the same as Lemma~\ref{proto-simplicity} except that we now
let $\mathcal{S}$ be the set of all all-dyadic surreal numbers between all $G^L$
and all $G^R$.  The only property of $\mathcal{S}$ we used were that every $x \in \mathcal{S}$
satisfies $G^L \lhd x \lhd G^R$, and that if $y$ is an option of $x \in \mathcal{S}$,
and $y$ also satisfies $G^L \lhd y \lhd G^R$, then $y \in \mathcal{S}$.  These conditions are still
satisfied if we restrict $\mathcal{S}$ to all-dyadic surreal numbers.

Finally, we prove the theorem.
We need to show that if $L$ and $R$ are finite sets of dyadic surreal
numbers, and every element of $L$ is less than every element of
$R$, then $\{L|R\}$ is also dyadic.  (All short surreal numbers are built up in this way, so this suffices.)
By the preceding paragraph,
it suffices to show that some dyadic surreal number is greater than every
element of $L$ and less than every element of $R$.  This follows
from the fact that the dyadic rational numbers are a dense total
order without endpoints, and that the dyadic surreal numbers are in
order-preserving correspondence with the dyadic rational numbers.
\end{proof}

From now on, we identify dyadic rationals and their corresponding
short surreal numbers.

We next determine the canonical form of every short number
and provide rules to decide which number
$\{L|R\}$ is, when $L$ and $R$ are sets of numbers.

\begin{theorem}
Let $b_0 \equiv \{|\}$ and $b_{n+1} \equiv \{b_n|\}$
for $n \ge 0$.  Then
$b_n$ is the canonical form of positive integers $n$ for $n \ge 0$.
\end{theorem}
\begin{proof}
It's easy to see that every $b_n$ is in canonical form: there
are no dominated moves because there are never two options for
Left or for Right, and there are no reversible moves,
because no $b_n$ has any right options.

Since $b_0 = 0$, it remains to see that $b_{n+1} = 1 + b_n$.
We proceed by induction.  The base case $n = 0$ is clear,
since $b_1 = \{b_0|\} = \{0|\} = 1$, by definition of the game $1$.

If $n > 0$, then
\[ 1 + b_n = \{0|\} + \{b_{n-1}|\} = \{1 + b_{n-1}, 0 + b_n|\}.\]
By induction $1 + b_{n-1} = b_n$, so this is just $\{b_n|\} = b_{n+1}$.
\end{proof}
So for instance, the canonical form of $7$ is $\{6|\}$.
Similarly, if we let $c_0 = 0$ and $c_{n+1} = \{|c_n\}$,
then $c_n = -n$ for every $n$, and these are in canonical form.
For example, the canonical form of $-23$ is $\{|-22\}$.

\begin{theorem}\label{intsimpl}
If $G \equiv \{G^L|G^R\}$ is a game, and there is at least one
integer $m$ such that $G^L \lhd m \lhd G^R$ for all $G^L$ and $G^R$,
then there is a unique such $m$ with minimal magnitude, and
$G$ equals it.
\end{theorem}
\begin{proof}
The proof is the same as the proof of Lemma~\ref{proto-simplicity},
but we let $\mathcal{S}$ be the set of \emph{integers}
between the left and right options of $G$, and we use their
canonical forms derived in the preceding theorem.

That is, we let $\mathcal{S}$ be the set
\[ \mathcal{S} = \{b_n\,:\, G^L \lhd b_n \lhd G^R \textrm{ for all $G^L$ and $G^R$} \}
\cup \{c_n\,:\, G^L \lhd c_n \lhd G^R \textrm{ for all $G^L$ and $G^R$} \}\]
Then by assumption (and the fact that every integer equals a $b_n$ or a $c_n$),
$\mathcal{S}$ is nonempty.  Let $m$ be an element of $\mathcal{S}$ with
minimal magnitude.  I claim that no option of $m$ is in $\mathcal{S}$.
If $m$ is $0 = b_0 = c_0$, this is obvious, since $m$ has no options.
If $m > 0$, then $m = b_m$, and the only option of $b_m$ is $b_{m-1}$,
which has smaller magnitude, and thus cannot be in $\mathcal{S}$.  Similarly
if $m < 0$, then $m = c_{-m}$, and the only option of $m$
is $m + 1$, which has smaller magnitude.

So no option of $m$ is in $\mathcal{S}$.  And in fact no option of $m$
is in the broader $\mathcal{S}$ of Lemma~\ref{proto-simplicity}, 
so $m = G$.
\end{proof}

\begin{theorem}
If $m/2^n$ is a non-integral dyadic rational, and $m$ is odd, then the canonical form
of $m/2^n$ is
\[ \{\frac{m-1}{2^n}|\frac{m+1}{2^n}\}.\]
\end{theorem}
So for instance, the canonical form of $1/2$ is $\{0|1\}$,
of $11/8$ is $\{5/4|3/2\}$, and so on.
\begin{proof}
Proceed by induction on $n$.  If $n = 1$, then we need to show that
for any $k \in \mathbb{Z}$,
\[ k + \frac{1}{2} = \{k|k+1\}\]
and that $\{k|k+1\}$ is in canonical form.  Letting $x = \{k|k+1\}$,
we see that
\[ x + x = \{k + x|k + 1 + x\}.\]
Since $k < x < k + 1$, we have $k + x < 2k + 1 < k + 1 + x$.  Therefore,
by Theorem~\ref{intsimpl} $x + x$ is an integer.  In fact, since
$k < x < k + 1$, we know that $2k < x + x < 2k + 2$, so that $x + x = 2k + 1$.
Therefore, $x$ must be $k + \frac{1}{2}$.

Moreover, $\{k|k+1\}$ is in canonical form:  it clearly has no dominated moves.
Suppose it had a reversible move, $k$ without loss of generality.
But $k$'s right option can only be $k + 1$, by canonical forms of
the integers.  And $k + 1 \not \le \{k|k+1\}$.

Now suppose that $n > 1$.  Then we need to show that for $m$ odd,
\[ \frac{m}{2^n} = \{\frac{m-1}{2^n}|\frac{m+1}{2^n}\}\]
and that the right hand side is in canonical form.  (Note that $\frac{m\pm 1}{2^n}$ have
smaller denominators than $\frac{m}{2^n}$ because $m$ is odd.)

Let $x = \{\frac{m-1}{2^n}|\frac{m+1}{2^n}\}$.  Then
\[ x + x = \{x + \frac{m-1}{2^n}|x + \frac{m+1}{2^n}\}.\]
Now since $\frac{m-1}{2^n} < x < \frac{m+1}{2^n}$, we know that
$x + \frac{m-1}{2^n} < \frac{m}{2^n} + \frac{m}{2^n} < x + \frac{m+1}{2^n}$.
So $\frac{m}{2^{n-1}}$ lies between the left and right options of $x + x$.  On the other
hand, we know by induction that $\frac{m}{2^{n-1}} = \{\frac{m-1}{2^{n-1}}|\frac{m+1}{2^{n-1}}\}$,
and we have
\[ x + \frac{m-1}{2^n} \not < \frac{m-1}{2^{n-1}} < x + \frac{m+1}{2^n}\]
(because $x > \frac{m-1}{2^n}$)
and
\[ x + \frac{m-1}{2^n} < \frac{m+1}{2^{n-1}} \not < x + \frac{m+1}{2^n}\]
(because $x < \frac{m+1}{2^n}$)
so by Lemma~\ref{proto-simplicity}, $x + x = \frac{m}{2^{n-1}}$.  Therefore,
$x = \frac{m}{2^n}$.

It remains to show that $x = \{\frac{m-1}{2^n}|\frac{m+1}{2^n}\}$ is in canonical form.
It clearly has no dominated moves.  And since $\frac{m-1}{2^n}$ has smaller denominator,
we know by induction that when it is in canonical form, any right option
must be at least $\frac{m-1}{2^n} + \frac{1}{2^{n-1}} = \frac{m+1}{2^n} \not > x$.
So $\frac{m-1}{2^n}$ is not reversible, and similarly neither is $\frac{m+1}{2^n}$.
\end{proof}

Using these rules, we can write out the canonical forms of some of the simplest numbers:
%\[ 0 = \{|\} \qquad \qquad 1 = \{0|\}\]
%\[ -1 = \{|0\} \qquad \qquad 2 = \{1|\}\]
%\[ \frac{1}{2} = \{0|1\} \qquad \qquad  \frac{-1}{2} = \{-1|0\}\]
%\[ -2 = \{|-1\} \qquad \qquad 3 = \{2|\}\]
%\[ \frac{3}{2} = \{1|2\} \qquad \qquad \frac{3}{4} = \{\frac{1}{2}|1\}\]
%\[ \frac{1}{4} = \{0|\frac{1}{2}\} \qquad \qquad \frac{-1}{4} = \{\frac{-1}{2}|0\}\]
%\[ \frac{-3}{4} = \{-1|\frac{-1}{2}\} \qquad \qquad \frac{-3}{2} = \{-2|-1\}\]
%\[ -3 = \{|-2\}. \]
\begin{align*}
 0 = \{|\} \qquad &  1 = \{0|\} \\
 -1 = \{|0\} \qquad &  2 = \{1|\}\\
 \frac{1}{2} = \{0|1\} \qquad &   \frac{-1}{2} = \{-1|0\}\\
 -2 = \{|-1\} \qquad &  3 = \{2|\}\\
 \frac{3}{2} = \{1|2\} \qquad &  \frac{3}{4} = \{\frac{1}{2}|1\}\\
 \frac{1}{4} = \{0|\frac{1}{2}\} \qquad &  \frac{-1}{4} = \{\frac{-1}{2}|0\}\\
 \frac{-3}{4} = \{-1|\frac{-1}{2}\} \qquad &  \frac{-3}{2} = \{-2|-1\}\\
 & -3 = \{|-2\}.
\end{align*}
But what about numbers that aren't in canonical form?  What is
$\{\frac{-1}{4}|27\}$ or $\{1,2|\frac{19}{8}\}$?

\begin{definition}
Let $x$ and $y$ be dyadic rational numbers (short surreal numbers).  We say that
$x$ is \emph{simpler than} $y$ if $x$ has a smaller denominator than $y$,
or if $|x| < |y|$ when $x$ and $y$ are both integers.
\end{definition}
Note that simplicity is a strict partial order on numbers.  Also, by the canonical forms
just determined, if $x$ is a number in canonical form, then all options of $x$ are simpler
than $x$.
\begin{theorem}\label{simplicity}(the simplicity rule)
Let $G = \{G^L|G^R\}$ be a game.  If there is any number
$x$ such that $G^L \lhd x \lhd G^R$ for all $G^L$ and $G^R$, then $G$
equals a number, and $G$ is the unique simplest such $x$.
\end{theorem}
\begin{proof}
A simplest possible $x$ exists because there are no infinite sequences of
numbers $x_1, x_2, \ldots$ such that $x_{n+1}$ is simpler than $x_n$
for every $n$.  The denominators in such a chain would necessarily decrease at each step
until reaching $1$, and then the magnitudes would decrease until $0$ was reached, from which
the sequence could not proceed.

Then taking $x$ to be in canonical form, the options of $x$ are simpler than $x$,
and therefore not in the class $\mathcal{S}$ of Lemma~\ref{proto-simplicity}, though $x$ itself is.
So by Lemma~\ref{proto-simplicity}, $x = G$.
\end{proof}

Here is a diagram showing the structure of some of the simplest short surreal numbers.  Higher numbers
in the tree are simpler.

\begin{figure}[H]
\begin{center}
\synttree[0 [-1 [-2 [-3 [-4] [-5/2]] [-3/2 [-7/4] [-5/4]]] [-1/2 [-3/4 ] [1/4 ]]] [1 [1/2 [1/4 ] [3/4]] [2 [3/2 [5/4] [7/4]] [3 [5/2] [4]]]]]
%\caption{The final move}
%\label{thirdmove}
\end{center}
\end{figure}

So for instance $\{10|20\}$ is $11$ (rather than $15$ as one might expect),
because $11$ is the simplest number between $10$ and $20$.  Or
$\{2|2.75\}$ is $2.5$ but $\{2|2.25\}$ is $2.125$.

\section{Numbers and Hackenbush}\label{sec:hack-num}
Surreal numbers are closely connected to the game of Hackenbush.
In fact, every Hackenbush position is a surreal number, and every short surreal number
occurs as a Hackenbush position.

\begin{lemma}
Let $G$ be a game, and suppose that for every position $H$ of $G$,
every $H^L \le H$ and every $H^R \ge H$.  Then $G$ is a surreal number.
\end{lemma}
\begin{proof}
If $H^L \le H$ then $H^L < H$ because $H^L \lhd H$ by Theorem~\ref{betwixt}, and similarly $H \le H^R \Rightarrow H < H^R$.
So by transitivity and the assumptions, $H^L < H^R$ for every
$H^L$ and $H^R$.  Since this is true for all positions of $G$, $G$ is a surreal number.
\end{proof}
\begin{theorem}
Every position of Hackenbush is a surreal number.
\end{theorem}
\begin{proof}
We need to show that if $G$ is a Hackenbush position, and $G^L$ is a left option,
then $G^L \le G$.  (We also need to show that if $G^R$ is a right
option, then $G \le G^R$.  But this follows by symmetry from the other claim.)
Equivalently, we need to show that
$G^L + (-G) \le 0$.  Note that $G^L$ is obtained from $G$ by removing a left
edge $e$, and then deleting all the edges that become disconnected from the ground
by this action.  Let $S$ be the set of deleted edges other than $e$.  Let
$e'$ and $S'$ be the corresponding edges in $-G$, so that $e'$ is a Right edge,
whose removal definitely deletes all the edges in $S'$.
\begin{figure}[H]
\begin{center}
\includegraphics[width=4in]
					{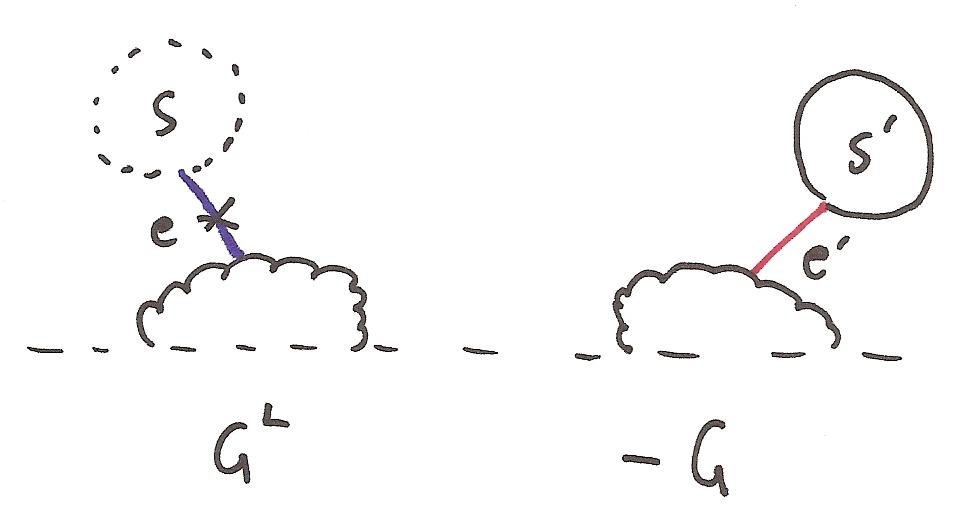}
%\caption{}
%\label{domineering-sum}
\end{center}
\end{figure}
To show that $G^L + (-G) \le 0$, we exhibit a strategy that Right can
use playing second.  Whenever Left plays in one component, Right makes the exact same
move in the other.  This makes sense as long as Left does not play at $e'$ or in $S'$.
However, Left cannot play at $e'$ because $e'$ is a Right edge.  On the other hand,
if Left plays in $S'$, we add a caveat to Right's strategy, by having Right
respond to any move in $S'$ with a move at $e'$.  After such an exchange,
all of $S'$ and $e'$ will be gone, and the resulting position will be of the form
$X + (-X) = 0$.  Since Right has moved to this position, Right will win.

Therefore, Right can always reply to any move of Left.  So Right will win,
if he plays second.
\end{proof}

As a simple example of numbers, the reader can verify that
a Hackenbush position containing only Blue (Left) edges is a positive integer,
equal to the number of edges.

More generally, every surreal number occurs:
\begin{theorem}\label{charcoal}
Every short surreal number is the value of some position in Hackenbush.
\end{theorem}
\begin{proof}
The sum of two Hackenbush positions is a Hackenbush position,
and the negative of a Hackenbush position is also a Hackenbush position.
Therefore, we only need to present Hackenbush positions
taking the values $\frac{1}{2^n}$ for $n \ge 0$.

Let $d_n$ denote a string of edges, consisting of a blue (left) edge
attached to the ground, followed by $n$ red edges.
\begin{figure}[htb]
\begin{center}
\includegraphics[width=2.5in]
					{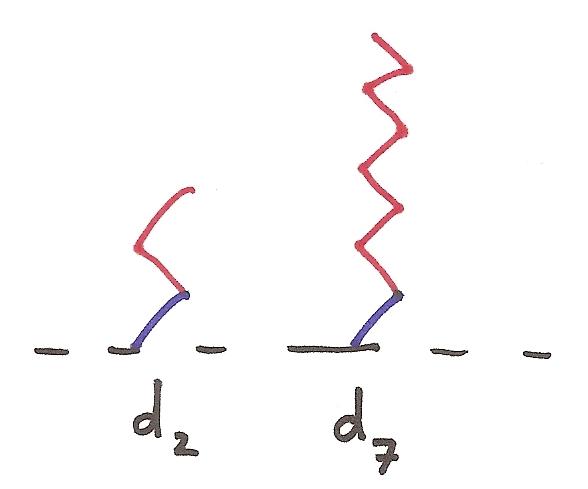}
\caption{$2^{-2}$ and $2^{-7}$ in Hackenbush}
\label{numbers-only-hackenbush}
\end{center}
\end{figure}

Then we can easily verify that for all $n \ge 0$,
\[ d_n \equiv \{0|d_0, d_1, \ldots, d_{n-1}\}.\]
So $d_0 \equiv \{0|\}$, $d_1 \equiv \{0|d_0\}$, $d_2 \equiv \{0|d_0,d_1\}$, and so on.

Then by the simplicity rule it easily follows inductively that $d_n \equiv \frac{1}{2^n}$.
\end{proof}

So we can assign a rational number to each Hackenbush position,
describing the advantage that the position gives to Left or to Right.
In many cases there are rules for finding these numbers,
though there is probably no general rule, because the problem of determining the
outcome of a general Hackenbush position is NP-hard, as shown on page 211-217 of \emph{Winning Ways}.

\section{The interplay of numbers and non-numbers}
Unfortunately, not all games are numbers.  However,
the numbers play a fundamental role in understanding the structure
of the other games.

First of all, they bound all other games:
\begin{theorem}
Let $G$ be a game.  Then there is some number $n$ such that $-n < G < n$.
\end{theorem}
\begin{proof}
We show by induction that every game is less than a number.

Let $G$ be a game, and suppose that every option of $G$ is less than a number.
Since all our games are short, $G$ has finitely many options.  So we can find some
$M$ such that every option $G^L, G^R < M$.  Without loss of generality,
$M$ is a positive integer, so it has no right option.
Then $G \le M$ unless $G$ is $\ge$ a right option of $M$ (but none exist),
or $M \le$ a left option of $G$ (but we chose $M$ to exceed all options of $G$).  Therefore
$G \le M$, by Theorem~\ref{betwixt}.  Then $G < M + 1$, and $M + 1$ is a number.
\end{proof}

Because of this, we can examine which numbers are greater than or less
than a given game.

\begin{definition}
If $G$ is a game then we let $L(G)$ be the infimum (in $\mathbb{R}$)
of all numbers $x$ such $G \le x$, and $R(G)$ be the supremum
of all numbers $x$ such that $x \le G$.
\end{definition}
These exist, because the previous theorem shows that the sets are non-empty,
and because if $x \le G$ for arbitrarily big $x$, then it could not be the
case that $G < n$ for some fixed $n$.

It's clear that $R(G) \le L(G)$, since if $x \le G$ and $G \le y$,
then $x \le y$.  It's not yet clear that $R(G)$ and $L(G)$ must
be dyadic rational numbers, but we will see this soon.
If $G$ is a number, then clearly $R(G) = G = L(G)$.
Another easily verified fact is that if $x$ is a number,
then $L(G + x) = L(G) + x$ and $R(G + x) = R(G) + x$
for any $x$.  Also, it's easy to show that $L(G + H) \le L(G) + L(H)$
and similarly that $R(G + H) \ge R(G) + R(H)$, using the fact
that if $x \le G$ and $y \le H$, then $x + y \le G + H$.

% Apricot One way to visualize a game is as a cloud sitting between $R(G)$ and $L(G)$.

Numbers are games in which any move makes the game worse for the player
who made the move.  Such games aren't very fun to play in,
so Left and Right might decide to simply stop as soon as the state of
the game becomes a number.  Suppose they take this number as the final
score of the game, with Left trying to maximize it, and Right trying
to minimize it.  Then the final score under perfect play is called
the ``stopping value.''  Of course it depends on who goes
first, so we actually get two stopping values:
\begin{definition}
Let $G$ be a short game.  We recursively define the
\emph{left stopping value} $LS(G)$ and the \emph{right stopping value}
$RS(G)$ by
\begin{itemize}
\item If $G$ equals a number $x$, then $LS(G) = RS(G) = x$.
\item Otherwise, $LS(G)$ is the maximum value of $RS(G^L)$
as $G^L$ ranges over the left options of $G$; and
$RS(G)$ is the minimum value of $LS(G^R)$ as $G^R$ ranges
over the right options of $G$.
\end{itemize}
\end{definition}

Then we have the following:
\begin{theorem}
Let $G$ be a game and $x$ be a number.  Then
\begin{itemize}
\item If $x > LS(G)$ then $x > G$.
\item If $x < LS(G)$ then $x \lhd G$.
\item If $x < RS(G)$ then $x < G$.
\item If $x > RS(G)$ then $x \rhd G$.
\end{itemize}
\end{theorem}
\begin{proof}
We proceed by joint induction on $G$ and $x$.
As usual, we need no base case.

If $G$ equals a number, then all results are obvious.  So suppose
that $G$ is not equal to a number, so that
$LS(G)$ is the maximum value of $RS(G^L)$
and $RS(G)$ is the maximum value of $LS(G^R)$.  We have
four things to prove.

Suppose $x > LS(G)$.  Then since $G$ does not equal
a number, we only need to show that
$G \le x$. This will be true unless $x^R \le G$
for some $x^R$, or $x \le G^L$ for some $G^L$.  In
the first case, we have $x^R > x > LS(G)$, so by
induction $x^R > G$, not $x^R \le G$.  In the second
case, note that $x > LS(G) \ge RS(G^L)$
so by induction $x \rhd G^L$, not $x \le G^L$.

Next, suppose that $x < LS(G)$.  Then there
is some $G^L$ such that $x < LS(G) = RS(G^L)$.
By induction, then $x < G^L$.  So
$x \le G^L \lhd G$, and thus $x \lhd G$.

The cases where $x > RS(G)$ and $x < RS(G)$
are handled similarly.
\end{proof}

\begin{corollary}
If $G$ is any short game, then $LS(G) = L(G)$ and
$RS(G) = R(G)$.
\end{corollary}
\begin{proof}
Clear from the definition of $L(G)$ and $R(G)$
and the previous theorem.
\end{proof}
Interestingly, this means that the left stopping
value of any non-numerical game is at least its right stopping
value
\[ LS(G) \ge RS(G).\]
So in some sense, in a non-numerical
game you usually want to be the first person to move.  Since
$LS(G)$ and $RS(G)$ are synonymous with $L(G)$ and $R(G)$,
we drop the former and write $L(G)$ and $R(G)$ for the left stopping
value and the right stopping value.

Using these results, we can easily compute $L(G)$ and $R(G)$ for
various games.  For $* = \{0|0\}$, the left and right
stopping values are easily seen to be 0, so
$L(*) = R(*) = 0$.  It follows that $*$ is less
than every positive number and greater than every negative
number.  Such games are called \emph{infinitesimal} 
or \emph{small} games, and will be discussed more in a later section.

For another example, the game $\pm 1 = \{1|-1\}$ has
$L(\pm 1) = 1$ and $R(\pm 1)$.  So it is less than
every number in $(1,\infty)$, and greater than
every number in $(-\infty,1)$, but fuzzy with
$(-1,1)$.

The next result formalizes the notion that ``numbers
aren't fun to play in.''
\begin{theorem}(Number Avoidance Theorem)
Let $x$ be a number and $G = \{G^L|G^R\}$ be a short game
\emph{that does not equal any number.}  Then
\[ G + x = \{G^L + x|G^R + x\}.\]
\end{theorem}
\begin{proof}
Let $G_x = \{G^L + x|G^R + x\}$.  Then consider
\[ G_x - G = \{G_x - G^R, (G^L + x) - G| G_x - G^L, (G^R + x) - G\}.\]
Now for any $G^R$, $G^R + x$ is a right option of $G_x$, so $G^R + x \rhd G_x$ and therefore $G_x - G^R \lhd x$.  Similarly, $G^L \lhd G$, so that
$(G^L + x) - G \lhd x$ for every $G^L$.  So every left option
of $G_x - G$ is $\lhd x$.

Similarly, for any $G^L$, we have $G^L + x \lhd G_x$ so that
$x \lhd G_x - G^L$.  And for any $G^R$, $G \lhd G^R$ so that $x \lhd (G^R + x) - G$.
So every right option of $G_x - G$ is $\rhd x$.  Then by the simplicity
rule, $G_x - G$ is a number, $y$.  We want to show that $y = x$.

Note that $G_x = G + y$, so that
\begin{equation} L(G_x) = L(G) + y.\label{yo}\end{equation}
Now $y$ is a number and $G$ is not, so $G_x$ is not a number.  Therefore
$L(G_x)$ is the maximum value of $R(G^L + x) = R(G^L) + x$ as $G^L$ ranges over the left options
of $G$.  Since the maximum value of $R(G^L)$ as $G^L$ ranges over the left options
of $G$ is $L(G)$, we must have $L(G_x) = L(G) + x$.  Combining this with (\ref{yo})
gives $y = x$.  So $G_x - G = x$ and we are done.
\end{proof}
This theorem needs some explaining.  Some simple examples of its use are
\[ * + x = \{x|x\}\]
and
\[ \{-1|1\} + x = \{-1 + x|1 + x\}\]
for any number $x$.  To see why it is called ``Number Avoidance,'' note that
the definition of $G + x$ is
\[ G + x = \{G^L + x, G + x^L|G^R + x, G + x^R\},\]
where the options $G + x^L$ and $G + x^R$ correspond to the options of moving
in $x$ rather than in $G$.  The Number Avoidance theorem says that such options
can be removed without affecting the outcome of the game.  The strategic implication
of this is that if you are playing a sum of games, you can ignore all moves
in components that are numbers.  This works even if your opponent
does move in a number, because by the gift-horse principle,
\[ G + x = \{G^L + x| G^R + x, G + x^R\}\]
in this case.
% apricot
%\section{The Opacity of Numbers}
%todo: write this section.

\section{Mean Value} % apricot and Thermography}
%We now leave the world of infinitesimals and consider arbitrary (short) games.
If $G$ and $H$ are short games, then $L(G + H) \le L(G) + L(H)$
and $R(G + H) \ge R(G) + R(H)$.  It follows that
the size of the \emph{confusion interval} $[R(G + H),L(G + H)]$ is at most the sum
of the sizes of the confusion intervals of $G$ and $H$.

Now if we add a single game to itself $n$ times, what happens in the limit?  We might expect
that $L(nG) - R(nG)$ will be approximately $n(L(G) - R(G))$.  But in fact,
the size of the confusion interval is bounded:
\begin{theorem}(Mean Value Theorem)
Let $nG$ denote $G$ added to itself $n$ times.  Then there is some bound $M$ dependent on
$G$ but not $n$ such that
\[ L(nG) - R(nG) \le M\]
for every $n$.
\end{theorem}
\begin{proof}
We noted above that
\[ L(G + H) \le L(G) + L(H).\]
But we can also say
that 
\[ R(G + H) \le R(G) + L(H)\]
 for arbitrary $G$ and $H$, because if
$x > R(G)$ and $y > L(H)$, then $x \rhd G$ and $y > H$, so that $x + y \rhd G + H$,
implying that $x + y \ge R(G + H)$.  Similarly, 
\[ L(G + H) \ge R(G) + L(H).\]

Every left option of $nG$ is of the form $G^L + (n-1)G$, and by
these inequalities
\[ R(G^L + (n-1)G) \le L(G^L) + R((n-1)G)\]
Therefore if we let $M_1$ be the maximum
of $L(G^L)$ over the left options of
$G$, then
\[ R(G^L + (n-1)G) \le M_1 + R((n-1)G)\]
and so every left option of
$nG$ has right stopping value at most $M_1 + R((n-1)G)$.
Therefore $L(nG) \le M_1 + R((n-1)G)$.

Similarly, every right option of $nG$ is of the form 
$G^R + (n-1)G$, and we have
\[ L(G^R + (n-1)G) \ge L(G^R) + R((n-1)G)\]
Letting $M_2$ be the minimum value of $L(G^R)$
over the right options of $G$, we have
\[ L(G^R + (n-1)G) \ge M_2 + R((n-1)G)\]
and so every right option of $nG$ has left stopping
value at least $M_2 + R((n-1)G)$.  Therefore,
$R(nG) \ge M_2 + R((n-1)G)$.

Thus
\[ L(nG) - R(nG) \le M_1 + R((n-1)G) - (M_2 + R((n-1)G)) = M_1 - M_2\]
regardless of $n$.
\end{proof}
Together with the fact that $L(G + H) \le L(G) + L(H)$ and $R(G + H) \ge R(G) + R(H)$ and $L(G) \ge R(G)$,
it implies that
\[ \lim_{n \to \infty} \frac{L(nG)}{n} \textrm{ and } \lim_{n \to \infty} \frac{R(nG)}{n}\]
converge to a common limit, called the \emph{mean value} of $G$, denoted $m(G)$.
It is also easily seen that $m(G + H) = m(G) + m(H)$ and $m(-G) = -m(G)$, and that $G \ge H$
implies $m(G) \ge m(H)$.  The mean value of $G$ can be thought of as a numerical approximation to $G$.

\chapter{Games near 0}
\section{Infinitesimal and all-small games}\label{sec:small-games}
%\emph{All of our games will be short games henceforth.}

As noted above, the game $*$ lies between all the positive
numbers and all the negative numbers.  Such games are called \emph{infinitesimals}
or \emph{small} games.

\begin{definition}
A game is \emph{infinitesimal} (or \emph{small}) if it is less than every positive
number and greater than every negative number, i.e., if $L(G) = R(G) = 0$.  A game is \emph{all-small in form}
 if every one of its positions (including itself) is infinitesimal.  A game is
\emph{all-small in value} if it equals an all-small game.
\end{definition}

Since $L(G + H) \le L(G) + L(H)$ and $R(G + H) \ge R(G) + R(H)$ and $R(G) \le L(G)$, it's clear
that infinitesimal games are closed under addition.  An easy inductive proof shows that
all-small games are also closed under addition: if $G$ and $H$ are all-small in form,
then every option of $G + H$ is all-small by induction, and $G + H$ is infinitesimal itself,
so $G + H$ is all-small.  Moreover, if $G$ is all-small in value, then the canonical
value of $G$ is all-small in form.

Their name might suggest that all-small games are the smallest of games, but as we will see
this is not the case: the game $+_2 = \{0|\{0|-2\}\}$ is smaller than every positive all-small game.

Infinitesimal games occur naturally in certain contexts.  For example, every position in Clobber is infinitesimal.
To see this, let $G$ be a position in Clobber.  We need to show that for any $n$,
\[ \frac{-1}{2^n} \le G \le \frac{1}{2^n}.\]
As noted in the proof of Theorem~\ref{charcoal}, $\frac{1}{2^n}$ is a Hackenbush position consisting of string of
edges attached to the ground: $1$ blue edge followed by $n$ red edges (see Figure~\ref{numbers-only-hackenbush}).  By symmetry,
we only need to show that $G \le \frac{1}{2^n}$, or in other words, that
$\frac{1}{2^n} - G$ is a win for Left when Right goes first.  Left plays as follows:
whenever it is Left's turn, she makes a move in $G$, unless there are no remaining
moves in $G$.  \emph{In this case there are no moves for Right either.}  This can be seen
from the rules of Clobber - see Figure~\ref{clobber-all-small}.

\begin{figure}[h]
\begin{center}
\includegraphics[width=4in]
					{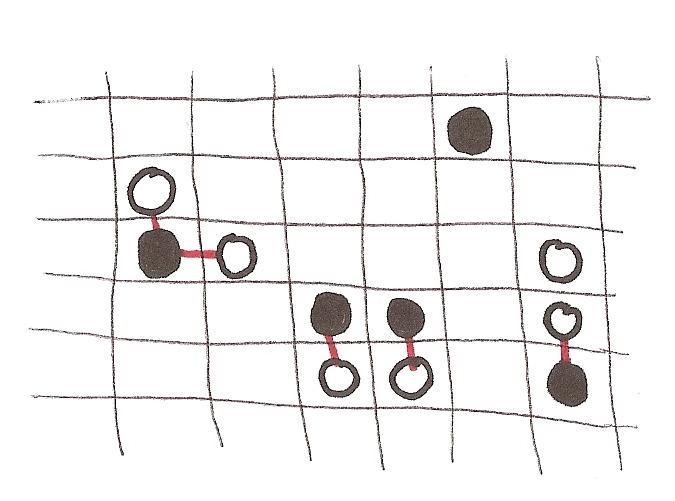}
\caption{In a position of Clobber, whenever one player has available moves,
so does the other.  The available moves for each player correspond to the pairs of adjacent black
and white pieces, highlighted with red lines in this diagram.}
\label{clobber-all-small}
\end{center}
\end{figure}

So once the game reaches a state where no more moves remain
in $G$, Left can make the final move in $\frac{1}{2^n}$, by cutting the blue
edge at the base of the stalk.  This ends the other component.

In other words, the basic reason why Left can win is that she retains the ability to end the Hackenbush position
at any time, and there's nothing that Right can do about it.

Now since every subposition of a Clobber position is itself a Clobber position,
it follows that Clobber positions are in fact \emph{all-small}.

The only property of Clobber that we used was that \emph{whenever Left can move, so can Right,
and vice versa}.  So we have the following
\begin{theorem}
If $G$ is a game for which Left has options iff Right has options, and the same holds 
of every subposition of $G$, then $G$ is all-small.
\end{theorem}
Conversely
\begin{theorem}
If $G$ is an all-small game in canonical form, then $G$ has the property that
Left can move exactly when Right can, and this holds in all subpositions.
\end{theorem}
\begin{proof}
By induction, we only need to show that this property holds for $G$.
Suppose that it didn't, so that $G = \{L|\emptyset\}$ or $G = \{\emptyset|R\}$.
In the first case, there is some number $n$ such that $n $ exceeds every element
of $L$.  Therefore by the simplicity rule, $G$ is a number.  Since $G$ is infinitesimal,
it must be zero, but the canonical form of $0$ has no left options.  So
$G$ cannot be of the form $\{L|\emptyset\}$.  The other possibility is ruled
out on similar grounds.
\end{proof}

The entire collection of all-small games is not easy to understand.  In fact,
we will see later that there is an order-preserving
homomorphism from the group $\mathcal{G}$ of (short) games into the group of all-small games.
So all-small games are as complicated as games in general.

Here are the simplest all-small games:
\[ 0 = \{|\}\]
\[ * = \{0|0\}\]
\[ \uparrow = \{0|*\}\]
\[ \downarrow = \{*|0\}\]
The reader can easily check that $\uparrow > 0$ but $\uparrow || *$.
As an exercise in reducing to canonical form, the reader can also verify that
\[ \uparrow + * = \{0,*|0\}\]
\[ \uparrow + \uparrow = \{0|\uparrow + *\}\]
\[ \uparrow + \uparrow + * = \{0|\uparrow\}\]
\[ \{\uparrow|\downarrow\} = *\]
Usually we use the abbreviations $\uparrow * = \uparrow + *$, $\Uparrow = \uparrow + \uparrow$,
$\Downarrow = \downarrow + \downarrow$, $\Uparrow* = \Uparrow + *$.  More generally,
$\hat{n}$ is the sum of $n$ copies of $\uparrow$ and $\hat{n}* = \hat{n} + *$.

These games occur in Clobber as follows:
\begin{figure}[h]
\begin{center}
\includegraphics[width=3in]
					{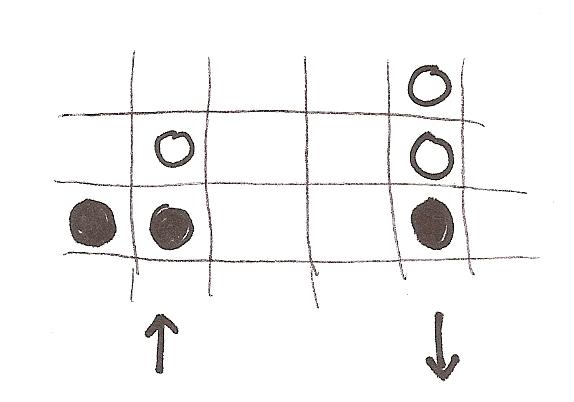}
%\caption{}
%\label{domineering-sum}
\end{center}
\end{figure}

\begin{theorem}
Letting $\mu_1 = \uparrow$, $\mu_{n+1} = \{0|\mu_n\}$, and $\nu_1 = \uparrow *$,
$\nu_{n+1} = \{0|\nu_n\}$, we have $\mu_{n+1} = \uparrow + \nu_n$ and
$\nu_{n+1} = \uparrow + \mu_n$ and $\mu_{n+1} = \nu_{n+1} + *$ for every $n \ge 1$.
\end{theorem}
\begin{proof}
We proceed by induction on $n$.  The base case is already verified above in the examples.
Otherwise
\[ \uparrow + \nu_n = \{0|*\} + \{0|\nu_{n-1}\} = \{\uparrow, \nu_n|\uparrow + \nu_{n-1}, \nu_n + *\}.\]
By induction, this is
\[ \{\uparrow, \nu_n| \mu_n, \mu_n\}\]
This value is certainly $\ge \{0|\mu_n\}$, since it is obtained by improving a left option ($0 \to \uparrow$)
and adding a left option of $\nu_n$.  So it remains to show that $\mu_{n+1} = \{0|\mu_n\} \le \{\uparrow, \nu_n|\mu_n\}$.
This will be true unless $\mu_{n+1} \ge \mu_n$ (impossible, since $\mu_n$ is a right option of $\mu_{n+1}$),
or $\{\uparrow, \nu_n | \mu_n\} \le 0$ (impossible because Left can make an initial winning move to $\uparrow$).
So $\uparrow + \nu_n = \mu_{n+1}$.

A completely analogous argument shows that $\uparrow + \mu_n = \nu_{n+1}$.  Then for the final claim,
note that $\mu_{n+1} = \uparrow + \nu_n = \uparrow + \mu_n - * = \uparrow + \mu_n + * = \nu_{n+1} + *$.
\end{proof}
So then $\mu_{2k - 1}$ is the sum of $2k - 1$ copies of $\uparrow$ and
$\nu_{2k}$ is the sum of $2k$ copies of $\uparrow$, because $* + * = 0$.

Using this, we get the following values of clobber positions:
\begin{figure}[h]
\begin{center}
\includegraphics[width=4in]
					{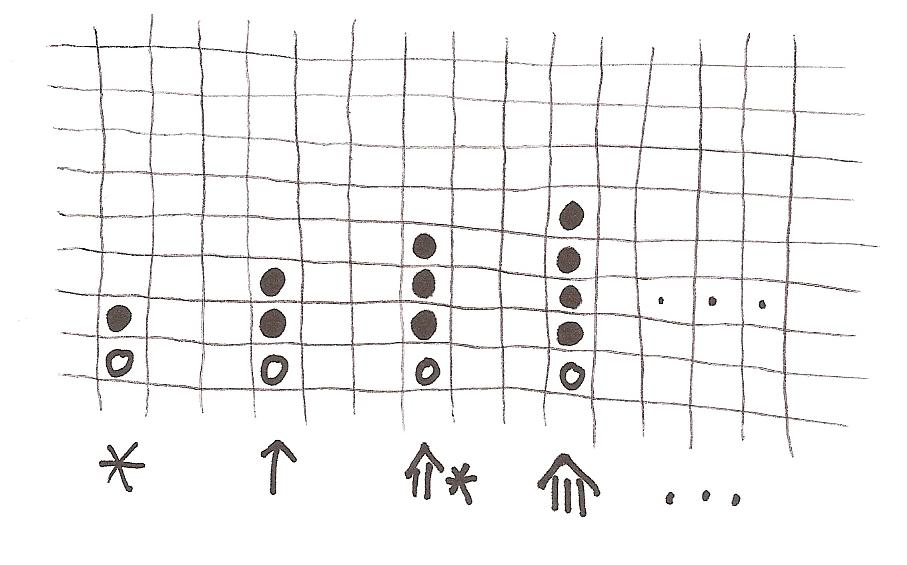}
%\caption{}
%\label{domineering-sum}
\end{center}
\end{figure}

We also can use this to show that the multiples of $\uparrow$ are among the largest of (short) infinitesimal games:
\begin{theorem}
Let $G$ be a short infinitesimal game.  Then there is some $N$ such that for $n > N$,
$G \le \mu_n$ and $G \le \nu_n$.  In particular, every infinitesimal game lies between two (possibly negative) multiples of $\uparrow$.
\end{theorem}
\begin{proof}
Take $N$ to be larger than five plus the number of positions in $G$.  We give a strategy for Left to use as the second
player in $\mu_n - G$.  The strategy is to always move to a position of the form $\mu_{m} + H$,
where $m > 5$ and $R(H) \ge 0$, until the very end.  The initial position is of this form.  From such a position,
Right can only move to $\mu_{m-1} + H$ or to $\mu_m + H^R$ for some $H$.  In the first case, we use the fact
that $L(H) \ge R(H) \ge 0$, and find a left option $H^L$ such that $R(H^L) \ge 0$.  Then
$\mu_{m-1} + H^L$ is of the desired form.  In the other case, since $R(H) \ge 0$, $L(H^R) \ge 0$, and therefore
we can find some $H^{RL}$ such that $R(H^{RL}) \ge 0$.  We make such a move.  Left uses this strategy
until $H$ becomes a number.

Note that if we follow this strategy, we (Left) never move in the $\mu_n$ component.  Therefore, the complexity
of the other component will decrease after each of our turns. By the time that $H$ becomes a number $x$, we will stil lbe in a position
$\mu_m + x$ with $m$ at least four or five, by choice of $n$.

Now by following our strategy, once $H$ becomes a number, the number will be nonnegative.  So we will be 
in a position $\mu_m + x$, where $x \ge 0$ and $m$ is at least four or five.  Either Left or Right
moved to this position.  Either way, this position is positive (because $x \ge 0$ and $\mu_m > 0$ for $m > 1$),
so therefore Left wins.

The same argument shows that $\nu_n - G$ is positive for sufficiently large $n$.

Since the positive multiples of $\uparrow$ are of the form $\mu_n$ or $\nu_n$, all sufficiently large multiples
of $\uparrow$ will exceed $G$.  By the same logic, all sufficiently large negative multiples of
$\uparrow$ (i.e.,  multiples of $\downarrow$) will be less than $G$.  So our claim is proven.
\end{proof}

We end this section by showing that some infinitesimal games are smaller than all positive
all-small games.  For any positive number $x$, let $+_x$ (pronounced ``tiny $x$'') be
the game $\{0||0|-x\} \equiv \{0|\{0|-x\}\}$.  The negative of $+_x$ is
$\{x|0||0\}$, which we denote $-_x$ (pronounced ``miny $x$'').
\begin{theorem}
For any positive number $x$, $+_x$ is a positive infinitesimal, and $+_x < G$ for any positive
all-small $G$.  Also, if $G$ is any positive infinitesimal, then $+_x < G$ for sufficiently large $x$.
\end{theorem}
\begin{proof}
It's easy to verify that $L(+_x) = R(+_x) = 0$, and that $+_x > 0$.

For the first claim, let $G$ be a positive all-small game.  We need to show that
$G + (-_x)$ is still positive. If Left goes first, she can win by making her first move
be $\{x|0\}$.  Then Right is forced to respond by moving to $0$ in this component,
or else Left can move to $x$ on her next turn, and win (because $x$ is a positive number,
so that $x$ plus any all-small game will be positive).  So Right is forced to move
to $0$.  This returns us back to $G$ alone, which Left wins by assumption.

If Left goes second, then she follows the same strategy, moving to $\{x|0\}$ at the first
moment possible.  Again, Right is forced to reply by moving $\{x|0\} \to 0$, and
the brief interruption has no effect.  The only time that this doesn't work is
if Right's first move is from $-_x$ to $0$. This leaves $G + 0$, but Left can win
this since $G > 0$.  This proves the first claim.

For the second claim, we use identical arguments, but choose $x$ to be so large
that $-x < G^* < x$ for every position $G^*$ occurring anywhere within
$G$.  Then Left's threat to move to $x$ is still strong enough to force a reply from Right.
\end{proof}

So just as the multiples of $\uparrow$ are the biggest infinitesimal games,
the games $+_x$ are the most miniscule.
\section{Nimbers and Sprague-Grundy Theory}
An important class of all-small games is the \emph{nimbers}
\[ *n = \{*0, *1, \ldots, *(n-1)\}\]
where we are using $\{A\}$ as shorthand for $\{A|A\}$.
For instance
\[ *0 = \{|\} = 0\]
\[ *1 = \{0|0\} = *\]
\[ *2 = \{0,*|0,*\}\]
\[ *3 = \{0,*,*2|0,*,*2\}\]
These games are all-small by the same criterion that made Clobber games
%and green jungles        apricot
all-small.  Note that if $m < n$, then
$*n$ has $*m$ as both a left and a right option, so $*m \lhd *n \lhd *m$.
Thus $*m || *n$.  So the nimbers are pairwise distinct, and in fact
pairwise fuzzy with each other.

There are standard shorthand notations for sums of numbers and nimbers:
\[ x*n \equiv x + *n\]
\[ x* \equiv x + *1 \equiv x + *\]
where $x$ is a number and $n \in \mathbb{Z}$.
Similarly, $7\uparrow$ means $7 + \uparrow$ and so on.This notation is usually justified
as an analog to mixed fraction notation like $5\frac{1}{2}$ for $\frac{11}{2}$.

Because nimbers are infinitesimal, we can compare expressions of this sort as follows:
$x_1*n_1 \le x_2*n_2$ if and only if $x_1 < x_2$, or $x_1 = x_2$ and $n_1 = n_2$.
We will see how to add these kind of values soon.

The nimbers are so-called because they occur in the game Nim.  In fact,
the nim-pile of size $n$ is identical to the nimber $*n$.

Nim is an example of an \emph{impartial game}, a game in which
every position is its own negative.  In other words, every left option of a position
is a right option, and vice versa.  When working with impartial
games, we can use ``option'' without clarifying whether we mean
left or right option.  The impartial games are exactly those
which can be built up by the operation $\{A,B,C,\ldots|A,B,C,\ldots\}$,
in which we require the left and right sides of the $|$ to be the same.
This is often abbreviated to $\{A,B,C,\ldots\}$, and we use this abbreviation
for the rest of the section.

Another impartial game is \emph{Kayles}.  Like Nim, it is played
using counters in groups.  However, now the counters are in rows,
and a move consists of removing one or two consecutive counters
from a row.  Doing so may split the row into two pieces.  Both
players have the same options, and as usual we play this game
using the normal play rule, where the last player to move is the winner.

\begin{figure}[htb]
\begin{center}
\includegraphics[width=4in]
					{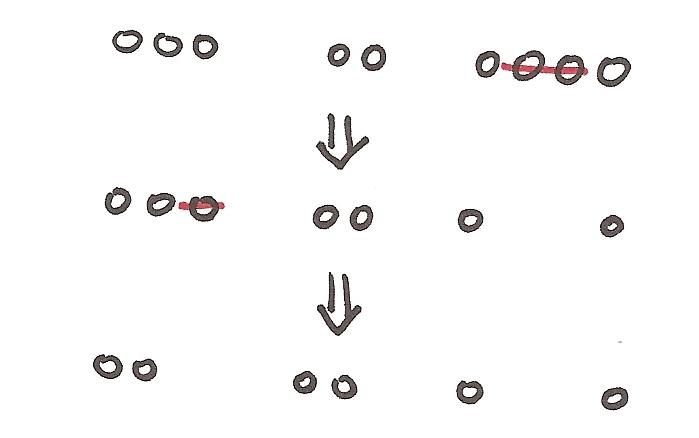}
\caption{Two moves in a game of Kayles.  Initially there are rows of length 3, 2, and 4.
The first player splits the row of 4 into rows of length 1 and 1, by removing the middle two
pieces.  The second player reduces the row of length 3 to length 2.}
\label{kayles-rules}
\end{center}
\end{figure}

So if $K_n$ denotes the Kayles-row of length $n$, then we have
\[ K_1 = \{0\} = *\]
\[ K_2 = \{0,K_1\}\]
\[ K_3 = \{K_1, K_2, K_1 + K_1\}\]
\[ K_4 = \{K_2, K_1 + K_1, K_3, K_2 + K_1\}\]
\[ K_5 = \{K_3, K_2 + K_1, K_4, K_3 + K_1, K_2 + K_2\}\]

Another similar game is \emph{Grundy's game}.  This game is played
with piles, like Nim, but rather than removing counters, the move is
to split a pile into two non-equal parts.
So if $G_n$ denotes a Grundy-heap of size $n$, then
\[ G_1 = \{\} = 0\]
\[ G_2 = \{\} = 0\]
\[ G_3 = \{G_1 + G_2\}\]
\[ G_4 = \{G_1 + G_3\}\]
\[ G_5 = \{G_1 + G_4, G_2 + G_3 \}\]
\[ G_6 = \{G_1 + G_5, G_2 + G_4 \}\]
\[ G_7 = \{G_1 + G_6, G_2 + G_5, G_3 + G_4\}\]
and so on.

%Another example is green hackenbush (?), Hackenbush Hotchpotch
%played with only green edges. apricot

The importance of nimbers is the following:
\begin{theorem}
Every impartial game equals a nimber.
\end{theorem}
\begin{proof}
Because of how impartial games are constructed,
it inductively suffices to show that if all options
of an impartial game are nimbers, then the game itself
is a nimber.  This is the following lemma:
\end{proof}
\begin{lemma}
If $a,b,c,\ldots$ are nonnegative integers, then
\[ \{*a,*b,*c,\ldots\} = *m,\]
where $m$ is the smallest nonnegative integer
not in the set $\{a,b,c,\ldots\}$, the
\emph{minimal excludent} of $a,b,c,\ldots$.
\end{lemma}
\begin{proof}
Since impartial games are their own negatives, we only need
to show that $*m \le x = \{*a,*b,*c,\ldots\}$.  This will
be true unless $*m \ge$ an option of $x$ (impossible
since $*m || *a, *b, *c, \ldots$ because $m \ne a,b,c,\ldots$),
or if $x \le$ an option of $*m$.  But by choice of
$*m$, every option of $*m$ is an option of $x$,
so $x$ is incomparable with every option of $*m$.
Thus this second case is also impossible, and so $*m \le x$.
\end{proof}

One can easily show that the sum of two impartial games
is an impartial game.  So which nimber is $*m + *n$?
\begin{lemma}
If $m,n < 2^k$, then $*m + *(2^k) = *(m + 2^k)$,
and $*m + *n = *q$ for some $q < 2^k$.
\end{lemma}
\begin{proof}
We proceed by induction on $k$.  The base case $k = 0$
is true because $*0 = 0$ and so $*0 + *0 = *0$
and $*0 + *(2^0) = *(0 + 2^0)$.

So suppose the hypothesis is true for $k$.  Let
$m,n < 2^{k+1}$.  We can write $m = m' + i2^k$
and $n = n' + j2^k$, where $i,j \in \{0,1\}$
and $m',n' < 2^k$.  Then by induction,
\[ *m = *m' + *(i2^k)\]
\[ *n = *n' + *(j2^k)\]
\[ *m + *n = *q' + *(i2^k) + *(j2^k)\]
where $*q' = *m' + *n'$, and $q' < 2^k$.
Now $i2^k$ is either $0$ or $2^k$ and similarly
for $j2^k$, so the correction term $*(i2^k) + *(j2^k)$ is either
$*0 + *0 = *0$, $*0 + *2^k = *2^k$, or $*2^k + *2^k = *0$ (using the fact that
impartial games are their own inverses).

So $*m + *n$ is either $*q'$ where $q' < 2^k < 2^{k+1}$, in which case we are done,
or $*q' + *2^k$, which by induction is $*(q' + 2^k)$.  Then we are done
since $q' + 2^k < 2^{k+1}$.

So if $m,n < 2^{k+1}$, then $*m + *n = *q$ for some $q < 2^{k+1}$.  Since addition
of games (modulo equality) is cancellative and associative, it follows that
$\{*q:q < 2^{k+1}\}$ forms a group.

It remains to show that for any $m < 2^{k+1}$, $*m + *2^{k+1} = *(m + 2^{k+1})$.
We show this by induction on $m$ ($k$ is fixed of course).  The options
of $*m + *2^{k+1}$ are of two forms:
\begin{itemize}
\item $*m + *n$ for $n < 2^{k+1}$.  Because $\{*q\,:\, q < 2^{k+1}\}$
with addition forms a group,
$\{*m + *n \,:\, n < 2^{k+1}\} = \{*n \,:\, n < 2^{k+1}\}$.
\item $*m' + *2^{k+1}$ for $m' < 2^{k+1}$.  By induction, this is just
$\{*(m' + 2^{k+1})\,:\, m' < m\}$.
\end{itemize}
So all together, the options of $*m + *2^{k+1}$ are just
\[ \{*n \,:\, n < 2^{k+1} \} \cup \{ *n \,:\, 2^{k+1} \le n < 2^{k+1} + m\} = \{*0,*1,\ldots, *(2^{k+1} + m - 1)\}.\]
Therefore the minimal excludent is $m + 2^{k+1}$, and so $*m + *2^{k+1} = *(m + 2^{k+1})$.
\end{proof}

Together with the fact that $*m + *m = 0$, we can use this to add any two nimbers:
\[ *9 + *7 = (*1 + *8) + (*3 + *4) = (*1 + *8) + (*1 + *2 + *4) =\]\[ (*1 + *1) + *2 + *4 + *8 = 0 + *2 + *4 + *8 = *6 + *8 = *14.\]
\[ *25 + *14 = (*1 + *8 + *16) + (*2 + *4 + *8) = *1 + *2 + *4 + (*8 + *8) + *16 =\]\[*1 + *2 + *4 + *16 = *23.\]
In general, the approach is to split up the summands into powers of two, and them combine and cancel out like terms.
The reader can show that the general rule for $*m + *n$ is to write
$m$ and $n$ in binary, and add without carries, to produce a number $l$ which will satisfy $*m + *n = *l$.

The number $l$ such that $*m + *n = *l$ is called the \emph{nim-sum} of $m$ and $n$, denoted
$m +_2 n$.  Here is an addition table:
\begin{center}
\begin{tabular}{|c||c|c|c|c|c|c|c|c|}
$+_2$ & 0 & 1 & 2 & 3 & 4 & 5 & 6 & 7 \\
\hline
\hline
0 & 0 & 1 & 2 & 3 & 4 & 5 & 6 & 7  \\ \hline
1 & 1 & 0 & 3 & 2 & 5 & 4 & 7 & 6  \\ \hline
2 & 2 & 3 & 0 & 1 & 6 & 7 & 4 & 5  \\ \hline
3 & 3 & 2 & 1 & 0 & 7 & 6 & 5 & 4  \\ \hline
4 & 4 & 5 & 6 & 7 & 0 & 1 & 2 & 3  \\ \hline
5 & 5 & 4 & 7 & 6 & 1 & 0 & 3 & 2  \\ \hline
6 & 6 & 7 & 4 & 5 & 2 & 3 & 0 & 1  \\ \hline
7 & 7 & 6 & 5 & 4 & 3 & 2 & 1 & 0  \\
\hline
\end{tabular}
\end{center}
Note also that sums of numbers and nimbers are added as follows: $x*n + y*m = (x+y)*(n +_2 m)$.  

Using nim-addition and the minimal-excludent rule, we can calculate values of some positions in Kayles and Grundy

\[ K_1 = \{0\} = *\]
\[ K_2 = \{0,K_1\} = \{0,*\} = *2\]
\[ K_3 = \{K_1, K_2, K_1 + K_1\} \]\[= \{*, *2, * + *\} = \{*,*2,0\} = *3\]
\[ K_4 = \{K_2, K_1 + K_1, K_3, K_2 + K_1\} \]\[= \{*2,*+*,*3,*2+*1\} = \{*2,0,*3,*3\} = *\]
\[ K_5 = \{K_3, K_2 + K_1, K_4, K_3 + K_1, K_2 + K_2\} \]\[= \{*3, *2 + *1, *, *3 + *1, *2 + *2\} = \{*3,*3,*,*2,*0\} = *4\]
\[ G_1 = \{\} = 0\]
\[ G_2 = \{\} = 0\]
\[ G_3 = \{G_1 + G_2\} = \{0 + 0\}= *\]
\[ G_4 = \{G_1 + G_3\} = \{0 + *\} = 0\]
\[ G_5 = \{G_1 + G_4, G_2 + G_3 \} \]\[ = \{0 + 0, 0 + *\} = \{0,*\} = *2\]
\[ G_6 = \{G_1 + G_5, G_2 + G_4 \}\]\[ = \{0 + *2, 0 + 0\} = \{0,*2\} = *\]
\[ G_7 = \{G_1 + G_6, G_2 + G_5, G_3 + G_4\} \]\[ = \{0 + *, 0 + *2, *\} = \{*,*2\} = 0\]
In general, there are sequences of integers $\kappa_n$ and $\gamma_n$ such that $K_n = *\kappa_n$ and
$G_n = *\gamma_n$. One can construct a table of these values, and use it to evaluate any small position
in Grundy's game or Kayles.  For the case of Kayles, this sequence is known to become
periodic after the first hundred or so values, but for Grundy's game periodicity is not yet know to occur.

The theory of impartial games is called \emph{Sprague-Grundy theory}, and was the original form of additive CGT which
Conway, Guy, Berlekamp and others extended to handle partizan games.

\chapter{Norton Multiplication and Overheating}
\section{Even, Odd, and Well-Tempered Games}\label{sec:evenandodd}
\begin{definition}
A short game $G$ is \emph{even} in form if every option is odd in form, and \emph{odd} in form
if every option is even in form and $G \ne 0$.  $G$ is \emph{even (odd) in value}
if it equals a short game that is even (odd) in form.
\end{definition}
For instance
\begin{itemize}
\item $0 = \{|\}$ is even and not odd (in form).
\item $* = \{0|0\}$ and $1 = \{0|\}$ are odd and not even (in form).
\item $2 = \{1|\}$ is even and not odd (in form).
\item $2 = \{0,1|\}$ is neither even nor odd (in form), but even (in value).
\item In general an integer is even or odd in value if it is even or odd in the usual sense.
\item $1/2 = \{0|1\}$ is neither even nor odd (in form).  By (b) of the following theorem,
it is neither even nor odd in value, too.
\end{itemize}

\begin{theorem}\label{parities}
Let $G$ be a short game.
\begin{description}
\item[(a)] If $G$ is a short game that is odd (even) in form, then the canonical form of $G$ is also odd (even) in form.
\item[(b)] $G$ is even or odd (in value) iff the canonical form of $G$ is even or odd (in form)
\item[(c)] $G$ is even or odd (in form) iff $-G$ is even or odd (in form).  $G$ is even or odd (in value) iff $-G$ is even or odd (in value).
\item[(d)] No game is both even and odd (in form or in value).
\item[(e)] The sum of two even games or two odd games is even.  The sum of an odd game and an even game is an odd game.  True for forms or values.
\item[(f)] If every option of $G$ is even (or odd) in value, and $G \ne 0$, then $G$ is odd (or even) in value.
\end{description}
\end{theorem}
\begin{proof}
\begin{description}
\item[(a)]
Let $G$ be odd or even in form.  By induction we can put all the options of $G$ in canonical form.  We can then
reduce $G$ to canonical form by bypassing reversible options and removing dominated options.  None of these operations
will introduce options of $G$ of the wrong parity: this is obvious in the case of removing dominated options,
and if, say, $G^R$ is a reversible option reversed by $G^{RL}$, then $G^R$ has the opposite parity
of $G$, and $G^{RL}$ has the same parity as $G$, so that every left option of $G^{RL}$ has opposite parity of $G$,
and can be added to the list of options of $G$ without breaking the parity constraint.  Of course since
removing dominated moves and bypassing reversible moves does not effect the value of $G$, the constraint
that $G \ne 0$ when $G$ is odd will never be broken.  So after reducing $G$ to canonical form,
it will still be odd or even, as appropriate.
\item[(b)] If $G$ is, say, odd in value, then $G = H$ for some $H$ that is odd in form.  Letting $H'$ be the canonical
form of $H$, by part (a) $H'$ is odd (in form).  Since $H'$ is also the canonical form of $G$, we see that the canonical
form of $G$ is odd (in form).  Conversely, if the canonical form of $G$ is odd (in form), then $G$ is odd in value
by definition, since $G$ equals its canonical form.
\item[(c)] Clear by induction - the definitions of even and odd are completely symmetric between the two players.
\item[(d)] We first show that no game is both even and odd in form, by induction.  Let $G$ be a short game,
and suppose it is both even and odd in form.  Then by definition, every option of $G$ is both even and odd in form.
By induction, $G$ has no options, so $G \equiv \{|\} = 0$, and then $G$ is not odd.

Next, suppose that $G$ is both even and odd in value. Then by part (b), the canonical form of $G$ is both even and odd in form,
contradicting what we just showed.
\item[(e)] We first prove the two statements about even and odd games \emph{in form}, by induction.  Suppose that both $G$
and $H$ have parities in form.  By induction, every option of $G + H$ will have the correct parity.
So we only need to show that if $G$ is odd and $H$ is even (or vice versa), then $G + H \ne 0$.  But
if $G + H = 0$, then $G = -H$, and since $H$ is even, so is $-H$, by part (c), so we have a contradiction of part (d),
since $G$ is odd and $-H$ is even.

Now suppose that $G$ and $H$ are both even or odd (in value).  Then $G = G'$ and $H = H'$, for some
games $G'$ and $H'$ having the same parities (in form) as $G$ and $H$ have (in value).  Then
$G + H = G' + H'$, and $G' + H'$ has the desired parity (in form), so $G + H$ has the desired parity (in value).
\item[(f)] For every option of $G$, there is an equivalent game having the same parity, in form.
Assembling these equivalent games into another game $H$, we have $G = H$ by Theorem~\ref{congruences}(c), and $H$ has the desired parity (in form),
so $G$ has the desired parity (in value).
\end{description}
\end{proof}

Henceforth ``even'' and ``odd'' will mean even and odd in value.  From this theorem, we see that
the even and odd values form a subgroup of $\mathcal{G}$ (the group of short games), with the even values as an index 2 subgroup.
Every even or odd game can be uniquely written as an even game plus an element of the order-2 subgroup
$\{0,*\}$, because $*$ is odd and has order 2.

Later, using Norton multiplication, we'll see that the group of (short) even games is isomorphic as a partially-ordered group
to the entire group $\mathcal{G}$ of short games.

A slight variation of even and odd is \emph{even and odd -temper}:
\begin{definition}
Let $G$ be a short game.  Then $G$ is \emph{even-tempered} in form
if it equals a (surreal) number, or every option is odd-tempered in form.
Similarly, $G$ is \emph{odd-tempered} in form if does not equal a number, and every option
is even-tempered in form.  Also, $G$ is \emph{odd- (even-)tempered in value}
if it equals a short game that is odd or even tempered in form.
\end{definition}
This notion behaves rather differently: now $0, 1, 1/2$ are all even-tempered,
while $*, 1*,$ and $\{1|0\}$ are odd-tempered, and $\uparrow$ is neither.  Intuitively,
a game is even- or odd-tempered iff a number will be reached in an even- or odd- number of moves.

We say that a game is \emph{well-tempered} if it is even-tempered or odd-tempered.

\begin{theorem}\label{tempers}
\item[(a)] If $G$ is a short game that is odd- or even-tempered in form, then the canonical form of $G$ is also odd- or even-tempered in form.
\item[(b)] $G$ is even- or odd-tempered (in value) iff the canonical form of $G$ is even- or odd-tempered (in form).
\item[(c)] Then $G$ is even- or odd-tempered (in form) iff $-G$ is even- or odd-tempered (in form).  $G$ is even- or odd-tempered (in value) iff
$-G$ is even- or odd-tempered (in value).
\item[(c')] If $G$ is even- or odd-tempered (in value), then so is $G + x$, for any number $x$.
\item[(d)] No game is both even- and odd-tempered (in form or in value).
\item[(e)] The sum of two even-tempered games or two odd-tempered games is even-tempered.  The sum of an odd-tempered game and an even-tempered game is an odd-tempered game.  True in values (not forms).
\item[(f)] If $G$ does not equal a number, and every option of $G$ is even- or odd-tempered in value,
then $G$ is odd- or even-tempered in value.
\end{theorem}
\begin{proof}
Most of the proofs are the same as in the case for even and odd games.  However, we have the following subtleties:
\begin{description}
\item[(a)] When bypassing reversible moves, we now need to check that the replacement options $G^{RLR}$ actually have the appropriate parity.
If $G$ itself equals a number,
then the parity of the bypassed options is irrelevant.  Otherwise, if $G^R$ equals a number,
then since we've reduced to canonical form, $G^{RL}$ and $G^{RLR}$ will also be numbers, so they will
have the same temper as $G^R$, and everything works out fine.

The only possible failure case is when $G^{RL}$ is a number, so every one of its left options $G^{RLR}$ is even tempered,
and $G$ and $G^R$ are even- and odd-tempered non-numerical games, respectively.  Since $G^{RL} \ge G$ (by definition of reversible move),
$G^{RL}$ must be greater than or fuzzy with every left option of $G$.  If $G^{RL}$ were additionally less than or fuzzy with every right option
of $G$, then by the simplicity rule $G$ would be a number.  So some right option $H$ of $G$ must be less than
or equal to $G^{RL}$.  This option cannot be $G^R$ itself, since $G^{RL} \lhd G^R$.  So $H$ will remain a right
option of $G$ after bypassing $G^R$.  But then $H \le G^{RL} \le G^{RLR}$ for every new option $G^{RLR}$.  Here $G^{RL} \le G^{RLR}$ because
$G^{RL}$ is a number.  So all the new moves will be dominated
by $H$ and can be immediately discarded, fixing the problem.
\item[(b-c)] These remain true for identical reasons as for even and odd games.
\item[(c')] Suppose $G$ is even tempered (in value).  Then it equals a game $H$ that is even-tempered (in form).
If $G$ equals a number, then so does $G + x$, so $G + x$ is also even-tempered (in form and value).  Otherwise,
$H$ does not equal a number, so by number avoidance
\[ H + x = \{H^L + x|H^R + x\}.\]
By induction, every $H^L + x$ and $H^R + x$ is odd-tempered in value, so by part (f), $H + x$
is even-tempered in value.

Similarly, if $G$ is odd tempered in value, then it equals a game $H$ that is odd-tempered (in form).  And since
$G$ and $H$ are not equal to numbers, neither is $G + x$, so it suffices to show that every option
of $\{H^L + x|H^R + x\}$ is even-tempered in value, which follows by induction.
\item[(d)] The proof that no game is both even- and odd-tempered in form is essentially the same:
unless $G$ is a number, $G$ can have no options by induction, and then it equals zero and is not odd.
And if $G$ is a number, then $G$ is not odd in form.  Extending this result to values
proceeds in the same way as before, using part (a).
\item[(e)] Note that this is not true in forms, since $1 + * = \{*,1|1\}$, and $*$ and $1$
are odd- and even-tempered respectively, so that $\{*,1|1\}$ is neither even- nor odd-tempered
\emph{in form}.  But it equals $\{1|1\}$ which is odd-tempered in form, so it \emph{is} odd-tempered
\emph{in value}.

We prove the result for values inductively, making use of part (f).
If $G$ and $H$ are even or odd-tempered games, then every option of $G + H$ will have the desired temper (in value),
except when one of $G$ or $H$ is a number.  But this case follows from part (c').  So the only remaining
thing we need to show is that if $G$ and $H$ are even- and odd-tempered in value, respectively,
then $G + H$ is not a number.  But if $G + H = x$ for some number $x$, then
$G = x + (-H)$, so by parts (c-c'), $x + (-H)$ is odd-tempered in value.  But then $G$ is both even-
and odd-tempered in value, contradicting part (d).
\item[(f)] This proceeds as in the previous theorem.
\end{description}
\end{proof}

So as in the previous case, even and odd-tempered values form a subgroup of $\mathcal{G}$,
with the even-tempered games as an index 2 subgroup, having $\{0,*\}$ as a complement.
But in this case, something more interesting happens: the group of all short games
is a direct sum of even-tempered games and infinitesimal games.

\begin{theorem}\label{decomposition}
Every short partizan game $G$ can be uniquely written as $E + \epsilon$,
where $E$ is even-tempered and $\epsilon$ is infinitesimal.
\end{theorem}
To prove this, we need some preliminary definitions and results:
\begin{definition}
If $G$ and $H$ are games, we say that $G$ is $H$-ish if
$G - H$ is an infinitesimal.
\end{definition}
Since infinitesimals form a group, this is an equivalence relation.
The suffix ``-ish'' supposedly stands for ``infinitesimally shifted,''
though it also refers to the fact that $G$ and $H$ are approximately equal.
For instance, they will have the same left and right stopping values.\footnote{This can be shown
easily from the fact that $L(G + H) \le L(G) + L(G)$ for short games $G$ and $H$, and related
inequalities, like $L(G + H) \ge L(G) + R(H)$.  Recall that if $\epsilon$ is infinitesimal,
then $L(\epsilon) = R(\epsilon) = 0$.}  We can rephrase Theorem~\ref{decomposition}
as saying that every short game is even-tempered-ish.

\begin{lemma}
If $G = \{A,B,\ldots|C,D,\ldots\}$ is a short game that does not equal a number,
and $A'$ is $A$-ish, $B'$ is $B$-ish, and so on, then
\[G' = \{A',B',\ldots|C',D',\ldots\}\] is $G$-ish.
\end{lemma}
\begin{proof}
Since $G$ does not equal a number, we know by number avoidance that for every positive
number $\delta$,
\[ G + \delta = \{A + \delta, B + \delta, \ldots | C + \delta, D + \delta, \ldots\}\]
But since $A'$ is $A$-ish, $A' - A \le \delta$, and $B' - B \le \delta$, and so on,
so that $A' \le A + \delta$, $B' \le B + \delta$, and so on.  Therefore,
\[ G' = \{A',B',\ldots|C',D',\ldots\} \le \{A + \delta, B + \delta, \ldots, C + \delta, D + \delta, \ldots\} = G + \delta.\]
So $G' - G \le \delta$ for every positive number $\delta$.  Similarly, $G' - G \ge \delta$ for every negative number $\delta$,
so that $G' - G$ is infinitesimal.
\end{proof}

\begin{corollary}\label{span}
For every short game $G$, there are $G$-ish even-tempered and odd-tempered games.
\end{corollary}
\begin{proof}
We proceed by induction on $G$.  If $G$ is a number, then $G$ is already even-tempered,
and $G + *$ is odd-tempered and $G$-ish because $*$ is infinitesimal.  If $G$ is not a number,
let $G = \{A,B,\ldots|C,D,\ldots\}$.  By induction, there are odd-tempered $A'$, $B'$, \ldots
such that $A'$ is $A$-ish, $B'$ is $B$-ish, and so on.  By the lemma,
\[ G' = \{A',B',\ldots|C',D',\ldots\}\]
is $G$-ish.  It is also even-tempered in value, by part (f) of Theorem~\ref{tempers},
unless $G'$ is a number.  But then it is even-tempered in form and value, by definition of even-tempered.
So either way $G'$ is even-tempered and $G$-ish.  Then as before, $G + *$ is odd-tempered and also
$G$-ish.
\end{proof}

It remains to show that $0$ is the only even-tempered infinitesimal game.
\begin{theorem}\label{blah}
Let $G$ be an even or even-tempered game.  If $R(G) \ge 0$, then $G \ge 0$.
Similarly, if $L(G) \le 0$, then $G \le 0$.
\end{theorem}
\begin{proof}
By symmetry we only need to prove the first claim.  Since right and left stopping values
depend only on value, not form, we can assume without loss of generality that
$G$ is even or even-tempered in form.  We proceed by induction.
If $G$ equals a number, then $R(G) = G$ and we are done.  Otherwise, every option of $G$
is odd or odd-tempered in form, and we have
\[ L(G^R) \ge R(G) \ge 0\]
for all $G^R$, by definition of stopping values.
We need to show that Left wins $G$ when Right goes first.  Suppose for the sake of contradiction
that Right wins, and let $G^R$ be Right's winning move.  So $G^R \le 0$.  If $G^R$ is a number,
then
\[ 0 \le L(G^R) = G^R \le 0,\]
so $G^R = 0$, contradicting the fact that $G^R$ is odd or odd-tempered.  Thus $G^R$ does not equal a number.
So again, by definition of stopping values,
\[ 0 \le L(G^R) = R(G^{RL})\]
for some left option $G^{RL}$ of $G^R$.  But then since $G^R$ is odd or odd-tempered, $G^{RL}$ is even or even-tempered,
and then by induction $0 \le G^{RL} \lhd G^R$, contradicting $G^R \le 0$.
\end{proof}

\begin{corollary}\label{nointersect}
If $G$ is an even or even-tempered game that is infinitesimal,
then $G = 0$.  If $G$ is odd or odd-tempered, then $G = *$.
\end{corollary}
\begin{proof}
Since $R(G) = L(G) = 0$, the previous theorem implies that $0 \le G \le 0$.
\end{proof}

Now we prove Theorem~\ref{decomposition}
\begin{proof}[Proof (of Theorem~\ref{decomposition})]
By Corollary~\ref{span}, we know that every short game $G$ can be written as the sum of
an even-tempered game and an infinitesimal game.  By Corollary~\ref{nointersect} the group of
even-tempered games has trivial intersection with the group of infinitesimal games.  So we are done.
\end{proof}
Therefore for every short game $G$, there is a unique $G$-ish even-tempered game.
%There was once a hope of assigning to each short game $G$ a reduced canonical form, that would be the
%simplest form of any $G$-ish game.  While the reduced canonical form turned out to exist, the
%map sending $G$ to its reduced canonical form's value isn't a homomorphism\footnote{According to some
%paper online and maybe even in one of the MSRI books, the original person who defined reduced canonical form got this wrong.},
%so the reduced canonical forms' values do not form a complement to the group of infinitesimal
%games.  Even-tempered games will do, but aren't as nice - for instance, the $\{1|-1\}$-ish even-tempered
%game is $\{1*|-1*\}$, which seems more complicated than $\{1|-1\}$.
We can also draw another corollary from Theorem~\ref{blah}
\begin{theorem}
If $G$ is any even or odd game (in value), then $L(G)$ and $R(G)$ are integers.
%and $G \ge 0 \iff R(G) > -1$ and $G \le 0 \iff L(G) < 1$. <- only if G is even.
\end{theorem}
\begin{proof}
If $G$ is any short game, then the values $L(G)$ and $R(G)$ actually occur
(as surreal numbers) within $G$.  So if $G$ is even or odd in form,
then $L(G)$ and $R(G)$ must be even or odd (not respectively) in value,
because every subposition of an even or odd game is even or odd (not respectively).
So to show that $L(G)$ and $R(G)$ are integers, it suffices to show
that every (surreal) number which is even or odd (in value) is an integer.

Suppose that $x$ is a short number which is even or odd in value, and $x$ is not an integer.  Then $x$ corresponds
to a dyadic rational, so some multiple of $x$ is a half-integer.  Since the set of even and odd games
forms a group, and since it contains the integers, it follows that $\frac{1}{2}$ must be
an even or an odd game.  But then $\frac{1}{2} + \frac{1}{2} = 1$ would be even,
when in fact it is odd.  Therefore $L(G)$ and $R(G)$ must be integers.
\end{proof}

\section{Norton Multiplication}\label{sec:norton}
If $H$ is any game, we can consider the multiples of $G$
\[ \ldots,~~(-2).H = -H-H,~~(-1).H = -H,~~0.H = 0,\]
\[ 1.H = H,~~2.H = H + H,~~3.H = H + H + H,~~\ldots\]
The map sending $n \in \mathbb{Z}$ to $G + G + \cdots + G$ ($n$ times, with obvious allowances for $n \le 0$)
establishes a homomorphism from
$\mathbb{Z}$ to $\mathcal{G}$.  If $G$ is positive, then the map is injective
and strictly order-preserving.  In this case, Simon Norton found a way to extend
the domain of the map to all short partizan games.  Unfortunately this definition depends on the form of $G$ (not just its value),
and doesn't have many of the properties that we expect from multiplication, but it does provide a good collection
of endomorphisms on the group of short partizan games.

We'll use Norton multiplication to prove several interesting results:
\begin{itemize}
\item If $G$ is any short game, then there is a short game $H$ with $H + H = G$.
By applying this to $*$, we get torsion elements of the group of games $\mathcal{G}$ having order
$2^k$ for arbitrary $k$.
\item The partially-ordered group of even games is isomorphic to the group
of all short games, and show how to also include odd games into the mix.
\item The group of all-small games contains a complete copy of the group
of short partizan games.
\item Later on, we'll use it to relate scoring games to $\mathcal{G}$.
\end{itemize}

The definition of Norton multiplication is very ad-hoc, but works nevertheless:
\begin{definition}
(Norton multiplication) % apricot some stuff about non-short games... could be used to get torsion elements of all degrees for \mathbf{Pg} in general.
Let $H$ be a positive short game.  For $n \in \mathbb{Z}$, define
$n.H$ to be $\underbrace{H + H + \cdots + H}_{\textrm{$n$ times}}$ if $n \ge 0$
or $-(\underbrace{H + H + \cdots + H}_{\textrm{$-n$ times}})$ if $n \le 0$.
If $G$ is any short game, then \emph{$G$ Norton multiplied by $H$ }(denoted $G.H$) is $n.H$ if $G$ equals an integer $n$.
and otherwise is defined recursively as 
\[ G.H \equiv \{G^L.H + H^L, G^L.H + 2H - H^R | G^R.H - H^L, G^R.H - 2H + H^R\}.\]
\end{definition}
To make more sense of this definition, note that $H^L$ and $2H - H^R$ can be
rewritten as $H + (H^L - H)$ and $H + (H - H^R)$.  The expressions $H^L - H$
and $H - H^R$ are called \emph{left and right incentives} of $H$,
since they measure how much Left or Right gains (improves her situation)
by making the corresponding option.  Unfortunately, incentives can never
be positive, because $H^L \lhd H \lhd H^R$ for every $H^L$ and $H^R$.

For instance, if $H \equiv \uparrow \equiv \{0|*\}$, then the left incentive
is $0 - \uparrow = \downarrow$, and the right incentive is $\uparrow - * = \uparrow*$.
Since $\uparrow* \ge \downarrow$, the options of the form
$G^L.H + H^L$ will be dominated by $G^L.H + 2H - H^R$ in this case, and we
get
\[ G.\uparrow \equiv \{G^L.\uparrow + \Uparrow* | G^R.\uparrow + \Downarrow*\}\]
when $G$ is not an integer.  Sometimes $G.\uparrow$ is denoted as $\hat{G}$.

Another important example is when $H \equiv 1 + * \equiv 1* \equiv \{1|1\}$. Then the incentives for both players are $1* - 1 = * = 1 - 1*$,
so $H + (H^L - H)$ and $H + (H - H^R)$ are both $1* + * = 1$.  So when $G$ is not an integer,
\[ G.(1*) \equiv \{ G^L.(1*) + 1 | G^R.(1*) - 1\}.\]

In many cases, Norton multiplication is an instance of the general \emph{overheating} operator
\[ \int_G^H K,\]
defined to be $K.G$ if $K$ equals an integer, and
\[ \{H + \int_G^H K^L | -H + \int_G^H K^R\}\]
otherwise.  For example, $\int_{\uparrow}^{\Uparrow*}$ is Norton multiplication by $\uparrow \equiv \{0|*\}$,
and $\int_{1*}^1$ is Norton multiplication by $\{1|1\}$.  Unfortunately, overheating is sometimes ill-defined
modulo equality of games.

We list the important properties of Norton multiplication in the following theorem:
\begin{theorem}\label{nortonstuff}
For every positive short game $A$,
the map $G \to G.A$ is a well-defined an order-preserving endomorphism on the group
of short-games, sending $1$ to $A$.  In other words
\[ (G + H).A = G.A + H.A\]
\[ (-G).A = -(G.A)\]
\[ 1.A = A\]
\[ 0.A = 0\]
\[ G = H \implies G.A = H.A\]
\[ G \ge H \iff G.A \ge H.A\]
\end{theorem}
Of course the last of these equations also implies that
$G < H \iff G.A < H.A$, $G \lhd H \iff G.A \lhd H.A$, and so on.

These identities show that $G.A$ depends only on the value of $G$.  But as a word
of warning, we note that \textbf{$G.A$ depends on the form of $A$}.  For instance,
it turns out that
\[ \frac{1}{2}.\{0|\} = \frac{1}{2},\]
while
\[ \frac{1}{2}.\{1*|\} = \{1|0\} \ne \frac{1}{2}\]
although $\{0|\} = 1 = \{1*|\}$.  By default, we will interpret $G.A$ using
the canonical form of $A$, when the form of $A$ is left unspecified.

Before proving Theorem~\ref{nortonstuff}, we use it to show some of the claims above:
\begin{corollary}\label{hats}
The map sending $G \to G.\uparrow$ is an order-presering embedding
of the group of short partizan games into the group of short all-small games.
\end{corollary}
\begin{proof}
We only need to show that $G.\uparrow$ is always all-small.  Since all-small
games are closed under addition, this is clear when $G$ is an integer.  In any
other case, $G$ has left and right options, so $G.\uparrow$ does too.
Moreover, the left options of $G.\uparrow$ are all of the form
$G^L.\uparrow + \Uparrow*$ and $G^L.\uparrow + 0$ (because $\uparrow = \{0|*\}$)
and by induction (and the fact that $\Uparrow*$ is all-small), all the left options
of $G.\uparrow$ are all-small.  So are all the right options.  So
every option of $G.\uparrow$ is all-small, and $G.\uparrow$ has options
on both sides.  Therefore $G.\uparrow$ is all-small itself.
\end{proof}

\begin{corollary}\label{halfpres}
The map sending $G \to G.(1*)$ is an order-preserving embedding
of the group of short partizan games into the group of (short) even games.
\end{corollary}
In fact, we'll see that this map is bijective, later on.
\begin{proof}
As before, we only need to show that $G.(1*)$ is even, for any $G$.
Since $1* = \{1|1\}$ is even, and even games are closed under addition and subtraction,
this is clear when $G$ is an integer.  Otherwise, note that
\[ G.(1*) = \{G^L.(1*) + 1|G^R.(1*) - 1\}\]
and by induction $G^L.(1*)$ is even and $G^R.(1*)$ is even, so that
$G^L.(1*) + 1$ and $G^R.(1*) - 1$ are odd (because $1$ and $-1$ are odd).
Thus every option of $G.(1*)$ is odd, and so $G.(1*)$ is even as desired.
\end{proof}

\begin{corollary}\label{twodivide}
If $G$ is any short game, then there is a short game $H$ such that
$H + H = G$.  If $G$ is infinitesimal or all-small, we can take $H$ to be likewise.
Either way, we can take $H$ to have the same sign (outcome) as $G$.
\end{corollary}
\begin{proof}
Since every short game is greater than some number, $G + 2n$ will be positive
for big enough $n$.  Let $H = (1/2).(G + 2n) - n$.  Then
\[ H + H = (1/2).(G + 2n) + (1/2).(G + 2n) - n - n =\]\[ (1/2 + 1/2).(G + 2n) - 2n = G + 2n - 2n = G.\]
If $G$ is infinitesimal, we can replace $2n$ with $\hat{2n}$ (i.e., $2n.\uparrow$),
since we know that every infinitesimal is less than some multiple of $\uparrow$.  Then
$G + \hat{2n}$ will be infinitesimal or all-small, as $G$ is, so $H = (1/2).(G + \hat{2n}) - \hat{n}$ will be
infinitesimal or all-small, by the following lemma:
\begin{lemma}
If $K$ is infinitesimal and positive, then $G.K$ is infinitesimal for every short game $G$.  Similarly
if $K$ is all-small, then $G.K$ is all-small too.
\end{lemma}
\begin{proof}
The all-small case proceeds as in Corollary~\ref{hats}, using the fact that the incentives
of $K$ will be all-small because $K$ and its options are, and all-small games form a group.

If $K$ is merely infinitesimal, then notice that since every short game is less than an integer,
there is some large $n$ for which $-n < G < n$, and so $-n.K < G.K < n.K$.  But since
$K$ is an infinitesimal, $n.K$ is less than every positive number and $-n.K$ is greater than
every negative number. Thus $G.K$ also lies between the negative and positive numbers,
so it is infinitesimal.
\end{proof}

To make the signs come out right, note that if $G = 0$, then we can trivially take $H = 0$.
If $G || 0$, then any $H$ satisfying $H + H = G$ must satisfy $H || 0$, since $H \ge 0 \Rightarrow H + H \ge 0$,
$H = 0 \Rightarrow H + H =0$, and $H \le 0 \Rightarrow H + H \le 0$.  So the $H$ chosen above works.
If $G > 0$, then we can take $n = 0$.  So $H = (1/2).G$ which is positive by Theorem~\ref{nortonstuff}.
If $G$ is negative, then by the same argument applied to $-G$, we can find $K > 0$ such that $K + K = -G$.
Then take $H = -K$.
\end{proof}

We now work towards a proof of Theorem~\ref{nortonstuff}.%\footnote{I should probably find a theorem
%since I made this one up from scratch, and there are likely better ways of showing it.}
\begin{lemma}\label{negation}
$(-G).H \equiv -(G.H)$
\end{lemma}
\begin{proof}
This is easily proven by induction.  If $G$ equals an integer, then it is obvious by definition
of Norton multiplication.  Otherwise,
\[ (-G).H \equiv \{(-G)^L.H + H^L, (-G)^L.H + 2H - H^R|\]\[(-G)^R.H - H^L, (-G)^R.H - 2H + H^R\}
\]\[\equiv \{(-(G^R)).H + H^L, (-(G^R)).H + 2H - H^R|\]\[(-(G^L)).H - H^L, (-(G^L)).H - 2H + H^R\}
\]\[\equiv \{-(G^R.H) + H^L, -(G^R.H) + 2H - H^R|\]\[-(G^L.H) - H^L, -(G^L.H) - 2H + H^R\}
\]\[\equiv \{-(G^R.H - H^L), -(G^R.H - 2H + H^R)|\]\[-(G^L.H + H^L), -(G^L.H + 2H - H^R)\}
\equiv\]\[ -\{G^L.H + H^L, G^L.H + 2H - H^R|\]\[G^R.H - H^L, G^R.H - 2H + H^R\}\]\[ \equiv -(G.H)\]
where the third identity follows by induction.
\end{proof}

The remainder is more difficult.  We'll need the following variant of number-avoidance
\begin{theorem}\label{intavoid}(Integer avoidance)
If $G$ is a short game that does not equal an integer, and $n$ is an integer,
then
\[ G + n = \{G^L + n|G^R + n\}\]
\end{theorem}
\begin{proof}
If $G$ does not equal a number, this is just the number avoidance theorem.  Otherwise, let $S$
be the set of all numbers $x$ such that $G^L \lhd x \lhd G^R$ for all
$G^L$ and $G^R$, and let $S'$ be the set of all numbers $x$ such that $G^L + n \lhd x \lhd G^R + n$
for all $G^L$ and $G^R$.  By the simplicity rule, $G$ equals the simplest number in $S$,
and $\{G^L + n|G^R + n\}$ equals the simplest number in $S'$.  But the elements of $S'$
are just the elements of $S$ shifted by $n$, that is $S' = \{s + n\,:\, s \in S\}$.
Let $x = G$ and $y = \{G^L + n|G^R + n\}$.  We want to show $y = x + n$, so suppose otherwise.
Then $x$ is simpler than $y - n \in S$, and $y$ is simpler than $x + n \in S'$.  Because
of how we defined simplicity, adding an integer to a number has no effect on how simple it is
unless a number is an integer.  So either $x$ or $y$ is an integer.  If $x = G$ is an integer
then we have a contradiction, and if $y$ is an integer, the fact that $x$ is simpler
than $y - n$ implies that $x$ is an integer too.
\end{proof}
The name integer avoidance comes from the following reinterpretation:
\begin{lemma}\label{explicit} % apricot move this to the number avoidance section
Let $G_1, G_2, \ldots, G_n$ be a list of short games.  If at least one $G_i$ does not equal an integer,
and $G_1 + G_2 + \cdots + G_n \rhd 0$, then there is some $i$ and some left option
$(G_i)^L$ such that $G_i$ does not equal an integer, and
\[ G_1 + \cdots + G_{i-1} + (G_i)^L + G_{i+1} + \cdots + G_n \ge 0\]
\end{lemma}
(If we didn't require $G_i$ to be a non-integer, this would be obvious from the fact
that some left option of $G_1 + G_2 + \cdots + G_n$ must be $\ge 0$.)
\begin{proof}
Assume without loss
of generality that we've sorted the $G_n$ so that $G_1,\ldots,G_j$ are all non-integers,
while $G_{j+1},\ldots,G_n$ are all integers.  (We don't assume that $j < n$, but we do assume that $j > 0$.)  Then
we can write $G_{j+1} + \cdots + G_n = k$ for some integer $k$, which will be the empty sum zero if $j = n$.  By integer avoidance,
\[ 0 \lhd G_1 + G_2 + \cdots + G_n = G_1 + \cdots + G_{j-1} + \{(G_j)^L + k|(G_j)^R + k\}\]
Therefore, there is some left option of $G_1 + \cdots + G_{j-1} + \{(G_j)^L + k|(G_j)^R + k\}$ which is $\ge 0$.
There are two cases: it is either of the form
\[ G_1 + \cdots + G_{i-1} + (G_i)^L + G_{i+1} + \cdots + G_{j-1} + \{(G_j)^L + k|(G_j)^R + k\}\]
for some $i < j$, or it is of the form
\[ G_1 + \cdots + G_{j-1} + (G_j)^L + k.\]
In the first case, we have
\[ 0 \le G_1 + \cdots + G_{i-1} + (G_i)^L + G_{i+1} + \cdots + G_{j-1} + \{(G_j)^L + k|(G_j)^R + k\}
 =\]\[ G_1 + \cdots + G_{i-1} + (G_i)^L + G_{i+1} + \cdots  + G_j + k = \]\[
G_1 + \cdots + G_{i-1} + (G_i)^L + G_{i+1} + \cdots + G_n\]
and $G_i$ is not an integer.  In the second case, we have
\[ 0 \le G_1 + \cdots + G_{j-1} + (G_j)^L + k = G_1 + \cdots + G_{j-1} + (G_j)^L + G_{j+1} + \cdots + G_n\]
and $G_j$ is not an integer.
\end{proof}
This result says that in a sum of games, not all integers, whenever you have a winning move, you have one in
a non-integer.  In other words, you never need to play in an integer if any non-integers are present on the board.
%
%Using integer avoidance, we can show something about incentives:
%\begin{lemma}
%Let $G$ be a game that does not equal an integer.  Then there is some left option
%$G^L$ with $G^L - G \ge -1$ or there is some right option $G^R$ with
%$G - G^R \ge -1$.  In other words, some incentive is at least $-1$.
%\end{lemma}
%\begin{proof}
%By integer-avoidance (Theorem~\ref{intavoid}), $G + 1 = \{G^L + 1|G^R + 1\}$.  Now
%$G + 1 - G > 0$, so there is some left option of
%\[ \{G^L + 1|G^R + 1\} - G\]
%which is $\ge 0$.  The left options of this game are of the form
%\[ G^L + 1 - G\]
%and
%\[ \{G^L + 1|G^R + 1\} - G^R = G + 1 - G^R\]
%So we either have $G^L + 1 - G \ge 0$, or $G + 1 - G^R \ge 0$, as desired.
%\end{proof}
%This fails if $G$ is an integer; for example every incentive of $\{-20|20\} = 0$ is
%at most $-20$.
%
%Hm, maybe I can't actually use that lemma.  I think I wanted incentives
%on both sides to be at least $-1$, but I don't know how to prove that, or I think
%it's not proven via integer avoidance.
% apricot What Happen?  Someone set us up the bomb.

Using this, we turn to our most complicated proof:
\begin{lemma}\label{triplesum}
Let $H$ be a positive short game and $G_1, G_2, G_3$ be short games.  Then
\[ G_1 + G_2 + G_3 \ge 0 \implies G_1.H + G_2.H + G_3.H \ge 0\]
\end{lemma}
\begin{proof}
If every $G_i$ equals an integer, then the claim follows easily from the definition of Norton multiplication.
Otherwise, we proceed by induction on the combined complexity of the \emph{non-integer} games
among $\{G_1, G_2, G_3\}$.

We need to show that Left has a good response to any Right option of $G_1.H + G_2.H + G_3.H$.
So suppose that Right moves in some component, $G_1.H$ without loss of generality.  We have several cases.

Case 1: $G_1$ is an integer $n$.
In this case, $(n+1) + G_2 + G_3 \ge 1 > 0$, and $G_2$ and $G_3$ are not both equal to integers,
so we can assume without loss of generality (by integer avoidance), that a winning left option
in $(n+1) + G_2 + G_3$ is in $G_2$, and $G_2$ is not an integer.  That is,
$G_2$ is not an integer and $(n+1) + (G_2)^L + G_3 \ge 0$ for some $G_2^L$.  By induction,
we get
\begin{equation} (n+1).H + (G_2)^L.H + G_3.H \ge 0 \label{handy}\end{equation}
Now we break into cases according to the sign of $n$.

Case 1a: $n = 0$.  Then $G_1.H \equiv 0$ so right could not have possibly moved in $G_1.H$.

Case 1b: $n > 0$.  Then $G_1.H \equiv n.H \equiv \underbrace{H + H + \cdots + H}_{\textrm{$n$ times}}$,
so that the right options of $G_1.H$ are all of the form $(n-1).H + H^R$.  If Right moves from $G_1.H = n.H$
to $(n-1).H + H^R$, we (Left) reply with a move from $G_2.H$ to $(G_2)^L.H + 2H - H^R$, which is legal
because $G_2$ is not an integer.  This leaves us in the position
\[ (n-1).H + H^R + (G_2)^L.H + H + H - H^R + G_3.H = (n+1).H + (G_2)^L.H + G_3.H \ge 0\]
using (\ref{handy}).

Case 1c: $n < 0$.  Then similarly, the right options of $n.H$ are all of the form
$(n+1).H - H^L$.  We reply to such a move with a move from $G_2.H$ to $(G_2)^L.H + H^L$, resulting in
\[ (n+1).H - H^L + (G_2)^L.H + H^L + G_3.H = (n+1).H + (G_2)^L.H + G_3.H \ge 0\]
using (\ref{handy}) again.

Case 2: $G_1$ is not an integer.  Then the right options of $G_1.H$ are of the form
$(G_1)^R.H - H^L$ and $(G_1)^R.H - 2H + H^R$.  We break into cases according
to the nature of $(G_1)^R + G_2 + G_3$, which is necessarily $\rhd 0$ because
$G_1 + G_2 + G_3 \ge 0$.

Case 2a: All of $(G_1)^R$, $G_2$, and $G_3$ are integers. Then
$(G_1)^R + G_2 + G_3 \rhd 0 \implies (G_1)^R + G_2 + G_3 \ge 1$.  After Right's move,
we will either be in
\[ (G_1)^R.H - H^L + G_2.H + G_3.H\]
or
\[ (G_1)^R.H - 2H + H^R + G_2.H + G_3.H\]
But since $(G_1)^R$, $G_2$, and $G_3$ are all integers, we can rewrite these possibilities
as
\[ m.H - H^L\]
and
\[ m.H - 2H + H^R\]
where $m = (G_1)^R + G_2 + G_3$ is an integer at least 1.  But since $m \ge 1$, we have
\[ m.H - H^L \ge H - H^L \rhd 0\]
and
\[ m.H - 2H + H^R \ge H - 2H + H^R = H^R - H \rhd 0\]
so Right's move in $G_1.H$ was bad.

Case 2b: Not all of $(G_1)^R$, $G_2$, and $G_3$ are integers.  Letting $\{A,B,C\} = \{(G_1)^R, G_2, G_3\}$,
we find outselves in a position
\begin{equation} A.H - H^L + B.H + C.H \label{firstcase}\end{equation}
or
\begin{equation} A.H - 2H + H^R + B.H + C.H \label{secondcase}\end{equation}
and we know that not all of $A,B,C$ are integers, and $A + B + C \rhd 0$.  By integer
avoidance, there is some winning left option in one of the non-integers.  Without loss of generality,
$A$ is not an integer and $A^L + B + C \ge 0$.  Then by induction,
\[ A^L.H + B.H + C.H \ge 0.\]
Now, if we were in situation (\ref{firstcase}), we move from $A.H$ to $A^L.H + H^L$, producing
\[ A^L.H + H^L - H^L + B.H + C.H = A^L.H + B.H + C.H \ge 0\]
while if we were in situation (\ref{secondcase}), we move from $A.H$ to $A^L.H + 2H - H^R$, producing
\[ A^L.H + 2H - H^R - 2H + H^R + B.H + C.H = A^L.H + B.H + C.H \ge 0\]
So in this case, we have a good reply, and Right's move could not have been any good.

So no matter how Right plays, we have good replies.
\end{proof}

Using this, we prove Theorem~\ref{nortonstuff}
\begin{proof}[Proof (of Theorem~\ref{nortonstuff})]
$1.A = A$ and $0.A = 0$ are obvious, and $(-G).A = -(G.A)$ was Lemma~\ref{negation}.  The implication
$G = H \implies G.A = H.A$ follows from the last line $G \ge H \iff G.A \ge H.A$, so we only need
to show
\begin{equation} G \ge H \iff G.A \ge H.A\label{order-preserving}\end{equation}
and
\begin{equation} (G + H).A = G.A + H.A.\label{additive0}\end{equation}
We use Lemma~\ref{triplesum} for both of these.  First of all, suppose that $G \ge H$.  Then
$G + (-H) + 0 \ge 0$, so by Lemma~\ref{triplesum}, together with Lemma~\ref{negation},
\[ G.A + (-H).A + 0.A \equiv G.A - H.A \ge 0.\]
Thus $G \ge H \implies G.A \ge H.A$.  Similarly, if $G$ and $H$ are any games,
$G + H + (-(G + H)) \ge 0$ and $(-G) + (-H) + (G + H) \ge 0$, so that
\[ G.A + H.A + (-(G+H)).A \equiv G.A + H.A - (G+H).A \ge 0\]
and
\[ (-G).A + (-H).A + (G + H).A \equiv -G.A - H.A + (G + H).A \ge 0.\]
Combining these shows (\ref{additive0}).  It remains to show the $\Leftarrow$ direction of (\ref{order-preserving}).

Then suppose that $G \not\ge H$, i.e., $G \lhd H$.  Then $H - G \rhd 0$, and
\[ (H - G).A = (H + (-G)).A = H.A + (-G).A = H.A - G.A.\]
If we can similarly show that $H.A - G.A \rhd 0$, when we'll have shown $G.A \not\ge H.A$, as desired.  So
it suffices to show that if $K \rhd 0$, then $K.A \rhd 0$.

We show this by induction on $K$.  If $K$ is an integer, this is obvious, since $K \rhd 0 \implies K \ge 1 \implies
K.A \ge A > 0$.  Otherwise, $K \rhd 0$ implies that some $K^L \ge 0$.  Then by the $\Rightarrow$ direction of (\ref{order-preserving}),
\[ K^L.A \ge 0\]
so that
\[ K^L.A + A^L \ge 0\]
if $A^L \ge 0$.  Such an $A^L$ exists because $A > 0$.
\end{proof}

\section{Even and Odd revisited}\label{sec:even-and-odd-revisited}
Now we show that the map sending $G$ to $G.(1*)$ is onto the even games, showing that the
short even games are isomorphic as a partially-ordered group to the whole group of short games.

\begin{lemma}\label{slippery}
For $G$ a short game, $G.(1*) \ge * \iff G \ge 1 \iff G.(1*) \ge 1*$.
Similarly, $G.(1*) \le * \iff G \le -1 \iff G.(1*) \le -1*$.
\end{lemma}
\begin{proof}
We already know that $G \ge 1$ iff $G.(1*) \ge 1*$, since
$1* = 1.1*$ and Norton multiplication by $1*$ is strictly order-preserving.

It remains to show that $G.(1*) \ge * \iff G \ge 1$.  If
$G$ is an integer, this is easy, since every positive multiple of $1*$
is greater than $*$ (as $*$ is an infinitesimal so $x$ and $x + *$ are greater than $*$
for positive numbers $x$), but $0.(1*) = 0 \not \ge *$.

If $G$ is not an integer, then $G.(1*) \equiv \{G^L.(1*) + 1 | G^R.(1*) - 1\}$, so by
Theorem~\ref{betwixt} we have $* \le G.(1*)$
unless and only unless $G.(1*) \le *^L = 0$ or $G^R.(1*) - 1 \le *$.
So $* \le G.(1*)$ unless and only unless $G.(1*) \le 0$ or some $G^R$ has $G^R.(1*) \le 1*$.
Because Norton multiplication with $1*$ is order-preserving, we see
that $* \le G.(1*)$ unless and only unless $G \le 0$ or some $G^R \le 1$.  This
is exactly the conditions for which $1 \not\le G$.  So
$* \le G.(1*)$ iff $1 \le G$.
\end{proof}

\begin{lemma}\label{goodenough}
If $G$ is a short game
and
\[ G.(1*) \ne \{G^L.(1*) + 1|G^R.(1*) - 1\},\]
then there is some integer $n$ such that $G^L \lhd n$ and $n + 1 \lhd G^R$ for
every $G^L$ and $G^R$.
\end{lemma}
In other words, the recursive definition of Norton multiplication works even when $G$ is an integer,
except in some bad cases.  Another way of saying this is that
as long as there is no more than one integer $n$ such that $G^L \lhd n \lhd G^R$,
then the recursive definition of $G.(1*)$ works.
\begin{proof}
Suppose that there is no integer $n$ such that $G^L \lhd n$ and $n + 1 \lhd G^R$
for all $G^L$ and $G^R$.  Then we want to show
that
\begin{equation} G.(1*) = \{G^L.(1*) + 1|G^R.(1*) - 1\}.\label{goal2}\end{equation}
This is obvious if $G$ does not equal an integer, so suppose that $G = m$ for some
integer $m$.  Then $G^L \lhd m \lhd G^R$ for every $G^L$ and $G^R$.
If $G^L.(1*) + 1 \ge G.(1*)$, then $G^L.(1*) - G.(1*) + 1* \ge *$, so by the
previous lemma $G^L - G + 1 \ge 1$, so $G^L \ge G$, contradicting $G^L \lhd G$.
Thus $G^L.(1*) + 1 \lhd G.(1*)$ for every $G^L$, and similarly one can show
$G.(1*) \lhd G^R.(1*) - 1$ for every $G^R$.  So by the gift-horse principle,
we can add the left options $G^L.(1*) + 1$ and the right options
$G^R.(1*) - 1$ to any presentation of $G.(1*)$.

Since $G = m$, one such presentation is $(1*) + \cdots + (1*)$ ($m$ times), which is
\[ \{(m-1).(1*) + 1 | (m- 1).(1*) + 1\}\]
This produces the presentation
\[ \{(m-1).(1*) + 1, G^L.(1*) + 1|(m-1).(1*) + 1, G^R.(1*) - 1\}\]
I claim that we can remove the old options $(m-1).(1*) + 1$ and $(m-1).(1*) + 1$ as dominated moves,
leaving behind (\ref{goal2}).

By assumption, $m$ is the only integer satisfying $G^L \lhd m \lhd G^R$, $\forall G^L, \forall G^R$.
Since $m - 1 \lhd G^R$ for all $G^R$, it must be the case that $G^L \ge m - 1$ for some $G^L$, or else $G^L \lhd m - 1 \lhd G^R$ would hold.
Then $G^L.(1*) \ge (m - 1).(1*)$, so $(m-1).(1*) + 1$ is dominated by $G^L.(1*) + 1$.
Similarly, $G^L \lhd m + 1$ for all $G^L$, so some $G^R$ must satisfy $G^R \le m + 1$.
Then $G^R.(1*) \le (m + 1).(1*) = (m-1).(1*) + 2$.  So $(m-1).(1*) + 1 \ge G^R.(1*) - 1$,
and $(m-1).(1*) + 1$ is dominated by $G^R.(1*) - 1$.  So after removing dominated
moves, we reach (\ref{goal2}), the desired form.
\end{proof}

\begin{lemma}\label{surge}
Every (short) even game $G$ equals $H.(1*)$ for some short game $H$.
\end{lemma}
\begin{proof}
We need induction that works in a slightly different way.

Recursively define the following sets of short games:
\begin{itemize}
\item $A_0$ contains all short games which equal numbers.
\item $A_{n+1}$ contains $A_n$ and all short games whose options
are all in $A_n$.
\end{itemize}
Note that $A_0 \subseteq A_1 \subseteq A_2 \subseteq \cdots$.

We first claim that $\cup_{n = 1}^\infty A_n$ is the set of all short games.  In other
words, every short game $G$ belongs to some $A_n$.  Proceeding by induction on $G$,
if $G$ is an integer, then $G \in A_0$, and otherwise, we can assume by induction
and shortness of $G$ that there is some $n$ such that every option of $G$ is in $A_n$,
so that $G$ itself is in $A_{n+1}$.

Next we claim that the sets $A_n$ are somewhat invariant under translation by integers.
Specifically, if $G \in A_n$ and $m$ is an integer, then $G +m = H$ for some $H \in A_n$.
We show this by induction on $n$.  If $n = 0$, this is obvious, since the integers are closed
under addition.  Now supposing that the hypothesis holds for $A_n$, let $G \in A_{n+1}$ and $m$ be
an integer.  If $G$ equals an integer, then $G + m$ does too, so $G + m$ equals an element
of $A_0 \subseteq A_{n+1}$ and we are done.  Otherwise, by integer avoidance
$G + m$ equals $\{G^L + m|G^R + m\}$.  By induction every $G^L + m$ and every $G^R + m$ equals an element
of $A_n$.  So $\{G^L + m|G^R +m\} = \{H^L|H^R\}$ for some $H^L, H^R \in A_n$.  Then
$H = \{H^L|H^R\} \in A_{n+1}$, so $G +m$ equals an element of $A_n$.

Next, we show by induction on $n$ that if $G$ is even in form
and $G \in A_n$, then $G = H.(1*)$ for some $H$.  If $n = 0$, then $G$ is an integer.  Since integers
are even as games if and only if they are even in the usual sense, $G = 2m = (2m).(1*)$ (because $*$ has order two,
so $1* + 1* = 2$).  For the inductive step, suppose
that the result is known for $A_n$, and $G \in A_{n+1}$.  Then by definition
of ``even'' and $A_{n+1}$, every option of $G$ is odd and in $A_n$.
So if $G^L$ is any left option of $G$, then $G^L - 1$ will be \emph{even},
and will equal some $X \in A_n$, by the previous paragraph.  By induction,
$X = H^L.(1*)$ for some $H^L$.  We can carry this out for every left option
of $G$, so that every $G^L$ is of the form $H^L.(1*) + 1$.  Similarly
we can choose some games $H^R$ such that the set of $G^R$
is the set of $(H^R).(1*) - 1$.  Thus
\[ G \equiv \{G^L|G^R\} = \{H^L.(1*) + 1|H^R.(1*) - 1\}\]
We will be done with our inductive step by Lemma~\ref{goodenough}, unless
there is some integer $n$ such that $H^L \lhd n$ and $n +1 \lhd H^R$
for every $H^L$ and $H^R$.  Now by the order-preserving property
of Norton multiplication, and Lemma~\ref{slippery}
\[ H^L \lhd n \iff H^L.(1*) \not \ge n.(1*) \iff
H^L.(1*) - n.(1*) + 1* \not \ge 1* \iff \]\[
H^L.(1*) - n.(1*) + 1* \not \ge * \iff H^L.(1*) + 1 \lhd n.(1*).\]
Similarly,
\[ n + 1 \lhd H^R \iff (n + 1).(1*) \not \ge H^R.(1*)
\iff (n + 1).(1*) - H^R.(1*) + 1* \not \ge 1* \iff \]\[
(n + 1).(1*) - H^R.(1*) + 1* \not \ge * \iff
(n + 1).(1*) \lhd H^R.(1*) - 1.\]
So it must be the case that
\[ H^L.(1*) + 1 \lhd n.(1*) \le (n+1).(1*) \lhd H^R.(1*) - 1\]
for every $H^L$ and $H^R$.  But since each $G^L$ equals $H^L.(1*) + 1$
and each $G^R$ equals $H^R.(1*) - 1$, we see that
\[ G^L \lhd n.(1*) \le (n + 1).(1*) \lhd G^R\]
for every $G^L$ and $G^R$.  But either $n$ or $n+1$ will be even,
so either $n.(1*)$ or $(n+1).(1*)$ will be an integer, and therefore
by the simplicity rule $G$ must equal an integer.  So $G \in A_0$
and we are done by the base case of induction.

So we have just shown, for every $n$, that if $G \in A_n$
and $G$ is even in form, then $G = H.(1*)$ for some $H$.
Thus if $K$ is any even game (in value), then as shown above
$K = G$ for some game $G$ that is even in form.  Since
every short game is in one of the $A_n$ for large enough $n$,
we see that $K = G = H.(1*)$ for some $H$.
\end{proof}

\begin{theorem}
The map $G \to G.(1*)$ establishes an isomorphism between the partially ordered group of
short games and the subgroup consisting of even games.  Moreover, every even or odd game
can be uniquely written in the form $G = H.(1*) + a$, for $a \in \{0,*\}$, (unique up
to equivalence of $G$), where $a = 0$ if $H$ is even and $a = *$ if $H$ is odd.  Such a game
is $\ge 0$ iff $H \ge 0$ when $a = 0$, and iff $H \ge 1$ when $a = *$.
\end{theorem}
\begin{proof}
From Corollary~\ref{halfpres}, we know that the map sending $G$ to $G.(1*)$
is an embedding of short partizan games into short even games.  From Lemma~\ref{surge}
we know that the map is a surjection.  We know from Theorem~\ref{parities}(d) that
no game is even and odd, so that the group of even and odd games is indeed
a direct product of $\{0,*\}$ with the even games.  Moreover,
we know that $H.(1*) \ge 0$ iff $H \ge 0$, by the order-preserving property of
Norton multiplication, and $H.(1*) + * \ge 0$ iff $H \ge 1$, by Lemma~\ref{slippery}.
\end{proof}

\chapter{Bending the Rules}
So far we have only considered loopfree partizan games played under the normal play rule,
where the last player able to move is the winner.  In this chapter we see how combinatorial
game theory can be used to analyze games that do not meet these criteria.  We first consider
cases where the standard partizan theory can be applied to other games.

\section{Adapting the theory}
\emph{Northcott's Game} is a game played on a checkerboard.  Each player starts with
eight pieces along his side of the board.  Players take alternating turns, and on each
turn a player may move one of her pieces left or right any number of squares, but may not jump
over her opponent's piece in the same row.  The winner is decided by the normal play rule: you lose
when you are unable to move.
\begin{figure}[H]
\begin{center}
\includegraphics[width=2.5in]
					{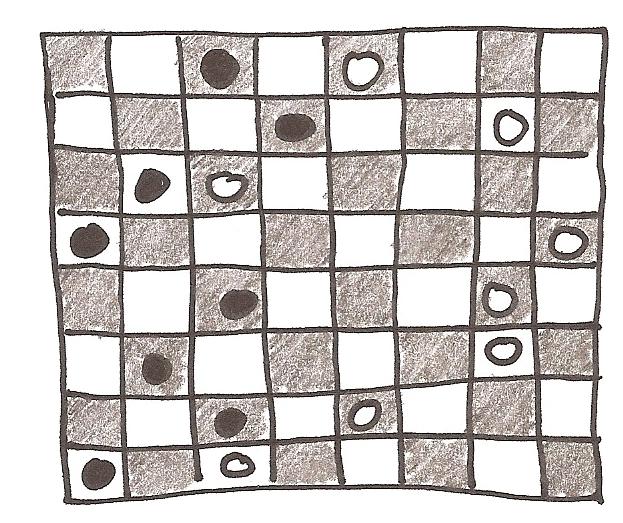}
\caption{A position of Northcott's Game}
\label{northcott}
\end{center}
\end{figure}
Clearly each row functions independently, so Northcott's Game is really a sum of eight independent games.
However, the standard partizan theory isn't directly applicable, because this game is \emph{loopy}, meaning
that the players can return the board to a prior state if they so choose:
\begin{figure}[H]
\begin{center}
\includegraphics[width=2.5in]
					{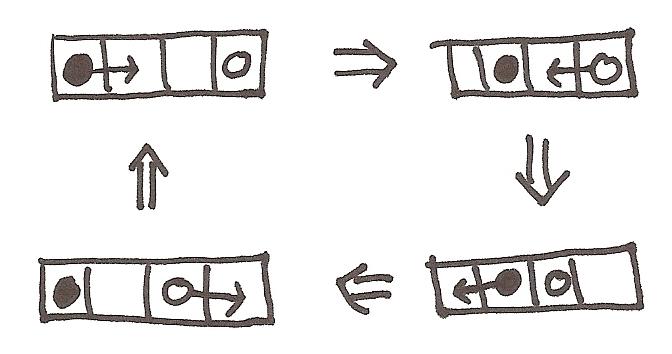}
\caption{Loops can occur in Northcott's Game.}
\label{northcott-loopy}
\end{center}
\end{figure}
Consequently, there is no guarantee that the game will ever come to an end, and \emph{draws} are possible.
We assume that each player prefers victory to a draw and a draw to defeat.  Because of this extra possibility,
it is conceivable that in some positions, neither player would have a winning strategy, but both players
would have a strategy guaranteeing a drwa.

Suppose that we changed the rules, so that a player could only move his pieces forward, towards his opponent's.  Then
the game would become loopfree, and in fact, it becomes nothing but Nim in disguise!  Given a position
of Northcott's Game, one simply counts the number of empty squares between the two pieces in each row, and creates
a Nim-heap of the same size:
\begin{figure}[H]
\begin{center}
\includegraphics[width=4in]
					{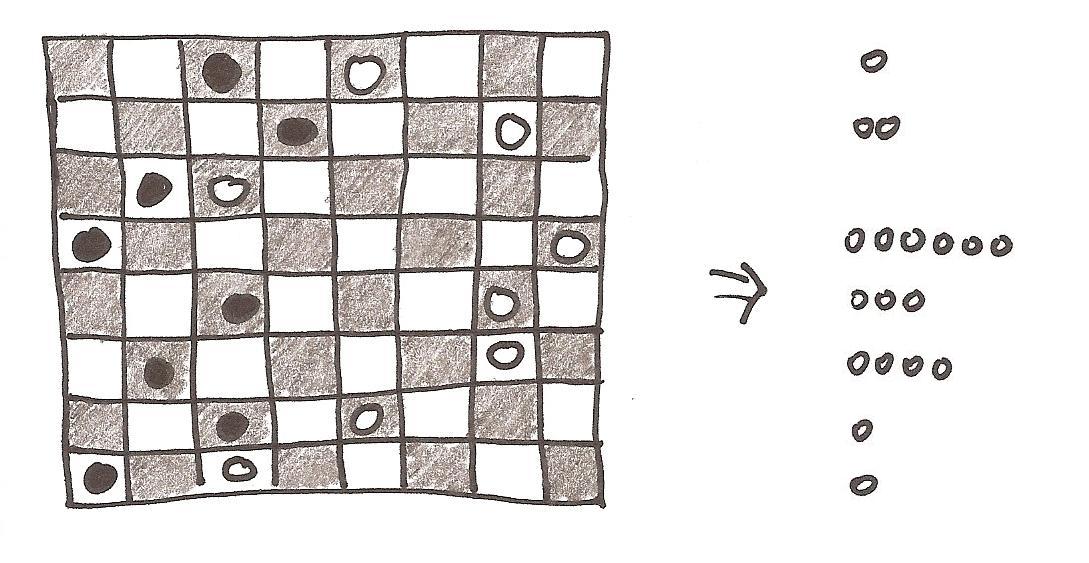}
\caption{Converting the position of Figure~\ref{northcott} into a Nim position. }
\label{northcott-is-nim}
\end{center}
\end{figure}
The resulting Nim position is equivalent: taking $n$ counters from a Nim pile corresponds to moving your piece in the corresponding
row $n$ squares forward.  This works as long as we forbid backwards moves.

However, it turns out that this prohibition has no strategic effect.  Whichever player has a winning strategy in the no-backwards-move variant
can use the same strategy in the full game.  If her opponent ever moves a piece backwards by $x$ squares, she moves her own piece
forwards by $x$ squares, cancelling her opponent's move.  This strategy guarantees that the game actually ends, because
the pieces of the player using the strategy are always moving forwards, which cannot go on indefinitely.
So Northcott's Game is still nothing but Nim in disguise.  The moral of the story is that loopy games can sometimes be analyzed using partizan theory (Sprague-Grundy theory in this case).

We now consider two case studies of real-life games that can be partially analyzed using the standard partizan theory, even though
they technically aren't partizan games themselves.
\section{Dots-and-Boxes}
Unlike Northcott's Game, \emph{Dots-and-Boxes} (also known as \emph{Squares}) is a game that people actually play.  This is a pencil and
paper game, played on a square grid of dots.  Players take turns drawing line segments between orthogonally adjacent dots.  Whenever
you complete the fourth side of a box, you claim the box by writing your initials in it, and get another move\footnote{Completing two boxes in one move \emph{does not} give you
two more moves.}.  It is possible to chain together these extra moves, and take many boxes in a single turn:
\begin{figure}[H]
\begin{center}
\includegraphics[width=3in]
					{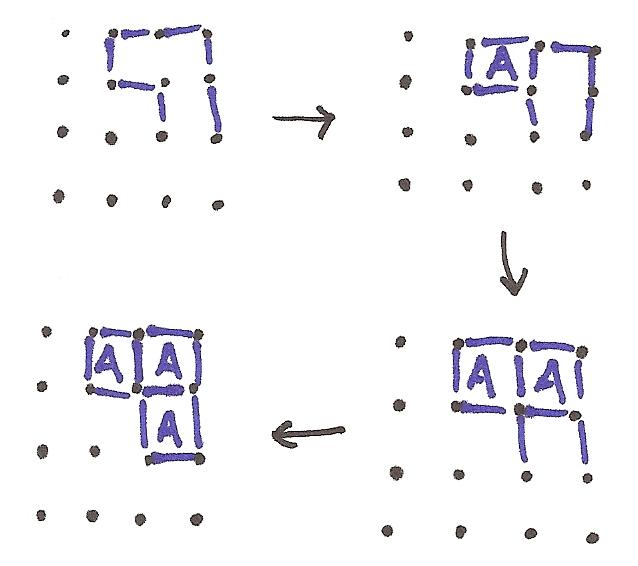}
\caption{Alice takes three boxes in one turn.}
\label{chain-of-three}
\end{center}
\end{figure}
Eventually the board fills up, and the game ends.  The player with the most boxes claimed wins.  \emph{Victory is not decided
by the normal play rule,} making the standard theory of partizan games inapplicable\footnote{The fact that you move again after completing a box
also creates problems}.

Most people play Dots-and-Boxes by making random moves until all remaining moves create a three-sided box.  Then the players
take turn giving each other larger and larger chains.
\begin{figure}[H]
\begin{center}
\includegraphics[width=5in]
					{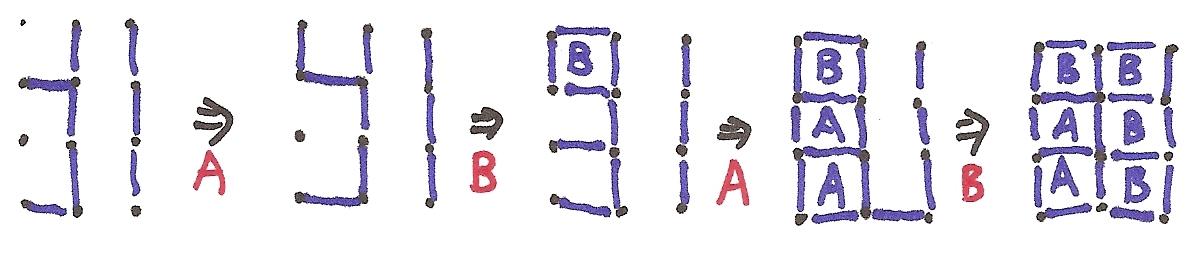}
\caption{Alice gives Bob a box.  Bob takes it and gives Alice two boxes.  Alice takes them and gives Bob three boxes.}
\label{naive-endgame}
\end{center}
\end{figure}
Oddly enough, there is a simple and little-known trick which easily beats the na\"ive strategy.  When an opponent gives
you three or more boxes, it is always possible to take all but two of them, and give two to your opponent.  Your opponent
takes the two boxes, and is then usually forced to give you another long chain of boxes.

For instance, in the following position,
\begin{figure}[H]
\begin{center}
\includegraphics[width=1.5in]
					{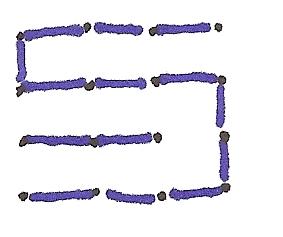}
%\caption{}
%\label{domineering-sum}
\end{center}
\end{figure}
the na\"ive strategy is to move to something like
\begin{figure}[H]
\begin{center}
\includegraphics[width=1.5in]
					{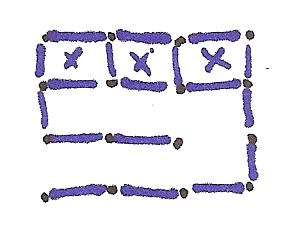}
%\caption{}
%\label{domineering-sum}
\end{center}
\end{figure}
which then gives your opponent more boxes than you obtained your self.
The better move is the following:
\begin{figure}[H]
\begin{center}
\includegraphics[width=1.5in]
					{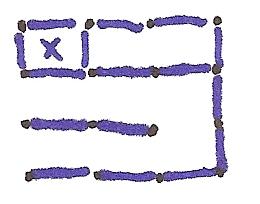}
%\caption{}
%\label{domineering-sum}
\end{center}
\end{figure}
From this position, your opponent might as well take the two boxes, but is then
forced to give you the other long chain:
\begin{figure}[H]
\begin{center}
\includegraphics[width=1.5in]
					{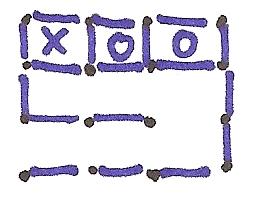}
%\caption{}
%\label{domineering-sum}
\end{center}
\end{figure}

This trick is tied to a general phenomenon, of \emph{loony positions}.  Rather
than giving a formal definition, we give an example.

Let $P_1$ be the following complicated position:
\begin{figure}[H]
\begin{center}
\includegraphics[width=2in]
					{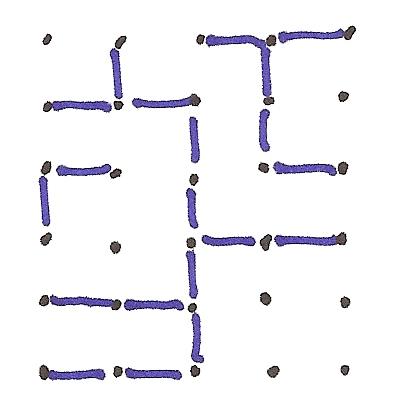}
%\caption{}
%\label{domineering-sum}
\end{center}
\end{figure}
Surprisingly, we can show that this position is a win for the first player, without even
exhibiting a specific strategy.  To see this, let $P_2$ be the following position, in which
Alice has the two squares in the bottom left corner:
\begin{figure}[H]
\begin{center}
\includegraphics[width=2in]
					{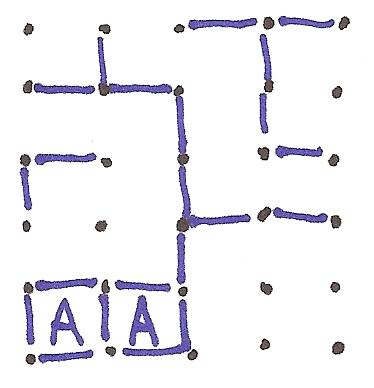}
%\caption{}
%\label{domineering-sum}
\end{center}
\end{figure}
Let $k$ be the final score for Alice if she moves first in $P_2$ and both players play optimally.

Since there are an odd number of boxes on the board,
$k$ cannot be zero.
Now break into cases according to the sign of $k$.
\begin{itemize}
\item If $k > 0$, then the first player can win $P_1$ by taking the two boxes
as well as the name ``Alice.''
\item If $k < 0$, then the first player can win $P_1$ by naming her opponent ``Alice'' and declining the two boxes, as follows:
\begin{figure}[H]
\begin{center}
\includegraphics[width=2in]
					{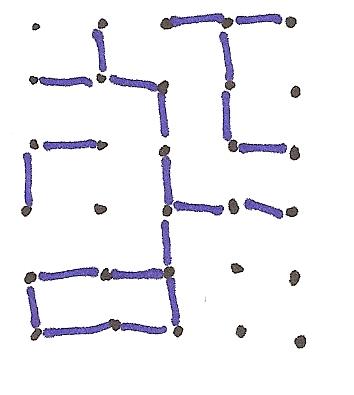}
%\caption{}
%\label{domineering-sum}
\end{center}
\end{figure}
Now ``Alice'' might as well take the two boxes, resulting in position $P_2$.  Then because $k < 0$, Alice's opponent
can guarantee a win.  If ``Alice'' doesn't take the two boxes, her opponent can just take them on her next turn, with no adverse effect.
\end{itemize}
So either way, the first player has a winning strategy in $P_1$.  Actually applying this strategy is made difficult by the fact
that we have to completely evaluate $P_2$ to tell which move to make in $P_1$.

In general, a \emph{loony position} is one containing two adjacent boxes, such that
\begin{itemize}
\item There is no wall between the two boxes
\item One of the two boxes has three walls around it.
\item The other box has exactly two walls around it.
\item The two boxes are not part of one of the following configurations:
\begin{figure}[H]
\begin{center}
\includegraphics[width=3in]
					{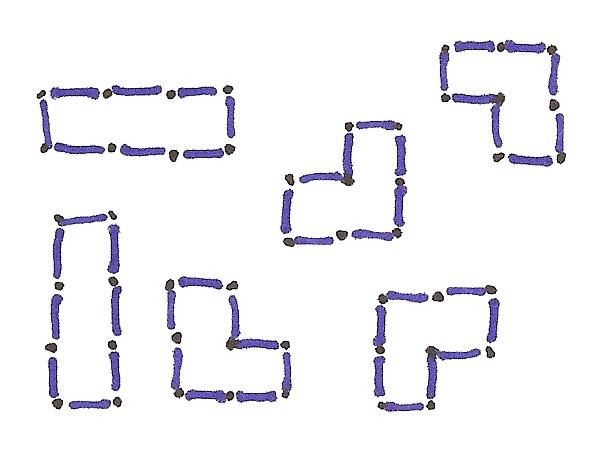}
%\caption{}
%\label{domineering-sum}
\end{center}
\end{figure}
\end{itemize}
The general fact about loony positions is that \emph{the first player is always able to win a weak majority of the remaining
pieces on the board}.  This follows by essentially the same argument used to analyze the complicated position above.  In the case
where there are an odd number of boxes on the board, and neither player has already taken any boxes, it follows that \emph{a loony position
is a win for the first player}.
Here are some examples of loony positions:
\begin{figure}[H]
\begin{center}
\includegraphics[width=2in]
					{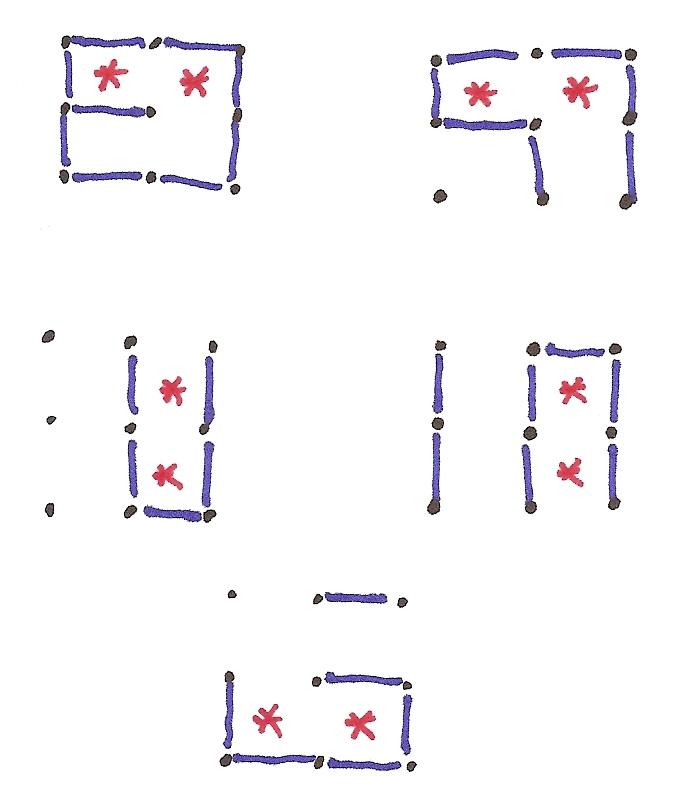}
%\caption{}
%\label{domineering-sum}
\end{center}
\end{figure}
The red asterisks indicate why each position is loony.
Here are some examples of non-loony positions:
\begin{figure}[H]
\begin{center}
\includegraphics[width=2in]
					{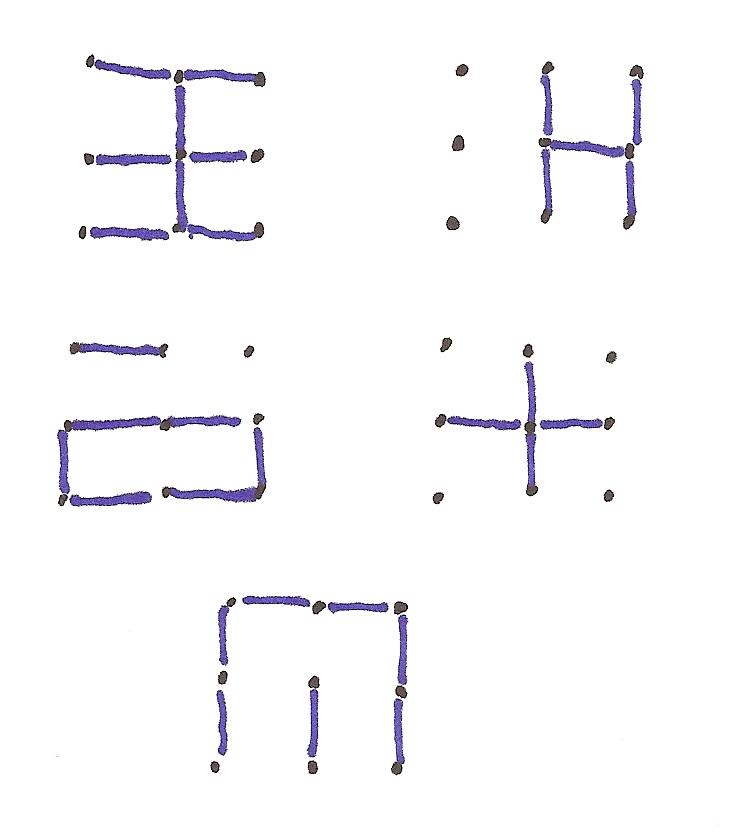}
%\caption{}
%\label{domineering-sum}
\end{center}
\end{figure}
It can be shown that whenever some squares are available to be taken, and the position is \emph{not} loony, then you might as well take them.

A \emph{loony move} is one that creates a loony position.  Note that giving away a long chain (three or more boxes) or a loop is always a loony move.
When giving away two boxes, it is always possible to do so in a non-loony way:
\begin{figure}[H]
\begin{center}
\includegraphics[width=5in]
					{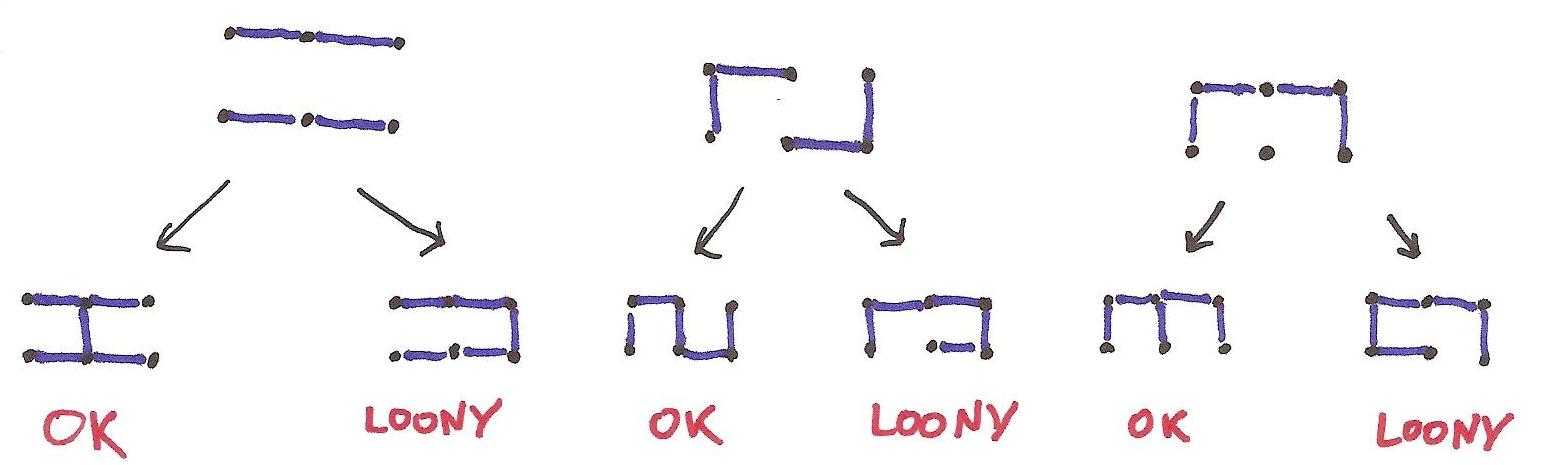}
%\caption{}
%\label{domineering-sum}
\end{center}
\end{figure}
In the vast majority of Dots-and-Boxes games, somebody eventually gives a way a long chain.  Usually, few boxes have been claimed when this first happens
(or both players have claimed about the same amount, because they have been trading chains of length one and two), so the player who gives
away the first long chain loses under perfect play.

Interestingly, the player who first makes a loony move can be predicted in terms of the parity of the number of long chains on the board.
As the game proceeds towards its end, chains begin to form and the number of long chains begins to crystallize.  Between experts, Dots-and-Boxes
turns into a fight to control this number.  For more information, I refer the interested reader to Elwyn Berlekamp's \emph{The Dots and Boxes Game.}

To connect Dots-and-Boxes to the standard theory of partizan games (in fact, to Sprague-Grundy theory), consider the variant game of
\emph{Nimdots}.  This is played exactly the same as Dots-and-Boxes except that the player who makes the last move \emph{loses}.  A few comments are in order:
\begin{itemize}
\item Despite appearances to the contrary, Nimdots is actually played by the normal play rule, not the mis\`ere rule.  The reason is that the
normal rule precisely says that \emph{you lose when it's your turn but you can't move.} In Nimdots, the player who makes the last move always
completes a box.  He then gets a bonus turn, which he is unable to complete, because the game is over!
\item Who claims each box is completely irrelevant, since the final outcome isn't decided by score.  This makes Nimdots be impartial.
\item As in Dots-and-Boxes, a loony move is generally bad. In fact, in Nimdots, a loony move is \emph{always} a losing move, by the same
arguments as above.  In fact, since we are using the normal play rule, we might as well make loony moves illegal, and consider no
loony positions.
\item If you give away some boxes without making a non-loony move, your opponent might as well take them.  But there is no score,
so it doesn't matter who takes the boxes, and we could simply have the boxes get magically eaten up after any move which gives away boxes.
\end{itemize}
With these rule modifications, there are no more entailed moves, and Nimdots becomes a bona fide impartial game, so we can apply Sprague-Grundy theory.
For example, here is a table showing the Sprague-Grundy numbers of some small Nimdots positions (taken from page 559 of \emph{Winning Ways}).
\begin{figure}[H]
\begin{center}
\includegraphics[width=4in]
					{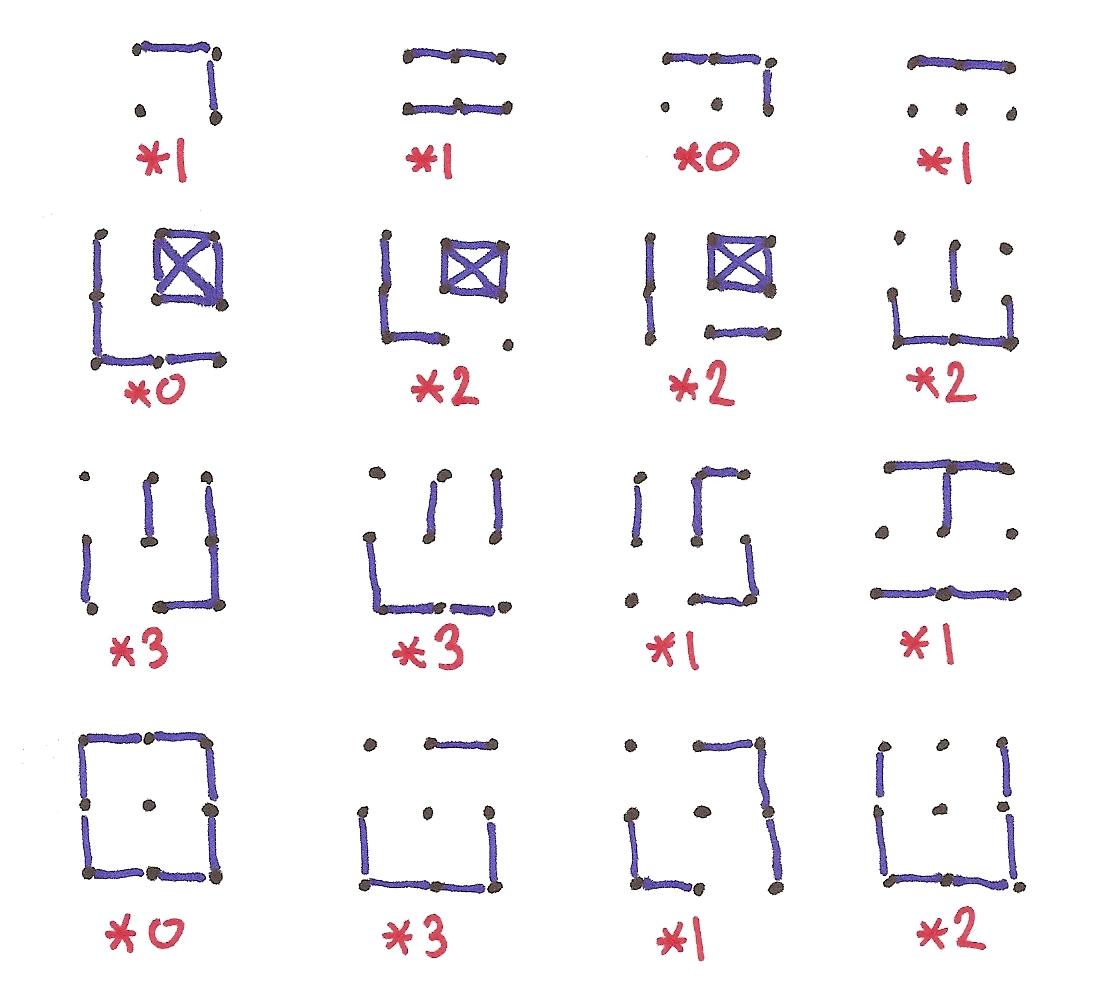}
%\caption{}
%\label{domineering-sum}
\end{center}
\end{figure}
This sort of analysis is actually useful because positions in Nimdots and Dots-and-Boxes can often decompose as sums of smaller positions.
And oddly enough, in some cases, a Nimdots positions replicate impartial games like Kayles (see chapter 16 of \emph{Winning Ways} for examples).

The connection between Dots-and-Boxes and Nimdots comes by seeing Nimdots as an approximation to Dots-and-Boxes.  In Dots-and-Boxes, the first
player to make a loony move usually loses.  in Nimdots, the first player to make a loony move \emph{always} loses.  So even though the winner
is determined by completely different means in the two games, they tend to have similar outcomes, at least early in the game.

This gives an (imperfect and incomplete) ``mathematical'' strategy for Dots-and-Boxes: pretend that the position is a Nimdots position,
and use this to make sure that your opponent ends up making the first loony move.  In order for the loony-move fight to even be worthwhile,
you also need to ensure that there are long enough chains.  In the process of using this strategy, one might actually sacrifice some boxes
to your opponent, for a better final score.  For instance, in Figure~\ref{sacrifice-necessary}, the only winning move is to prematurely
sacrifice two boxes.
\begin{figure}[H]
\begin{center}
\includegraphics[width=2.5in]
					{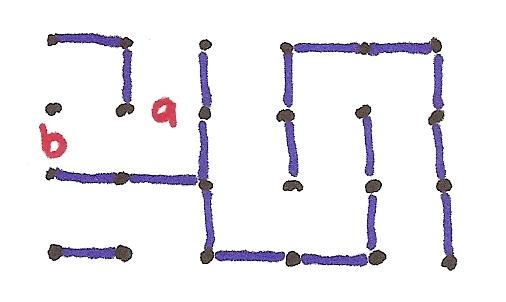}
\caption{The only winning move is (a), which sacrifices two boxes.  The alternative move at (b) sacrifices zero boxes,
but ultimately loses.}
\label{sacrifice-necessary}
\end{center}
\end{figure}

The mathematical strategy is imperfect, so some people have advocated alternative strategies. On his now-defunct Geocities page, Ilan Vardi
suggested a strategy based on
\begin{description}
\item[(a)] Making lots of shorter chains, and loops, which tend to decrease the value of winning the Nimdots fight.
\item[(b)] ``Nibbling,'' allowing your opponent to win the Nimdots/loony-move fight, but at a cost.
\item[(c)] ``Pre-emptive sacrifices,'' in which you make a loony-move in a long chain before the chain gets especially long.  This breaks up chains
early, helping to accomplish (a).  Such moves can only work if you are already ahead score-wise, via (b).
\end{description}
As Ilan Vardi notes, there are some cases in Dots-and-Boxes in which the \emph{only} winning move is loony:
\begin{figure}[H]
\begin{center}
\includegraphics[width=3.5in]
					{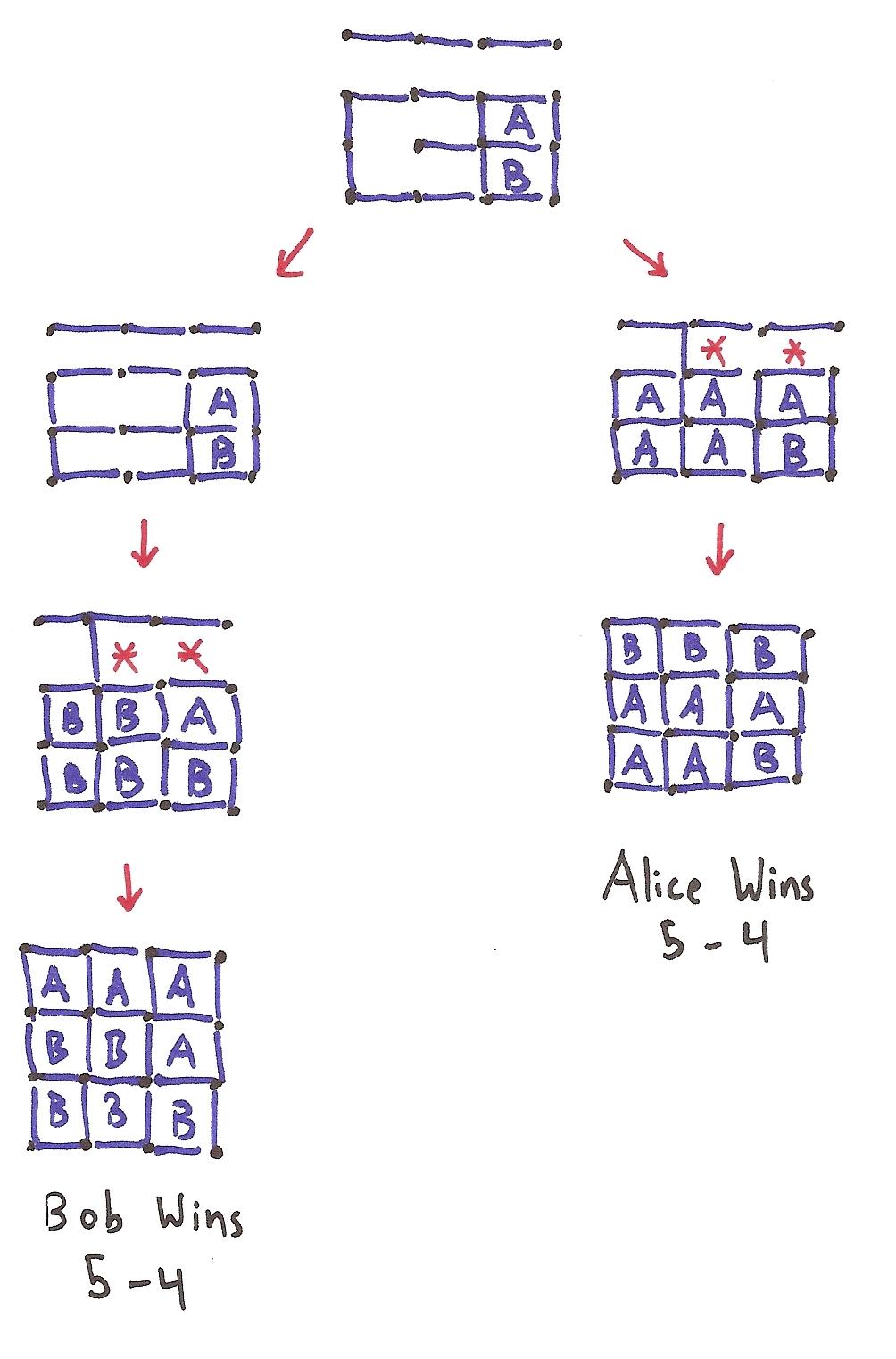}
\caption{In the top position, with Alice to move, the move at the left is the only non-loony move.  However,
it ultimately loses, giving Bob most of the boxes.  On the other hand, the move on the right is
technically loony, but gives Alice the win, with 5 of the 9 boxes already.}
\label{be-crazy}
\end{center}
\end{figure}
According to Vardi, some of the analyses of specific positions in Berlekamp's \emph{The Dots and Boxes Game} are incorrect because of the false assumption that
loony moves are always bad.

Unlike many of the games we have considered so far, there is little hope of giving a general analysis of Dots-and-Boxes, since determining
the outcome of a Dots-and-Boxes position is NP-hard, as shown by Elwyn Berlekamp in the last chapter of his book.

\section{Go}
\emph{Go} (also known as Baduk and Weiqi) is an ancient boardgame that is popular in China, Japan, the United States, and New Zealand, among other places.  It
is frequently considered to have the most strategic depth of any boardgame commonly played, more than Chess.\footnote{In fact, while computers
can now beat most humans at Chess, computer Go programs are still routinely defeated by novices and children.}

In Go, two players, Black and White, alternatively place stones on a $19 \times 19$ board.  Unlike Chess or Checkers, pieces are played on the corners
of the squares, as in Figure~\ref{go-1}.
A \emph{group} of stones is a set of stones of one color that is connected (by means of direct orthogonal connections).  So in the following position,
Black has 4 groups and White has 1 group:
\begin{figure}[H]
\begin{center}
\includegraphics[width=2.5in]
					{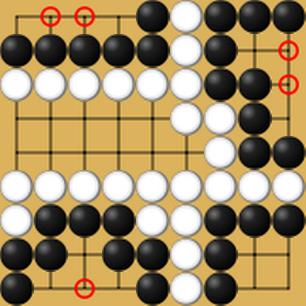}
\caption{Image taken from the Wikipedia article \href{http://en.wikipedia.org/wiki/Life_and_death}{\emph{Life and Death}} on June 6, 2011.}
\label{go-1}
\end{center}
\end{figure}
The \emph{liberties} of a group are the number of empty squares.  Once a group has no liberties, its pieces are captured and removed from the board,
and given to the opposing player.  There are some additional prohibitions against suicidal moves and moves which exactly reverse the previous move or
return the board to a prior state.  Some of these rules vary between different rulesets.

Players are also allowed to pass, and the game ends when both players pass.  The rules for scoring are actually very complicated and vary by ruleset,
but roughly speaking
you get a point for each captured opponent stone, and a point for each empty space that is surrounded by pieces of your own color.\footnote{What if the losing player decides to never pass?  If understand
the rules correctly, he will eventually be \emph{forced} to pass, because his alternative to passing is filling up his own territory.  He could also try
invading the empty spaces in his opponent's territory, but then his pieces would be captured and eventually the opponent's territory would also fill up.
After a very long time, all remaining spots on the board would become illegal to move to, by the no suicide rule, and then he would be forced to pass.
At any rate, it seems like this would be a pointless exercise in drawing out the game, and the sportsmanlike thing to do is to resign, i.e., to pass.}

In the following position, if there were no stones captured, then Black would win by four points:
\begin{figure}[H]
\begin{center}
\includegraphics[width=2.5in]
					{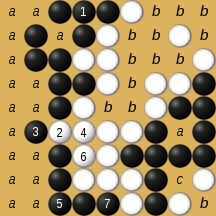}
\caption{Black has 17 points of territory (the \emph{a}'s) and White has 13 (the \emph{b}'s). Image taken from the Wikipedia article \href{http://en.wikipedia.org/wiki/Rules_of_Go}{\emph{Rules of Go}} on June 6, 2011.}
\label{go-2}
\end{center}
\end{figure}
(The scoring rule mentioned above is the one used in Japan and the United States.  In China, you also get points for your own pieces on the board,
but not for prisoners, which tends to make the final score difference almost identical to the result of Japanese scoring.)

There is a great deal of terminology and literature related to this game, so we can barely scratch the surface.  One thing worth pointing out
is that it is sometimes possible for a group of stones to be indestructible.  This is called \emph{life}.  Here is an example:
\begin{figure}[H]
\begin{center}
\includegraphics[width=2.5in]
					{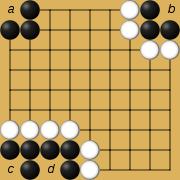}
\caption{The black group in the bottom left corner has two eyes, so it is alive.  There is no way
for White to capture it, since White would need to move in positions $c$ and $d$ simultaneously.
The other black groups do not have two eyes, and could be taken.  For example, if White moves at $b$, the top right black group
would be captured.  (Image taken from the Wikipedia article \href{http://en.wikipedia.org/wiki/Life_and_Death}{\emph{Life and Death}}
on June 6, 2011.)}
\label{go-3}
\end{center}
\end{figure}
This example shows the general principle that two ``eyes'' ensures life.

Another strategic concept is \emph{seki}, which refers to positions in which neither player wants to move, like the following:
\begin{figure}[H]
\begin{center}
\includegraphics[width=2.5in]
					{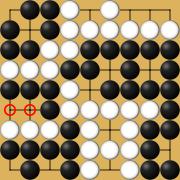}
\caption{If either player moves in one of the red circled positions, his opponent will move in the other and take
one of his groups.  So neither player will play in those positions, and they will remain empty. (Image taken from the Wikipedia article
\href{http://en.wikipedia.org/wiki/Go_(game)}{\emph{Go (game)}} on June 6, 2011.)}
\label{go-4}
\end{center}
\end{figure}
Because neither player has an obligation to move, both players will simply ignore this area until the end of the game, and the spaces
in this position will count towards neither player.

Like Dots-and-Boxes, Go is not played by the normal play rule, but uses scores instead.  However, there is a na\"ive way to turn
a Go position into a partizan game position that actually works fairly well, and is employed by Berlekamp and Wolfe in their book
\emph{Mathematical Go: Chilling Gets the Last Point.}  Basically, each final position in which no moves remain is replaced by its score,
interpreted as a surreal number.

For instance, we have
\begin{figure}[H]
\begin{center}
\includegraphics[width=4in]
					{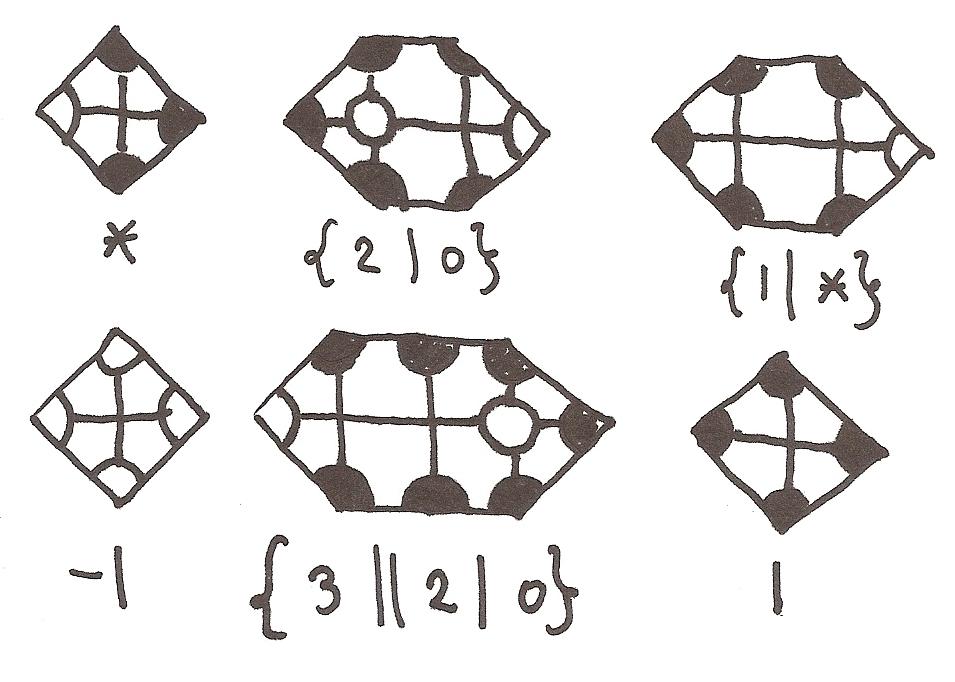}
\caption{Small positions in Go, taken from Berlekamp and Wolfe.  The pieces along the boundary are assumed to be alive.}
\label{go-dictionary}
\end{center}
\end{figure}
This approach works because of number avoidance.  Converting Go endgames into surreal numbers adds extra options, but
we can assume that the players never use these extra options, because of number avoidance.  For this to work,
we need the fact that a non-endgame Go position isn't a number.  Unfortunately, some are, like the following:
\begin{figure}[H]
\begin{center}
\includegraphics[width=2in]
					{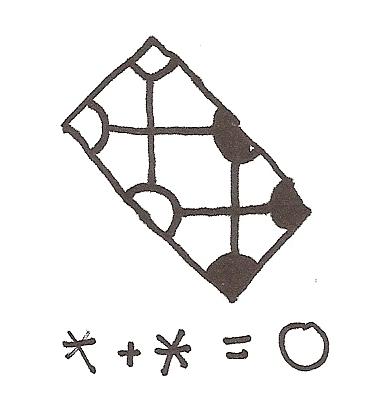}
%\caption{}
%\label{domineering-sum}
\end{center}
\end{figure}
However, something slightly stronger than number-avoidance is actually true:
\begin{theorem}
Let $A,B,C,\ldots,D,E,F,\ldots$ be short partizan games, such that
\[ \max( R(A), R(B),\ldots) \ge \min(L(D), L(E), \ldots),\]
and let $x$ be a number. Then
\[ \{A,B,C,\ldots|D,E,F,\ldots\} + x = \{A + x, B + x, \ldots | D + x, E + x, \ldots \}.\]
\end{theorem}
\begin{proof}
If $\{A,B,C,\ldots |D,E,F,\ldots\}$ is not a number, then this follows by number avoidance.  It also follows
by number avoidance if $\{A + x, B + x, \ldots | D + x, E + X, \ldots\}$ is not a number.  Otherwise,
there is some number $y$, equal to $\{A,B,\ldots|D,E,\ldots\}$ such that $A,B,C,\ldots \lhd y \lhd D, E, F, \ldots$.  But by definition of $L(\cdot)$
and $R(\cdot)$, it follows that
\[ \max(R(A),R(B),\ldots) \le y \le \min(L(D),L(E),\ldots),\]
since it is a general fact that $y \lhd G$ implies that $y \le L(G)$ and similarly $G \lhd y \Rightarrow R(G) \le y$.
So it must be the case that $\max(R(A),R(B),\ldots) = y = \min(L(A),L(B),\ldots)$.  Thus
\[ \{A,B,C,\ldots|D,E,F,\ldots\} = \max(R(A),R(B),\ldots).\]
By the same token,
\[ \{A + x, B + x, \ldots | D + x, E + x, \ldots\} = \max(R(A + x), R(B + x),\ldots) \]\[ = \max(R(A) + x, R(B) + x,\ldots) = \{A,B,\ldots|D,E,\ldots\} + x.\]
\end{proof}
Now Go positions always have the property that $\max_{G^L}(R(G^L)) \ge \min_{G^R}(L(G^R))$, because players are not under compulsion to pass.
There is no way to create a position like $\{0|4\}$ in Go (which would asymmetrically be given the value $1$ by our translation), because in such a position,
neither player wants to move, and the position will be a \emph{seki} endgame position that should have been directly turned into a number.

Interestingly enough, many simple Go positions end up taking values that are Even or Odd, in the sense of Section \ref{sec:evenandodd}.  This comes about
because we can assign a parity to each Go position, counting the number of prisoners and emtpy spaces on the board, and the parity is reversed after each
move.  And an endgame will have an odd score iff its positional parity is odd (unless there are \emph{dame}).

Then because most of the values that arise are even and odd, we can describe them as Norton multiples of $1*$.  Replacing the position $X.(1*)$
with $X$ creates a simpler description of the same position.  This operation is the ``chilling'' operation referenced in the title
of Wolfe and Berlekamp's books.  It is an instance of the \emph{cooling} operation of ``thermography,'' which is closely related to the
mean value theory.

A lot of research has gone into studying \emph{ko} situations like the following:
\begin{figure}[H]
\begin{center}
\includegraphics[width=4in]
					{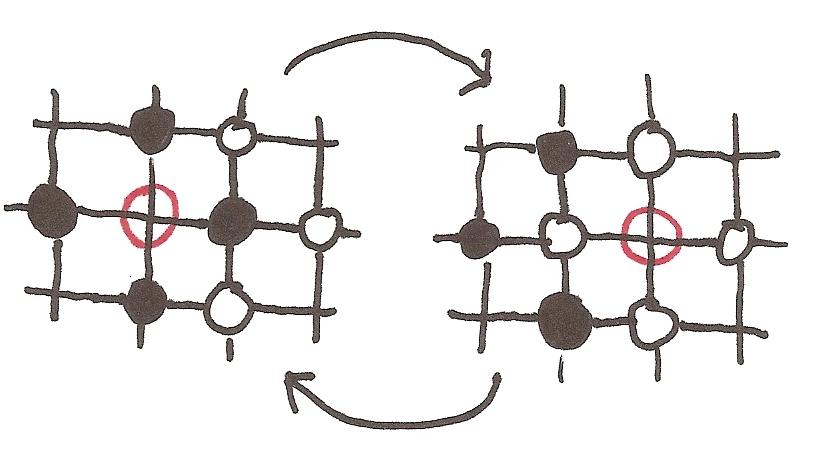}
\caption{From the position on the left, White can move to the position on the right by playing in the circled position.  But from the position
on the right, Black can move directly back to the position on the right, by playing in the circled position.}
\label{ko}
\end{center}
\end{figure}
%
%\begin{equation} \text{ example of a ko }\label{ko1}\end{equation}
%In this position, Black can move to
%\begin{equation} \text{ other }\label{ko2}\end{equation}
%from which White can move right back to (\ref{ko1})!
The rules of Go include a proviso that forbids directly undoing a previous move.
However, nothing prevents White from making a threat somewhere else, which Black must respond to - and then
after Black responds elsewhere, White can move back in the ko.  This can go back and forth several rounds,
in what is known as a \emph{kofight}.  While some players of Go see kofights as mere randomness (a player once
told me it was like shooting craps), many combinatorial game theorists have mathematically examined
ko positions, using extensions of mean value theory and ``thermography'' for loopy games.

\section{Changing the theory}
In some cases, we need a different theory from the standard one discussed so far.
The examples in this section are not games people play, but show how alternative theories
can arise in constructed games.

Consider first the following game, \emph{One of the King's Horses}: a number of chess knights
sit on a square board, and players take turns moving them towards the top left corner.  The pieces
move like knights in chess, except only in the four northwestern directions:
\begin{figure}[H]
\begin{center}
\includegraphics[width=2in]
					{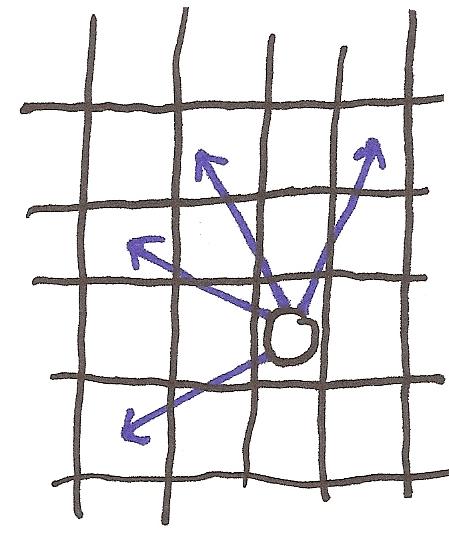}
%\caption{}
%\label{domineering-sum}
\end{center}
\end{figure}
Each turn, you move \emph{one} piece, but multiple pieces are allowed
to occupy the same square.  You lose when you cannot move.

Clearly, this is an impartial game with the normal play rule, and each piece is moving completely independently of all the others,
so the game decomposes as a sum.  Consequently we can ``solve'' the game by figuring out the Sprague-Grundy value
of each position on the board.  Here is a table showing the values in the top left corner of the board:
\begin{center}
\begin{tabular}{|c|c|c|c|c|c|c|c|c|c|c|c|c|c|c|c|}
\hline
0 & 0 & 1 & 1 & 0 & 0 & 1 & 1 & 0 & 0 & 1 & 1 & 0 & 0 & 1 & 1 \\
\hline
0 & 0 & 2 & 1 & 0 & 0 & 1 & 1 & 0 & 0 & 1 & 1 & 0 & 0 & 1 & 1 \\
\hline
1 & 2 & 2 & 2 & \textcircled{3} & 2 & 2 & 2 & 3 & 2 & 2 & 2 & 3 & 2 & 2 & 2 \\
\hline
1 & 1 & 2 & 1 & 4 & \textcircled{3} & 2 & 3 & 3 & 3 & 2 & 3 & 3 & 3 & 2 & 3 \\
\hline
0 & 0 & 3 & \textcircled{4} & 0 & 0 & 1 & 1 & 0 & 0 & 1 & 1 & 0 & 0 & 1 & 1 \\
\hline
0 & 0 & \textcircled{2} & \textcircled{3} & 0 & 0 & 2 & 1 & 0 & 0 & 1 & 1 & 0 & 0 & 1 & 1 \\
\hline
1 & 1 & 2 & 2 & 1 & 2 & 2 & 2 & 3 & 2 & 2 & 2 & 3 & 2 & 2 & 2 \\
\hline
1 & 1 & 2 & 3 & 1 & 1 & 2 & 1 & 4 & 3 & 2 & 3 & 3 & 3 & 2 & 3 \\
\hline
0 & 0 & 3 & 3 & 0 & 0 & 3 & 4 & 0 & 0 & 1 & 1 & 0 & 0 & 1 & 1 \\
\hline
0 & 0 & 2 & 3 & 0 & 0 & 2 & 3 & 0 & 0 & 2 & 1 & 0 & 0 & 1 & 1 \\
\hline
1 & 1 & 2 & 2 & 1 & 1 & 2 & 2 & 1 & 2 & 2 & 2 & 3 & 2 & 2 & 2 \\
\hline
1 & 1 & 2 & 3 & 1 & 1 & 2 & 3 & 1 & 1 & 2 & 1 & 4 & 3 & 2 & 3 \\
\hline
0 & 0 & 3 & 3 & 0 & 0 & 3 & 3 & 0 & 0 & 3 & 4 & 0 & 0 & 1 & 1 \\
\hline
0 & 0 & 2 & 3 & 0 & 0 & 2 & 3 & 0 & 0 & 2 & 3 & 0 & 0 & 2 & 1 \\
\hline
1 & 1 & 2 & 2 & 1 & 1 & 2 & 2 & 1 & 1 & 2 & 2 & 1 & 2 & 2 & 2 \\
\hline
1 & 1 & 2 & 3 & 1 & 1 & 2 & 3 & 1 & 1 & 2 & 3 & 1 & 1 & 2 & 1 \\
\end{tabular}
\end{center}
So for instance, if there are pieces on the circled positions,
the combined value is
\[ 3 +_2 3 +_2 2 +_2 3 +_2 4 = 5 \]
Since $5 \ne 0$, this position is a first-player win.  The table shown above
has a fairly simple and repetitive pattern, which gives the general
strategy for One of the King's Horses.

Now consider the variant \emph{All of the King's Horses}, in which you move \emph{every} piece on your turn, rather than
selecting one.  Note that once \emph{one} of the pieces reaches the top left corner of the board, the game is over, since you are required
to move all the pieces on your turn, and this becomes impossible once one of the pieces reaches the home corner.

This game no longer corresponds to a \emph{sum}, but instead to what \emph{Winning Ways} calls a \emph{join}.  Whereas
a sum is recursively defined as
\[ G + H = \{G^L + H, G + H^L|G^R + H, G + H^R\},\]
 a join is defined recursively as
\[ G \wedge H = \{G^L \wedge H^L| G^R \wedge H^R\},\]
where $G^*$ and $H^*$ range over the options of $G$ and $H$.  In a join of two games, you must move in both components on each turn.

Just as sums of impartial games are governed by Sprague-Grundy numbers, joins of impartial games are governed by \emph{remoteness}.
If $G$ is an impartial game, its remoteness $r(G)$ is defined recursively as follows:
\begin{itemize}
\item If $G$ has no options, then $r(G)$ is zero.
\item If some option of $G$ has even remoteness, then $r(G)$ is one more than the minimum $r(G')$ where
$G'$ ranges over options of $G$ such that $r(G')$ is even.
\item Otherwise, $r(G)$ is one more than the maximum $r(G')$ for $G'$ an option of $r(G)$.
\end{itemize}
Note that $r(G)$ is odd if and only if some option of $G$ has even remoteness.  Consequently,
a game is a second-player win if its remoteness number is even,
and a first-player win otherwise.  The remoteness of a Nim heap with $n$ counters is $0$ if $n = 0$, and $1$ otherwise, since
every Nim-heap after the zeroth one has the zeroth one as an option.

The remoteness is roughly a measure of how quickly the winning player can bring the game
to an end, assuming that the losing player is trying to draw out the game as long as possible.

Remoteness governs the outcome of joins of impartial games in the same way that
Sprague-Grundy numbers govern the outcome of sums:
\begin{theorem}\label{remoteness}
If $G_1, \ldots, G_n$ are impartial games, then the join $G_1 \wedge \cdots \wedge G_n$
is a second-player win if $\min(r(G_1),\ldots,r(G_n))$ is even, and a first-player win otherwise.
\end{theorem}
\begin{proof}
First of all note that if $G$ is any nonzero impartial game, then $r(G) = r(G') - 1$ for some option $G'$ of $G$.
Also, if $r(G)$ is odd then some option of $G$ has even remoteness.

To prove the theorem, first consider the case where one of the $G_i$ has no options, so $r(G_i) = 0$.  Then neither
does the $G_1 \wedge \cdots \wedge G_n$. A game with no options
is a second-player win (because whoever goes first immediately loses).  And as expected,
and $\min(r(G_1),\ldots,r(G_n)) = 0$ which is even.  

Now suppose that every $G_i$ has an option.  First consider the case where $\min(r(G_1),\ldots,r(G_n))$ is odd.
Then for every $i$ we can find an option $G_i'$ of $G_i$, such that $r(G_i) = r(G_i') + 1$.  In particular then,
\[ \min(r(G_1'),\ldots,r(G_n')) = \min(r(G_1),\ldots,r(G_n)) \text{ is even,},\]
so by induction $G_1' \wedge \cdots \wedge G_n'$ is a second-player win.  Therefore $G_1 \wedge \cdots \wedge G_n$
is a first-player win, as desired.

On the other hand, suppose that $\min(r(G_1),\ldots,r(G_n))$ is even.  Let $r(G_i) = \min(r(G_1),\ldots,r(G_n))$.
Suppose for the sake of contradiction that there is some option $G_1' \wedge \cdots \wedge G_n'$
of $G_1 \wedge \cdots \wedge G_n$ such that $\min(r(G_1'),\ldots,r(G_n'))$ is also even.  Let
$r(G_j') = \min(r(G_1'),\ldots,r(G_n'))$.  Since $r(G_j')$ is even, it follows that
$r(G_j)$ is odd and at most $r(G_j') + 1$.  Then
\begin{equation} r(G_i) = \min(r(G_1),\ldots,r(G_n)) \le r(G_j) \le r(G_j') + 1,\label{remotenesseq}\end{equation}
On the other hand, since $r(G_i)$ is even, every option of $G_i$ has odd
remoteness, and in particular $r(G_i')$ is odd and at most $r(G_i) - 1$.  Then
\[ r(G_j') = \min(r(G_1'),\ldots,r(G_n')) \le r(G_i') \le r(G_i) - 1.\]
Combining with (\ref{remotenesseq}), it follows that $r(G_j') = r(G_i) - 1$, contradicting the fact
that $r(G_i)$ and $r(G_j')$ are both even.
\end{proof}

In fact, from this we can determine the remoteness of a join of two games:
\begin{corollary}\label{remoteness2}
Let $G$ and $H$ be impartial games.  Then $r(G \wedge H) = \min(r(G),r(H))$.
\end{corollary}
\begin{proof}
Let $a_0 = \{|\}$, and $a_n = \{a_{n-1}|a_{n-1}\}$ for $n > 0$.  Then $r(a_n) = n$.
Now if $K$ is any impartial game, then $r(K)$ is uniquely determined by the outcomes
of $K \wedge a_n$ for every $n$.  To see this, suppose that $K_1$ and $K_2$ have differing
remotenesses, specifically $n = r(K_1) < r(K_2) $.  Then $\min(r(K_1),r(a_{n+1})) = \min(n,n+1) = n$,
while $r(K_2) \ge n+1$, so that $\min(r(K_2),r(a_{n+1})) = n + 1$.  Since $n$ and $n+1$ have different
parities, it follows by the theorem that $K_1 \wedge a_{n+1}$ and $K_2 \wedge a_{n+1}$ have different outcomes.

Now let $G$ and $H$ be impartial games, and let $n = \min(r(G),r(H))$.  Then for every $k$,
\[ \min(r(G),r(H),r(a_k)) = \min(r(a_n),r(a_k)),\]
so that $G \wedge H \wedge a_k$ has the same outcome as $a_n \wedge a_k$ for all $k$.  But
then since $G \wedge H \wedge a_k = (G \wedge H) \wedge a_k$, it follows by the previous paragraph
that $(G \wedge H)$ and $a_n$ must have the same remoteness.  But since the remoteness of
$a_n$ is $n$, $r(G \wedge H)$ must also be $n = \min(r(G),r(H))$.
\end{proof}

Using these rules, we can evaluate a position of All the King's Horses using the following table
showing the remoteness of each location:
\begin{center}
\begin{tabular}{|c|c|c|c|c|c|c|c|c|c|c|c|c|c|c|c}
\hline
0 & 0 & 1 & 1 & 2 & 2 & 3 & 3 & 4 & 4 & 5 & 5 & 6 & 6 & 7 & 7 \\
\hline
0 & 0 & 1 & 1 & 2 & 2 & 3 & 3 & 4 & 4 & 5 & 5 & 6 & 6 & 7 & 7 \\
\hline
1 & 1 & 1 & 1 & \textcircled{3} & 3 & 3 & 3 & 5 & 5 & 5 & 5 & 7 & 7 & 7 & 7 \\
\hline
1 & 1 & 1 & 3 & 3 & \textcircled{3} & 3 & 5 & 5 & 5 & 5 & 7 & 7 & 7 & 7 & 9 \\
\hline
2 & \textcircled{2} & 3 & 3 & 4 & 4 & 5 & 5 & \textcircled{6} & 6 & 7 & 7 & 8 & 8 & 9 & 9 \\
\hline
2 & 2 & 3 & 3 & 4 & 4 & 5 & 5 & 6 & 6 & 7 & 7 & 8 & 8 & 9 & 9 \\
\hline
3 & 3 & 3 & 3 & 5 & 5 & 5 & 5 & 7 & 7 & 7 & 7 & 9 & 9 & 9 & 9 \\
\hline
3 & 3 & 3 & 5 & 5 & 5 & 5 & 7 & 7 & 7 & 7 & 9 & 9 & 9 & 9 & 11 \\
\hline
4 & 4 & 5 & 5 & 6 & 6 & 7 & 7 & 8 & 8 & 9 & 9 & 10 & 10 & 11 & 11 \\
\hline
4 & 4 & 5 & 5 & 6 & 6 & 7 & 7 & 8 & 8 & 9 & 9 & 10 & 10 & 11 & 11 \\
\hline
5 & 5 & 5 & 5 & 7 & 7 & 7 & 7 & 9 & 9 & 9 & 9 & 11 & 11 & 11 & 11 \\
\hline
5 & 5 & 5 & 7 & 7 & 7 & 7 & 9 & 9 & 9 & 9 & 11 & 11 & 11 & 11 & 13 \\
\hline
6 & 6 & 7 & 7 & 8 & 8 & 9 & 9 & 10 & 10 & 11 & 11 & 12 & 12 & 13 & 13 \\
\hline
6 & 6 & 7 & 7 & 8 & 8 & 9 & 9 & 10 & 10 & 11 & 11 & 12 & 12 & 13 & 13 \\
\hline
7 & 7 & 7 & 7 & 9 & 9 & 9 & 9 & 11 & 11 & 11 & 11 & 13 & 13 & 13 & 13 \\
\hline
7 & 7 & 7 & 9 & 9 & 9 & 9 & 11 & 11 & 11 & 11 & 13 & 13 & 13 & 13 & 15 \\
\end{tabular}
\end{center}
So for instance, if there are pieces on the circled positions, then 
the combined remoteness is
\[ \min(3,3,2,6) = 2,\]
and $2$ is even, so the combined position is a win for the second-player.

The full partizan theory of joins isn't much more complicated than the impartial theory,
because of the fact that play necessarily alternates in each component (unlike in the
theory of sums, where a player might make two moves in a component without an intervening
move by the opponent).

A third operation, analogous to sums and joins, is the \emph{union}, defined recursively as
\[ G \vee H = \{G^L \vee H, G^L \vee H^L, G \vee H^L|G^R \vee H, G^R \vee H^R, G \vee H^R\}\]
In a union of two games, you can move in one or both components.  More generally, in a union
of $n$ games, you can on your turn move in any (nonempty) set of components.  The corresponding
variant of All the King's Horses is \emph{Some of the King's Horses}, in which you can move
any positive number of the horses, on each turn.

For impartial games, the theory of unions turns out to be trivial: a union of two games is a second-player
win if and only if both are second-player wins themselves.  If $G$ and $H$ are both second-player wins, then any move in $G$, $H$, or both, will result in at least
one of $G$ and $H$ being replaced with a first-player win - and by induction such a union is itself a first player win.
On the other hand, if at least one of $G$ and $H$ is a first-player win, then the first player to move
in $G \vee H$ can just move in whichever components are not second-player wins, creating a position whose every
component is a second-player win.

So to analyze Some of the King's Horses, we only need to mark whether each position is a first-player win or a second-player win:
\begin{center}
\begin{tabular}{|c|c|c|c|c|c|c|c}
\hline
2 & 2 & 1 & 1 & 2 & 2 & 1 & 1 \\
\hline
2 & 2 & 1 & 1 & 2 & 2 & 1 & 1 \\
\hline
1 & 1 & 1 & 1 & 1 & 1 & 1 & 1 \\
\hline
1 & 1 & 1 & 1 & 1 & 1 & 1 & 1 \\
\hline
2 & 2 & 1 & 1 & 2 & 2 & 1 & 1 \\
\hline
2 & 2 & 1 & 1 & 2 & 2 & 1 & 1 \\
\hline
1 & 1 & 1 & 1 & 1 & 1 & 1 & 1 \\
\hline
1 & 1 & 1 & 1 & 1 & 1 & 1 & 1 \\
\end{tabular}
\end{center}

\section{Highlights from \emph{Winning Ways} Part 2}
The entire second volume of \emph{Winning Ways} is an exposition of alternate theories to the
standard partizan theory.

\subsection{Unions of partizan games}
In the partizan case, unions are much more interesting.  Without proofs, here is a summary of what happens, taken
from Chapter 10 of \emph{Winning Ways}:
\begin{itemize}
\item To each game, we associate an expression of the form $x_ny_m$ where $x$ and $y$ are dyadic rationals and
$n$ and $m$ are nonnegative integers.  The $x_n$ part is the ``left tally'' consisting of a ``toll'' of $x$
and a ``timer'' $n$, and similarly $y_m$ is the ``right tally.''  The expression $x$ is short for
$x_0x_0$.
\item These expressions are added as follows:
\[ x_ny_m + w_iz_j = (x+w)_{\max(n,i)}(y+z)_{\max(m,j)}.\]
\item In a position with value $x_ny_m$, if Left goes first then she wins iff $x > 0$
or $x = 0$ and $n$ is odd.  If Right goes first then he wins iff $y < 0$ or $y = 0$ and $m$ is odd.
\end{itemize}
Given a game $G = \{G^L|G^R\}$, the tallies can be found by the following complicated procedure, copied verbatim
out of page 308 of \emph{Winning Ways}:
\begin{quote}
\begin{center}\emph{To find the tallies from options:}\end{center}
Shortlist all the $G^L$ with the GREATEST RIGHT toll, and all the $G^R$ with the LEAST LEFT toll.  Then on each side select
the tally with the LARGEST EVEN timer if there is one, and otherwise the LEAST ODD timer, obtaining the form
\[ G = \{\ldots x_a|y_b \ldots\}.\]
\begin{itemize}
\item If $x > y$ (HOT), the tallies are $x_{a+1}y_{b+1}$.
\item If $x < y$ (COLD), $G$ is the simplest number between $x$ and $y$, including $x$ as a possibility just if $a$ is odd,
$y$ just if $b$ is odd.
\item If $x = y$ (TEPID), try $x_{a+1}y_{b+1}$.  But if \emph{just one} of $a+1$ and $b+1$ is an even number,
increase the other (if necessary) by just enough to make it a larger odd number.  If \emph{both} are even, replace each of
them by 0.
\end{itemize}
\end{quote}
Here they are identifying a number $z$ with tallies $z_0z_0$.  If I understand \emph{Winning Ways} correctly,
the cases where there are no left options or no right options fall under the COLD case.

\subsection{Loopy games}
Another part of \emph{Winning Ways} Volume 2 considers \emph{loopy} games, which have no guarantee of ever
ending.  The situation where play continues indefinitely are \emph{draws}, are ties by default.  However,
we actually allow games to specify the winner of every infinite sequence of plays.  Given a game
$\gamma$, the variants $\gamma^+$ and $\gamma^-$ are the games formed by resolving all ties in favor of
Left and Right, respectively.

Sums of infinite games are defined in the usual way, though to specify the winner,
the following rules are used:
\begin{itemize}
\item If the sum of the games comes to an end, the winner is decided by the normal rule, as usual.
\item Otherwise, if Left or Right wins every component in which play never came to an end, then Left or Right wins the sum.
\item Otherwise the game is a tie.
\end{itemize}

Given sums, we define equivalence by $G = H$ if $G + K$ and $H + K$ have the same outcome under perfect play
for every loopy game $K$.

A \emph{stopper} is a game which is guaranteed to end when played in isolation, because it has no infinite
\emph{alternating} sequences of play, like
\[ G \to G^L \to G^{LR} \to G^{LRL} \to G^{LRLR} \to \cdots\]
Finite stoppers have canonical forms in the same way that loopfree partizan games do.

For most (but not all\footnote{For some exceptional game $\gamma$, $\gamma^+$ and $\gamma^-$
can not be taken to be stoppers, but this does not happen for most of the situations considered in \emph{Winning Ways}.})
loopy games $\gamma$, there exist stoppers $s$ and $t$ such that
\[ \gamma^+ = s^+ \text{ and } \gamma^- = t^-.\]
These games are called the \emph{onside} and \emph{offside} of $\gamma$ respectively, and we
write $\gamma = s\&t$ to indicate this relationship.  These stoppers can be found by the operation
of ``sidling'' described on pages 338-342 of \emph{Winning Ways}.  It is always the case that $s \ge t$.
When $\gamma$ is already a stopper,
$s$ and $t$ can be taken to be $\gamma$.  

Given two games $s\&t$ and $x\&y$, the sum $(s\&t) + (x\&y)$ is $u\&v$ where
$u$ is the \emph{upsum} of $s$ and $x$, while $v$ is the \emph{downsum}
of $t$ and $y$.  The upsum of two games is the onside of their sum,
and the downsum is the offside of their sum.

\subsection{Mis\`ere games}
\emph{Winning ways} Chapter 13 ``Survival in the Lost World'' and \emph{On Numbers and Games} Chapter 12 ``How to Lose when you Must''
both consider the theory of impartial mis\`ere games.  These are exactly like normal impartial games, except
that we play by a different rule, the \emph{mis\`ere rule} in which the last player able to move loses.  The theory
turns out to be far more complicated and lest satisfactory than the Sprague-Grundy theory for normal impartial games.
For games $G$ and $H$, Conway says that $G$ is \emph{like} $H$ if $G + K$ and $H + K$ have the same mis\`ere outcome
for all $K$, and then goes on to show that every game $G$ has a canonical simplest form, modulo this relation.
However, the reductions allowed are not very effective, in the sense that the number of mis\`ere impartial games born on
day $n$ grows astronomically, like the sequence
\[ \left\lceil \gamma_0 \right\rceil, \left\lceil  2^{\gamma_0} \right\rceil, \left\lceil  2^{2^{\gamma_0}} \right\rceil, \left\lceil  2^{2^{2^{\gamma_0}}}\right\rceil,
\ldots\]
for $\gamma_0 \approx 0.149027$ (see page 152 of \emph{ONAG}).
Conway is able to give more complete analyses of certain ``tame'' games which behave similar to mis`ere Nim positions,
and \emph{Winning Ways} contains additional comments about games that are almost tame
but in general, the theory is very spotty.  For example, these results do not provide a complete analysis
of Mis\`ere Kayles.

\section{Mis\`ere Indistinguishability Quotients}
However, a solution of Mis\`ere Kayles was obtained through other means by William Sibert.  Sibert found a complete description
of the Kayles positions for which mis\`ere outcome differs from normal outcome.  His solution can be found on page
446-451 of \emph{Winning Ways}.

Let $\mathcal{K}$ be the set of all Kayles positions.  We say that $G$ and $H \in \mathcal{K}$ are \emph{indistinguishable}
if $G + X$ and $H + X$ have the same mis\`ere outcome for every $X$ \emph{in $\mathcal{K}$}.  If we let $X$ range over
\emph{all} impartial games, this would be the same as Conway's relation $G$ ``is like'' $H$.  By limiting $X$ to range
over only positions that occur in Kayles, the equivalence relation becomes coarser, and the quotient space becomes
smaller.  In fact, using Sibert's solution, one can show that the quotient space has size 48.  An alternate way of describing Sibert's
solution is to give a description of this monoid, a table showing which equivalence classes have which outcomes, and a table
showing which element of the monoid corresponds to a Kayles row of each possible length.

This sort of analysis has been extended to many other mis\`ere games by Plambeck, Siegel, and others.  For a given class of
mis\`ere games, let $\mathcal{G}$ be the closure of this class under addition.  Then for $X, Y \in \mathcal{G}$, we say
that $X$ and $Y$ are \emph{indistinguishable} if $X + Z$ and $Y + Z$ have the same mis\`ere outcome for all $z \in \mathcal{G}$.
We then let the \emph{indistinguishability quotient} be $\mathcal{G}$ modulo indistinguishability.  The point of this construction
is that
\begin{itemize}
\item The indistinguishablity quotient is a monoid, and there is a natural surjective monoid homomorphism (the ``pretending function'') from $\mathcal{G}$
(as a monoid with addition)
to the indistinguishability quotient.
\item There is a map from the indistinguishability quotient to the set of outcomes, whose composition with the pretending function yields
the map from games to their outcomes.
\end{itemize}
Using these two maps, we can then analyze any sum of games in $\mathcal{G}$, assuming the structure of the indistinguishability quotient is manageable.

For many cases, like Kayles, the indistinguishability quotient is finite.  In fact Aaron Siegel has written software to
calculate the indistinguishability quotient when it is finite, for a large class of games.
This seems to be the best way to solve or analyze mis\`ere games so far.

\section{Indistinguishability in General}
The general setup of (additive) combinatorial game theory could be described as follows: we have
a collection of \emph{games,} each of which has an \emph{outcome}.  Additionally, we have various \emph{operations} - ways of combining games.
We want to characterize each game with a simpler object, a \emph{value}, satisfying two conditions.  First of all, the outcome
of a game must be determined by the value, and second, the value of a combination of games must be determined by the values of the games
being combined.  The the value of a game contains all the information about the game that we care about, and two games having
the same value can be considered \emph{equivalent}.

Our goal is to make the set of values as simple and small as possible.  We first throw out values that correspond to no games,
making the map from games to values a surjection.  Then the set of values becomes the quotient space of games modulo equivalence.
This quotient space will be smallest when the equivalence relation is coarsest.

However, there are two requirements on the equivalence relation.  First of all, it needs to respect outcomes: if two games are
equivalent, then they must have the same outcome.  And second, it must be \emph{compatible} with the operations on games, so that
the operations are well-defined on the quotient space.

\emph{Indistinguishability} is the uniqe coarsest equivalence relation satisfying these properties, and the \emph{indistinguishability quotient}
of games modulo indistinguishability is thus the smallest set of values that are usable.  It thus provides a canonical and optimal
solution to the construction of the set of ``values.''

The basic idea of indistinguishability is that two games should be indistinguishable if they are fully \emph{interchangeable},
meaning that they can be exchanged in any context within a larger combination of games, without changing the outcome of the entire combination.
For example if $G$ and $H$ are two indistinguishable partizan games, then
\begin{itemize}
\item $G$ and $H$ must have the same outcome.
\item $23 + G + \downarrow$ and $23 + H + \downarrow$ must have the same outcome.
\item $\{17*|G,6 - G\}$ and $\{17*|H,6 - H\}$ must have the same outcome
\item And so on\ldots
\end{itemize}
Conversely, if $G$ and $H$ are not indistinguishable, then there must be some context in which they cannot be interchanged.

The notion of indistinguishability is a relative one, that depends on the class of games being considered, the map from games to
outcomes, and the set of operations being considered.  Restricting the class of games makes indistinguishability coarser,
which is how mis\`ere indistinguishability quotients are able to solve games like Mis\`ere Kayles, even when we cannot
classify positions of Mis\`ere Kayles up to indistinguishability in the broader context of all mis\`ere impartial games.

Similarly, adding new operations into the mix makes indistinguishability finer.  In the case of partizan games, by a lucky coincidence
indistinguishability for the operation of addition alone is already compatible with negation and game-construction, so adding in
these other two operations does not change indistinguishability.  In other contexts this might not always work.

While never defined formally, the notion of indistinguishability is implicit in every chapter of the second volume of \emph{Winning Ways}.
For example, one can show that if our class of games is partizan games and our operation is unions, then two games $G$
and $H$ are indistinguishable if and only if they have the same tally.  Similarly, if we are working with
impartial games and joins, then two games $G$ and $H$ are indistinguishable if and only if they have the same remoteness
(this follows by Theorem~\ref{associndist} below and what was shown in the first paragraph of the proof of Corollary~\ref{remoteness2} above).
For mis\`ere impartial games, our indistinguishability agrees with the usual definition used by Siegel and Plambeck, because of
Theorem~\ref{associndist} below.  And for the standard theory of sums of partizan games, indistinguishability
will just be the standard notion of equality that we have used so far.

To formally define indistinguishability, we need some notation.  Let $S$ be a set of ``games,'' $O$ a set of ``outcomes,''
and $\outcome:S \to O$ a map which assigns an outcome to each game.  Let $f_1, \ldots, f_k$ be ``operations''
$f_i:S^{n_i} \to S$ on the set of games.
\begin{theorem}\label{indist}
There is a unique \emph{largest}
equivalence relation $\sim$ on $S$ having the following properties:
\begin{description}
\item[(a)] If $x \sim y$ then $\outcome(x) = \outcome(y)$.
\item[(b)] If $1 \le i \le k$, and if $x_1, \ldots, x_{n_i}, y_1, \ldots, y_{n_i}$ are games in $S$
for which $x_j \sim y_j$ for every $1 \le j \le n_i$, then $f_i(x_1,\ldots,x_{n_i}) \sim f_i(y_1,\ldots,y_{n_i})$.
\end{description}
\end{theorem}
So if we have just one operation, say $\oplus$, then $\sim$ is the largest equivalence relation
such that
\[ x_1 \sim y_1 \text{ and } x_2 \sim y_2 \implies x_1 \oplus x_2 \sim y_1 \oplus y_2,\]
and such that $x \sim y$ implies that $x$ and $y$ have the same outcome.
These conditions are equivalent to the claim
that $\oplus$ and $\outcome(\cdot)$ are well-defined on the quotient space of $\sim$.
\begin{proof}
For notational simplicity, we assume that there is only one $f$, and that its arity is 2: $f:S^2 \to S$.  The proof
works the same for more general situations.

Note that as long as $\sim$ is an equivalence relation, (b) is logically equivalent to the following assumptions
\begin{description}
\item[(c1)] If $x \sim x'$, then $f(x,y) \sim f(x',y)$.
\item[(c2)] If $y \sim y'$, then $f(x,y) \sim f(x,y')$.
\end{description}
For if
(b) is satisfied, then (c1) and (c2) both follow by reflexitivity of $\sim$.  On the other hand, given
(c1) and (c2), $x_1 \sim y_1$ and $x_2 \sim y_2$ imply that
\[ f(x_1,y_1) \sim f(x_2,y_1) \sim f(x_2,y_2),\]
using (c1) and (c2) for the first and second $\sim$, so that by transitivity $f(x_1,y_1) \sim f(x_2,y_2)$.
These proofs easily generalize to the case where there is more than one $f$ or higher arities, though we need
to replace (c1) and (c2) with $n_1 + n_2 + \cdots + n_k$ separate conditions, one for each parameter
of each function.

We show that there is a unique largest relation satisfying (a), (c1) and (c2), and that it is an equivalence relation.
This clearly implies our desired result.

Let $\mathcal{T}$ be the class of all relations $R$ satisfying
\begin{description}
\item[(A)] If $\outcome(x) \ne \outcome(y)$, then $x ~R~ y$.
\item[(C1)] If $f(x,a) ~R~ f(y,a)$, then $x ~R~ y$.
\item[(C2)] If $f(a,x) ~R~ f(a,y)$, then $x ~R~ y$.
\end{description}
It's clear that $R$ satisfies (A), (C1), and (C2) if an only if
the complement of $R$ satisfies (a), (c1), and (c2).  Moreover,
there is a unique smallest element $\not\sim$ of $\mathcal{T}$, the intersection
of all relations in $\mathcal{T}$, and its complement is the unique largest
relation satisfying (a), (c1), and (c2).  We need to show that the complement $\sim$
of this minimal relation $\not\sim$ is an equivalence relation.

First of all, the relation $\ne$ also satisfies (A), (C1), (C2).  By minimality
of $\not\sim$, it follows that $x \not\sim y \implies x \ne y$, i.e., $x = y \implies x \sim y$.
So $\sim$ is reflexive.

Second of all, if $R$ is any relation in $\mathcal{T}$, then the transpose relation $R'$ given by
$x~R'~y \iff y~R~x$ also satisfies (A), (C1), and (C2).  Thus $\not\sim$ must lie inside
its transpose: $x \not\sim y \implies y \not\sim x$,  and therefore $\sim$ is symmetric.

Finally, to see that $\not\sim$ is transitive, let $R$ be the relation given by
\[ x~R~z \iff \forall y \in S:~x\not\sim y \vee y \not\sim z\]
where $\vee$ here means logical ``or.''  I claim that $R \in \mathcal{T}$.  Indeed
\[ \outcome(x) \ne \outcome(z) \implies \forall y \in S:~ \outcome(x) \ne \outcome(y) \vee \outcome(y) \ne \outcome(z)
\]\[\implies  \forall y \in S:~ x \not \sim y \vee y \not \sim z \implies x ~ R ~ z\]
so (A) is satisfied.  Similarly, for (C1):
\[ f(x,a) ~ R ~ f(z,a) \implies \forall y \in S:~f(x,a) \not\sim y \vee y \not\sim f(z,a)
\]\[ \implies \forall y \in S:~f(x,a) \not\sim f(y,a) \vee f(y,a) \not\sim f(z,a) \implies \]\[
\forall y \in S:~x \not \sim y \vee y \not\sim z \implies x ~ R ~ z,\]
using the fact that $\not\sim$ satisfies (C1).  A similar argument shows that
$R$ satisfies (C2).  Then by minimality of $\not\sim$, we see that $x\not\sim y \implies x~R~y$, i.e.,
\[ x \not \sim z \implies \forall y \in S:~ x\not\sim y \vee y \not\sim z\]
which simply means that $\sim$ is transitive.
\end{proof}

\begin{definition}
Given a class of games and a list of operations on games, we define \emph{indistinguishability (with respect to the given operations)}
to be the equivalence relation from the previous theorem, and denote it as
$\approx_{f_1,f_2,\ldots,f_k}$.  The quotient space of $S$ is the \emph{indistinguishability quotient}.
\end{definition}

In the case where there is a single binary operation, turning the class of games into a commutative monoid,
indistinguishability has a simple definition:
\begin{theorem}\label{associndist}
Suppose that $\otimes:G \times G \to G$ is commutative and associative and has an identity $e$.  Then $G, H \in S$ are indistinguishable (with
respect to $\otimes$)
if and only if $\outcome(G \otimes X) = \outcome(G \otimes X)$ for every $X \in S$.
\end{theorem}
\begin{proof}
Let $\rho$ be the relation $G\,\rho\,H$ iff $\outcome(G \otimes X) = \outcome(H \otimes X)$ for every $X \in S$.
I first claim that $\rho$ satisfies conditions (a) and (b) of Theorem~\ref{indist}.  For (a), note that
\[ G \,\rho\,H \Rightarrow \outcome(G \otimes e) = \outcome(H \otimes e).\]
But $G \otimes e = G$ and $H \otimes e = H$, so $G\,\rho\,H \Rightarrow \outcome(G) = \outcome(H)$.
For (b), suppose that $G \,\rho\, G'$ and $H\,\rho,H'$.  Then $G \otimes H \,\rho\, G' \otimes H'$, because for any
$X \in S$,
\[ \outcome((G \otimes H) \otimes X) = \outcome(G \otimes (H \otimes X)) = \outcome(G' \otimes (H \otimes X)) =\]\[
\outcome(H \otimes (G' \otimes X)) = \outcome(H' \otimes (G' \otimes X)) = \outcome((G' \otimes H') \otimes X).\]
It then follows that if $\sim$ is true indisinguishability, then $\sim$ must be coarser than $\rho$, i.e.,
$\rho \subseteq (\sim)$.  On the other hand, suppose that $G$ and $H$ are indistinguishable, $G \sim H$.  Then
for any $X \in S$ we must have
\[ G \otimes X \sim H \otimes X,\]
so that $\outcome(G \otimes X) = \outcome(H \otimes X)$.  Thus $(\sim) \subseteq \rho$, and so $\rho$ is true
indistinguishability and we are done.
\end{proof}

For the standard theory of sums of normal play partizan games, indistinguishability is just equality:
\begin{theorem}
In the class of partizan games with normal outcomes, indistinguishability with respect to addition
is equality.
\end{theorem}
\begin{proof}
By the previous theorem, $G$ and $H$ are indistinguishable if and only if $G + X$ and $H + X$ have the same
outcome for all $X$.  Taking $X \equiv -G$, we see that $G + (-G)$ is a second player win (a zero game), and
therefore $H + (-G)$ must also be a second player win.  But this is the definition of equality, so $G = H$.

Conversely, if $G = H$, then $G + X$ and $H + X$ are equal, and so have the same outcome, for any $X$.
\end{proof}
Note that this is indistinguishability for the operation of \emph{addition}.  We could also throw the operations
of negation and game-building ($\{\cdots|\cdots\}$) into the mix, but they would not change indistinguishability,
because they are already compatible with equality, by Theorem~\ref{congruences}.

In the case where there is a poset structure on the class of outcomes $O$, the indistinguishability quotient
inherits a partial order, by the following theorem:
\begin{theorem}
Suppose $O$ has a partially ordered structure.  Then there is a maximum reflexive and transitive relation
$\lesssim$ on the set of games $S$ such that
\begin{itemize}
\item If $G \lesssim H$ then $\outcome(G) \le \outcome(H)$.
\item For every $i$, if $G_1, \ldots, G_{n_i}$ and $H_1, \ldots, H_{n_i}$ are such that
$G_j \lesssim H_j$ for every $j$, then $f_i(G_1,\ldots,G_{n_i}) \lesssim f_i(H_1,\ldots,H_{n_i})$.
\end{itemize}
Moreover, $G \lesssim H$ and $H \lesssim G$ if and only if $G$ and $H$ are indistinguishable.
\end{theorem}
For example, in the case of partizan games, the four outcomes are arranged into a poset as in Figure~\ref{outcome-poset},
and this partial order gives rise to the $\le$ order on the class of games modulo equivalence.\footnote{But note that
throwing negation into the mix now breaks everything, because negation is order-reversing!  It's probably possible to
flag certain operations as being order-reversing, and make everything work out right.}
\begin{figure}[tb]
\begin{center}
\includegraphics[width=2in]
					{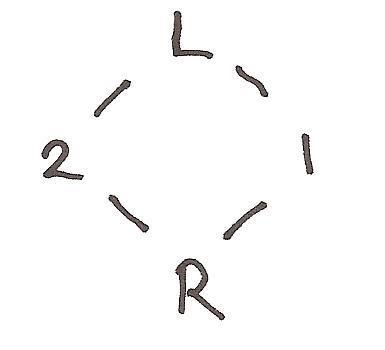}
\caption{From Left's point of view, $L$ is best, $R$ is worst, and $1$ and $2$ are inbetween,
and incomparable with each other.  Here $L$ denotes a win for Left, $R$ denotes a win for Right, $1$
denotes a win for the first player, and $2$ denotes a win for the second player.}
\label{outcome-poset}
\end{center}
\end{figure}
\begin{proof}
The proof is left as an exercise to the reader, though it seems like it is probably completely analogous
to the proof of Theorem~\ref{indist}.  To show that $\lesssim \cap \gtrsim$ is $\sim$, use the fact
that $\lesssim \cap \gtrsim$ satisfies (a) and (b) of Theorem~\ref{indist}, while $\sim$ satisfies
(a) and (b) of this theorem.
\end{proof}
In the case where we have a single commutative and associative operation with identity, we have
the following analog of Theorem~\ref{associndist}:
\begin{theorem}
With the setup of the previous theorem, if $\otimes$ is the sole operation, and $\otimes$ has an identity,
then $G \lesssim H$ if and only if $\outcome(G \otimes X) \le \outcome(H \otimes X)$ for all $X$.
\end{theorem}
The proof is left as an exercise to the reader.

\part{Well-tempered Scoring Games}
\chapter{Introduction}
\section{Boolean games}
The combinatorial game theory discussed so far doesn't seem very relevant to the game \textsc{To Knot or Not to Knot}.
The winner of TKONTK is decided by neither the normal rule or the mis\`ere rule, but is instead specified
explicitly by the game.  On one hand, TKONTK feels impartial, because at each position, both players have
identical options, but on the other hand, the positions are clearly not symmetric between the two players - a position
can be a win for Ursula no matter who goes first, unlike any impartial game.
Moreover, our way of combining games is asymmetric, favoring King Lear in the case where each player
won a different component.

By the philosophy of indistinguishability, we should consider the class of all positions in TKONTK, and the operation
of connected sum, and should determine the indistinguishability quotient.  This enterprise is very complicated, so
we instead consider a \emph{larger} class of games, with a combinatorial definition, and apply the same methodology to them.
\begin{definition}
A \emph{Boolean game} born on day $n$ is
\begin{itemize}
\item One of the values \textsc{True} or \textsc{False} if $n = 0$.
\item A pair $(L,R)$ of two finite sets $L$ and $R$ of Boolean games born on day $n - 1$, if $n > 0$.  The elements
of $L$ and $R$ are called the \emph{left options} and \emph{right options} of $(L,R)$.
\end{itemize}
We consider a game born on day $0$ to have no options of either sort.

The sum $G \vee H$ of two Boolean games $G$ and $H$ born on days $n$ and $m$ is the logical OR of $G$ and $H$ when $n = m = 0$,
and is otherwise recursively defined as
\[ G \vee H = (\{G^L \vee H, G \vee H^L\},\{G^R \vee H, G \vee H^R\}),\]
where $G^L$ ranges over the left options of $G$, and so on.

The \emph{left outcome} of a Boolean game $G$ is \textsc{True} if $G = $\textsc{True}, or
some right option of $G$ has right outcome \textsc{True}.  Otherwise, the left outcome of $G$ is \textsc{False}.

Similarly, the \emph{right outcome} of a Boolean game $G$ is \textsc{False} if $G = $\textsc{False}, or
some left option of $G$ has left outcome \textsc{False}.  Otherwise, the left outcome of $G$ is \textsc{True}.
\end{definition}
In other words, a Boolean game is a game between two players that ends with either Left = True winning, or Right = False winning.
But we require that all sequences of play have a prescribed length.  In other words, if we make a gametree, every leaf must be at the same depth:
\begin{figure}[H]
\begin{center}
\includegraphics[width=5in]
					{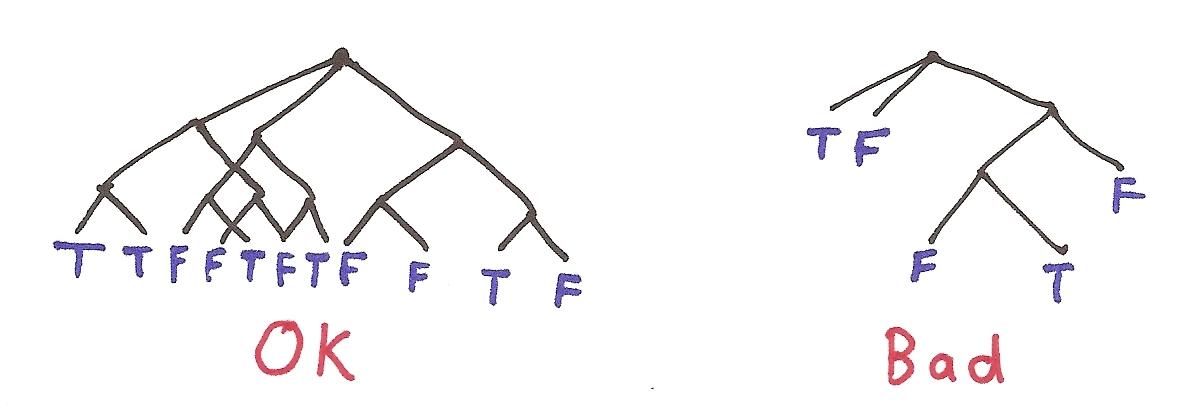}
%\caption{}
%\label{domineering-sum}
\end{center}
\end{figure}
This includes the case of TKONTK, because the length of a game of TKONTK is a fixed number, namely the number of unresolved crossings initially present.

The indistinguishability-quotient program can be carried out for Boolean games, by brute force means.  When I first tried to analyze
these games, this was the approach that I took.  It turns out that there are exactly 37 types of Boolean games, modulo
indistinguishability.  Since addition of Boolean games is commutative and associative, the quotient space (of size 37) has a monoid structure,
and here is part of it, in my original notation:
\begin{center}
\begin{figure}[H]
\begin{center}
\begin{tabular}{l || l | l | l | l | l | l | l | l}
 & $00$ & $01^-$ & $01^+$ & $11$ & $02$ & $12^-$ & $12^+$ & $22$ \\
\hline
\hline
$00$ & $00$ & $01^-$ & $01^+$ & $11$ & $02$ & $12^-$ & $12^+$ & $22$ \\ \hline
$01^-$ & $01^-$ & $02$ & $12^-$ & $12^-$ & $12^+$ & $12^+$ & $22$ & $22$ \\ \hline
$01^+$ & $01^+$ & $12^-$ & $12^+$ & $12^+$ & $12^+$ & $22$ & $22$ & $22$ \\ \hline
$11$ & $11$ & $12^-$ & $12^+$ & $22$ & $12^+$ & $22$ & $22$ & $22$ \\ \hline
$02$ & $02$ & $12^+$ & $12^+$ & $12^+$ & $22$ & $22$ & $22$ & $22$ \\ \hline
$12^-$ & $12^-$ & $12^+$ & $22$ & $22$ & $22$ & $22$ & $22$ & $22$ \\ \hline
$12^+$ & $12^+$ & $22$ & $22$ & $22$ & $22$ & $22$ & $22$ & $22$ \\ \hline
$22$ & $22$ & $22$ & $22$ & $22$ & $22$ & $22$ & $22$ & $22$ \\
\end{tabular}
\caption{There is no clear rule governing this table, which is mostly verified by a long case-by-case analysis.  But compare
with Figure~\ref{cap-cup-tables} below!}
\label{old-school}
\end{center}
\end{figure}
\end{center}

Subsequently, I found a better way of describing Boolean games, by viewing them as part of a larger class of \emph{well-tempered scoring games}.
While the end result takes longer to prove, it seems like it gives better insight into what is actually happening. For instance, it helped me relate
the analysis of Boolean games to the standard theory of partizan games.  We will return to Boolean games in Chapter~\ref{chap:bool},
and give a much cleaner explanation of the mysterious 37 element monoid mentioned above.

\section{Games with scores}
A \emph{scoring game} is one in which the winner is determined by a final score rather than by the normal play rule or the mis\`ere rule.
In a loose sense this includes games like Go and Dots-and-Boxes.  Such games can be added in the usual way, by playing two in parallel.
The final score of a sum is obtained by adding the scores of the two summands.  This is loosely how independent positions
in Go and Dots-and-Boxes are added together.

Scoring games were first studied by John Milnor in a 1953 paper ``Sums of Positional Games'' in \emph{Contributions to the Theory of Games},
which was one of the earliest papers in unimpartial combinatorial game theory.  Milnor's paper
was followed in 1957 by Olof Hanner's paper ``Mean Play of Sums of Positional Games,'' which studied the mean values of games, well before the later
work of Conway, Guy, and Berlekamp.

The \emph{outcome} of a scoring game is the final score under perfect play (Left trying to maximize the score, Right trying to minimize the score).
There are actually two outcomes, the left outcome and the right outcome, depending on which player goes first.  Milnor and Hanner
only considered games in which there was a non-negative incentive to move, in the sense that the left outcome was always
as great as the right outcome - so each player would prefer to move rather than pass.
This class of games forms a group modulo indistinguishability, and is closely connected to the later theory of partizan games.

Many years later, in the 1990's, J. Mark Ettinger studied the broader class of all scoring games, and tried to show that scoring games
formed a cancellative monoid.  Ettinger refers to scoring games as ``positional games,'' following Milnor and Hanner's terminology.  However,
the term ``positional game'' is now a standard synonym for \emph{maker-breaker games}, like Tic-Tac-Toe or Hex.\footnote{These are games in which
players take turns placing pieces of their own color on the board, trying to make one of a prescribed list of configurations with their pieces.
In Tic-Tac-Toe, a player wins by having three pieces in a row.  In Hex, a player wins by having pieces in a connected path from
one side of the board to the other.}  Another name might be ``Milnor game,'' but Ettinger uses this to refer to the restricted
class of games studied by Milnor and Hanner, in which there is a nonnegative incentive to move.  So I will instead call the general class of games
``scoring games,'' following Richard Nowakowski's terminology in his \emph{History of Combinatorial Game Theory}.

To notationally separate scoring games from Conway's partizan games, we will use angle brackets rather than curly brackets to construct games.
For example
\[ X = \langle 0 | 4 \rangle\]
is a game in which Left can move to $0$ and Right can move
to $4$, with either move ending the game.  Similarly, $Y = \langle X, 4 | X, 0\rangle$ is a game in which Left can move to $4$
and Right can move to $0$ (with either move ending the game), but either player can also move to $X$.
To play $X$ and $Y$ together, we add the final scores, resulting in
\[ X + Y = \langle 0 + Y, X + X, X + 4 \,|\, 4 + Y, X + X, X + 0 \rangle =\]\[
\langle Y, \langle X \,|\, 4 + X \rangle, 4 + X \,|\, 4 + Y, \langle X \,|\, 4 + X \rangle, X \rangle,\]
where $4 + X = \langle 4 \,|\, 8 \rangle$ and $4 + Y = \langle 4 + X, 8 \,|\, 4 + X, 4 \rangle$.

Unlike the case of partizan games, we rule out games like
\[ \langle 3 \,|\, \rangle,\]
in which one player has options but the other does not.  Without this prohibition, we would need
an ad hoc rule for deciding the final outcome of $\langle 3 \,|\, \rangle$ in the case where Right goes first:
perhaps Right passes and lets Left move to 3, or perhaps Right gets a score of $0$ or $-\infty$.  Rather
than making an arbitrary rule, we follow Milnor, Hanner, and Ettinger and exclude this possibility.

\section{Fixed-length Scoring Games}
Unfortunately I have no idea how to deal with scoring games in general.
However, a nice theory falls out
if we restrict to
\emph{fixed-length scoring games} - those in which the duration of the game
from start to finish is the same under all lines of play.  The Boolean games defined in the previous section
are examples, if we identify \textsc{False} with 0, and \textsc{True} with 1.  So in particular, \textsc{To Knot or Not to Knot}
is an example.  But because its structure is opaque and unplayable, we present a couple alternative examples, that also
demonstrate a wider range of final scores.

\emph{Mercenary Clobber} is a variant of Clobber (see Section~\ref{sec:examples}) in which players have two types of moves allowed.
First, they can make the usual clobbering move, moving one of their own pieces onto one of their opponent's pieces.
But second, they can \emph{collect} any piece which is part of an \emph{isolated} group of pieces - a connected group of
pieces of a single color which is not connected to any opposing pieces.  Such isolated groups are no longer
accessible to the basic clobbering rule.  So in the following position:
\begin{figure}[H]
\begin{center}
\includegraphics[width=2.5in]
					{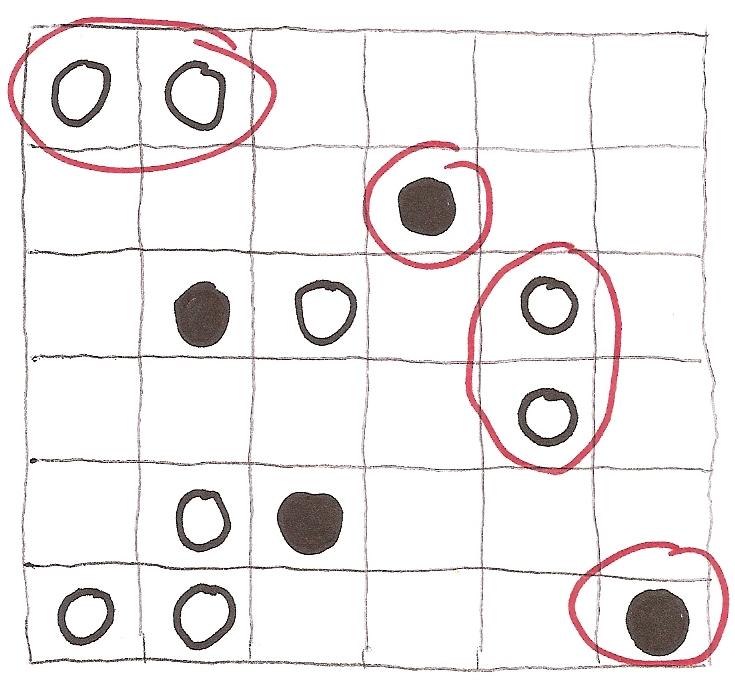}
%\caption{}
%\label{domineering-sum}
\end{center}
\end{figure}
the circled pieces are available for collection.  Note that you can collect your own pieces or your opponent's.
You get one point for each of your opponent's pieces you collect, and zero points for each of your own pieces.
Which player makes the last move is immaterial.  Your goal is to maximize your score (minus your opponent's).

Each move in Mercenary Clobber reduces the number of pieces on the board by one.  Moreover, the game does not end
until every piece has been removed: as long as at least one piece remains on the board, there is either an available
clobbering move, or at least one isolated piece.  Thus Mercenary Clobber is an example of a fixed-length scoring game.

\emph{Scored Brussel Sprouts} is a variant of the joke game \emph{Brussel Sprouts}, which is itself a variant of the
pen-and-paper game \emph{Sprouts}.  A game of Brussel Sprouts begins with a number of crosses:
\begin{figure}[H]
\begin{center}
\includegraphics[width=2in]
					{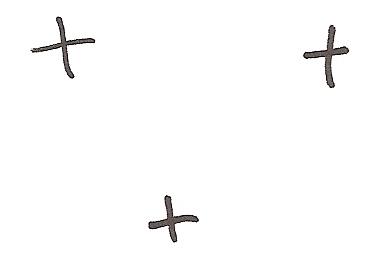}
%\caption{}
%\label{domineering-sum}
\end{center}
\end{figure}
Players take turns drawing lines which connect to of the loose ends.  Every time you draw a line, you make an additional
cross in the middle of the line:
\begin{figure}[H]
\begin{center}
\includegraphics[width=2in]
					{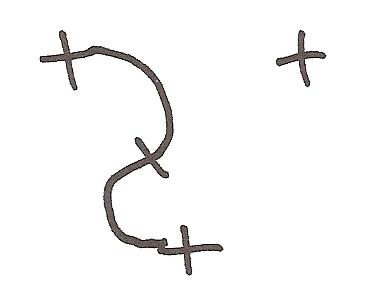}
%\caption{}
%\label{domineering-sum}
\end{center}
\end{figure}
Play continues in alternation until there are no available moves.  The first player unable to move loses.
\begin{figure}[H]
\begin{center}
\includegraphics[width=2in]
					{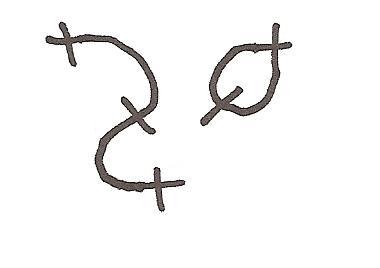}
%\caption{}
%\label{domineering-sum}
\end{center}
\end{figure}
\begin{figure}[H]
\begin{center}
\includegraphics[width=2in]
					{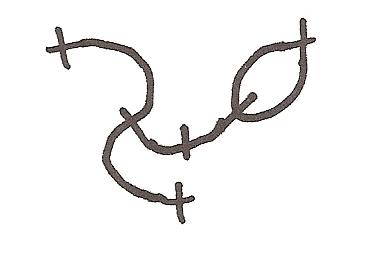}
%\caption{}
%\label{domineering-sum}
\end{center}
\end{figure}
\begin{figure}[H]
\begin{center}
\includegraphics[width=2in]
					{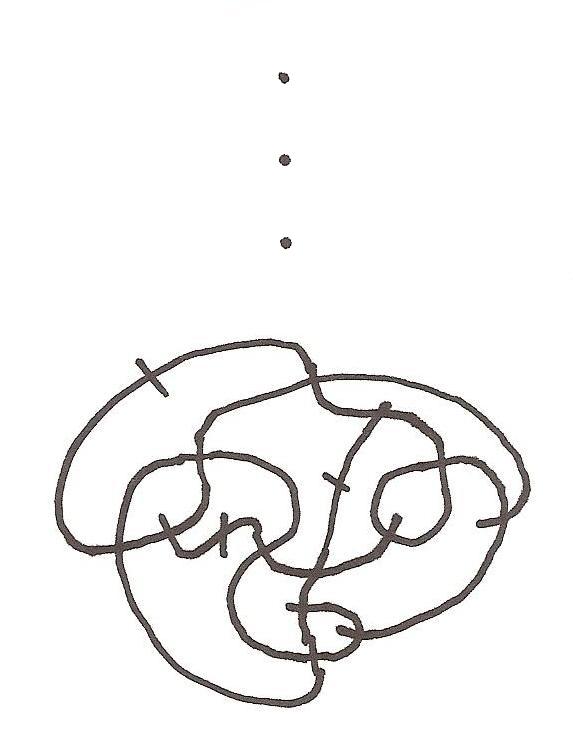}
%\caption{}
%\label{domineering-sum}
\end{center}
\end{figure}
The ``joke'' aspect of Brussel Sprouts is that the number of moves that the game will last is completely predictable
in advance, and therefore so is the winner.  In particular, the winner is not determined in any way by the actual decisions
of the players.  If a position starts with $n$ crosses, it will last exactly $5n - 2$ moves.  To see this, note first of all
that the total number of loose ends is invariant.  Second, each move either creates a new ``region'' or decreases the number
of connected components by one (but not both).  In particular, $regions - components$ increases by exactly one on each turn.
Moreover, it is impossible to make a region which has no loose ends inside of it, so in the final position, each region
must have exactly one loose end:
\begin{figure}[H]
\begin{center}
\includegraphics[width=3.5in]
					{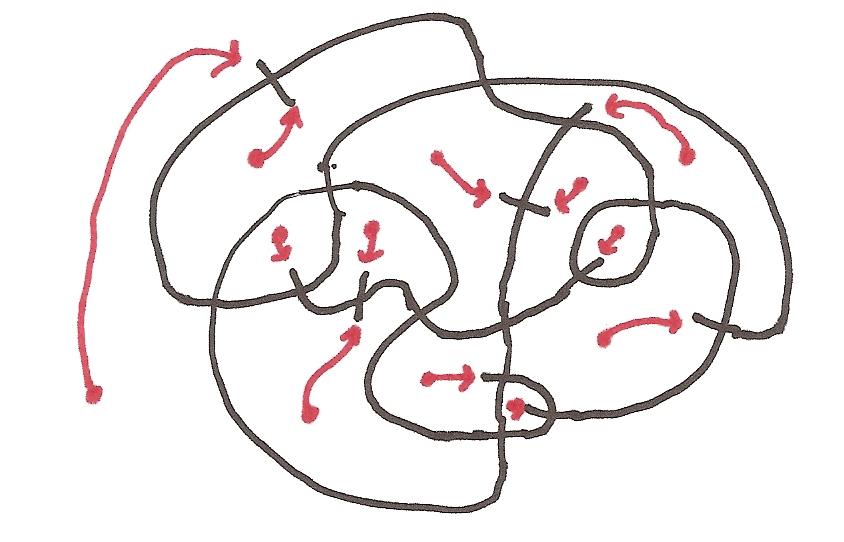}
%\caption{}
%\label{domineering-sum}
\end{center}
\end{figure}
or else there would be more moves possible.  Also, there must be exactly one connected component, or else there would
be a move connecting two of them:
\begin{figure}[H]
\begin{center}
\includegraphics[width=2in]
					{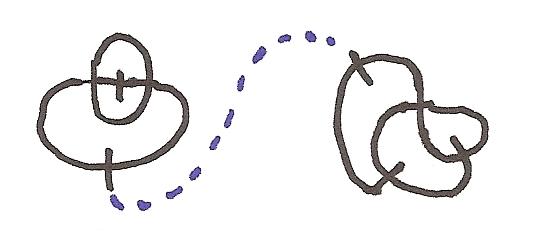}
\caption{If two connected components remain, each will have a loose end on its outside, so a move remains that can
connect the two.  A similar argument works in the case where one connected component lies inside another.}
\label{b-s-conn-comp}
\end{center}
\end{figure}

So in the final position, $regions - components = 4n - 1$ because there are $4n$ loose ends.  But initially,
there is only one region and $n$ components, so $regions - components = 1 - n$.  Therefore the total number
of moves is $(4n - 1) - (1 - n) = 5n - 2$.  So in particular, if $n$ is odd, then $5n - 2$ is odd, so the first
player to move will win, while if $n$ is even, then $5n - 2$ is even, and so the second player will win.

To make Brussel Sprouts more interesting, we assign a final score based on the regions that arise.  We give Left
one point for every triangLe, and Right one point for every squaRe.  We are counting a region as a ``triangle''
or a ``square'' if it has 3 or 4 corners (not counting the ones by the loose end).
\begin{figure}[H]
\begin{center}
\includegraphics[width=3.5in]
					{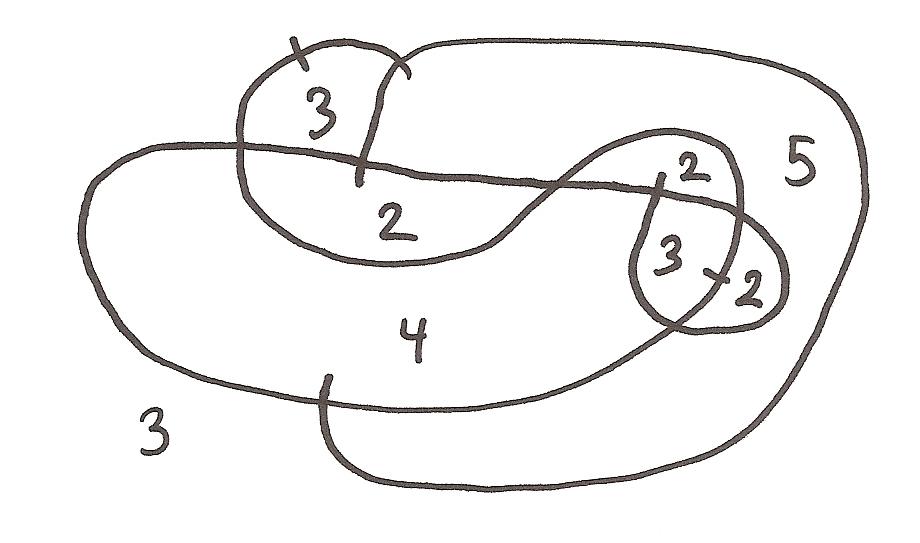}
\caption{The number of ``sides'' of each region.  Note that the outside counts as a region,
in this case a triangle.}
\label{sbsp-side-count}
\end{center}
\end{figure}
\begin{figure}[H]
\begin{center}
\includegraphics[width=3.5in]
					{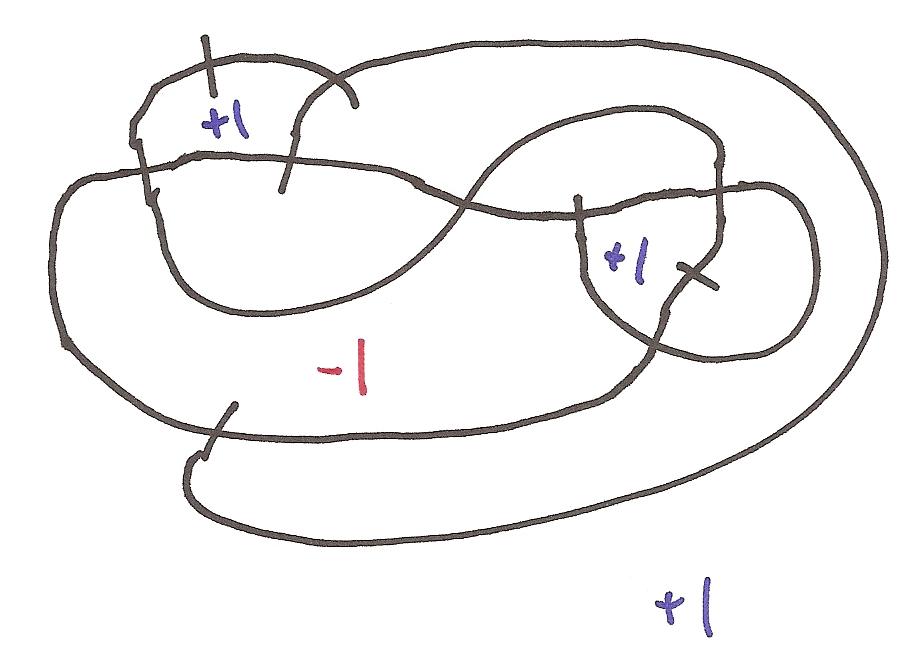}
\caption{A scored version of Figure~\ref{sbsp-side-count}.  There is one square and four triangles,
so Left wins by three points.}
\label{sbsp-scored}
\end{center}
\end{figure}
We call this variant Scored Brussel Sprouts.  Note that the game is no longer impartial, because
we have introduced an asymmetry between the two players in the scoring.

Both Mercenary Clobber and Scored Brussel Sprouts have a tendency to decompose into independent positions whose scores are combined by addition:
\begin{figure}[H]
\begin{center}
\includegraphics[width=2.7in]
					{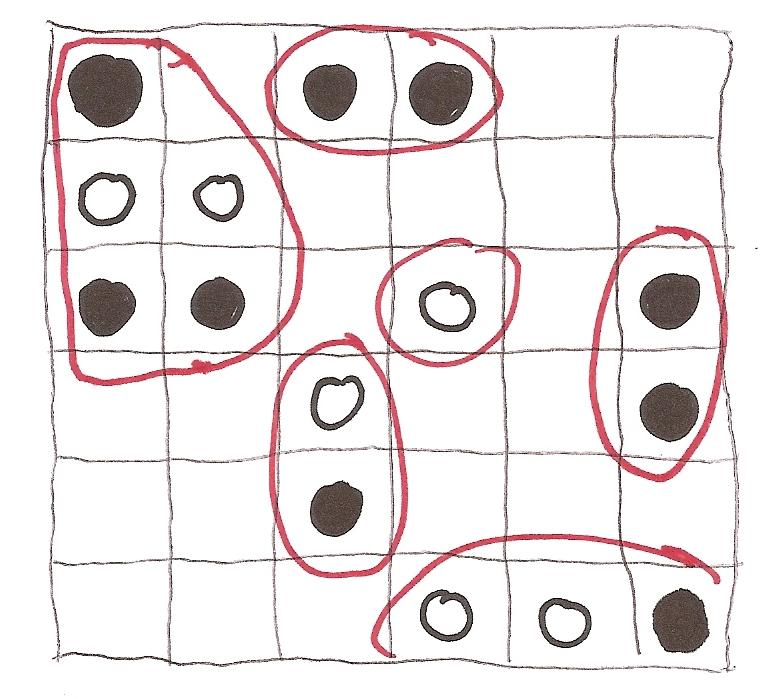}
\caption{Each circled region is an independent subgame.}
\label{mercenary-clobber-subdivision}
\end{center}
\end{figure}
In Scored Brussel Sprouts, something more drastic happens: each
individual region becomes its own subgame.  This makes the gamut of indecomposable positions
smaller and more amenable to analysis.
\begin{figure}[H]
\begin{center}
\includegraphics[width=2in]
					{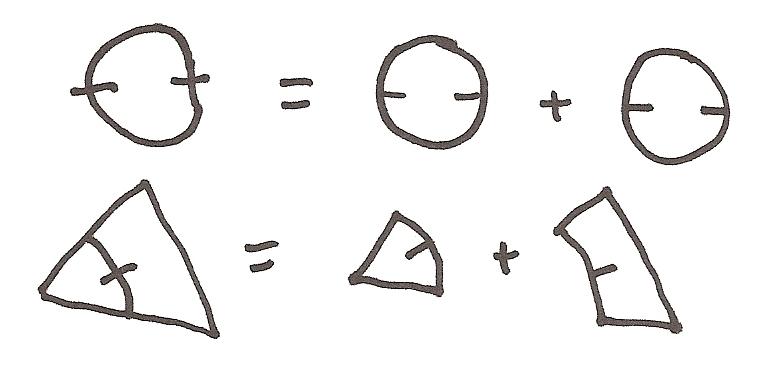}
\caption{A Scored Brussel Sprouts position is the sum of its individual cells.}
\label{sbsp-subdivision}
\end{center}
\end{figure}

Because both Mercenary Clobber and Scored Brussel Sprouts decompose into
independent subpositions, they will be directly amenable to the theory we develop.
In contrast, \textsc{To Knot or Not to Knot} does
not \emph{add} scores, but instead combines them by a maximum, or a Boolean \textsc{OR}.  We will see in Chapter~\ref{chap:bool}
why it can still be analyzed by the theory of fixed-length scoring games and addition.

For technical reasons we will actually consider a slightly larger class of games than fixed-length games.  We will study
\emph{fixed-parity} or \emph{well-tempered} games, in which the parity of the game's length is predetermined, rather than its exact length.
In other words, these are the games where we can say at the outset which player will make the last move.
While this class of games is much larger, we will see below (Corollary~\ref{fixed-length-exhausts}) that every fixed-parity game is equivalent
to a fixed-length game, so that the resulting indistinguishability quotients are identical.  I cannot think of any
natural examples of games which are fixed-parity but not fixed-length, since it seems difficult to arrange
for the game's length to vary but not its parity.

By restricting to well-tempered games, we are excluding strategic concerns of getting the last move, and thus one might expect that the resulting theory
would be orthogonal to the standard theory of partizan games.  However, it actually turns out to closely duplicate it, as we will
see in Chapter~\ref{chap:psi}.

There is one more technical restriction we make: we will only consider games taking values in the \emph{integers}.  Milnor, Hanner,
and Ettinger all considered games taking values in the full real numbers, but we will only allow integers, mainly so that Chapter~\ref{chap:psi}
works out.  Most of the results we prove will be equally valid for real-valued well-tempered games, and the full indistinguishability
quotient of real-valued well-tempered games can be described in terms of integer-valued games using the results of Chapter~\ref{chap:Distortions}.  We leave
these generalizations as an exercise to the motivated reader.

\chapter{Well-tempered Scoring Games}
\section{Definitions}
We will focus on the following class of games:\footnote{Our main goal will be determining the structure of this class of games, and the appropriate notion
of equivalence.  Because of this, we will \emph{not} use $\equiv$ and $=$ to stand for identity and equivalence (indistinguishability).
Instead, we will use $=$ and $\approx$ respectively.  For comparison, Ettinger uses $=$ and $\equiv$ (respectively!), while Milnor
uses $=$ and $\sim$ in his original paper.}
\begin{definition}
Let $\mathcal{S} \subseteq \mathbb{Z}$.  Then an \emph{even-tempered $\mathcal{S}$-valued game} is either an element
of $\mathcal{S}$ or a pair  $\langle L \,|\, R \rangle$
where $L$ and $R$ are finite nonempty sets of odd-tempered $\mathcal{S}$-valued games.
An \emph{odd-tempered $\mathcal{S}$-valued game} is a pair $\langle L \,|\, R \rangle$ where $L$ and $R$ are finite nonempty sets
of even-tempered $\mathcal{S}$-valued games.  A \emph{well-tempered $\mathcal{S}$-valued game} is an even-tempered $\mathcal{S}$-valued game or an odd-tempered
$\mathcal{S}$-valued game.  The set of well-tempered $\mathcal{S}$-valued games is denoted $\mathcal{W}_\mathcal{S}$.  The subsets of
even-tempered and odd-tempered games are denoted $\mathcal{W}_\mathcal{S}^0$ and $\mathcal{W}_\mathcal{S}^1$, respectively.  If $G = \langle L \,|\, R \rangle$,
then the elements of $L$ are called the \emph{left options} of $G$, and the elements of $R$ are called the \emph{right options}.
If $G = n$ for some $n \in \mathcal{S}$, then we say that $G$ has no left or right options.  In this case, we call
$G$ a \emph{number}.
\end{definition}
As usual, we omit the curly braces in $\langle \{L_1, L_2, \ldots \} \,|\, \{R_1, R_2, \ldots \} \rangle$, writing it
as $\langle L_1, L_2, \ldots \,|\, R_1, R_2, \ldots \rangle$ instead.  We also adopt some of the same notational conventions
from partizan theory, like letting $\langle x || y | z \rangle$ denote $\langle x | \langle y | z \rangle \rangle$.  We also
use $*$ and $x*$ to denote $\langle 0 | 0 \rangle$ and $\langle x | x \rangle$.

We will generally refer to well-tempered scoring games simply as ``games'' or ``$\mathbb{Z}$-valued games'' in what follows,
and refer to the games of Conway's partizan game theory as ``partizan games'' when we need them.

We next define outcomes.
\begin{definition}
For $G \in \mathcal{W}_\mathcal{S}$ we define $\loutcome(G) = \routcome(G) = n$ if $G = n \in \mathcal{S}$,
and otherwise, if $G = \langle L_1, L_2, \ldots \,|\, R_1, R_2, \ldots \rangle$, then we define
\[ \loutcome(G) = \max\{\routcome(L_1), \routcome(L_2), \ldots\}\]
\[ \routcome(G) = \min\{\loutcome(R_1), \loutcome(R_2), \ldots \}.\]
For any $G \in \mathcal{W}_\mathcal{S}$, $\loutcome(G)$ is called the \emph{left outcome}
of $G$, $\routcome(G)$ is called the \emph{right outcome} of $G$, and the ordered pair
$(\loutcome(G),\routcome(G))$ is called the \emph{(full) outcome} of $G$, denoted $\outcome(G)$.
\end{definition}
The outcomes of a game $G$ are just the final scores of the game when Left and Right play first,
and both players play perfectly.  It is clear from the definition that if
$G \in \mathcal{W}_\mathcal{S}$, then $\outcome(G) \in \mathcal{S} \times \mathcal{S}$.
We compare outcomes of games using the obvious partial order on $\mathcal{S} \times \mathcal{S}$.
So for example $\outcome(G_1) \le \outcome(G_2)$ iff $\routcome(G_1) \le \routcome(G_2)$ and $\loutcome(G_1) \le \loutcome(G_2)$.
In what follows, we will use bounds like
\[ \loutcome(G^R) \ge \routcome(G) \text{ for all $G^R$}\]
without explanation.

We next define operations on games.  For $\mathcal{S}, \mathcal{T} \subseteq \mathbb{Z}$, we let
$\mathcal{S} + \mathcal{T}$ denote $\{s + t\,:\, s \in \mathcal{S},\, t \in \mathcal{T}\}$,
and $-\mathcal{S}$ denote $\{-s\,:\, s \in \mathcal{S}\}$.
\begin{definition}
If $G$ is an $\mathcal{S}$-valued game, then its \emph{negative} $-G$ is the $(-\mathcal{S})$-valued game
defined recursively as $-n$ if $G = n \in \mathcal{S}$, and as
\[ -G = \langle -R_1, -R_2, \ldots \,|\, -L_1, -L_2, \ldots \rangle \]
if $G = \langle L_1, L_2, \ldots \,|\, R_1, R_2, \ldots \rangle$.
\end{definition}
Negation preserves parity: the negation of an even-tempered or odd-tempered game is an even-tempered or odd-tempered game.  Moreover,
$-(-G) = G$ for any game $G$.  It is also easy to check that $\loutcome(-G) = -\routcome(G)$
and $\routcome(-G) = -\loutcome(G)$.

We next define the sum of two games, in which we play the two games in parallel (like a sum of partizan games),
and add together the final scores at the end.
\begin{definition}
If $G$ is an $\mathcal{S}$-valued game and $H$ is a $\mathcal{T}$-valued game, then the \emph{sum}
$G + H$ is defined in the usual way if $G$ and $H$ are both numbers, and otherwise defined recursively as
\begin{equation} G + H = \langle G + H^L,\, G^L + H \,|\, G + H^R,\, G^R + H \rangle \label{combo}\end{equation}
where $G^L$ and $G^R$ range over all left and right options of $G$, and $H^L$ and $H^R$ range over all left and right
options of $H$.
\end{definition}
Note that (\ref{combo}) is used even when one of $G$ and $H$ is a number but the other isn't.  For instance,
\[ 2 + \langle 3 | 4 \rangle = \langle 5 | 6 \rangle .\]
In this sense, number avoidance is somehow built in to our theory.

It is easy to verify that $0 + G = G = G + 0$ for any $\mathbb{Z}$-valued game $G$, and that
addition is associative and commutative.  Moreover, the sum of two even-tempered games or two odd-tempered games
is even-tempered, while the sum of an even-tempered and an odd-tempered game is odd-tempered.  Another important fact which we'll need later is the following:
\begin{proposition}\label{numavoid}
If $G$ is a $\mathbb{Z}$-valued game and $n$ is a number, then
\[ \loutcome(G + n) = \loutcome(G) + n\]
and
\[ \routcome(G + n) = \routcome(G) + n.\]
\end{proposition}
\begin{proof}
Easily seen inductively from the definition.  If $G$ is a number, this is obvious, and otherwise,
if $G = \langle L_1, L_2, \ldots \,|\, R_1, R_2, \ldots \rangle$, then $n$ has no options, so
\[ G + n = \langle L_1 + n, L_2 + n, \ldots \,|\, R_1 + n, R_2 + n, \ldots \rangle.\]
Thus by induction
\[ \loutcome(G + n) = \max\{\routcome(L_1 + n),\routcome(L_2 + n),\ldots\} =
\max\{\routcome(L_1) + n, \routcome(L_2) + n, \ldots\} = \loutcome(G) + n,\]
and similarly $\routcome(G + n) = \routcome(G) + n$.
\end{proof}

We also define $G - H$ in the usual way, as $G + (-H)$.

%
%TODO: this next section shouldn't appear till way later\ldots
%More generally, given sets $\mathcal{S}_1, \mathcal{S}_2, \ldots, \mathcal{S}_n, \mathcal{T} \subseteq \mathbb{Z}$ and a function
%$f\,:\, \mathcal{S}_1 \times \mathcal{S}_2 \times \cdots \times \mathcal{S}_n \to \mathcal{T}$, $f$ can
%be extended in an obvious way to
%\[ \tilde{f}\,:\, \mathcal{S}_1\mathbf{Par} \times \mathcal{S}_2\mathbf{Par} \times \cdots \times \mathcal{S}_n\mathbf{Par} \to \mathcal{T}\mathbf{Par}\]
%by letting $\tilde{f}(G_1, G_2, \ldots G_n)$ be the $\mathcal{T}$-valued game obtained by playing
%$G_1, G_2, \ldots, G_n$ in parallel, and combining the final scores at the end using $f$.  Usually we are only interested
%in the cases where $f$ is \emph{weakly-order-preserving}, in the sense that
%\[ f(x_1, x_2, \ldots) \le f(y_1, y_2, \ldots)\]
%whenever $x_i \le y_i$ for all $i$, though negation is an exception.  As with addition, the parity of
%such a composite game is determined by the parity of the components.
%
%For instance, in To Knot or Not to Knot, the games are $\{0,1\}$-valued, and the natural operation is Boolean or
%$\wedge$, given by $x \wedge y = \max(x,y)$ for $x, y \in \{0,1\}$.  We will later on consider
%$\{0,1\}\mathbf{Par}$ modulo $\wedge$-indistinguishability.

\section{Outcomes and Addition}
In the theory of normal partizan games, $G \ge 0$ and
$H \ge 0$ implied that $G + H \ge 0$, because Left could combine
her strategies in the two games to win in their sum.  Similarly,
for $\mathbb{Z}$-valued games, we have the following:
\begin{claim}\label{baseclaim}
If $G$ and $H$ are \emph{even-tempered} $\mathbb{Z}$-valued games,
and $\routcome(G) \ge 0$ and $\routcome(H) \ge 0$, then
$\routcome(G + H) \ge 0$.
\end{claim}
Left combines her strategies in $G$ and $H$.  Whenever Right
moves in either component, Left responds in the same component,
playing responsively.
Similarly, just as $G \rhd 0$ and $H \ge 0$ implied $G + H \rhd 0$ for partizan
games, we have
\begin{claim}
If $G, H$ are $\mathbb{Z}$-valued games, with $G$ odd-tempered and $H$ even-tempered, and
$\loutcome(G) \ge 0$ and $\routcome(H) \ge 0$, then $\loutcome(G + H) \ge 0$.
\end{claim}
Since $G$ is odd-tempered (thus not a number) and $\loutcome(G) \ge 0$, there must be some
left option $G^L$ with $\routcome(G^L) \ge 0$.  Then by the previous claim,
$\routcome(G^L + H) \ge 0$.  So moving first, Left can ensure a final score
of at least zero, by moving to $G^L + H$.

Of course there is nothing special about the score 0.  More generally, we have
\begin{itemize}
\item If $G$ and $H$ are even-tempered, $\routcome(G) \ge m$ and $\routcome(H) \ge n$,
then $\routcome(G + H) \ge m + n$.  In other words, $\routcome(G + H) \ge \routcome(G) + \routcome(H)$.
\item If $G$ is odd-tempered, $H$ is even-tempered, $\loutcome(G) \ge m$, and $\routcome(H) \ge n$,
then $\loutcome(G + H) \ge m + n$.  In other words, $\loutcome(G + H) \ge \loutcome(G) + \routcome(H)$.
\end{itemize}

We state these results in a theorem, and give formal inductive proofs:
\begin{theorem}\label{outsum}
Let $G$ and $H$ be $\mathbb{Z}$-valued games.  If $G$ and $H$ are both even-tempered,
then
\begin{equation} \routcome(G + H) \ge \routcome(G) + \routcome(H) \label{uno}\end{equation}
Likewise, if $G$ is odd-tempered and $H$ is even-tempered, then
\begin{equation} \loutcome(G + H) \ge \loutcome(G) + \routcome(H) \label{dos}\end{equation}
\end{theorem}
\begin{proof}
Proceed by induction on the complexity of $G$ and $H$.  If $G$ and $H$ are both even-tempered,
then (\ref{uno}) follows from Proposition~\ref{numavoid} whenever $G$ or $H$ is a number,
so suppose both are not numbers.  Then every right-option of $G + H$ is either of the
form $G^R + H$ or $G + H^R$.  Since $G^R$ is odd-tempered, by induction (\ref{dos}) tells
us that $\loutcome(G^R + H) \ge \loutcome(G^R) + \routcome(H)$.  Clearly
$\loutcome(G^R) + \routcome(H) \ge \routcome(G) + \routcome(H)$, because $\routcome(G)$
is the minimum value of $\loutcome(G^R)$.  So $\loutcome(G^R + H)$ is always at least
$\routcome(G) + \routcome(H)$.  Similarly, $\loutcome(G + H^R)$ is always at least
$\routcome(G) + \routcome(H)$.  So every right option of $G + H$ has left-outcome
at least $\routcome(G) + \routcome(H)$, and so the best right can do with $G + H$ is
$\routcome(G) + \routcome(H)$, proving (\ref{uno}).

If $G$ is odd-tempered and $H$ is even-tempered, then $G$ is not a number so there is some left option
$G^L$ with $\loutcome(G) = \routcome(G^R)$.  Then by induction, (\ref{uno}) gives
\[ \routcome(G^R + H) \ge \routcome(G^R) + \routcome(H) = \loutcome(G) + \routcome(H).\]
But clearly $\loutcome(G + H) \ge \routcome(G^R + H)$, so we are done. 
\end{proof}
Similarly we have
\begin{theorem}\label{outsum2}
Let $G$ and $H$ be $\mathbb{Z}$-valued games.  If $G$ and $H$ are both even-tempered,
then
\begin{equation} \loutcome(G + H) \le \loutcome(G) + \loutcome(H) \label{tres}\end{equation}
Likewise, if $G$ is odd-tempered and $H$ is even-tempered, then
\begin{equation} \routcome(G + H) \le \routcome(G) + \loutcome(H) \label{cuatro}\end{equation}
\end{theorem}

Another key fact in the case of partizan games was that $G + (-G) \ge 0$.  Here
we have the analogous fact that
\begin{theorem}\label{selfdestruct}
If $G$ is a $\mathbb{Z}$-valued game (of either parity), then
\begin{equation} \routcome(G + (-G)) \ge 0 \label{cinco}\end{equation}
\begin{equation} \loutcome(G + (-G)) \le 0 \label{seis}\end{equation}
\end{theorem}
\begin{proof}
Consider the game $G + (-G)$.  When Right goes first, Left has an obvious
Tweedledum and Tweedledee Strategy mirroring moves in the two components, which guarantees a score of exactly zero.
This play may not be optimal, but it at least shows that $\routcome(G + (-G)) \ge 0$.
The other case is similar.
\end{proof}

Unfortunately, some of the results above are contingent on parity.  Without the
conditions on parity, equations (\ref{uno}-\ref{cuatro}) would fail.  For example,
if $G$ is the even-tempered game $\langle -1|-1||1|1 \rangle = \langle \langle -1 | -1 \rangle | \langle 1 | 1 \rangle \rangle$ and
$H$ is the odd-tempered game $\langle G \,|\, G \rangle$, then the reader can easily check
that $\routcome(G) = 1$, $\routcome(H) = -1$, but $\routcome(G + H) = -2$ (Right moves from $H$ to $G$),
and $-2 \not \ge 1 + (-1)$, so that (\ref{uno}) fails.  The problem here is that since $H$ is odd-tempered,
Right can end up making the last move in $H$, and then Left is forced to move in
$G$, breaking her strategy of only playing responsively.

To amend this situation, we consider a restricted class of games, in which being forced
to unexpectedly move is not harmful.
\begin{definition}
An \emph{i-game} is a $\mathbb{Z}$-valued game $G$ which has the property that
every option is an i-game, and if $G$ is even-tempered, then $\loutcome(G) \ge \routcome(G)$.
\end{definition}
So for instance, numbers are always i-games, $*$ and $1*$ and even $\langle -1 | 1 \rangle$
are i-games, but the game $G = \langle -1 + * \,|\, 1 + * \rangle$ mentioned above is not,
because it is even-tempered and $\loutcome(G) = -1 < 1 = \routcome(G)$.  It may seem arbitrary
that we only require $\loutcome(G) \ge \routcome(G)$ when $G$ is \emph{even}-tempered, but later
we will see that this definition is more natural than it might first appear.

Now we can extend Claim~\ref{baseclaim} to the following:
\begin{claim}
If $G$ and $H$ are $\mathbb{Z}$-valued games, $G$ is even-tempered
and an i-game, $H$ is odd-tempered,
and $\routcome(G) \ge 0$ and $\routcome(H) \ge 0$, then
$\routcome(G + H) \ge 0$.
\end{claim}
In this case, Left is again able to play responsively in each component,
but in the situation where Right makes the final move in the second component,
Left is able to leverage the fact that the first component's left-outcome
is at least its right-outcome, because the first component will be an even-tempered i-game.
And similarly, we also have
\begin{claim}
If $G$ and $H$ are odd-tempered $\mathbb{Z}$-valued games, $G$ is an i-game,
$\loutcome(G) \ge 0$ and $\routcome(H) \ge 0$, then
$\loutcome(G + H) \ge 0$.
If $G$ and $H$ are even-tempered $\mathbb{Z}$-valued games, $G$ is an i-game,
$\routcome(G) \ge 0$ and $\loutcome(H) \ge 0$, then
$\loutcome(G + H) \ge 0$.
\end{claim}

As before, these results can be generalized to the following:
\begin{theorem}
Let $G$ and $H$ be $\mathbb{Z}$-valued games, and $G$ an i-game.
\begin{itemize}
\item If $G$ is even-tempered and $H$ is odd-tempered, then
\begin{equation} \routcome(G + H) \ge \routcome(G) + \routcome(H) \label{siete}\end{equation}
\item If $G$ and $H$ are both odd-tempered, then
\begin{equation} \loutcome(G + H) \ge \loutcome(G) + \routcome(H) \label{ocho}\end{equation}
\item If $G$ and $H$ are both even-tempered, then
\begin{equation} \loutcome(G + H) \ge \routcome(G) + \loutcome(H) \label{nueve}\end{equation}
\end{itemize}
\end{theorem}
\begin{proof}
We proceed by induction on $G$ and $H$.
If $G$ or $H$ is a number, then every equation follows from Proposition~\ref{numavoid}
and the stipulation that $\loutcome(G) \ge \routcome(G)$ if $G$ is even-tempered.
So suppose that $G$ and $H$ are both not numbers.

To see (\ref{siete}), note that every right option of $G + H$ is either of the form
$G^R + H$ or $G + H^R$.  Since $\loutcome(G^R) \ge \routcome(G)$ and $G^R$ is an odd-tempered i-game,
(\ref{ocho}) tells us inductively that
\[ \loutcome(G^R + H) \ge \loutcome(G^R) + \routcome(H) \ge \routcome(G) + \routcome(H).\]
And likewise since $\loutcome(H^R) \ge \routcome(H)$ and $H^R$ is even-tempered,
(\ref{nueve}) tells us inductively that
\[ \loutcome(G + H^R) \ge \routcome(G) + \loutcome(H^R) \ge \routcome(G) + \routcome(H).\]
So no matter how Right moves in $G + H$, he produces a position with left-outcome at least
$\routcome(G) + \routcome(H)$.  This establishes (\ref{siete}).

Equations (\ref{ocho} - \ref{nueve}) can easily be seen by having Left make an optimal
move in $G$ or $H$, respectively, and using (\ref{siete}) inductively.  I leave the details
to the reader.
\end{proof}
Similarly we have
\begin{theorem}
Let $G$ and $H$ be $\mathbb{Z}$-valued games, and $G$ an i-game.
\begin{itemize}
\item If $G$ is even-tempered and $H$ is odd-tempered, then
\begin{equation} \loutcome(G + H) \le \loutcome(G) + \loutcome(H) \label{diez}\end{equation}
\item If $G$ and $H$ are both odd-tempered, then
\begin{equation} \routcome(G + H) \le \routcome(G) + \loutcome(H) \label{once}\end{equation}
\item If $G$ and $H$ are both even-tempered, then
\begin{equation} \routcome(G + H) \le \loutcome(G) + \routcome(H) \label{doce}\end{equation}
\end{itemize}
\end{theorem}

As an application of this pile of inequalities, we prove some useful results about i-games.
\begin{theorem}\label{closure}
If $G$ and $H$ are i-games, then $-G$ and $G + H$ are i-games.
\end{theorem}
\begin{proof}
Negation is easy, and left to the reader as an exercise.  We show $G + H$ is an i-game inductively.
For the base case, $G$ and $H$ are both numbers, so $G + H$ is one too, and is therefore an i-game.  Otherwise,
by induction, every option of $G + H$ is an i-game, so it remains to show that $\loutcome(G + H) \ge \routcome(G + H)$
if $G + H$ is even-tempered.  In this case $G$ and $H$ have the same parity.  If both are even-tempered, then
by equations~(\ref{nueve}) and (\ref{doce}), we have
\[ \routcome(G + H) \le \loutcome(G) + \routcome(H) \le \loutcome(G + H),\]
and if both are odd-tempered the same follows by equations~(\ref{ocho}) and (\ref{once}) instead.
\end{proof}
\begin{theorem}\label{annihilate}
If $G$ is an i-game, then $G + (-G)$ is an i-game and
\[ \loutcome(G + (-G)) = \routcome(G + (-G)) = 0.\]
\end{theorem}
\begin{proof}
We know in general, by equations~(\ref{cinco}-\ref{seis}), that
\[ \loutcome(G + (-G)) \le 0 \le \routcome(G + (-G)).\]
By the previous theorem we know that $G + (-G)$ is an i-game,
and it is clearly even-tempered, so
\[ \loutcome(G + (-G)) \ge \routcome(G + (-G))\]
and we are done.
\end{proof}
\begin{theorem}\label{nonnegative}
If $G$ is an even-tempered i-game, and $\routcome(G) \ge 0$,
then for any $X \in \mathcal{W}_\mathbb{Z}$, we have
$\outcome(G + X) \ge \outcome(X)$.
\end{theorem}
\begin{proof}
If $X$ is even-tempered, then by Equation~(\ref{uno}),
\[ \routcome(X) \le \routcome(G) + \routcome(X) \le \routcome(G + X), \]
and by Equation~(\ref{nueve}) we have
\[ \loutcome(X) \le \routcome(G) + \loutcome(X) \le \loutcome(G + X).\]
If $X$ is odd-tempered, then by Equation~(\ref{siete}),
\[ \routcome(X) \le \routcome(G) + \routcome(X) \le \routcome(G + X),\]
and by Equation~(\ref{dos}) we have
\[ \loutcome(X) \le \routcome(G) + \loutcome(X) \le \loutcome(G + X).\]
\end{proof}
\begin{theorem}\label{nonpositive}
If $G$ is an even-tempered i-game, and $\loutcome(G) \le 0$,
the for any $X \in \mathcal{W}_\mathbb{Z}$, we have
$\outcome(G + X) \le \outcome(X)$.
\end{theorem}
\begin{proof}
Analogous to Theorem~\ref{nonnegative}.
\end{proof}
\begin{theorem}\label{zero}
If $G$ is an even-tempered i-game, and $\loutcome(G) = \routcome(G) = 0$,
then for any $X \in \mathcal{W}_\mathbb{Z}$, we have
$\outcome(G + X) = \outcome(X)$.
\end{theorem}
\begin{proof}
Combine Theorems~\ref{nonnegative} and \ref{nonpositive}.
\end{proof}
This last result suggests that if $G$ is an even-tempered i-game with vanishing outcomes,
then $G$ behaves very much like $0$.  Let us investigate this indistinguishability further\ldots

\section{Partial orders on integer-valued games}
\begin{definition}
If $G_1, G_2 \in \mathcal{W}_\mathbb{Z}$, then we say that $G_1$ and $G_2$ are \emph{equivalent},
denoted $G_1 \approx G_2$, iff $\outcome(G_1 + X) = \outcome(G_2 + X)$ for all $X \in \mathcal{W}_\mathbb{Z}$.
We also define a preorder on $\mathcal{W}_\mathbb{Z}$ by
$G_1 \lesssim G_2$ iff $\outcome(G_1 + X) \le \outcome(G_2 + X)$ for all $X \in \mathcal{W}_\mathbb{Z}$.
\end{definition}
So in particular, $G_1 \approx G_2$ iff $G_1 \gtrsim G_2$ and $G_1 \lesssim G_2$.  Taking $X = 0$ in the
definitions, we see that $G_1 \approx G_2$ implies $\outcome(G_1) = \outcome(G_2)$, and similarly,
$G_1 \gtrsim G_2$ implies that $\outcome(G_1) \ge \outcome(G_2)$.  It is straightforward
to see that $\approx$ is indeed an equivalence relation, and a congruence with respect to addition and negation,
so that the quotient space $\mathcal{W}_\mathbb{Z}/\approx$ retains its commutative monoid structure.
If two games $G_1$ and $G_2$ are equivalent, then they are interchangeable in all context involving addition.
Later we'll see that they're interchangeable in all contexts made of weakly-order preserving functions.
Our main goal is to understand the quotient space $\mathcal{W}_\mathbb{Z}/\approx$.

We restate the results at the end of last section in terms of $\approx$ and its quotient space:
\begin{corollary}\label{restatement}
If $G$ is an even-tempered i-game, then $G \lesssim 0$ iff $\loutcome(G) \le 0$, $G \gtrsim 0$ iff $\routcome(G) \ge 0$,
and $G \approx 0$ iff $\loutcome(G) = \routcome(G) = 0$.  Also, if $G$ is any i-game, then $G + (-G) \approx 0$,
so every i-game is invertible modulo $\approx$ with inverse given by negation.
\end{corollary}
\begin{proof}
Theorems \ref{nonnegative}, \ref{nonpositive}, and \ref{zero} give the implications in the direction $\Leftarrow$.
For the reverse directions, note that if $G \lesssim 0$, then by definition $\loutcome(G + 0) \le \loutcome(0 + 0) = 0$.
And similarly $G \gtrsim 0$ implies $\routcome(G) \ge 0$, and $G \approx 0$ implies $\outcome(G) = \outcome(0)$.
For the last claim, note that by Theorem~\ref{annihilate}, $G + (-G)$ has vanishing outcomes,
so by what has just been shown $G + (-G) \approx 0$.
\end{proof}
Note that this gives us a test for $\approx$ and $\lesssim$ between i-games: $G \lesssim H$
iff $G + (-H) \lesssim H + (-H) \approx 0$, iff $\loutcome(G + (-H)) \le 0$.  Here we have used the fact
that $G \lesssim H$ implies $G + X \lesssim H + X$, which is easy to see from the definition.  And if $X$ is
invertible, then the implication holds in the other direction.

Also, by combining Corollary~\ref{restatement}, and Theorem~\ref{closure}, we see that i-games modulo
$\approx$ are an abelian group, partially ordered by $\lesssim$.

We next show that even-tempered and odd-tempered games are never comparable.
\begin{theorem}\label{sameparity}
If $G_1, G_2 \in \mathcal{W}_\mathbb{Z}$ but $G_1$ is odd-tempered and $G_2$ is even-tempered, then
$G_1$ and $G_2$ are incomparable with respect to the $\gtrsim$ preorder.
Thus no two games of differing parity are equivalent.
\end{theorem}
\begin{proof}
Since we are only considering finite games, there is some $N \in \mathbb{Z}$ such
that $N$ is greater in magnitude than all numbers occuring within $G_1$ and $G_2$.
Since $\langle -N | N \rangle \in \mathcal{W}_\mathbb{Z}$,
it suffices to show that $\loutcome(G_1 + \langle -N | N \rangle)$ is positive while
$\loutcome(G_2 + \langle -N | N \rangle)$ is negative.  In both sums, $G_1 + \langle -N \,|\, N \rangle$
and $G_2 + \langle -N \,|\, N \rangle$, $N$ is so large that no player will move in $\langle -N \,|\, N \rangle$
unless they have no other alternative.  Moreover, the final score will be positive iff Right had to move in this
component, and negative iff Left had to move in this component, because $N$ is so large that it dwarves the
score of the other component.  Consequently, we can assume that the last move of the game will be in
the $\langle -N \,|\, N \rangle$ component.  Since $G_1 + \langle -N | N \rangle$ is even-tempered,
Right will make the final move if Left goes first, so $\loutcome(G_1 + \langle -N \,|\, N \rangle) > 0$.
But on the other hand, $G_2 + \langle -N | N \rangle$ is odd-tempered, so Left will make the final
move if Left goes first, and therefore $\loutcome(G_2 + \langle -N \,|\, N \rangle) < 0$, so we are done.
\end{proof}
Thus $\mathcal{W}_\mathbb{Z}/\approx$ is naturally the disjoint union of its even-tempered and odd-tempered components:
$\mathcal{W}_\mathbb{Z}^0/\approx$ and $\mathcal{W}_\mathbb{Z}^1/\approx$.

Although they are incomparable with each other, $\mathcal{W}_\mathbb{Z}^0$ and $\mathcal{W}_\mathbb{Z}^1$ are
very closely related, as the next results show:
\begin{theorem}\label{fundinv}
Let $*$ be $\langle 0 | 0 \rangle$.
If $G \in \mathcal{W}_\mathbb{Z}$, then $G + * + * \approx G$.
The map $G \to G + *$ establishes an involution on $\mathcal{W}_\mathbb{Z}/\approx$ interchanging
$\mathcal{W}_\mathbb{Z}^0/\approx$ and
$\mathcal{W}_\mathbb{Z}^1/\approx$.  In fact, as a commutative monoid, $\mathcal{W}_\mathbb{Z}$ is isomorphic to the direct product
of the cyclic group $\mathbb{Z}_2$ and the submonoid $\mathcal{W}_\mathbb{Z}^0/\approx$.
\end{theorem}
\begin{proof}
It's clear that $*$ is an i-game, and it is its own negative, so that $* + * \approx 0$.  Therefore
$G + * + * \approx G$ for any $G$.
Because $\approx$ is a congruence with respect to addition, $G \to G + *$ is a well-defined map on $\mathcal{W}_\mathbb{Z}/\approx$.
Because $*$ is odd-tempered, this map will certainly interchange even-tempered and odd-tempered games.  It is an involution by the first claim.
Then, since every $\overline{G} \in \mathcal{W}_\mathbb{Z}/\approx$ is of the form $\overline{H}$ or
$\overline{H} + \overline{*}$, and since $\overline{*} + \overline{*} = \overline{0 + * + *} = \overline{0}$, the desired direct sum
decomposition follows.
\end{proof}
As a corollary, we see that every fixed-parity (well-tempered) game is equivalent to a fixed-\emph{length} game:
\begin{definition}
Let $G$ be a well-tempered scoring game.  Then $G$ has \emph{length} 0 if $G$ is a number,
and has \emph{length} $n+1$ if every option of $G$ has length 0.  A well-tempered game is \emph{fixed-length}
if it has length $n$ for any $n$.
\end{definition}
It is easy to see by induction that if $G$ and $H$ have lengths $n$ and $m$, then $G + H$ has length $n + m$.
\begin{corollary}\label{fixed-length-exhausts}
Every well-tempered game $G$ is equivalent ($\approx$) to a fixed-length game.
\end{corollary}
\begin{proof}
We prove the following claims by induction on $G$:
\begin{itemize}
\item If $G$ is even-tempered, then for all large enough even $n$, $G \approx H$ for some game $H$ of length $n$.
\item If $G$ is odd-tempered, then for all large enough odd $n$, $G \approx H$ for some game $H$ of length $n$.
\end{itemize}
If $G$ is a number, then $G$ is even-tempered.  By definition, $G$ already has length 0.  On the other hand,
$*$ has length 1, so $G + * + *$ has length 2, $G + * + * + * + *$ has length 4, and so on.  By Theorem~\ref{fundinv},
these games are all equivalent to $G$.  This establishes the base case.

For the inductive step, suppose that $G = \langle A, B, C, \ldots | D, E, F, \ldots \rangle$.  If $G$ is even-tempered,
then $A,B,C,\ldots, D, E, F, \ldots$ are all odd-tempered.  By induction, we can find length $n - 1$ games $A', B', \ldots$
with
\[ A \approx A',\]
\[ B \approx B',\]
and so on, for all large enough even $n$.  By Theorem~\ref{monotone} below, we then have
\[ G = \langle A, B, C, \ldots | D, E, F, \ldots \rangle \approx \langle A', B', C', \ldots | D', E', F', \ldots \rangle,\]
where $H = \langle A', B', C', \ldots | D', E', F', \ldots \rangle$ has length $n$.  So $G$ is equivalent to a game
of length $n$, for large enough even $n$.  The case where $G$ is odd-tempered is handled in a completely analogous way.
\end{proof}
Because of this corollary, the indistinguishability quotient of well-tempered games is the same as the
indistinguishability quotient of fixed-length games.

Unfortunately our definition of $\approx$
is difficult to use in practice, between non-i-games, since we have to check all $X \in \mathcal{W}_\mathbb{Z}$.  Perhaps
we can come up with a better equivalent definition?

\section{Who goes last?}
The outcome of a partizan game or a $\mathbb{Z}$-valued game depends on which player goes first.  However,
since our $\mathbb{Z}$-valued games are fixed-parity games, saying who goes first is the same as saying who goes last.
We might as well consider the following:
\begin{definition}
The \emph{left-final outcome} $\lfout(G)$ of a $\mathcal{S}$-valued game $G$ is the outcome of $G$ when Left makes
the final move, and the \emph{right-final outcome} $\rfout(G)$ is the outcome when Right makes the final move.
In other words, if $G$ is odd-tempered (whoever goes first also goes last), then
\[ \lfout(G) = \loutcome(G) \]
\[ \rfout(G) = \routcome(G) \]
while if $G$ is even-tempered, then
\[ \lfout(G) = \routcome(G) \]
\[ \rfout(G) = \loutcome(G).\]
\end{definition}
This may seem arbitrary, but it turns out to be the key to understanding well-tempered scoring games.  We will see in Section~\ref{sec:sides}
that a $\mathbb{Z}$-valued game is schizophrenic in nature, acting as one of two different i-games depending on which
player will make the final move.  With this in mind, we make the following definitions:
\begin{definition} (Left's partial order)
If $G$ and $H$ are two $\mathbb{Z}$-valued games of the same parity, we define
$G \lesssim_+ H$ if $\lfout(G + X) \le \lfout(H + X)$ for all $X \in \mathcal{W}_\mathbb{Z}$,
and $G \approx_+ H$ if $\lfout(G + X) = \lfout(H + X)$ for all $X \in \mathcal{W}_\mathbb{Z}$.
\end{definition}
\begin{definition} (Right's partial order)
If $G$ and $H$ are two $\mathbb{Z}$-valued games of the same parity, we define
$G \lesssim_- H$ if $\rfout(G + X) \le \rfout(H + X)$ for all $X \in \mathcal{W}_\mathbb{Z}$,
and $G \approx_- H$ if $\rfout(G + X) = \rfout(H + X)$ for all $X \in \mathcal{W}_\mathbb{Z}$.
\end{definition}
It is clear that $\lesssim_\pm$ are preorders and $\approx_\pm$ are equivalence relations,
and that $G \lesssim H$ iff $G \lesssim_+ H$ and $G \lesssim_- H$, and that
$G \approx H$ iff $G \approx_+ H$ and $G \approx_- H$, in light of Theorem~\ref{sameparity}.
Also, $\approx_\pm$ are still congruences with respect to addition (though not negation), i.e.,
if $G \approx_+ G'$ and $H \approx_+ H'$, then $G + H \approx_+ G' + H'$.

All three equivalence relations are also congruences with respect to the operation of game formation.
In fact, we have
\begin{theorem}\label{monotone}
Let $\Box$ be one of $\approx$, $\approx_+$, $\approx_-$, $\lesssim$, $\lesssim_+$, $\lesssim_-$,
$\gtrsim$, $\gtrsim_-$, or $\gtrsim_+$. Suppose that
$L_1\ \Box\ L_1'$, $L_2\ \Box\ L_2'$, \ldots, $R_1\ \Box\ R_1'$, $R_2\ \Box R_2'$, \ldots.
Then
\[ \langle L_1, L_2, \ldots\ |\ R_1, R_2, \ldots \rangle \ \Box\ \langle L_1', L_2', \ldots \ |\ R_1', R_2', \ldots \rangle.\]
\end{theorem}
\begin{proof}
All have obvious inductive proofs.  For example, suppose $\Box$ is $\lesssim_-$.
Let $G = \langle L_1, L_2, \ldots\ |\ R_1, R_2, \ldots \rangle$ and $H = \langle L_1', L_2', \ldots \ |\ R_1', R_2', \ldots \rangle$.
Then for any $X \in \mathcal{W}_\mathbb{Z}$, if $G + X$ is even-tempered then
\[ \lfout(G + X) = \min\{\lfout(G^R + X), \lfout(G + X^R)\} \le \min\{\lfout(H^R + X), \lfout(H + X^R)\},\]
where the inequality follows by induction, and if $G + X$ is odd-tempered, then
\[ \lfout(G + X) = \max\{\lfout(G^L + X), \lfout(G + X^L)\} \le \max\{\lfout(H^L + X), \lfout(H + X^L)\}.\]
Note that the induction is on $G$, $H$, and $X$ all together. 
\end{proof}

The next four lemmas are key:
\begin{lemma}\label{topcompare}
If $G, H \in \mathcal{W}_\mathbb{Z}$ and $H$ is an i-game, then
$G \lesssim H \iff G \lesssim_+ H$.
\end{lemma}
\begin{proof}
The direction $\Rightarrow$ is obvious.  So suppose that $G \lesssim_+ H$. Since $H$ is invertible modulo
$\approx$ and therefore $\approx_+$, assume without loss of generality that $H$ is zero.  In this case, we are given
that $G$ is even-tempered and $\lfout(G + X) \le \lfout(X)$ for every $X \in \mathcal{W}_\mathbb{Z}$,
and we want to show $\rfout(G + X) \le \rfout(X)$ for every $X \in \mathcal{W}_\mathbb{Z}$.
Taking $X = -G$, we see by Equation~(\ref{cinco}) above that
\[ 0 \le \routcome(G + (-G)) = \lfout(G + X) \le \lfout(-G) = -\rfout(G) = -\loutcome(G),\]
so that
\[ \loutcome(G) \le 0.\]
Now let $X$ be arbitrary.  If $X$ is even-tempered, then by Equation~(\ref{tres})
\[ \rfout(G + X) = \loutcome(G + X) \le \loutcome(G) + \loutcome(X) \le \loutcome(X) = \rfout(X).\]
Otherwise, by Equation~(\ref{cuatro})
\[ \rfout(G + X) = \routcome(G + X) \le \loutcome(G) + \routcome(X) \le \routcome(X) = \rfout(X).\]
\end{proof}
\begin{lemma}
If $G, H \in \mathcal{W}_\mathbb{Z}$ and $G$ is an i-game, then
$G \lesssim H \iff G \lesssim_- H$.
\end{lemma}
\begin{proof}
Analogous to the previous lemma.
\end{proof}
\begin{proposition}\label{allthesame}
If $G$ and $H$ are i-games, then $G \lesssim H \iff G \lesssim_- H \iff G \lesssim_+ H$.
\end{proposition}
\begin{proof}
Follows directly from the preceding two lemmas.
\end{proof}
With Corollary~\ref{restatement}, this gives us a way of testing $\lesssim_\pm$ between i-games, but it doesn't seem to help us compute
$\lesssim$ for arbitrary games!

\begin{lemma}\label{numcompare}
If $G$ is an even-tempered $\mathbb{Z}$-valued game, and $n$ is an integer, then $G \lesssim_- n$
iff $\loutcome(G) \le n$.  Similarly, $n \lesssim_+ G$
iff $\routcome(G) \ge n$.
\end{lemma}
\begin{proof}
First, note that if $G \lesssim_- n$, then certainly $\loutcome(G) = \rfout(G + 0) \le \rfout(n + 0) = n$.
Conversely, suppose that $\loutcome(G) \le n$.  Let $X$ be arbitrary.  If $X$ is even-tempered
then by Proposition~\ref{numavoid} and Equation~(\ref{tres}),
\[ \rfout(G + X) = \loutcome(G + X) \le \loutcome(G) + \loutcome(X) \le n + \loutcome(X) = \rfout(n + X).\]
If $X$ is odd-tempered, we use Equation~(\ref{cuatro}) instead:
\[ \rfout(G + X) = \routcome(G + X) \le \loutcome(G) + \routcome(X) \le n + \routcome(X) = \rfout(n + X).\]
So for every $X$, $\rfout(G + X) \le \rfout(n + X)$ and we have shown the first sentence.  The second
is proven analogously.
\end{proof}

\begin{lemma}\label{gap}
Let $G$ be an even-tempered game, whose options are all i-games.  If $\loutcome(G) \le \routcome(G)$,
then $G \approx_- \loutcome(G)$ and $G \approx_+ \routcome(G)$.
\end{lemma}
\begin{proof}
We show $G \approx_- \loutcome(G)$, because the other result follows by symmetry.  Let $n = \loutcome(G)$.
Now since
$\loutcome(G) \le n$, we have $G \lesssim_- n$ by the preceding lemma,
so it remains to show that $n = \lesssim_- G$.  If $G$ is a number, it must be $n$,
and we are done.  Otherwise, by definition of $\loutcome$,
it must be the case that every $G^L$ satisfies $\routcome(G^L) \le n$.

Let $X$ be an arbitrary game.  We show by induction on $X$
that $n + \rfout(X) \le \rfout(G + X)$.  (This suffices because $n + \rfout(X) = \rfout(n + X)$.)
If $X$ is even-tempered, then we need to show
\[ n + \loutcome(X) \le \loutcome(G + X).\]
This is obvious if $X$ is a number (since then $\loutcome(G + X) = \loutcome(G) + \loutcome(X)$),
so suppose $X$ is not a number.  Then there is some left option $X^L$ of $X$
with $\loutcome(X) = \routcome(X^L)$, and by the inductive hypothesis applied to $X^L$,
\[ n + \loutcome(X) = n + \routcome(X^L) \le \routcome(G + X^L) \le \loutcome(G + X),\]
where the last inequality follows because $G + X^L$ is a left option in $G + X$.

Alternatively, if $X$ is odd-tempered, we need to show that
\[ n + \routcome(X) \le \routcome(G + X).\]
We show that every right option of $G + X$ has left outcome at least $n + \routcome(X)$.
The first kind of right option is of the form $G^R + X$.  By assumption,
$\loutcome(G^R) \ge \routcome(G) \ge \loutcome(G) = n$, and $G^R$ is an odd-tempered i-game, so by Equation~(\ref{ocho})
\[ \loutcome(G^R + X) \ge \loutcome(G^R) + \routcome(X) \ge n + \routcome(X).\]
The second kind of right option is of the form $G + X^R$.  By induction we know that
\[ n + \routcome(X) \le n + \loutcome(X^R) \le \loutcome(G + X^R),\]
So every right option of $G + X$ has left outcome at least $n + R(X)$, and we are done.
\end{proof}

\section{Sides}\label{sec:sides}
We now put everything together.
\begin{theorem}\label{sides-complete}
If $G$ is any $\mathbb{Z}$-valued game, then there exist i-games $G^-$
and $G^+$ such that $G \approx_- G^-$ and $G \approx_+ G^+$,
These games are unique, modulo $\approx$.
\end{theorem}
\begin{proof}
We prove that $G^+$ exists, by induction on $G$. If $G$ is a number,
this is obvious, taking $G^+ = G$.  Otherwise,
let $G = \langle L_1, L_2, \ldots \,|\, R_1, R_2, \ldots \rangle$.  By
induction, $L_i^+$ and $R_i^+$ exist.
Consider the game
\[ H = \langle L_1^+, L_2^+, \ldots \,|\, R_1^+, R_2^+, \ldots \rangle.\]
By the inductive assumptions and Theorem~\ref{monotone}, we clearly have
$H \approx_+ G$.  However,
$H$ might not be an i-game.  This can only happen if $G$ and $H$ are even-tempered,
and $\loutcome(H) < \routcome(H)$.  In this case, by Lemma~\ref{gap},
$\routcome(H) \approx_+ H \approx_+ G$, and $\routcome(H)$ is a number, so
take $G^+ = \routcome(H)$.

A similar argument shows that $G^-$ exists.  Uniqueness follows by Proposition~\ref{allthesame}.
\end{proof}
The games $G^+$ and $G^-$ in the theorem are called the \emph{upside} and \emph{downside} of $G$,
respectively, because they have formal similarities with the ``onside'' and ``offside''
of loopy partizan game theory.  An algorithm for producing the upside and downside can be extracted from the proof of the
theorem.

Here are the key facts about these games
\begin{theorem}\label{plethora}$\quad$
\begin{description}
\item[(a)]
If $G$ and $H$ are $\mathbb{Z}$-valued games, then $G \lesssim_- H$ iff
$G^- \lesssim H^-$, and $G \lesssim_+ H$ iff $G^+ \lesssim H^+$.
\item[(b)]
If $G$ is an i-game, then $G^+ \approx G^- \approx G$.
\item[(c)]
If $G$ is a $\mathbb{Z}$-valued game, then $(-G)^- \approx -(G^+)$ and $(-G)^+ \approx -(G^-)$.
\item[(d)]
If $G$ and $H$ are $\mathbb{Z}$-valued games,then $(G+H)^+ \approx G^+ + H^+$
and $(G + H)^- \approx G^- + H^-$.
\item[(e)]
For any $G$, $G^- \lesssim G \lesssim G^+$.
\item[(f)]
For any $G$, $G$ is invertible modulo $\approx$ iff it is equivalent to
an i-game, iff $G^+ \approx G^-$.  When $G$ has an inverse, it is given by $-G$.
\item[(g)] If $\mathcal{S} \subset \mathbb{Z}$ and  $G$ is an $\mathcal{S}$-valued game, then $G^+$ and $G^-$ can be taken
to be $\mathcal{S}$-valued games too.
\item[(h)] If $G$ is even-tempered, then $\loutcome(G) = \loutcome(G^-)$, and this is the smallest $n$ such that
$G^- \lesssim n$.  Similarly, $\routcome(G) = \routcome(G^+)$, and this is the biggest $n$ such that
$n \lesssim G^+$.
\item[(i)] If $G$ is odd-tempered, then $\loutcome(G) = \loutcome(G^+)$ and $\routcome(G) = \routcome(G^-)$.
\end{description}
\end{theorem}
\begin{proof}$\quad$
\begin{description}
\item[(a)]
Since $G^- \approx_- G$ and $H^- \approx_- H$, it's clear that
$G^- \lesssim_- H^-$ iff $G \lesssim_- H$.  But by Proposition~\ref{allthesame}, $G^- \lesssim_- H^- \iff G^- \lesssim H^-$,
because $G^-$ and $H^-$ are i-games.
The other case is handled similarly.
\item[(b)]
By definition, $G^- \approx_- G$.  But since both are i-games,
Proposition~\ref{allthesame} implies that $G^- \approx G$.  The other case is handled similarly.
\item[(c)]
Obvious by symmetry.  This can be proven by noting that $A \approx_- B$ iff $(-A) \approx_+ (-B)$.
\item[(d)]
Note that since $\approx_+$ is a congruence with respect to addition, we have
\[ (G + H)^+ \approx_+ (G + H) \approx_+ G^+ + H^+.\]
But since i-games are closed under addition (Theorem~\ref{closure}), $G^+ + H^+$ is an i-game,
and since $(G + H)^+$ is too, Proposition~\ref{allthesame} shows that $(G + H)^+ \approx G^+ + H^+$.
The other case is handled similarly.
\item[(e)] By definition $G \approx_+ G^+$, so $G \lesssim_+ G^+$.  But $G^+$ is an i-game,
so by Lemma~\ref{topcompare}, $G \lesssim G^+$.  The other case is handled similarly.
\item[(f)] We already showed in part (b) that if $G$ is an i-game, then
$G^+ \approx G^-$.  Conversely, if $G^+ \approx G^-$, then $G \approx_+ G^+$ and
$G \approx_- G^- \approx G^+$, so $G \approx_\pm G^+$, so $G$ is equivalent to the i-game $G^+$.
Moreover, we showed in Corollary~\ref{restatement} that i-games are invertible.  Conversely,
suppose that $G$ is invertible.  Define the \emph{deficit} $\defi(G)$ to be the i-game
$G^+ - G^-$.  This is an even-tempered i-game.  By part (e), $\defi(G) \gtrsim 0$, and by
part (d), $\defi(G + H) \approx \defi(G) + \defi(H)$.  By part (b), $\defi(G) \approx 0$
when $G$ is an i-game.   Now suppose that $G$ is invertible, and $G + H \approx 0$.
Then $\defi(G + H) \approx \defi(G) + \defi(H) \approx 0$. Since i-games are a partially
ordered abelian group, it follows that $0 \lesssim \defi(G) \approx -\defi(H) \lesssim 0$, so that
$\defi(G) \approx 0$, or in other words, $G^+ \approx G^-$.  This then implies that $G$
is equivalent to an i-game.

If $G$ has an inverse, then $G$ is equivalent to an i-game, so the inverse of $G$ is
$-G$, by Corollary~\ref{restatement}.
\item[(g)] This is clear from the construction of $G^+$ and $G^-$ given in Theorem~\ref{sides-complete}.
At some points, we replace a game by one of its outcomes.  However, the outcome of
an $\mathcal{S}$-valued game is always in $\mathcal{S}$, so this doesn't create any new outcomes.
\item[(h)] Since $G^- \approx_- G$, it follows that $\loutcome(G^-) = \rfout(G^-) = \rfout(G) = \loutcome(G)$.
Then by Lemma~\ref{numcompare}, $\loutcome(G)$ is the smallest $n$ such that $G \lesssim_- n$,
which by part (a) is the smallest $n$ such that $G^- \lesssim n$.  The other case is handled similarly.
\item[(i)] Since $G^+ \approx_+ G$, it follows that $\loutcome(G^+) = \lfout(G^+) = \lfout(G) = \loutcome(G)$.
The other case is handled similarly.
\end{description}
\end{proof}

Borrowing notation from loopy partizan theory, we use $A\&B$ to denote a well-tempered game
that has $A$ as its upside and $B$ as its downside.  By Theorem~\ref{plethora}(a), $A\&B$ is well-defined
modulo $\approx$, as long as it exists: if $G^+ \approx H^+$ and $G^- \approx H^-$, then $G \approx_+ H$
and $G \approx_- H$, so
$G \approx H$.  So the elements of $\mathcal{W}_\mathbb{Z}/\approx$ correspond to certain pairs
$A\&B$ of i-games, with $A \gtrsim B$.  In fact, all such pairs $A\&B$ with $A \gtrsim B$ occur.
\begin{theorem}\label{allpairs-early}
If $A \gtrsim B$ are i-games, then there is some game $G$ with $G^+ \approx A$ and
$G^- \approx B$.  Moreover, if $A$ and $B$ are both $\mathcal{S}$-valued games
for some $\mathcal{S} \subset \mathbb{Z}$, then $G$ can also be taken to be
$\mathcal{S}$-valued.
\end{theorem}
This will be proven below, in Chapter \ref{chap:Distortions} Theorem~\ref{allpairs}.

A generic well-tempered game $G$ acts like its upside $G^+$ when Left is going to move last,
and like its downside $G^-$ when Right is going to move last.  The two sides act fairly independently.
For example, we have
\[ A\&B + C\&D \approx (A+C)\&(B+D)\]
by Theorem~\ref{plethora}(d), and
\[ \langle A\&B, C\&D, \ldots| E\&F,\ldots \rangle = \langle A, C, \ldots |E,\ldots \rangle\&\langle B,D,\ldots|F,\ldots \rangle\]
by Theorem~\ref{monotone} (applied to $\approx_\pm$).

\section{A summary of results so far}\label{sec:so-far}
Let's review what we've done so far.

For every set of integers $S$, we constructed a class $\mathcal{W}_S$ of (well-tempered) $S$-valued games.
We then focused on $\mathbb{Z}$-valued games, and considered the operations of addition
and negation.  We defined $\approx$ to be the appropriate indistinguishability relation
to deal with addition and negation, and then considered the structure of the quotient monoid $M = \mathcal{W}_\mathbb{Z}/\approx$.
The monoid $M$ is a partially ordered commutative monoid, and has an additional
order-reversing map of negation, which does not necessarily correspond to the inverse of addition.
We showed that parity was
well defined modulo $\approx$, so that $M$ can be partitioned into even-tempered games $M_0$ and odd-tempered games
$M_1$, and that even-tempered and odd-tempered games are incomparable with respect to the partial order, and that $M_0$ is a submonoid,
and in fact $M \cong \mathbb{Z}_2 \times M_0$.

Moreover, letting $\mathcal{I}$ denote the invertible elements
of $M$, we showed that $M$ is in bijective correspondence with the set of all ordered pairs $(a,b) \in \mathcal{I}$ such that $a \ge b$.
Moreover, addition is pairwise, so that $(a,b) + (c,d) = (a + c, b + d)$, and negation is pairwise with a flip:
$-(a,b) = (-b,-a)$.  The elements of $\mathcal{I}$ themselves are in correspondence with the pairs $(a,a)$.  The maps
$(a,b) \to (a,a)$ and $(a,b) \to (b,b)$ are monoid homomorphisms.  The set $\mathcal{I}$ forms a partially ordered
abelian group.  The even-tempered i-games form an index two subgroup $\mathcal{J}$ containing a copy of the integers,
and in fact every even-tempered i-game is $\le$ some integers (the least of which is the left outcome), and is $\ge$ some integers
(the greatest of which is the right outcome).  Moreover, the left outcome of an arbitrary even-tempered game $(a,b)$ is the
left outcome of $b$, and the right outcome is the right outcome of $a$.

\chapter{Distortions}\label{chap:Distortions}

\section{Order-preserving operations on Games}
So far, we have been playing the sum of two games $G$ and $H$ by playing them in parallel,
and combining the final scores by addition.  Nothing stops us from using another operation,
however.  In fact, we can take a function with any number of arguments.
\begin{definition}\label{def:extension}
Let $S_1, S_2, \ldots, S_k, T \subseteq \mathbb{Z}$, and let
$f$ be a function $f:S_1 \times \cdots \times S_k \to T$.  Then the \emph{extension
of $f$ to games} is a function \[\tilde{f}:\mathcal{W}_{S_1} \times \cdots \times \mathcal{W}_{S_k} \to \mathcal{W}_T\]
defined recursively by
\[ \tilde{f}(G_1,G_2,\ldots,G_k) = f(G_1,G_2,\ldots,G_k)\]
when $G_1, \ldots, G_k$ are all numbers, and otherwise,
\[ \tilde{f}(G_1, G_2, \ldots, G_k) = \]\[
\langle \tilde{f}(G_1^L,G_2, \ldots, G_k), \tilde{f}(G_1,G_2^L,G_3,\ldots,G_k),\ldots
\tilde{f}(G_1,\ldots,G_{k-1},G_k^L) |\]\[
\tilde{f}(G_1^R,G_2, \ldots, G_k), f(G_1,G_2^R,G_3,\ldots,G_k),\ldots
\tilde{f}(G_1,\ldots,G_{k-1},G_k^R) \rangle \]
\end{definition}
So for example, if $f:\mathbb{Z} \times \mathbb{Z} \to \mathbb{Z}$ is ordinary addition of integers,
then $\tilde{f}$ is the addition of games that we've been studying so far.  In general,
$\tilde{f}(G_1,G_2,\ldots,G_k)$ is a composite game in which the players play
$G_1, G_2, \ldots,$ and $G_k$ in parallel, and then combine the final score of each game
using $f$. Structurally, $\tilde{f}(G_1,G_2,\ldots,G_k)$ is just like
$G_1 + \cdots + G_k$, except with different final scores.  In particular, the parity
of $\tilde{f}(G_1,G_2,\ldots,G_k)$ is the same as the parity
of $G_1 + \cdots + G_k$.

\begin{figure}[htb]
\begin{center}
\includegraphics[width=4in]
					{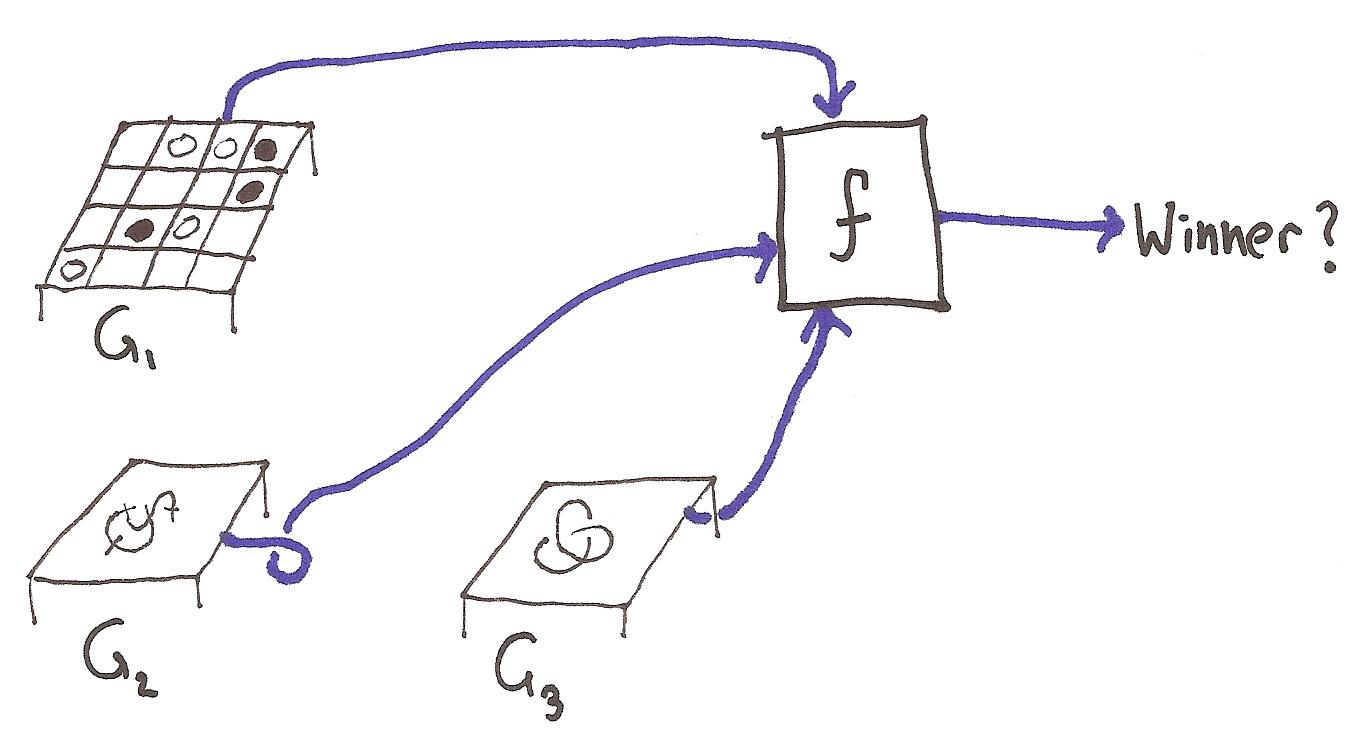}
\caption{Schematically, we are playing several games in parallel and using $f$ to combine their
final scores.}
\label{extension-schematic}
\end{center}
\end{figure}

Algebraic properties of $f$ often lift to algebraic properties of $\tilde{f}$. For example,
since addition of integers is associative and commutative, so is addition of games.  
Or if $f, g$, and $h$ are functions from $\mathbb{Z}$ to $\mathbb{Z}$,
then $f \circ g = h \implies \tilde{f} \circ \tilde{g} = \tilde{h}$.  Similar results hold
for compositions of functions of higher arities.  We will use these facts without comment
in what follows.

\begin{exercise}
Show than an algebraic identity will be maintained as long as each variable occurs
exactly once on each side.  So associativity and commutativity are maintained.
\end{exercise}
On the other hand, properties like idempotence and distributivity are not preserved,
because even structurally, $\tilde{f}(G,G)$ is
very different than $G$, having many more positions.  In fact if $G$ is odd-tempered,
then $\tilde{f}(G,G)$ will be even-tempered and so cannot equal or even be equivalent to $G$.

We are mainly interested in cases where $f$ has the following property:
\begin{definition}
We say that $f~S_1 \times \cdots \times S_k \to T$ is \emph{(weakly) order-preserving}
if whenever $(a_1, \ldots, a_n) \in S_1 \times \cdots \times S_k$ and
$(b_1, \ldots, b_n) \in S_1 \times \cdots \times S_k$ satisfy $a_i \le b_i$ for every
$1 \le i \le k$, then $f(a_1, \ldots, a_n) \le f(b_1, \ldots, b_n)$.
\end{definition}

Order-preserving operations are closed under composition.  Moreover, unary
order-preserving functions have the following nice property, which will be used later:
\begin{lemma}\label{outcomeslide}
If $S, T$ are subsets of $\mathbb{Z}$, and $f:S\to T$ is order-preserving, then
for any $G \in \mathcal{W}_S$, \[\routcome(\tilde{f}(G)) = f(\routcome(G)),\] and
\[ \loutcome(\tilde{f}(G)) = f(\loutcome(G))\]
\end{lemma}
\begin{proof}
Easy by induction; left as an exercise to the reader.
\end{proof}

We also have
\begin{lemma}\label{boundingfunctions}
Let $f$ and $g$ be two functions from $S_1 \times \cdots \times S_k \to T$,
such that $f(x_1,\ldots,x_k) \le g(x_1,\ldots,x_k)$ for every
$(x_1,\ldots,x_k) \in S_1 \times \cdots \times S_k$.  Then for
every $(G_1,\ldots,G_k) \in \mathcal{W}_{S_1} \times \cdots \times \mathcal{W}_{S_k}$,
\[ \tilde{f}(G_1,\ldots,G_k) \lesssim \tilde{g}(G_1,\ldots,G_k) \]
and in particular
\[ \outcome(\tilde{f}(G_1,\ldots,G_k)) \le \outcome(\tilde{g}(G_1,\ldots,G_k))\]
\end{lemma}
\begin{proof}
An obvious inductive proof using Theorem~\ref{monotone}.
\end{proof}

Another easy fact is the following:
\begin{lemma}\label{staradd}
Let $f : S_1 \times \cdots \times S_k \to T$ be a function.
Then for any games $(G_1, G_2, \ldots, G_k) \in \mathcal{W}_{S_1} \times \cdots \times \mathcal{W}_{S_k}$,
we have
\[ f(G_1 + *, G_2, \ldots, G_k) = f(G_1, G_2 + *, \ldots, G_k) = \cdots
=\]\begin{equation} f(G_1, G_2, \ldots, G_k + *) = f(G_1, \ldots, G_k) + *,\label{staraddeq}\end{equation}
where $*$ is $\langle 0 | 0 \rangle$ as usual.
\end{lemma}
\begin{proof}
Note that for $(x_1,\ldots, x_k, y) \in S_1 \times \cdots \times S_k \times \{0\}$,
\[ f(x_1 + y, x_2, \ldots, x_k) = f(x_1, x_2 + y, \ldots, x_k) = \cdots =\]
\[ f(x_1, x_2, \ldots, x_k + y) = f(x_1, \ldots, x_k) + y.\]
It follows that (\ref{staraddeq}) is true more generally if we replace $*$ by any $\{0\}$-valued game.
\end{proof}

\section{Compatibility with Equivalence}
We defined $\approx$ to be indistinguishability for the operation of addition.  By adding new operations
into the mix, indistinguishability could conceivably become a finer relation.  In this section,
we'll see that this does not occur when our operations are extensions of order-preserving functions.
In other words, the $\approx$ equivalence relation is already compatible with extensions
of order-preserving functions.

\begin{theorem}\label{fineenough}
Let $S_1, \ldots, S_k, T$ be subsets of $\mathbb{Z}$, $f:S_1 \times \cdots \times S_k \to T$
be order-preserving, and $\tilde{f}$ be its extension to
$\mathcal{W}_{S_1} \times \cdots \times \mathcal{W}_{S_k} \to \mathcal{W}_T$.  Let $\Box$ be one
of $\lesssim$, $\gtrsim$, $\lesssim_\pm$, $\gtrsim_\pm$, $\approx$, or $\approx_\pm$. If
$(G_1,\ldots,G_k)$ and $(H_1,\ldots,H_k)$ are elements of $\mathcal{W}_{S_1} \times \cdots \mathcal{W}_{S_k}$,
such that $G_i \Box H_i$ as integer-valued games, then
\[ \tilde{f}(G_1,\ldots,G_k) \Box \tilde{f}(H_1,\ldots,H_k).\]
\end{theorem}
This theorem says that extensions of order-preserving maps are compatible with equivalence and all
our other relations.  In particular, order-preserving extensions are well-defined on the quotient spaces
of $\approx$ and $\approx_\pm$.
%(TODO: note that I just pulled out Corollary~\ref{fine2} and moved it somewhere else\ldots rethink this.)

To prove Theorem~\ref{fineenough}, we reduce to the case where $G_i = H_i$ for all but one $i$, by the usual means.
By symmetry, we only need to show that
$\tilde{f}(G,G_2,\ldots,G_k) \lesssim_- \tilde{f}(H,G_2,\ldots,G_k)$, when $G \lesssim_- H$.  We also reduce
to the case where the codomain $T$ is $\{0,1\}$.  This makes $T^{S_1}$, the set of order-preserving functions
from $S_1$ to $T$ be itself finite and linearly ordered.  We then view $f(\cdot, G_2, \ldots, G_k)$, the context into which
$G$ and $H$ are placed, as a $T^{S_1}$-valued game whose score is combined with the final score of $G$ or $H$.  (See Figure~\ref{context-argument}).

\begin{figure}[htb]
\begin{center}
\includegraphics[width=4in]
					{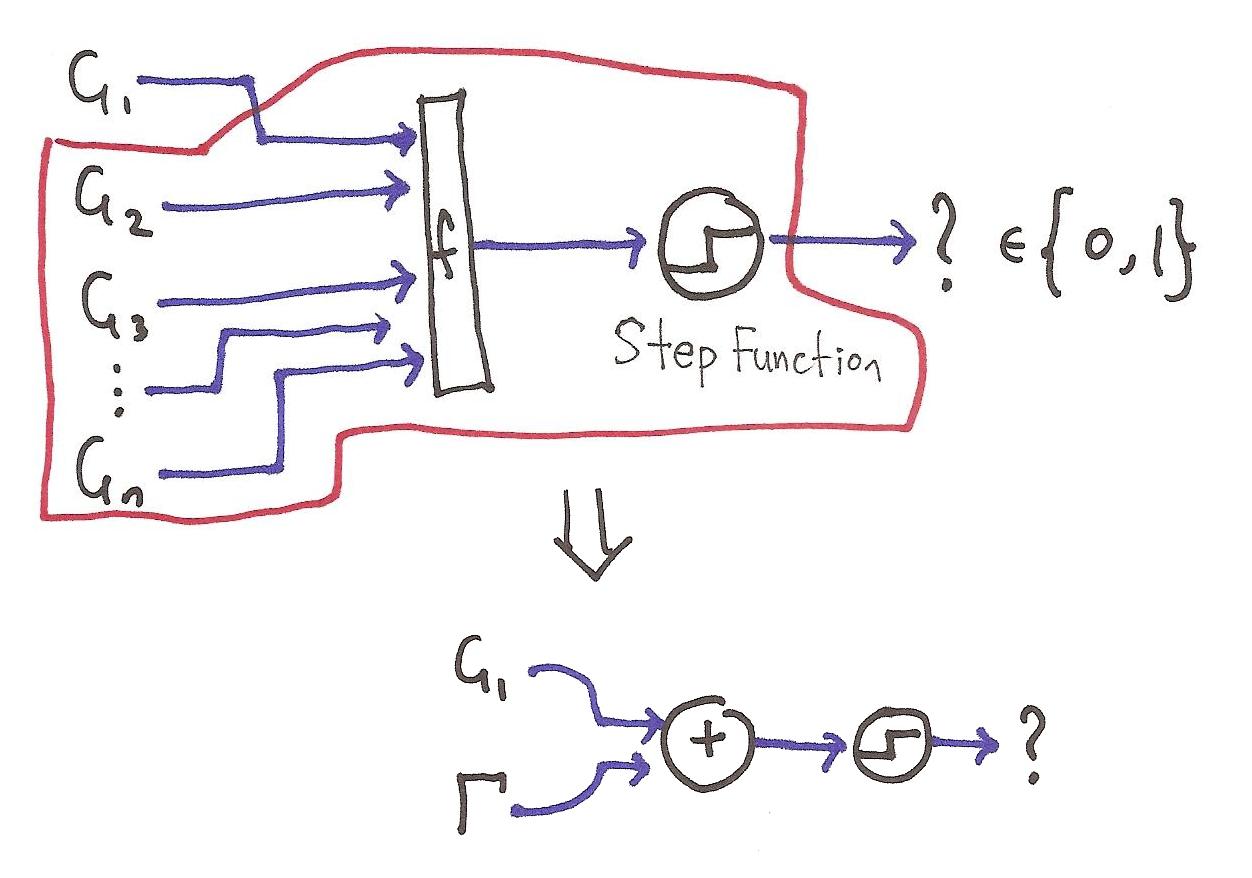}
\caption{Schematically, we are taking all the component games other than $G$, as well as the function $f$, and bundling them up
into a ``context'' $\Gamma$.  In order to pull this off, we have to run the output of the game through a step function.  By varying
the cutoff of the step function, the true outcome of $\tilde{f}(G_1,\ldots,G_n)$ is recoverable, so this is no great loss.}
\label{context-argument}
\end{center}
\end{figure}

As a finite total order, $T^{S_1}$ can then be identified with a set of integers slightly larger than $S_1$,
and applying this identification to $f(\cdot, G_2, \ldots, G_k)$, we make
an integer-valued game $A$ such that
\[ \outcome(G + A) \le \outcome(H + A) \implies \outcome(G,G_2,\ldots,G_k) \le \outcome(H,G_2,\ldots,G_k).\]
Many of these steps will not be spelled out explicitly in what follows.

\begin{lemma}\label{functions}
Let $S, S_2, \ldots, S_{k+1}$ be subsets of the integers, with $S$ finite, and let $f':S \times S_2 \times \cdots \times S_{k+1} \to \mathbb{Z}$
be order-preserving.  Then for any $n \in \mathbb{Z}$, there is a function $g_n:S_2\times\cdots\times S_{k+1} \to \mathbb{Z}$, such that
for any $(x_1, \ldots, x_{k+1}) \in S \times S_2 \times \ldots S_{k+1}$,
\begin{equation} x_1 + g_n(x_2, x_3, \ldots, x_{k+1}) > 0 \iff f'(x_1, x_2, \ldots, x_{k+1}) > n.\label{t0}\end{equation}
\end{lemma}
(Note that there is no stipulation that $g_n$ be order-preserving.)
\begin{proof}
Fix $x_2, x_3, \ldots, x_{k+1}$.  Partition $S$ as $A \cup B$, where
\[ A = \{x \in S~:~ f'(x,x_2,\ldots,x_{k+1}) > n\}\]
\[ B = \{x \in S~:~ f'(x,x_2,\ldots,x_{k+1}) \le n\}\]
Then because $f'$ is order-preserving, every element of $A$ is greater than every element of $B$.  Since $S$ is finite, this
implies that there is some
integer $m$ such that $A = \{x \in S: x > m\}$ and $B = \{x \in S:x \le m\}$.  Let $g_n(x_2,\ldots,x_{k+1}) = -m$.  Then
clearly (\ref{t0}) will hold.
\end{proof}

\begin{lemma}\label{applyit}
With the setup of the previous lemma, if $G$ is an $S$-valued game, and $G_i$ is a $S_i$-valued game for $2 \le i \le k+1$,
then
\[ \loutcome(G + \tilde{g}(G_2,\ldots,G_{k+1})) > 0 \iff \loutcome(\tilde{f'}(G,G_2,\ldots,G_{k+1})) > n\]
and similarly
\[ \routcome(G + \tilde{g}(G_2,\ldots,G_{k+1})) > 0 \iff \routcome(\tilde{f'}(G,G_2,\ldots,G_{k+1})) > n\]
\end{lemma}
\begin{proof}
For any integer $m$, let $\delta_m~:~\mathbb{Z} \to \{0,1\}$ be the function $\delta_m(x) = 1$ if
$x > m$ and $\delta_m(x) = 0$ if $x \le m$.  Then using Lemma~\ref{outcomeslide}, we have
\begin{equation} \loutcome(G + \tilde{g_n}(G_2,\ldots,G_{k+1})) > 0 \iff
\loutcome(\tilde{\delta_0}(G + \tilde{g_n}(G_2,\ldots,G_{k+1}))) = 1,\label{t1}\end{equation}
and similarly,
\begin{equation} \loutcome(\tilde{f'}(G,G_2,\ldots,G_{k+1})) > n \iff
\loutcome(\tilde{\delta_n}(\tilde{f'}(G,G_2,\ldots,G_{k+1}))) = 1.\label{t2}\end{equation}
Now by the previous lemma, we know that for any $(x_1,\ldots,x_{k+1}) \in S \times S_2 \times \cdots \times S_{k+1}$, we have
\[ \delta_0(x_1 + g_n(x_2,x_3,\ldots,x_k)) = \delta_n(x_1,x_2,\ldots,x_k)\]
But then this equation continues to be true when we extend everything to games, so that
\[ \tilde{\delta_0}(G + \tilde{g_n}(G_2,\ldots,G_{k+1})) = \tilde{\delta_n}(\tilde{f'}(G,G_2,\ldots,G_{k+1})) \]
and then we are done, after combining with (\ref{t1}) and (\ref{t2}) above.
\end{proof}
Now since $G + \tilde{g_n}(G_2,\ldots,G_k)$ has the same parity as $\tilde{f}(G,G_2,\ldots,G_k)$,
the two equations in Lemma~\ref{applyit} are equivalent to the following two:
\begin{equation} \lfout(G + \tilde{g_n}(G_2,\ldots,G_k)) > 0 \iff \lfout(\tilde{f}(G,G_2,\ldots,G_k)) > n\label{distort-lfout}\end{equation}
\begin{equation} \rfout(G + \tilde{g_n}(G_2,\ldots,G_k)) > 0 \iff \rfout(\tilde{f}(G,G_2,\ldots,G_k)) > n\label{distort-rfout}\end{equation}

\begin{lemma}\label{onebyone}
If $(G_1,\ldots,G_k)$ and $(H_1,\ldots,H_k)$ are
in $\mathcal{W}_{S_1} \times \mathcal{W}_{S_2} \times \cdots \times \mathcal{W}_{S_k}$, and
$G_i \lesssim_- H_i$ for every $i$, and $G_i = H_i$ for all but one $i$, then
\[ \tilde{f}(G_1,\ldots,G_k) \lesssim_- \tilde{f}(H_1,\ldots,H_k).\]
Also, the same holds if we replace $\lesssim_-$ with $\lesssim_+$.
\end{lemma}
\begin{proof}
Without loss of generality, $G_i = H_i$ for all $i \ne 1$.  We only need to consider the case
where $G_1 \lesssim_- H_1$, as the $G_1 \lesssim_+ H_1$ case follows by symmetry.

We want to show that
\[ \tilde{f}(G,G_2,\ldots,G_k) \lesssim_- \tilde{f}(H,G_2,\ldots,G_k)\]
given that $G \lesssim_- H$ are $S_1$-valued games.  In particular, we need to show that for every game $K$,
\begin{equation} \lfout(\tilde{f}(G,G_2,\ldots,G_k) + K) \le \lfout(\tilde{f}(H,G_2,\ldots,G_k) + K).\label{gaol}\end{equation}
Suppose for the sake of contradiction that there is some $K$ for which (\ref{gaol}) doesn't hold.  Then there is some
integer $n$ such that
\[ \lfout(\tilde{f}(H,G_2,\ldots,G_k) + K) \not > n\]
but
\[ \lfout(\tilde{f}(G,G_2,\ldots,G_k) + K) > n.\]
Let $S_{k+1}$ be $\mathbb{Z}$, so that $K$ is an $S_{k+1}$-valued game.  Since all our games are finite, only finitely many
values occur within each of $G$ and $H$.  Thus there is some finite subset $S$ of $S_1$ so that $G$ and $H$ are both $S$-valued games.
Let $f'~:~S \times S_1 \times \cdots \times S_k \times S_{k+1}$ be the function
\[ f'(x_1,\ldots,x_{k+1}) = f(x_1,\ldots,x_k) + x_{k+1},\]
which is still order-preserving.
Then $\tilde{f'}(G_1,\ldots,G_{k+1}) = \tilde{f}(G_1,\ldots,G_k) + G_{k+1}$ for any appropriate $G_1,\ldots,G_{k+1}$.  In
particular, then, we have that
\[ \lfout(\tilde{f'}(H,G_2,\ldots,G_k,K)) \not > n\]
and
\[ \lfout(\tilde{f'}(G,G_2,\ldots,G_k,K)) > n\]
Let $g_n$ be the function from Lemma~\ref{applyit}.  Then by (\ref{distort-lfout}), it follows that
\[ \lfout(H + \tilde{g_n}(G_2,\ldots,G_k,K)) \not > 0\]
and
\[ \lfout(G + \tilde{g_n}(G_2,\ldots,G_k,K)) > 0.\]
Thus, if $J = \tilde{g_n}(G_2,\ldots,G_k,K)$, we have $\lfout(G + J) \not\le \lfout(H + J)$, contradicting the
fact that $G \lesssim_- H$.

The other case, in which $G \lesssim_+ H$, follows by symmetry.
\end{proof}

\begin{proof}[Proof (of Theorem~\ref{fineenough})]
With the setup of Theorem~\ref{fineenough}, first consider the case where $\Box$ is $\lesssim_-$.
So $G_i \lesssim_- H_i$ for every $i$.  Then by Lemma~\ref{onebyone},
\[ \tilde{f}(G_1,\ldots,G_k) \lesssim_- \tilde{f}(H_1,G_2\ldots,G_k)
\lesssim_- \tilde{f}(H_1,H_2,G_3,\ldots,G_k) \]\[\lesssim_- \cdots
\lesssim_- \tilde{f}(H_1,\ldots,H_{k-1},G_k) \lesssim_- \tilde{f}(H_1,\ldots,H_k).\]
So $\tilde{f}(G_1,\ldots,G_k) \lesssim_- \tilde{f}(H_1,\ldots,H_k)$ by transitivity
of $\lesssim_-$.  This establishes Theorem~\ref{fineenough} when $\Box$ is $\lesssim_-$.

The cases where $\Box$ is one of $\lesssim_+$, $\gtrsim_+$ or $\gtrsim_-$ follow immediately.
All the other remaining possibilities for $\Box$ can be written as intersections
of $\lesssim_\pm$ and $\gtrsim_\pm$, so the remaining cases follow easily.  For example,
if $G_i \approx_+ H_i$ for all $i$, then we have $G_i \lesssim_+ H_i$ and $G_i \gtrsim_+ H_i$
for all $i$, so that
\[ \tilde{f}(G_1,\ldots,G_k) \lesssim_+ \tilde{f}(H_1,\ldots,H_k)\]
and
\[ \tilde{f}(G_1,\ldots,G_k) \gtrsim_+ \tilde{f}(H_1,\ldots,H_k)\]
by the cases where $\Box$ is $\lesssim_+$ or $\gtrsim_+$.  Thus
\[ \tilde{f}(G_1,\ldots,G_k) \approx_+ \tilde{f}(H_1,\ldots,H_k).\]
\end{proof}

% Key things that are worth proving:
% The sides of an [n]-valued game are [n]-valued games.

% apricot figure out what to do with this corollary
% It uses the yet-unspecified notation for \{0,1,\ldots,n-1\}
% and it cites a corollary that got moved way later.
%\begin{corollary}
%If $S$ is finite and $|S| = n$, then there's a natural bijection between
%$S$-valued games modulo $\approx$ and $\{0,1,\ldots,n-1\}$-valued games modulo $\approx$.
%\end{corollary}
%\begin{proof}
%Let $f:S \to \{0,\ldots,n-1\}$ and $g: \{0,\ldots,n-1\} \to S$ be order preserving bijections,
%and thus inverses of each other.  Then $\tilde{f} \circ \tilde{g}$
%and $\tilde{g} \circ \tilde{f}$ are both identity maps.  Because $f$ and $g$ are order-preserving,
%by Corollary~\ref{fine2}, $\tilde{f}$
%and $\tilde{g}$ are well defined modulo $\approx$.
%\end{proof}
%It is for this reason that we mainly focus on $n$-valued games in what follows.

%Theorem~\ref{fineenough} and Corollary~\ref{fine2} show that
%the $\approx$ relation is sufficiently fine/divisive on $n$-valued games  if we are considering reasonable
%operations on the class of $n$-valued games.  Recollecting results from previous sections,
%we know that $n$-valued games modulo $\approx$ are in one-to-one correspondence with
%pairs $(A,B)$ of $n$-valued i-games satisfying $A \lesssim B$.  The next theorem, which
%generalizes Lemma~\ref{distortions},
%shows that $\tilde{f}$ acts on $n$-valued i-games, and this action determines
%the action of $\tilde{f}$ on $n$-valued games in general:
%
%

As a corollary of Theorem~\ref{fineenough}, we see that the action of an order-preserving
extension on $S$-valued games is determined by its action on even-tempered i-games.

% apricot what is this comment even referring to?
% (TODO: this corollary will be repeated somewhere later - resolve the situation! <- I don't know what this is talking about)
\begin{corollary}\label{determined}
Let $S_1, S_2, \ldots, S_k, T$ be subsets of $\mathbb{Z}$, and
$f:S_1 \times \cdots \times S_k \to T$ be order-preserving.  For $1 \le i \le k$,
let $G_i$ be an $S_i$-valued game.  Let $e_i$ be $0$ or $*$, so that
$e_i$ has the same parity as $G_i$.
For each $G_i$, choose an upside and downside $G_i^+$ and $G_i^-$ which are $S_i$-valued,
possible by Theorem~\ref{plethora}(g).  Then for every $i$, $G_i^+ + e_i$
and $G_i^- + e_i$ are even-tempered $S_i$-valued i-games, and
\[ \tilde{f}(G_1,\ldots,G_k) \approx A\&B + (e_1 + \cdots + e_k),\]
where
\[ A \approx \tilde{f}(G_1^+ + e_1, G_2^+ + e_2, \ldots, G_k^+ + e_k)^+\]
\[ B \approx \tilde{f}(G_1^- + e_2, G_2^- + e_2, \ldots, G_k^- + e_k)^-\]
are even-tempered i-games.
\end{corollary}
\begin{proof}
First consider the case where every $G_i$ is even-tempered, so that all the $e_i$ vanish.
Then we need to show that
\[ \tilde{f}(G_1,\ldots,G_k)^+ \approx A\]
\[ \tilde{f}(G_1,\ldots,G_k)^- \approx B\]
or equivalently,
\[ \tilde{f}(G_1,\ldots,G_k) \approx_+ \tilde{f}(G_1^+,G_2^+,\ldots,G_k^+)\]
\[ \tilde{f}(G_1,\ldots,G_k) \approx_- \tilde{f}(G_1^-,G_2^-,\ldots,G_k^-).\]
But these follow directly from Theorem~\ref{fineenough} in the case where
$\Box$ is $\approx_\pm$, since $G_i \approx_\pm G_i^\pm$ for all $i$.

Now suppose that some of the $G_i$ are odd-tempered.  Since $*$ is an i-game,
every $G_i^\pm + e_i$ is an i-game too, and is $S_i$-valued because
$G_i^\pm$ is $S_i$-valued and $e_i$ is $\{0\}$-valued.  Now
$G_i$, $G_i^\pm$, and $e_i$ all have the same parity, so
$G_i^\pm + e_i$ will be even-tempered $S_i$-valued i-games.

Letting $H_i = G_i + e_i$, we see that $H_i$ is an even-tempered $S_i$-valued
game for every $i$, and that $H_i^\pm \approx G_i^\pm + e_i$ because
$e_i$ is an i-game.  Now $0 + 0 = 0$, and $* + * \approx 0$ (as $*$ equals
its own negative), so $G_i \approx H_i + e_i$ for every $i$.  Then
by Theorem~\ref{fineenough} and repeated applications of Lemma~\ref{staradd}, we see that
\[ \tilde{f}(G_1,\ldots,G_k) \approx
\tilde{f}(H_1 + e_1, \ldots, H_k + e_k) = \tilde{f}(H_1, \ldots, H_k) + (e_1 + \cdots + e_k).\]
But by the even-tempered case that we just proved,
\[ \tilde{f}(H_1,\ldots,H_k) \approx A\&B.\]
\end{proof}

\section{Preservation of i-games}\label{sec:presigame}
By Corollary~\ref{determined}, any order-preserving extension $\tilde{f}$ is determined
by two maps on even-tempered i-games, one that sends
\[ (G_1,\ldots,G_k) \to \tilde{f}(G_1,\ldots,G_k)^+\]
and one that sends
\[ (G_1,\ldots,G_k) \to \tilde{f}(G_1,\ldots,G_k)^-.\]
In this section we show that in fact $\tilde{f}(G_1,\ldots,G_k)$ will always be an
i-game itself, so that these two maps in fact agree.\footnote{In this way the theory diverges
from the case of loopy partizan games, where there are distinct upsums and downsums used
to add onsides and offsides.}  Thus every order-preserving map $f$
induces a single map on equivalence classes of even-tempered i-games, and this map determines the action
of $f$ on all games.

We first prove that i-games are preserved for the case where $f$ is unary, and use it to answer a question from
a previous chapter: for which $A \gtrsim B$ does $A\&B$ exist?

\begin{lemma}\label{distortions}
Let $S$ be a set of integers, and $f:S\to \mathbb{Z}$ be weakly order-preserving.
Then for any $S$-valued i-game $G$, $\tilde{f}(G)$ is an i-game.
\end{lemma}
\begin{proof}
By induction, we only need to show that if $G$ is even-tempered, then
\[ \loutcome(\tilde{f}(G)) \ge \routcome(\tilde{f}(G)),\]
which follows by Lemma~\ref{outcomeslide} and the fact that $\loutcome(G) \ge \routcome(G)$.
\end{proof}
\begin{lemma}\label{sidedistort}
If $G$ is an $S$-valued game and $f:S \to \mathbb{Z}$ is order-preserving, then
\[ \tilde{f}(G)^+ \approx \tilde{f}(G^+)\]
\[ \tilde{f}(G)^- \approx \tilde{f}(G^-)\]
where we take $G^\pm$ to be $S$-valued.
\end{lemma}
\begin{proof}
We have $G^\pm \approx_\pm G$, so by Theorem~\ref{fineenough},
\[ \tilde{f}(G) \approx_\pm \tilde{f}(G^\pm).\]
But by Lemma~\ref{distortions}, $\tilde{f}(G^\pm)$ is an i-game because $G^\pm$ is.
So the desired result follows.
\end{proof}

We now complete the description of $\mathcal{W}_\mathbb{Z}$ in terms of i-games:
\begin{lemma}
If $A$ is an i-game (necessarily even-tempered) with $A \gtrsim 0$, then there is some $\mathbb{Z}$-valued game
$H$ with $H^- \approx 0$ and $H^+ \approx A$.
\end{lemma}
\begin{proof}
For any integers $n \le m$, let $D_{n,m}$ denote the even-tempered game $\langle n + * \,|\, m + *\rangle$.
The reader can verify from the definitions that $D_{n,m}^+ = m$ and $D_{n,m}^- = n$, so that $D_{n,m}$ acts like either
$n$ or $m$ depending on its context.

We create $H$ from $A$ by substituting $D_{0,n}$ for every positive number $n$ occurring within $A$.
As $D_{0,n}^+ = n$, it follows by an inductive argument using the $\approx_+$ case of Theorem~\ref{monotone} that $A$ is still the upside
of $H$.

It is also clear by the $\approx_-$ case of Theorem~\ref{monotone} that the downside of $H$ is the downside of the game obtained by replacing
every positive number in $A$ with $0$.  Letting $f(n) = \min(n,0)$, this game
is just $\tilde{f}(A)$.  So $H^- \approx f(A)^-$.  But $A \approx_- A^-$, so by Theorem~\ref{fineenough} $f(A) \approx_- f(A^-)$.
Then by Lemma~\ref{distortions}, $f(A^-)$ is an i-game, so $f(A)^- \approx f(A^-)$.  Thus
\[ H^- \approx f(A)^- \approx f(A^-) \approx f(A).\]
Moreover, since $A \gtrsim 0$, Corollary~\ref{restatement}
implies that $\routcome(A) \ge 0$.  As an even-tempered i-game, $\loutcome(A) \ge \routcome(A) \ge 0$.
Therefore $f(\loutcome(A)) = f(\routcome(A)) = 0$.  By Lemma~\ref{outcomeslide}, it then follows
that $\loutcome(f(A)) = \routcome(f(A)) = 0$.  Then by Corollary~\ref{restatement},
$f(A) \approx 0$, so $H^- \approx 0$ and we are done.
\end{proof}
Using this, we see that all possible pairs $A \& B$ occur:
\begin{theorem}\label{inorder2}
If $A$, $B$ are i-games with $A \gtrsim B$, then there is some $\mathbb{Z}$-valued game
$G$ with $G^+ \approx A$ and $G^- \approx B$.
\end{theorem}
\begin{proof}
Since $A - B \gtrsim 0$, we can produce a game $H$ with $H^+ \approx A - B$
and $H^- \approx 0$ by the lemma.  Then letting $G = H + B$, we have
\[ G^+ \approx H^+ + B^+ \approx A - B + B \approx A,\]
\[ G^- \approx H^- + B^- \approx 0 + B = B.\] 
\end{proof}
Moreover, we can refine this slightly:
\begin{theorem}\label{allpairs}
Let $\mathcal{S} \subseteq \mathbb{Z}$ and let $A$ and $B$ be $\mathcal{S}$-valued i-games
with $A \gtrsim B$.  Then there is some $\mathcal{S}$-valued game $G$ with
$G^+ \approx A$ and $G^- \approx B$.
\end{theorem}
\begin{proof}
By the previous theorem, we can construct a game $G_0$ with $G_0^+ \approx A$
and $G_0^- \approx B$.  Let $f\,:\,\mathbb{Z} \to \mathbb{Z}$
be a weakly order-preserving function that projects $\mathbb{Z}$ onto $\mathcal{S}$.
That is, $f \circ f = f$, and $f(\mathbb{Z}) = \mathcal{S}$.  We can construct such an
$f$ by sending every integer to the closest element of $\mathcal{S}$, breaking ties arbitrarily.
(Note that by existence of $A$ and $B$, $\mathcal{S}$ cannot be empty.)  Then by Lemma~\ref{sidedistort}
and Theorem~\ref{fineenough},
\[ f(G_0)^+ \approx f(G_0^+) \approx f(A) = A\]
and
\[ f(G_0)^- \approx f(G_0^-) \approx f(B) = B,\]
so taking $G = f(G_0)$, we have $G \in \mathcal{W}_\mathcal{S}$, and $G^+ \approx A$ and $G^- \approx B$.
\end{proof}
So for any $\mathcal{S} \subseteq \mathbb{Z}$, the $\mathcal{S}$-valued games modulo $\approx$ are in
one-to-one correspondence with the pairs $(a,b) \in \mathcal{I}_\mathcal{S} \times \mathcal{I}_\mathcal{S}$
for which $a \gtrsim b$, where $\mathcal{I}_\mathcal{S}$ is the $\mathcal{S}$-valued i-games modulo $\approx$.
In particular, $A \& B$ exists and can be $S$-valued whenever $A$ and $B$ are $S$-valued i-games with $A \gtrsim B$.

We now return to proving that order-preserving extensions preserve i-games:
\begin{theorem}\label{generaldistortions}
Let $S_1, \ldots,  S_k, T$ be subsets of $\mathbb{Z}$, $f:S_1 \times \cdots \times S_k \to T$
be order-preserving, and $\tilde{f}$ be the extension of $f$ to
$\mathcal{W}_{S_1} \times \cdots \times \mathcal{W}_{S_k} \to \mathcal{W}_T$.  If
$G_1, \ldots, G_k$ are i-games, with $G_i \in \mathcal{W}_{S_i}$, then
$\tilde{f}(G_1,\ldots,G_k)$ is also an i-game.  Moreover, if
$H_1,\ldots, H_k$ are general games, with $H_i \in \mathcal{W}_{S_i}$, then
\[ (\tilde{f}(H_1,\ldots,H_k))^+ \approx \tilde{f}(H_1^+,H_2^+,\ldots,H_k^+)\]
and
\[ (\tilde{f}(H_1,\ldots,H_k))^- \approx \tilde{f}(H_1^-,H_2^-,\ldots,H_k^-).\]
\end{theorem}
In other words, $f$ preserves i-games, and interacts nicely with upsides and downsides.
\begin{proof}
The first claim is the more difficult to show.  It generalizes Theorem~\ref{closure} and Lemma~\ref{distortions}.
We first prove a slightly weaker form:

\begin{lemma}
If $(G_1, \ldots, G_k) \in \mathcal{W}_{S_1} \times \cdots \times \mathcal{W}_{S_k}$ are all i-games, then $\tilde{f}(G_1,\ldots,G_k) \approx H$
for some i-game $H \in \mathcal{W}_\mathbb{Z}$.
\end{lemma}
\begin{proof}
Since each $G_i$ is finite, it has only finitely many elements of $S_i$ as subpositions - so we can take finite
subsets $S'_i \subseteq S_i$ such that $G_i \in \mathcal{W}_{S_i'}$ for every $i$.  Restricting
$f$ from $S_i$ to $S'_i$, we can assume without loss of generality that $S_i = S_i'$ is finite.
Then there is some positive integer $M$ so that
\[ |f(x_1,\ldots,x_k)| < \frac{M}{2}\]
for every $(x_1,\ldots,x_k) \in S_1 \times \cdots \times S_k$.

For each $i$, let $Z_i = \{-s\,:\,s \in S\}$ and let $g:~Z_1\times \cdots \times Z_k \to \mathbb{Z}$ be the order
preserving function
\[ g(x_1,\ldots,x_k) = -f(-x_1,\ldots,-x_k).\]
I claim that
\begin{equation} \tilde{f}(G_1,\ldots,G_k) + \tilde{g}(-G_1,-G_2,\ldots,-G_k) \approx 0\label{temp0}\end{equation}
so that $\tilde{f}(G_1,\ldots,G_k) \approx$ an i-game by Theorem~\ref{plethora}(f).
To show (\ref{temp0}) we need to show that for any integer-valued game $K$,
\begin{equation} \outcome(\tilde{f}(G_1,\ldots,G_k) + \tilde{g}(-G_1,\ldots,-G_k) + K) = \outcome(K).\label{h2form}\end{equation}
We show that the left hand side is $\ge$ than the right hand side by essentially showing
that Left can play the sum $G_1 + \cdots + G_k + (-G_1) + \cdots + (-G_k) + K$ in such a way
that her score in each $G_i$ component outweighs the score in the corresponding
$-G_i$ component.  The other direction of the inequality follows by symmetry.

First of all, notice that since the $G_i$ are i-games by assumption, $G_i - G_i \approx 0$ for every
$i$ and so
\[ \outcome((G_1 - G_1) + \cdots + (G_k - G_k) + K) = \outcome(K).\]
Now we distort this sum by putting exorbitant penalties on Left for failing to
ensure that the final score of any of the $G_i - G_i$ components is $\ge 0$.  Specifically,
let $\delta:~ \mathbb{Z} \to \mathbb{Z}$ be the function given by
$\delta(x) = 0$ if $x \ge 0$, and $\delta(x) = -M$ if $x < 0$.  Then by Lemma~\ref{distortions}
$\delta(G_i - G_i)$ is an i-game because $G_i - G_i$ is.  Moreover, $\delta(G_i - G_i)$ must
have outcomes $(\delta(0),\delta(0)) = (0,0)$, so that by Lemma~\ref{numcompare} $\delta(G_i - G_i) \approx 0$.
Therefore,
\begin{equation} \outcome(\delta(G_1 - G_1) + \cdots + \delta(G_k - G_k) + K) = \outcome(K).\label{h1form}\end{equation}
Now let $h_1$ and $h_2$ be functions $S_1 \times \cdots \times S_k \times Z_1 \times \cdots \times Z_k \times \mathbb{Z} \to \mathbb{Z}$
given by
\[ h_1(x_1,\ldots,x_k,y_1,\ldots,y_k,z) = \delta(x_1 + y_1) + \delta(x_2 + y_2) + \cdots + \delta(x_k + y_k) + z\]
and
\[ h_2(x_1,\ldots,x_k,y_1,\ldots,y_k,z) = f(x_1,\ldots,x_k) + g(y_1,\ldots,y_k) + z\]
Then we can write (\ref{h1form}) as
\[ \outcome(\tilde{h_1}(G_1,\ldots,G_k,-G_1,\ldots,-G_k,K)) = \outcome(K).\]
Suppose we know that $h_1 \le h_2$ for all possible inputs.  Then Lemma~\ref{boundingfunctions}
implies that
\[ \outcome(\tilde{h_2}(G_1,\ldots,G_k,-G_1,\ldots,-G_k,K)) \ge \outcome(K)\]
which is just the $\ge$ direction of (\ref{h2form}).  By symmetry the $\le$ direction of (\ref{h2form}) also
follows and we are done.  So it remains to show that $h_1 \le h_2$, i.e.,
\begin{equation} \delta(x_1 + y_1) + \cdots + \delta(x_k + y_k) \le f(x_1, \ldots, x_k) + g(y_1,\ldots,y_k) \label{finalgoal}\end{equation}
for all $(x_1,\ldots,x_k,y_1,\ldots,y_k) \in S_1 \times \cdots \times S_k \times Z_1 \times \cdots \times Z_k$.

Suppose first that $x_i + y_i \ge 0$ for all $i$.  Then
$\delta(x_i + y_i) = 0$ for all $i$ so the left hand side of (\ref{finalgoal}) is zero.  On the other hand,
since $-y_i \le x_i$ for every $i$, and $f$ is order preserving,
\[ -g(y_1,\ldots,y_k) = f(-y_1,\ldots,-y_k) \le f(x_1,\ldots,x_k),\]
so (\ref{finalgoal}) holds.  Otherwise, $x_i + y_i < 0$ for some $i$, and so
the left hand side of (\ref{finalgoal}) is $\le -M$.  On the other hand,
the right hand side is at least $-M$, by choice of $M$ (and the fact that
the range of $g$ is also bounded between $-M$ and $M$).  Therefore
(\ref{finalgoal}) again holds, and we are done.
\end{proof}

Now the Lemma shows that $\tilde{f}(G_1,\ldots,G_k)$ is equivalent to an i-game.  We can easily use this to show
that it is in fact an i-game.  Note that every subposition of $\tilde{f}(G_1,\ldots,G_k)$
is of the form $\tilde{f}(G_1',\ldots,G_k')$ where $G_i'$ is a subposition of $G_i$ for every $i$.  By definition
of i-game, the $G_i'$ will also be equivalent to i-games, and so by the previous lemma
every subposition of $\tilde{f}(G_1,\ldots,G_k)$ is equivalent to an i-game.  So by the following lemma, $\tilde{f}(G_1,\ldots,G_k)$
is itself an i-game:

\begin{lemma}
If $G$ is an integer-valued game, and every subposition of $G$ is equivalent ($\approx$) to an i-game, then $G$ is an i-game.
\end{lemma}
\begin{proof}
$G$ is an i-game as long as every even-tempered subposition $G'$ satisfies $\loutcome(G') \ge \routcome(G')$.
But if $G'$ is an even-tempered subposition of $G$, then $G$ equals an i-game $H$ by assumption, and
$H$ is even-tempered by Theorem~\ref{sameparity}.  So $\loutcome(H) \ge \routcome(H)$.  But
$G' \approx H$ implies that $\outcome(G') = \outcome(H)$.
\end{proof}

For the last claim of Theorem~\ref{generaldistortions}, let $H_1, \ldots, H_k$ be general games
with $H_i \in \mathcal{W}_{S_i}$.  Then $H_i \lesssim_+ H_i^+$
and $H_i^+ \lesssim H_i$ for every $i$, so by Theorem~\ref{fineenough},
\[ \tilde{f}(H_1,\ldots,H_k) \lesssim_+ \tilde{f}(H_1^+,\ldots,H_k+)\]
and
\[ \tilde{f}(H_1^+,\ldots,H_k^+) \lesssim_+ \tilde{f}(H_1,\ldots,H_k)\]
so
\[ \tilde{f}(H_1,\ldots,H_k) \approx_+ \tilde{f}(H_1^+,\ldots,H_k^+).\]
But we just showed that $\tilde{f}(H_1^+,\ldots,H_k^+)$ is equivalent
to an i-game, so therefore
\[ \tilde{f}(H_1,\ldots,H_k)^+ \approx \tilde{f}(H_1^+,\ldots,H_k^+).\]
The proof of the claim for downsides is completely analogous.  This completes the proof
of Theorem~\ref{generaldistortions}.
\end{proof}
%TODO: I just pulled out a block and moved it back up to the section on
%$n$-valued games.  It was the fact that $\tilde{f}$'s action
%on $S$-valued games is determined by its action on $S$-valued i-games.
%(And I even really should have said \emph{even} i-games)

\chapter{Reduction to Partizan theory}\label{chap:psi}
\section{Cooling and Heating}
%In analogy to the cooling and heating operations of standard partizan games\footnote{TODO: that aren't mentioned above}, we make the following definition:
The following definition is an imitation of the standard cooling and heating operators for partizan games, defined and discussed on pages 102-108 of \emph{ONAG}
and Chapter 6 of \emph{Winning Ways}.
\begin{theorem}
If $G$ is a $\mathbb{Z}$-valued game and $n$ is an integer (possibly negative), we define $G$ \emph{cooled by $n$}, $G_n$,
to be $G$ is $G$ is a number, and otherwise $\langle (G^L)_n - n \,|\, (G^R)_n + n\rangle$, where
$G^L$ and $G^R$ range over the left and right options of $G$.  We define $G$ \emph{heated by $n$}
to be $G_{-n}$, $G$ cooled by $-n$.
\end{theorem}

The following results are easily verified from the definition:
\begin{theorem}\label{coolfacts}
Let $G$ and $H$ be games, and $n$ and $m$ be integers.
\begin{description}
\item[(a)] $(-G)_n = -(G_n)$.
\item[(b)] $(G + H)_n = G_n + H_n$.
\item[(c)] If $G$ is even-tempered, then $\loutcome(G_n) = \loutcome(G)$ and $\routcome(G_n) = \routcome(G)$.
\item[(d)] If $G$ is odd-tempered, then $\loutcome(G_n) = \loutcome(G) - n$ and $\routcome(G_n) = \routcome(G) + n$.
\item[(e)] $(G_n)_m = G_{n+m}$.
\item[(f)] $G$ is an i-game iff $G_n$ is an i-game.
\end{description}
Note that there are no references to $\approx$.
\end{theorem}
\begin{proof}
We proceed with inductive proofs in the style of \emph{ONAG}.  We mark the inductive
step with $\stackrel{!}{=}$.
\[ (-G)_n = \langle (-G^R)_n - n\,|\, (-G^L)_n + n \rangle \stackrel{!}{=}\]\[ \langle -(G^R_n) - n | -(G^L_n) + n \rangle
= -\langle (G^L)_n - n | (G^R)_n + n \rangle = -(G_n),\]
unless $G$ is a number, in which case (a) is obvious.
\[ (G + H)_n = \langle (G + H)^L_n - n | (G + H)^R_n + n \rangle =\]\[ \langle (G^L + H)_n - n,\,(G + H^L)_n - n\,|\,(G^R + H)_n + n,\, (G + H^R)_n + n \rangle \stackrel{!}{=} \]\[\langle G^L_n + H_n - n,\,G_n + H^L_n - n\,|\, G^R_n + H_n + n,\, G_n + H^R_n + n \rangle
= G_n + H_n,\]
unless $G$ and $H$ are both numbers, in which case (b) is obvious.
If $G$ is even-tempered, then
\[ \loutcome(G_n) = \max\{\routcome(G^L_n - n)\} \stackrel{!}{=} \max\{\routcome(G^L) + n - n\} = \max\{\routcome(G^L)\} = \loutcome(G),\]
unless $G$ is a number, in which case $\loutcome(G_n) = \loutcome(G)$ is obvious.  And similarly, $\routcome(G_n) = \routcome(G)$.
If $G$ is odd-tempered, then
\[ \loutcome(G_n) = \max\{\routcome(G^L_n - n)\} \stackrel{!}{=} \max\{\routcome(G^L) - n\} = \loutcome(G) - n,\]
and similarly $\routcome(G_n) = \routcome(G) + n$.
\[ (G_n)_m = \langle (G_n)^L_m - m \,|\, (G_n)^R_m + m \rangle = \langle (G^L_n - n)_m - m \,|\, (G^R_n + n)_m + m \rangle \stackrel{*}{=}\]\[
\langle (G^L_n)_m - (n + m) \,|\, (G^R_n)_m + (n + m) \rangle \stackrel{!}{=} \langle G^L_{n + m} - (n + m) \,|\, G^R_{n +m} + (n + m)\rangle =
G_{n + m},\]
unless $G$ is a number, in which case (e) is obvious.  Here the $\stackrel{*}{=}$ follows by part (b).

Finally, part (f) follows by an easy induction using part (c).
\end{proof}

Using these, we show that heating and cooling are meaningful modulo $\approx$ and $\approx_\pm$:
\begin{theorem}\label{heating-preserves-sides}
If $G$ and $H$ are games, $n \in \mathbb{Z}$, and $\Box$ is one of $\approx$, $\approx_\pm$, $\lesssim$,
$\lesssim_\pm$, etc., then $G\ \Box\ H \iff G_n\ \Box\ H_n$.
\end{theorem}
\begin{proof}
It's enough to consider $\lesssim_-$ and $\lesssim_+$. By symmetry, we only consider
$\lesssim_-$.  By part (e) of the preceding theorem, cooling by $-n$ is exactly the inverse of
cooling by $n$, so we only show that $G \lesssim_- H \Rightarrow G_n \lesssim_- H_n$.
Let $G \lesssim_- H$.  Then $G$ and $H$ have the same parity.  Let $X$ be arbitrary.
Note that $G_n + X$ and $H_n + X$ are just $(G + X_{-n})_n$ and $(H + X_{-n})$.  By
assumption, $\rfout(G + X_{-n}) \le \rfout(H + X_{-n})$, so by part (c) or (d) of the previous theorem
(depending on the parities of $G, H,$ and $X$) we see that
\[ \rfout(G_n + X) = \rfout((G + X_{-n})_n) \le \rfout((H + X_{-n})_n) = \rfout(H_n + X).\]
Then since $X$ was arbitrary, we are done.
\end{proof}

So heating and cooling induce automorphisms of the commutative monoid $\mathcal{W}_\mathbb{Z}/\approx$
that we are interested in.

\begin{definition}
For $n \in \mathbb{Z}$, let $I_n$ be recursively defined set of $\mathbb{Z}$-valued games such that
$G \in I_n$ iff every option of $G$ is in $I_n$, and
\begin{itemize}
\item If $G$ is even-tempered, then $\loutcome(G) \ge \routcome(G)$.
\item If $G$ is odd-tempered, then $\loutcome(G) - \routcome(G) \ge n$.
\end{itemize}
\end{definition}
It's clear that $I_n \subseteq I_m$ when $n > m$, and that the i-games are precisely the elements
of $I = \bigcup_{n \in \mathbb{Z}} I_n$.  Also, the elements of $\cap_{n \in \mathbb{Z}} I_n$ are nothing but
the numbers, since any odd-tempered game fails to be in $I_n$ for some $n$.

The class $I_0$ consists of the games in which being unexpectedly forced to move is harmless.\footnote{This makes $I_0$ the class of well-tempered
games which are also Milnor games, in the sense used by Ettinger in the papers ``On the Semigroup of Positional Games'' and ``A Metric for Positional Games.''}
Just as looking at i-games allowed us to extends Equations (\ref{uno}-\ref{dos}) to other parities,
the same thing happens here: if $G$ and $H$ are two games in $I_0$ with $\routcome(G) \ge 0$ and $\routcome(H) \ge 0$,
then $\routcome(G + H) \ge 0$, \emph{regardless of parity}.  We have already seen this if
$G$ and $H$ are even-tempered (Equation~(\ref{uno})) or if one of $G$ is odd-tempered (Equation~(\ref{siete})), and the final case,
where both games are odd-tempered, comes from the following results:
\begin{theorem}
Let $G$ and $H$ be $\mathbb{Z}$-valued games in $I_n$ for some $n$.
\begin{itemize}
\item If $G$ is odd-tempered and $H$ is odd-tempered, then
\begin{equation} \routcome(G + H) \ge \routcome(G) + \routcome(H) + n \label{trece}\end{equation}
\item If $G$ is even-tempered and $H$ is odd-tempered, then
\begin{equation} \loutcome(G + H) \ge \loutcome(G) + \routcome(H) + n \label{catorce}\end{equation}
\end{itemize}
\end{theorem}
\begin{proof}
We proceed by induction as usual.  To see (\ref{trece}), note that $G + H$ is not a number,
and every right option of $G + H$ is of the form $G^R + H$ or $G + H^R$.  By induction,
Equation~(\ref{catorce}) tells us that $\loutcome(G^R + H) \ge \loutcome(G^R) + \routcome(H) + n \ge \routcome(G) + \routcome(H) + n$.
Similarly, every option of the form $G + H^R$ also has left-outcome at least $\routcome(G) + \routcome(H) + n$,
establishing (\ref{trece}).
Likewise, to see (\ref{catorce}), note that if $G$ is a number, this follows from Proposition~\ref{numavoid}
and the definition of $I_n$, and otherwise, letting $G^L$ be a left option of $G$
such that $\routcome(G^L) = \loutcome(G)$, we have by induction
\[ \loutcome(G + H) \ge \routcome(G^L + H) \ge \routcome(G^L) + \routcome(H) + n = \loutcome(G) + \routcome(H) + n.\]
\end{proof}
Similarly we have
\begin{theorem}
Let $G$ and $H$ be $\mathbb{Z}$-valued games in $I_n$ for some $n$.
\begin{itemize}
\item If $G$ and $H$ are both odd-tempered, then
\begin{equation} \loutcome(G + H) \le \loutcome(G) + \loutcome(H) - n \label{quince}\end{equation}
\item If $G$ is even-tempered and $H$ is odd-tempered, then
\begin{equation} \routcome(G + H) \le \routcome(G) + \loutcome(H) - n \label{pear}\end{equation}
\end{itemize}
\end{theorem}

Mimicking the proof that i-games are closed under negation and addition, we also have
\begin{theorem}\label{filtration}
The class of games $I_n$  is closed under negation and addition.
\end{theorem}
\begin{proof}
Negation is easy.  We show closure under addition.  Let $G, H \in I_n$.  By induction,
every option of $G + H$ is in $I_n$.  It remains to show that $\loutcome(G + H) - \routcome(G + H)$
is appropriately bounded.  If $G + H$ is even-tempered, we already handled this case in Theorem~\ref{closure}, which
guarantees that $G + H$ is an i-game so that $\loutcome(G + H) \ge \routcome(G + H)$.
So assume $G + H$ is odd-tempered.  Without loss of generality, $G$ is odd-tempered and $H$ is even-tempered.  Then by Equation~(\ref{catorce})
\[ \loutcome(G + H) \ge \routcome(G) + \loutcome(H) + n \]
while by Equation~(\ref{cuatro}),
\[ \routcome(G + H) \le \routcome(G) + \loutcome(H).\]
Therefore
\[ \loutcome(G + H) - \routcome(G + H) \ge n.\]
\end{proof}

Next we relate the $I_n$ to cooling and heating:
\begin{theorem}\label{slide}
For any $G$, $G \in I_n$ iff $G_m \in I_{n - 2m}$.
\end{theorem}
\begin{proof}
By Theorem~\ref{coolfacts}(d), we know that whenever
$H$ is an odd-tempered game,  $\loutcome(H) - \routcome(H) \ge n$ iff
$\loutcome(H_m) - \routcome(H_m)  = (\loutcome(H) - m) - (\routcome(H) + m) \ge n - 2m$,
while if $H$ is even-tempered, then $\loutcome(H) - \routcome(H) \ge 0$ iff $\loutcome(H_m) - \routcome(H_m) \ge 0$,
by Theorem~\ref{coolfacts}(c).
\end{proof}
So heating by 1 unit establishes a bijection from $I_n$ to $I_{n + 2}$ for all $n$.
Also, note that $G$ is an i-game iff $G_{-n} \in I_0$ for some $n \ge 0$.

Let $\mathcal{I}$ denote the invertible elements of $\mathcal{W}_\mathbb{Z}/\approx$, i.e., the equivalence
classes containing i-games.  Also, let $\mathcal{I}_n$ be the equivalence classes containing games
in $I_n$.  By Theorem~\ref{filtration} each $\mathcal{I}_n$ is a subgroup of $\mathcal{I}$.  And we have a filtration:
\[ \cdots \subseteq \mathcal{I}_2 \subseteq \mathcal{I}_1 \subseteq \mathcal{I}_0 \subseteq \mathcal{I}_{-1} \subseteq \cdots \subseteq \mathcal{I}.\]
Furthermore, heating by $m$ provides an isomorphism of partially ordered abelian groups from
$\mathcal{I}_n$ to $\mathcal{I}_{n + 2m}$.  Because $\mathcal{I}$ is the union $\bigcup_{k \in \mathbb{Z}} \mathcal{I}_{2k}$,
it follows that as a partially-ordered abelian group, $\mathcal{I}$ is just the direct limit (colimit) of
\[ \cdots \hookrightarrow \mathcal{I}_{-2} \hookrightarrow \mathcal{I}_{-2} \hookrightarrow \mathcal{I}_{-2} \hookrightarrow \cdots \]
where each arrow is heating by 1.
In the next two sections we will show that the even-tempered component of $\mathcal{I}_{-2}$ is isomorphic to $\mathcal{G}$,
the group of (short) partizan games, and that the action of heating by 1 is equivalent to the Norton multiplication by $\{1*|\}$, which is the
same as overheating from $1$ to $1*$.

\section{The faithful representation}
We construct a map $\psi$ from $I_{-2}$ to $\mathcal{G}$, the abelian group of short partizan games.  This map does
the most obvious thing possible:
\begin{definition}
If $G$ is in $I_{-2}$, then the \emph{representation of $G$}, denoted $\psi(G)$, is defined recursively by $\psi(n) = n$ (as a surreal number)
if $n \in \mathbb{Z}$, and by
\[ \psi(G) = \{\psi(G^L)|\psi(G^R)\}\]
if $G = \langle G^L | G^R \rangle$ is not a number.
\end{definition}
So for instance, we have
\[ \psi(2) = 2,\]
\[ \psi(\langle 3 | 4 \rangle) = \{3|4\} = 3.5\]
\[ \psi(\langle 2|2 || 1|3 \rangle) = \{2|2||1|3\} = \{2*|2\} = 2 + \downarrow.\]
Usually, turning angle brackets into curly brackets causes chaos to ensue: for example $\langle 0 \,|\, 3 \rangle + \langle 0 \,|\, 3 \rangle
\approx 3$ but $\{0\,|\,3\} + \{0\,|\,3\} = 2$.  But $\langle 0\,|\, 3\rangle$ isn't in $I_{-2}$.

It is clearly the case that $\psi(-G) = -\psi(G)$.  But how does $\psi$ interact with other operations?

The next two results are very straightforward:
\begin{theorem}\label{unused-label}
If $G$ is odd-tempered, then $\psi(G) \ge 0$ iff $\routcome(G) > 0$, and $\psi(G) \le 0$ iff $\loutcome(G) < 0$.
\end{theorem}
\begin{proof}
By definition $\psi(G) \ge 0$ iff $\psi(G)$ is a win for Left when Right goes first.  If Right goes first, then Right goes last,
so a move to 0 is a loss for Left.  Thus Left needs the final score of $G$ to be at least 1.  The other case is handled similarly.
\end{proof}
Similarly,
\begin{theorem}
If $G$ is even-tempered, then $\psi(G) \ge 0$ iff $\routcome(G) \ge 0$, and $\psi(G) \le 0$ iff $\loutcome(G) \le 0$.
\end{theorem}
\begin{proof}
The same as before, except now when Right makes the first move, Left makes the last move, so a final score of zero is a win for \emph{Left}.
\end{proof}

But since we are working with $I_{-2}$ games, we can strengthen these a bit:
\begin{theorem}\label{subtler}
If $G$ is an even-tempered $I_{-2}$ game, and $n$ is an integer, then $\psi(G) \ge n$ iff $\routcome(G) \ge n$, and $\psi(G) \le n$ iff $\loutcome(G) \le n$.
If $G$ is an odd-tempered $I_{-2}$ game instead, then $\psi(G) \ge n$ iff $\routcome(G) > n$, and $\psi(G) \le n$ iff $\loutcome(G) < n$.
\end{theorem}
\begin{proof}
We proceed by induction on $G$ and $n$ (interpreting $n$ as a partizan game).  If $G$ is a number, the result is obvious.

Next, suppose $G$ is not a number, but is even-tempered. If $n \le \psi(G)$, then
$\psi(G^R) \not \le n$ for any $G^R$.  By induction, this means that $\loutcome(G^R) \not < n$ for every
$G^R$, i.e., $\loutcome(G^R) \ge n$ for every $G^R$.  This is the same as $\routcome(G) \ge n$.
Conversely, suppose that $\routcome(G) \ge n$.  Then reversing our steps, every $G^R$ has $\loutcome(G^R) \ge n$,
so by induction $\psi(G^R) \not \le n$ for any $G^R$.  Then the only way that $n \le \psi(G)$ can fail to be true
is if $\psi(G) \le n'$ for some $n'$ that is less than $n$ and simpler than $n$.  By induction, this implies
that $\loutcome(G) \le n' < n \le \routcome(G)$, contradicting the definition of $I_n$.  So we
have shown that $n \le \psi(G) \iff n \le \routcome(G)$.  The proof that $n \ge \psi(G) \iff n \ge \routcome(G)$ is similar.

Next, suppose that $G$ is odd-tempered.  If $n \le \psi(G)$, then $\psi(G^R) \not \le n$ for any $G^R$.  By induction,
this means that $\loutcome(G^R) \not \le n$ for every $G^R$, i.e., $\loutcome(G^R) > n$ for every $G^R$. This is the
same as $n < \routcome(G)$.  Conversely, suppose that $n < \routcome(G)$.  Reversing our steps,
every $G^R$ has $\loutcome(G^R) > n$, so by induction $\psi(G^R) \not \le n$ for every $G^R$.  Then the only
way that $n \le \psi(G)$ can fail to be true is if $\psi(G) \le n'$ for some $n' < n$, $n'$ simpler than $n$.
But then by induction, this implies that $\loutcome(G) < n' < n < \routcome(G)$, so that $\routcome(G) - \loutcome(G) \ge 3$,
contradicting the definition of $I_{-2}$.
\end{proof}
The gist of this proof is that the $I_{-2}$ condition prevents the left and right stopping values of $\psi(G)$ from being too spread out.

Next we show
\begin{theorem}\label{additive1}
If $G$ and $H$ are in $I_{-2}$, then $\psi(G + H) = \psi(G) + \psi(H)$.
\end{theorem}
\begin{proof}
We proceed inductively.  If $G$ and $H$ are both numbers, this is obvious.  If both are not numbers,
this is again straightforward:
\[ \psi(G + H) = \{ \psi(G^L + H), \psi(G + H^L)\,|\, \psi(G^R + H), \psi(G + H^R) \} \]\[
 \stackrel{!}{=} \{ \psi(G^L) + \psi(H), \psi(G) + \psi(H^L)\,|\, \psi(G^R) + \psi(H), \psi(G) + \psi(H^R)\} = \psi(G) + \psi(H),\]
where the middle equality follows by induction on subgames.  The one remaining case is when exactly one of $G$ and $H$
is a number.  Consider $G + n$, where $n$ is a number and $G$ is not.  If $\psi(G)$ is not an integer,
then by integer avoidance,
\[ \psi(G) + n = \{ \psi(G^L) + n \,|\, \psi(G^R) + n \} = \{ \psi(G^L + n) \,|\, \psi(G^R + n) \} = \psi(G + n),\]
where the middle step is by induction.

So suppose that $\psi(G)$ is an integer $m$.  Thus $m \le \psi(G) \le m$.  If $G$ is even-tempered, then by Theorem~\ref{subtler},
$\loutcome(G) \le m \le \routcome(G)$.  Then by Proposition~\ref{numavoid}, $\loutcome(G + n) \le m + n \le \routcome(G + n)$,
so by Theorem~\ref{subtler} again, $m + n \le \psi(G + n) \le m + n$.  Therefore $\psi(G + n) = m + n = \psi(G) + n$.
Similarly, if $G$ is odd-tempered, then by Theorem~\ref{subtler},
$\loutcome(G) < m < \routcome(G)$.  So by Proposition~\ref{numavoid}, $\loutcome(G + n) < m + n < \routcome(G + n)$,
which by Theorem~\ref{subtler} implies that $m + n \le \psi(G + n) \le m + n$, so that again, $\psi(G + n) = m + n = \psi(G) + n$.
\end{proof}

As in the previous section, let $\mathcal{I}_{-2}$ denote the quotient space of $I_{-2}$ modulo $\approx$.  Putting
everything together,
\begin{theorem}\label{embedding}
If $G, H \in I_{-2}$ have the same parity, then $\psi(G) \le \psi(H)$ if and only if $G \lesssim H$.
In fact $\psi$ induces a weakly-order preserving homomorphism from $\mathcal{I}_{-2}$ to $\mathcal{G}$ (the group of short partizan games).
Restricted to even-tempered games in $\mathcal{I}_{-2}$, this map is strictly order-preserving.  For $G \in I_{-2}$,
$\psi(G) = 0$ if and only if $G \approx 0$ or $G \approx \langle -1 | 1 \rangle$.  The kernel of the homomorphism
from $\mathcal{I}_{-2}$ has two elements.
\end{theorem}
\begin{proof}
First of all, suppose that $G, H \in I_{-2}$ have the same parity.  Then $G - H$ is an even-tempered game in $I_{-2}$.  So by
Theorem~\ref{subtler} and Corollary~\ref{restatement}
\[ \psi(G - H) \le 0 \iff \loutcome(G - H) \le 0 \iff G - H \lesssim 0.\]
Since i-games are invertible, $G - H \lesssim 0 \iff G \lesssim H$.  And by Theorem~\ref{additive1} and the remarks
before Theorem~\ref{unused-label}, $\psi(G - H) = \psi(G) - \psi(H)$.  Thus
\[ \psi(G) \le \psi(H) \iff \psi(G) - \psi(H) \le 0 \iff G \lesssim H.\]

It then follows that if $G \approx H$, then $\psi(G) = \psi(H)$, so $\psi$ is well-defined on the quotient space $\mathcal{I}_{-2}$.
And if $G \gtrsim H$, then $G$ and $H$ have the same parity, so by what was just shown $\psi(G) \ge \psi(H)$.  Thus
$\psi$ is weakly order-preserving.  It is a homomorphism by Theorem~\ref{additive1}.

When $G$ and $H$ are both even-tempered, then $G$ and $H$ have the same parity, so $\psi(G) \le \psi(H) \iff G \lesssim H$.  Thus,
restricted to even-tempered games, $\psi$ is strictly order preserving on the quotient space.

Suppose that $\psi(G) = 0$.  Then $G$ has the same parity as either $0$ which is even-tempered, or $\langle -1 | 1 \rangle$, which is odd-tempered.
Both $0$ and $\langle -1 | 1 \rangle$ are in $I_{-2}$.
Now
\[ 0 = \psi(0) = \psi(G) = \psi(\langle -1 | 1 \rangle),\]
so by what has just been shown, either $G \approx 0$ or $G \approx \langle -1 | 1 \rangle$.

Then considering games modulo $\approx$, the kernel of $\psi$ has two elements, because $0 \not \approx \langle -1 | 1 \rangle$.
\end{proof}

Thus the even-tempered part of $\mathcal{I}_{-2}$ is isomorphic to a subgroup of $\mathcal{G}$.  In fact, it's isomorphic
to all of $\mathcal{G}$:
\begin{theorem}\label{psisurj}
The map $\psi$ is surjective.  In fact, for any $X \in \mathcal{G}$, there is some even-tempered $H$ in $I_{-2}$ with
$\psi(H) = X$.  In fact, if $X$ is not an integer, then we can choose $H$ such that the left options of $\psi(H)$ are equal to the left options
of $X$ and the right options of $\psi(H)$ are equal to the right options of $X$.
\end{theorem}
\begin{proof}
We proceed by induction on $X$.  If $X$ equals an integer, the result is obvious.  If not,
let $X = \{L_1, L_2, \ldots\,|\, R_1, R_2, \ldots \}$.  By induction, we can produce even-tempered $I_{-2}$ games
$\lambda_1, \lambda_2, \ldots, \rho_1, \rho_2, \ldots \in I_2$ with $\psi(\lambda_i) = L_i$
and $\psi(\rho_i) = R_i$.  Replacing $\lambda_i$ and $\rho_i$ by $\lambda_i + \langle -1 | 1 \rangle$
and $\rho_i + \langle -1 | 1 \rangle$, we can instead take the $\lambda_i$ and $\rho_i$ to be odd-tempered.

Then consider the even-tempered game
\[ H = \langle \lambda_1, \lambda_2, \ldots\,|\, \rho_1, \rho_2, \ldots \rangle.\]
As long as $H \in I_{-2}$, then $H$ will have all the desired properties.  So suppose that
$H$ is not in $I_{-2}$.  As $H$ is even-tempered, and all of its options are in $I_{-2}$,
this implies that $\loutcome(H) < \routcome(H)$.
 Let $n = \loutcome(H)$.  Then we have
\[ \routcome(\lambda_i) \le n \]
for every $i$, which by Theorem~\ref{subtler} is the same as $L_i = \psi(\lambda_i) \not \ge n$
for every $i$.  Similarly, we have
\[ n < \routcome(H) \le \loutcome(\rho_i)\]
for every $i$, so by Theorem~\ref{subtler} again, $R_i = \psi(\rho_i) \not \le n$ for every $i$.
Therefore, every left option of $X$ is less than or fuzzy with $n$, and every right option of
$X$ is greater than or fuzzy with $n$.  So $X$ is $n$, or something simpler.  But nothing
is simpler than an integer, so $X$ is an integer, contradicting our assumption that it wasn't.
\end{proof}

Therefore, the even-tempered subgroup of $\mathcal{I}_{-2}$ is isomorphic to $\mathcal{G}$.  This subgroup
has as a complement the two-element kernel of $\psi$, so therefore
we have the isomorphism
\[ \mathcal{I}_{-2} \cong \mathbb{Z}_2 \oplus \mathcal{G}.\]

\section{Describing everything in terms of $\mathcal{G}$}
As noted above, $\mathcal{I}$ as a whole is the direct limit of
\[ \cdots \hookrightarrow \mathcal{I}_{-2} \hookrightarrow \mathcal{I}_{-2} \hookrightarrow \mathcal{I}_{-2} \hookrightarrow \cdots \]
where each arrow is the injection $G \to G_{-1}$.  What is the corresponding injection
in $\mathcal{G}$?  It turns out to be Norton multiplication by $\{1*|\}$.

Let $E$ be the partizan game form $\{1*|\}$. Note that $E = 1$, and $E + (E^L - E) = E^L = 1*$.  Thus
\[ n.E = n\]
for $n$ an integer, and
\[ G.E = \langle G^L.E + 1*| G^R.E - 1* \rangle\]
when $G$ is not an integer.  So Norton multiplication by $\{1*|\}$ is the same as
overheating from $1$ to $1*$:
\[ G.E = \int_1^{1*} G\]
(On the other hand, in Sections~\ref{sec:norton} and \ref{sec:even-and-odd-revisited} we considered overheating from $1*$ to $1$!)

\begin{theorem}

If $G$ is an even-tempered game in $I_{-2}$, then
\begin{equation} \psi(G_{-1}) = \psi(G).E \label{heating1}\end{equation}
and if $G$ is odd-tempered and in $I_{-2}$, then
\begin{equation} \psi(G_{-1}) = * + \psi(G).E. \label{heating2}\end{equation}
\end{theorem}
\begin{proof}

Let $O = \langle -1 | 1 \rangle$.
Then $O$ is an odd-tempered game in $I_{-2}$ with
$\psi(O) = 0$, so $O + O \approx 0$ and the map $G \to G + O$
is an involution on $\mathcal{I}_{-2}$ interchanging odd-tempered and even-tempered games, and leaving
$\psi(G)$ fixed.  Also, $O_{-1} = \langle 0 | 0 \rangle$, so $\psi(O_{-1}) = *$.

Let $H = \psi(G)$.  We need to show that $\psi(G_{-1}) = H.E$ if $G$ is even-tempered
and $\psi(G_{-1}) = * + H.E$ if $G$ is odd-tempered.  We proceed by induction on $H$, rather than $G$.

We first reduce the case where $G$ is odd-tempered to the case where $G$ is even-tempered (without changing $H$).
If $G$ is odd-tempered, then $G \approx G' + O$, where $G' = G + O$ is even-tempered.
Then
\[ H = \psi(G) = \psi(G' + O) = \psi(G') + \psi(O) = \psi(G').\]
If we can show the claim for $H$ when $G$ is even-tempered, then $\psi(G'_{-1}) = \psi(G').E = H.E$.
But in this case,
\[ \psi(G_{-1}) = \psi((G' + O)_{-1}) = \psi(G'_{-1}) + \psi(O_{-1}) = H.E + *,\]
establishing (\ref{heating2}).

So it remains to show that if $G$ is even-tempered, and $\psi(G) = H$, then $\psi(G_{-1}) = H.E$.  For the base case,
if $H$ equals an integer $n$, then $\psi(G) = H = n = \psi(n)$, and $G$ and $n$ are both
even-tempered, so $G \approx n$.  Therefore $G_{-1} \approx n_{-1} = n$, and so $\psi(G_{-1}) = \psi(n) = n = n.E = H.E$
and we are done.

So suppose that $H$ does not equal any integer.  By Theorem~\ref{psisurj} there exists an even-tempered game
$K \in \mathcal{I}_{-2}$ for which $\psi(K) = H$, and the $H^L$ and $H^R$ are exactly the $\psi(K^L)$ and $\psi(K^R)$.
Then by faithfulness, $\psi(G) = H = \psi(K)$, so $G \approx K$.  Moreover,
\[ \psi(K_{-1}) = \psi(\langle K^L_{-1} + 1 | K^R_{-1} - 1\rangle)
= \{ \psi(K^L_{-1}) + 1 | \psi(K^R_{-1}) - 1 \}.\]
Now every $\psi(K^L)$ or $\psi(K^R)$ is an $H^L$ or $H^R$, so by induction,
\[ \psi(K^L_{-1}) = \psi(K^L).E + *\]
\[ \psi(K^R_{-1}) = \psi(K^R).E + *\]
since $K^L$ and $K^R$ are odd-tempered.  Thus
\[ \psi(K_{-1}) = \{\psi(K^L).E + * + 1 | \psi(K^R).E + * - 1\} = \]\[\{\psi(K)^L.E + 1* | \psi(K)^R.E - 1*\} = \psi(K).E,\]
where the last step follows because $\psi(K) = H$ equals no integer.  But then since $G \approx K$,
$G_{-1} \approx K_{-1}$ and so $\psi(G_{-1}) = \psi(K_{-1}) = \psi(K).E = H.E$.
\end{proof}

In light of all this, the even-tempered component of $\mathcal{I}$ is isomorphic to the direct limit of
\[ \cdots \stackrel{(-).E}{\longrightarrow} \mathcal{G} \stackrel{(-).E}{\longrightarrow}  \mathcal{G} \stackrel{(-).E}{\longrightarrow}  \cdots \]
where each map is $x \to x.E$.
Together with the comments of Section~\ref{sec:so-far}, this gives a complete description of $\mathcal{W}_\mathbb{Z}$ in terms of $\mathcal{G}$.
The map assigning outcomes to $\mathbb{Z}$-valued games
can be recovered using Theorem~\ref{subtler} and Theorem~\ref{coolfacts}(c,d) - we leave this as an exercise to the reader.
% apricot TODO: outcomes come how?

The reduction of $\mathcal{W}_\mathbb{Z}$ to $\mathcal{G}$ has
a number of implications, because much of the theory of partizan games carries over.  For example, every even-tempered $\mathbb{Z}$-valued
game is divisible by two:
\begin{theorem}
If $G$ is an even-tempered $\mathbb{Z}$-valued game, then there exists an even-tempered $\mathbb{Z}$-valued game $H$
such that $H + H = G$.
\end{theorem}
\begin{proof}
Choose $n$ big enough that $(G^+)_{-n}$ and $(G^-)_{-n}$ are in $I_{-2}$.  Then $(G^+)_{-n}$ is an i-game,
by Theorem~\ref{coolfacts}(f), and
\[ (G^+)_{-n} \approx_+ G_{-n}\]
by Theorem~\ref{heating-preserves-sides}, so therefore $(G^+)_{-n} \approx (G_{-n})^+$.  Similarly
$(G^-)_{-n} \approx (G_{-n})^-$.  So both the upside and downside of $G_{-n}$ are in $I_{-2}$.

Let $K = G_{-n}$.  Then $\psi(K^+) \ge \psi(K^-)$, so by Corollary~\ref{twodivide} we can find partizan games
$H_1$ and $H_2$ with
\[ H_1 + H_1 = \psi(K^-)\]
\[ H_2 + H_2 = \psi(K^+) - \psi(K^-)\]
\[ H_2 \ge 0.\]
Then by Theorem~\ref{psisurj} there are i-games $X_1$ and $X_2$ in $I_{-2}$ with
\[ \psi(X_1) = H_1\]
\[ \psi(X_2) = H_2\]
Adding $\langle -1 | 1 \rangle$ to $X_1$ or $X_2$, we can assume $X_1$ and $X_2$ are even-tempered.  Then
$\psi(X_2) = H_2 \ge 0$, so $X_2 \gtrsim 0$.  Thus $X_1 + X_2 \gtrsim X_1$, so by Theorem~\ref{inorder2}
there is a $\mathbb{Z}$-valued game $J \approx (X_1 + X_2)\& X_1$.  Then
\[ \psi((J+J)^+) = \psi(J^+ + J^+) = \psi(J^+) + \psi(J^+) = \psi(X_1 + X_2) + \psi(X_1 + X_2) =\]\[
\psi(X_1) + \psi(X_2) + \psi(X_1) + \psi(X_2) = H_1 + H_1 + H_2 + H_2 = \]\[\psi(K^-) + \psi(K^+) - \psi(K^-) = \psi(K^+)\]
and
\[ \psi((J+J)^-) = \psi(J^- + J^-) = \psi(J^-) + \psi(J^-) =\]\[ \psi(X_1) + \psi(X_1) = H_1 + H_1 = \psi(K^-).\]
Then $J + J \approx\pm K$ by Theorem~\ref{embedding}, so $J + J \approx K$.  Then taking $H = J_n$, we have
\[ H + H = J_n + J_n = (J+J)_n \approx K_n = (G_{-n})_n = G.\]
\end{proof}
As another example, a theorem of Simon Norton (proven on page 207-209 of \emph{On Numbers and Games})
says that no short partizan game has odd order.  For instance, if $G + G + G = 0$, then $G = 0$.  By our results, one can easily show
that the same thing
is true in $\mathcal{W}_\mathbb{Z}$.  Or for another corollary, the problem of determining the outcome
of a sum of $\mathbb{Z}$-valued games, given in extensive form, is PSPACE-complete, because the same problem
is PSPACE-complete for partizan games, as shown by Morris and Yedwab (according to David Wolfe's ``Go Endgames are PSPACE-Hard''
in \emph{More Games of No Chance}).

It also seems likely that i-games have canonical simplest forms, just like partizan games, and that this can
be shown using the map $\psi$.  Moreover, the mean-value theorem carries over for i-games, though for non-invertible
games, the two sides can have different mean values.  In this case, $\lfout(n.G)$ and $\rfout(n.G)$ will gradually
drift apart as $n$ goes to $\infty$ - but at approximately linear rates.
We leave such explorations to the reader.
% apricot: TODO work that all out.
% apricot: TODO computational complexity intractability.
%            - also later on with 3-valued games!!  Perhaps.

%(Incidentally, the operation of Norton-multiplying with $\{1*|1*\}$ is the same as twice multiplying by
%$\{1/2|\}$.  In fact, cooling by 1/2 (which is well-defined for even $\mathbb{Z}$-valued games), corresponds
%to the Norton product by $\{1/2|\}$.  Moreover, the image of \textbf{ShUg} under Norton-multiplication
%by $\{1/2|\}$ is a set similar to the set of even-tempered games or even games: it is the set of games
%in which the integers show up after an even number of moves.  Incidentally, there are automorphisms
%of \textbf{ShUg} that double or half every stopping value:
%\[ f(G) = \{f(G^L)|f(G^R)\} \text{ unless $G$ is a number, and then $f(G) = G/2$}\]
%\[ g(G) = \{g(G^L)|g(G^R)\} \text{ unless $G$ is a number, and then $g(G) = 2G$}\]
%these two maps are inverses of eachother, and $g(f(G).\{1/2|\}) = G.(1*)$, so the image of
%$G \to G.\{1/2|\}$ is merely the set of even games, except that we uniformly divide all stopping values by two.)

% apricot
%\chapter{Examples}
%% put in examples of Scored Brussel Sprouts and Mercenary Clobber.

\chapter{Boolean and n-valued games}\label{chap:bool}
\section{Games taking only a few values}
For any positive integer $n$, we follow the convention for von Neumann Ordinals and identify $n$ with the set $\{0,1,\ldots,n-1\}$.
In this chapter, we examine the structure of $n$-valued games.
For $n = 2$, this gives us Boolean games, the
theory we need to analyze \textsc{To Knot or Not to Knot}.

When considering one of these restricted classes of games, we can no longer let addition and negation
be our main operations, because $\{0,1,\ldots,n-1\}$ is not closed under either operation.
Instead, we will use order-preserving operations like those of Chapter~\ref{chap:Distortions}.  These
alternative sets of games and operations can yield different indistinguishability relations from $\approx$.

We will examine four order-preserving binary operations on $n = \{0,\ldots,n-1\}$:
\begin{itemize}
\item $x \wedge y = \min(x,y)$.
\item $x \vee y = \max(x,y)$.
\item $x \oplus_n y = \min(x+y,n-1)$.
\item $x \odot_n y = \max(0,x + y - (n-1))$.
\end{itemize}
All four of these operations are commutative and associative, and each has an identity
when restricted to $n$.  The operation $\oplus_n$ corresponds to adding and rounding down in case
of an overflow,
and $\odot$ is a dual operation going in the other direction.  Here are tables showing what
these fuctions look like for $n = 3$:
\begin{center}
\begin{tabular}{c  c}
\begin{tabular}{|c||c|c|c|}
\hline
$\oplus_3$ & 0 & 1 & 2 \\
\hline
\hline
0 & 0 & 1 & 2 \\
\hline
1 & 1 & 2 & 2 \\
\hline
2 & 2& 2 & 2 \\
\hline
\end{tabular} &
\begin{tabular}{|c||c|c|c|}
\hline
$\odot_3$ & 0 & 1 & 2 \\
\hline
\hline
0 & 0 & 0 & 0 \\
\hline
1 & 0 & 0 & 1 \\
\hline
2 & 0 & 1 & 2 \\
\hline
\end{tabular} \\ \\
\begin{tabular}{|c||c|c|c|}
\hline
$\wedge$ & 0 & 1 & 2 \\
\hline
\hline
0 & 0 & 0 & 0 \\
\hline
1 & 0 & 1 & 1 \\
\hline
2 & 0 & 1 & 2 \\
\hline
\end{tabular} & 
\begin{tabular}{|c||c|c|c|}
\hline
$\wedge$ & 0 & 1 & 2 \\
\hline
\hline
0 & 0 & 1 & 2 \\
\hline
1 & 1 & 1 & 2 \\
\hline
2 & 2 & 2 & 2 \\
\hline
\end{tabular}
\end{tabular}
\end{center}

\begin{definition}
If $G$ and $H$ are $n$-valued games, we define
$G \wedge H$, $G \vee H$, $G \oplus_n H$ and $G \odot_n H$ by extending
these four operations to $n$-valued games, in the sense of Definition~\ref{def:extension}.
\end{definition}

Clearly $G \wedge H$ and $G \vee H$ don't depend on $n$,
justifying the notational lack of an $n$.  For $n = 2$, $\wedge$ is the same as $\odot_2$
and $\vee$ is the same as $\oplus_2$.  It is this case, specifically the operation
of $\vee = \oplus_2$, that is needed to analyze sums of positions in \textsc{To Knot or Not to Knot}.

%TODO: block removed
%Order-preserving operations are closed under composition.  Moreover, unary
%order-preserving functions have the following nice property, which will be used later:
%\begin{lemma}\label{outcomeslide}
%If $S, T$ are subsets of $\mathbb{Z}$, and $f~:~S\to T$ is order-preserving, then
%for any $G \in \mathcal{W}_S$, \[\routcome(\tilde{f}(G)) = f(\routcome(G)),\] and
%\[ \loutcome(\tilde{f}(G)) = f(\loutcome(G))\]
%\end{lemma}
%\begin{proof}
%Easy by induction; left as an exercise to the reader.
%\end{proof}
%
%We also have
%\begin{lemma}\label{boundingfunctions}
%Let $f$ and $g$ be two functions from $S_1 \times \cdots \times S_k \to T$,
%such that $f(x_1,\ldots,x_k) \le g(x_1,\ldots,x_k)$ for every
%$(x_1,\ldots,x_k) \in S_1 \times \cdots \times S_k$.  Then for
%every $(G_1,\ldots,G_k) \in \mathcal{W}_{S_1} \times \cdots \times \mathbf{Par}$,
%\[ \tilde{f}(G_1,\ldots,G_k) \lesssim \tilde{g}(G_1,\ldots,G_k) \]
%and in particular
%\[ \outcome(\tilde{f}(G_1,\ldots,G_k)) \le \outcome(\tilde{g}(G_1,\ldots,G_k))\]
%\end{lemma}
%\begin{proof}
%An obvious inductive proof using the fact (missing above) that when
%$A_i \lesssim A_i'$ and $B_i \lesssim B_i'$, then
%\[ \langle A_1, A_2, \ldots | B_1, B_2, \ldots \rangle \lesssim
%\langle A_1', A_2', \ldots | B_1', B_2', \ldots \rangle.\]
%\end{proof}

%TODO: more importantly, we just pulled out the section ``Distorions Revisited.'' It contained the
%following, material which seemed possibly inappropriate in the earlier location:

Our goal is to understand the indistinguishability quotients of $n$-valued games for various combinations of these operations.
The first main result, which follows directly from Theorem~\ref{fineenough}, is that
indistinguishability is always as coarse as the standard $\approx$ relation.

\begin{theorem}\label{fine2}
Let $f_1, f_2, \ldots, f_k$ be order-preserving operations $f_i:~(n)^i \to n$, and
$\sim$ be indistinguishability on $n$-valued games with respect to $\tilde{f}_1, \tilde{f}_2, \ldots, \tilde{f}_k$.
Then $\sim$ is as coarse as $\approx$.
\end{theorem}
\begin{proof}\
If $A \approx B$, then $\outcome(A) = \outcome(B)$, so $\approx$ satisfies condition (a)
of Theorem~\ref{indist}.  Part (b) of Theorem~\ref{indist} follows from Theorem~\ref{fineenough}.
\end{proof}
%
%Taking $S_1,\ldots,S_k,T = S$, we get the following Corollary:
%\begin{corollary}
%If $S$ is a subset of $\mathbb{Z}$, then for any order-preserving
%map $f~:~S^k \to S$, $\tilde{f}~:~\mathcal{W}_S^k \to \mathcal{W}_S$
%restricts to a map from $S$-valued i-games to $S$-valued i-games.  Moreover,
%modulo $\approx$, the action of $\tilde{f}$ on i-games
%determines its action on arbitrary games, as
%\[ \tilde{f}(G_1,\ldots,G_k)^+
%\approx \tilde{f}(G_1^+,\ldots,G_k^+)\]
%and
%\[ \tilde{f}(G_1,\ldots,G_k)^-
%\approx \tilde{f}(G_1^-,\ldots,G_k^-)\]
%\end{corollary}
%TODO: you wanted even i-games.

So at this point, we know that $\approx$ does a good enough job
of classifying $n$-valued games, and much of the theory for addition
and negation carries over to this case.  However we can possibly do better,
in specific cases, by considering the coarser relation of indistinguishability.

\section{The faithful representation revisited: n = 2 or 3}
Before examining the indistinguishability quotient in the cases of $\wedge$, $\vee$,
$\oplus_n$ and $\odot_n$, we return to the map $\psi : I_{-2} \to \mathcal{G}$.

\begin{theorem}
If $G$ is an $n$-valued i-game, then $G \in I_{-n+1}$.
\end{theorem}
\begin{proof}
In other words, whenever $G$ is an $n$-valued odd-tempered i-game, $\loutcome(G) - \routcome(G) \ge -n + 1$.  This
follows from the fact that $\loutcome(G)$ and $\routcome(G)$ must both lie in the range $n = \{0,1,\ldots,n-1\}$.
\end{proof}
So in particular, if $n = 2$ or 3, then $G$ is in $I_{-2}$, the domain of $\psi$, so we can apply
the map $\psi$ to $G$.  Thus if $G$ and $H$ are two $2$- or $3$-valued i-games, then $G \approx H$
if and only if
\[ \psi(G) = \psi(H) \text{ and $G$ and $H$ have the same parity.}\]

In fact, in these specific cases we can do better, and work with non-i-games, taking sides and representing
them in $\mathcal{G}$ in one fell swoop:
\begin{definition}
If $G$ is a $3$-valued game, then we recursively define $\psi^+(G)$ to be
\[ \psi^+(n) = n\]
when $n = 0,1,2$, and otherwise
\[ \psi^+(G) = \text{``}\{\psi^+(G^L)|\psi^+(G^R) \}\text{''},\]
where ``$\{H^L|H^R\}$'' is $\{H^L|H^R\}$ unless there is more than
one integer $x$ satisfying $H^L \lhd x \lhd H^R$ for every $H^L$ and
$H^R$, in which case we take the \emph{largest} such $x$.

Similarly, we define $\psi^-(G)$ to be
\[ \psi^-(n) = n\]
when $n = 0,1,2$, and otherwise
\[ \psi^-(G) = \text{,,}\{\psi^-(G^L)|\psi^-(G^R)\}\text{,,},\]
where ,,$\{H^L|H^R\}$,, is $\{H^L|H^R\}$ unless there is more than
one integer $x$ satisfying $H^L \lhd x \lhd H^R$ for every $H^L$ and
$H^R$, in which case we take the \emph{smallest} such $x$.
\end{definition}
As an example of funny brackets,
\[ \text{``}\{*|2*\}\text{''} = 2 \ne 0 = \{*|2*\}\]
%The opposite convention of taking the smallest integer between $H^L$ and $H^R$ is actually what we should have
%used to define $\psi^+$, but since we are working with $\{0,1,2\}$ valued games,
%you get the same result as usual brackets, by the simplicity rule, the fact that $0,1,2$ are all positive, and the following
%theorem, which implies that

The point of these functions $\psi^\pm$ is the following:
\begin{theorem}\label{psipm}
If $G$ is a $3$-valued game, then $\psi^+(G) = \psi(G^+)$ and $\psi^-(G) = \psi(G^-)$,
where we take $G^+$ and $G^-$ to be $3$-valued games, as made possible by Theorem~\ref{plethora}(g).
\end{theorem}
\begin{proof}
We prove $\psi^+(G) = \psi(G^+)$; the other equation follows similarly.
Proceed by induction.  If $G$ is one of $0,1,2$, this is obvious.  Otherwise, let
$G = \langle G^L | G^R \rangle $ and let $H = \langle H^L | H^R \rangle$ be a game whose
options are $H^L \approx (G^L)^+$ and $H^R \approx (G^R)^+$, similar to the proof of Theorem~\ref{sides-complete}.
By Theorem~\ref{plethora}(g), we can assume that $H^L,H^R,H$ are $3$-valued games, because $G$ is.
Then $G \approx_+ H$, and in fact $G^+ \approx H$ as long as $H$ is an i-game.  We break into
two cases:

(Case 1) $H$ is an i-game.  Then $G^+ \approx H$, so we want to show that
$\psi^+(G) = \psi(H)$.  By induction, $\psi^+(G^L) = \psi(H^L)$ and
$\psi^+(G^R) = \psi(H^R)$.  Then we want to show the equality of
\[ \psi(H) = \{\psi(H^L)|\psi(H^R)\}\]
and
\[ \psi^+(G) = \text{``}\{\psi^+(G^L)|\psi^+(G^R)\}\text{''}
= \text{``}\{\psi(H^L)|\psi(H^R)\text{''}.\]
So, in light of the simplicity rule, it suffices to show that there 
is a most one integer $n$ with $\psi(H^L) \lhd n \lhd \psi(H^R)$ for all $H^L$
and $H^R$.  Suppose for the sake of contradiction
that
\begin{equation} \psi(H^L) \lhd n \le n + 1 \lhd \psi(H^R)\label{dubious}\end{equation}
 for all $H^L$ and $H^R$.

Now if $H$ is even-tempered, then by Theorem~\ref{subtler} this indicates that
$\routcome(H^L) \le n$ and $n + 1 \le \loutcome(H^R)$ for every $H^L$
and $H^R$.  Thus $\loutcome(H) \le n < n + 1\le \routcome(H)$ contradicting
the assumption that $H$ is an i-game.  Similarly, if $H$ is odd-tempered, then Theorem~\ref{subtler}
translates (\ref{dubious}) into $\routcome(H^L) < n < n + 1 < \loutcome(H^R)$
for every $H^L$ and $H^R$ so
that $\loutcome(H) < n$ and $\routcome(H) > n + 1$.  Thus $\loutcome(H) - \routcome(H) \le 3$,
which is impossible since $H$ is a $3$-valued game.

(Case 2) $H$ is not an i-game. Then $H$ is even-tempered (and thus $G$ is also) and $\loutcome(H) - \routcome(H) < 0$.
Then by Lemma~\ref{gap}, $G \approx_+ H \approx_+ \routcome(H)$.  Since $\routcome(H)$
is a number, it is an i-game and $G^+ \approx \routcome(H)$.

Since $H^R$ is odd-tempered, if $n$ is an integer then $n \lhd \psi(H^R) \iff n \le \loutcome(H^R)$,
by Theorem~\ref{subtler}.
Similarly, $\psi(H^L) \lhd n \iff \routcome(H^L) \le n$.  So an integer $n$ satisfies
\[ \psi(H^L) \lhd n \lhd \psi(H^R) \]
for every $H^L$ and $H^R$ iff $\routcome(H^L) \le n \le \loutcome(H^R)$ for every $H^L$ and $H^R$,
which is the same as saying that $\loutcome(H) \le n \le \routcome(H)$.  Since
$\loutcome(H) - \routcome(H) < 0$, it follows that
\[ \psi^+(G) = \text{``}\{\psi^+(G^L)|\psi^+(G^R)\}\text{''}
= \text{``}\{\psi(H^L)|\psi(H^R)\text{''} = \routcome(H) = \psi(\routcome(H)).\]
Since $G^+ \approx \routcome(H)$, $\psi(G^+) = \psi(\routcome(H))$ and we are done.
\end{proof}
It then follows that a $3$-valued game $G$ is determined up to $\approx$ by its parity, $\psi^+(G)$, and
$\psi^-(G)$.

Also, we see that $\psi^-(G)$ could have been defined using ordinary $\{\cdot|\cdot\}$ brackets
rather than funny ,,$\{\cdot|\cdot\}$,, brackets, since by the simplicity rule, a difference could only arise if
$\psi^-(G)$ equaled an integer $n < 0$, in which case $\psi(G^-) = \psi^-(G) = n$, so that $G^- \approx n$
or $G^- \approx \langle n - 1 | n + 1 \rangle = n + \langle - 1 | 1 \rangle$.  But neither
$n$ nor $\langle n - 1 | n + 1 \rangle$ could equal a $3$-valued game, because both games have
negative left outcome, and the left outcome of a $3$-valued game should be 0, 1, or 2.  Then taking
$G^-$ to be a 3-valued game, we would get a contradiction.

\section{Two-valued games}
If we restrict to $2$-valued games, something nice happens: there are only finitely
many equivalence classes, modulo $\approx$.
\begin{theorem}\label{twoval}
Let $G$ be a $2$-valued game.  If $G$ is even-tempered, then $\psi^-(G)$
is one of the following eight values:
\[ 0,~a = \{\frac{1}{2}|*\},~b = \{\{1|\frac{1}{2}*\}|*\},~c = \frac{1}{2}*,\]
\[ d = \{1*|*\},~e = \{1*|\{\frac{1}{2}*|0\}\},~f = \{1*|\frac{1}{2}\},~1\]
Similarly, if $G$ is odd-tempered, then $\psi^-(G)$ is one of the following
eight values:
\[ *,~a* = \{\frac{1}{2}*|0\},~b* = \{\{1*|\frac{1}{2}\}|0\},~c* = \frac{1}{2},\]
\[ d* = \{1|0\},~e* = \{1|\{\frac{1}{2}|*\}\},~f* = \{1|\frac{1}{2}*\},~1*\]
\end{theorem}
\begin{proof}
Let $S = \{0,a,b,c,d,e,f,1\}$ and $T = \{*,a*,b*,c*,d*,e*,f*,1*\}$.
Because of the recursive definition of $\psi^-$, it suffices to show that
\begin{enumerate}
\item $0,1 \in S$.
\item If $A, B$ are nonempty subsets of $S$, then $\{A|B\} \in T$.
\item If $A, B$ are nonempty subsets of $T$, then $\{A|B\} \in S$.
\end{enumerate}
Because $S$ and $T$ are finite, all of these can be checked by inspection: (1) is obvious,
but (2) and (3) require a little more work.
To make life easier, we can assume that $A$ and $B$ have no dominated moves,
i.e., that $A$ and $B$ are antichains.  Now as posets $S$ and $T$ look like:
\begin{figure}[H]
\begin{center}
\includegraphics[width=2.7in]
					{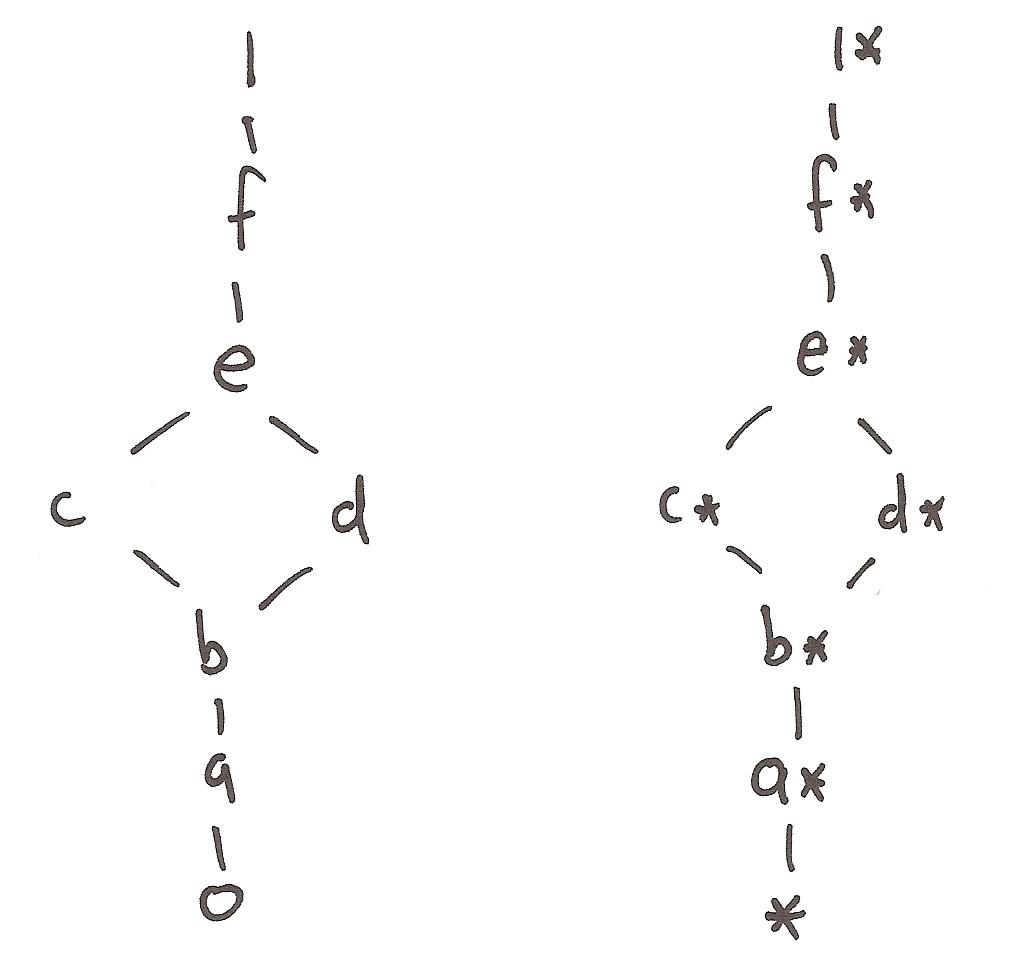}
%\caption{}
%\label{domineering-sum}
\end{center}
\end{figure}
and it is clear that these posets have very few antichains.  In particular,
each of $S$ and $T$ has only nine nonempty antichains.

Using David Wolfe's gamesman's toolkit,
I produced the following tables.  In each table, Left's options
are along the left side and Right's options are along the top.
For even-tempered games:
\begin{center}
\begin{tabular}{l || c|c|c|c|c|c|c|c|c }
A$~\setminus~$B & 0 & a & b & c,d & c & d & e & f & 1 \\
\hline
\hline
0 &\emph{\textbf{*}}&\emph{\textbf{c*}}&c*&c*&c*&c*&c*&c*&c* \\ \hline
a &*&c*&c*&c*&c*&c*&c*&c*&c* \\ \hline
b &*&c*&c*&c*&c*&c*&c*&c*&c* \\ \hline
d &\emph{\textbf{*}}&c*&c*&c*&c*&c*&c*&c*&c* \\ \hline
c &\emph{\textbf{a*}}&c*&c*&c*&c*&c*&c*&c*&c* \\ \hline
c,d &a*&c*&c*&c*&c*&c*&c*&c*&c* \\ \hline
e &\emph{\textbf{a*}}&c*&c*&c*&c*&c*&c*&c*&c* \\ \hline
f &\emph{\textbf{b*}}&c*&c*&c*&c*&c*&c*&c*&\emph{\textbf{c*}} \\ \hline
1 &\emph{\textbf{d*}}&\emph{\textbf{e*}}&\emph{\textbf{f*}}&f*&\emph{\textbf{f*}}&\emph{\textbf{1*}}&1*&1*&\emph{\textbf{1*}} \\
\end{tabular}
\end{center}
and for odd-tempered games:
\begin{center}
\begin{tabular}{l || c|c|c|c|c|c|c|c|c }
A$~\setminus~$B & * & a* & b* & c*,d* & c* & d* & e* & f* & 1* \\
\hline
\hline
* &\emph{\textbf{0}}&0&0&0&0&0&0&0&0 \\ \hline
a* &0&0&0&0&0&0&0&0&0 \\ \hline
b* &0&0&0&0&0&0&0&0&0 \\ \hline
d* &0&0&0&0&0&0&0&0&\emph{\textbf{0}} \\ \hline
c* &\emph{\textbf{a}}&\emph{\textbf{c}}&c&c&c&\emph{\textbf{1}}&1&1&1 \\ \hline
c*,d* &a&c&c&c&c&1&1&1&1 \\ \hline
e* &\emph{\textbf{a}}&c&c&c&c&1&1&1&1 \\ \hline
f* &\emph{\textbf{b}}&c&c&c&\emph{\textbf{c}}&1&1&1&1 \\ \hline
1* &\emph{\textbf{d}}&\emph{\textbf{e}}&\emph{\textbf{f}}&f&\emph{\textbf{f}}&1&1&1&\emph{\textbf{1}} \\
\end{tabular}
\end{center}
In fact, by monotonicity, only the bold entries need to be checked.
\end{proof}

\begin{corollary}
If $G$ is an even-tempered $2$-valued game, then $\psi^+(G)$ and $\psi^-(G)$
are among $S = \{0,a,b,c,d,e,f,1\}$, and if $G$ is an odd-tempered $2$-valued game,
then $\psi^+(G)$ and $\psi^-(G)$ are among $T = \{*,a*,b*,c*,d*,e*,f*,1*\}$.
Moreover, all these values can occur: if $x,y \in S$ or $x,y \in T$ have $x \le y$
then there is a game $G$ with $\psi^-(G) = x$ and $\psi^+(G) = y$.  Modulo
$\approx$, there are exactly sixteen $2$-valued i-games and seventy $2$-valued games.
\end{corollary}
\begin{proof}
All eight values of $\psi^-$ actually occur, because they are (by inspection) built up in a parity-respecting
way from $0$, $1$, $\frac{1}{2} = \{0|1\}$, and $* = \{0|0\}$.  Now if $G$ is an i-game,
then $\psi(G) = \psi(G^-) = \psi^-(G) \in S \cup T$, and so if $H$ is any two-valued game,
then $\psi^+(H) = \psi(H^+) = \psi^-(H^+) \in S \cup T$.  Moreover, $\psi(G)$
and $\psi^+(H)$ will clearly be in $S$ if $G$ or $H$ is even-tempered, and $T$ if odd-tempered.  All pairs
of values occur because of Theorem~\ref{allpairs}.  Since $S$ and $T$ have eight elements,
and an i-game is determined by its image under $\psi$, it follows that there are exactly eight even-tempered i-games
and eight odd-tempered i-games, making sixteen total.  Similarly, by inspecting $S$ and $T$
as posets, we can see that there are exactly 35 pairs $(x,y) \in S \times S$
with $x \le y$.  So there are exactly 35 even-tempered games and similarly 35 odd-tempered games, making 70 in total.
\end{proof}

A couple of things should be noted about the values in $S$ and in $T$.
First of all, $S \cap T = \emptyset$.  It
follows that a $2$-valued game $G$ is determined modulo $\approx$
by $\psi^+(G)$ and $\psi^-(G)$, since they in turn determine the parity of $G$.
Second and more importantly,
by direct calculation one can verify that the values of $S$ are actually
all obtained by Norton multiplication with $\mathbf{1} \equiv \{\frac{1}{2}|\}$:
\[ 0 = 0.\mathbf{1},\quad a = \frac{1}{4}.\mathbf{1},\quad b = \frac{3}{8}.\mathbf{1},\quad c = \frac{1}{2}.\mathbf{1}\]
\[ d = \frac{1}{2}*.\mathbf{1},\quad e = \frac{5}{8}.\mathbf{1},\quad f = \frac{3}{4}.\mathbf{1},\quad 1 = 1.\mathbf{1}\]
So the poset structure of $S$ comes directly from the poset structure of
\[ U = \{0,\frac{1}{4},\frac{3}{8},\frac{1}{2},\frac{1}{2}*,\frac{5}{8},\frac{3}{4},1\}.\]
Similarly, $T$ is just $\{s + *\,:\,s \in S\}$, so $T$ gets its structure in the same way.  %From now on, we let $Q = \{\frac{1}{2}|\}$
%and we refer to $0.\mathbf{1}$, $1/4.\mathbf{1}$, $1/4.\mathbf{1} + *$, and so on, instead of $0,a,a*,\ldots$.

Understanding these values through Norton multiplication makes the structure of $2$-valued games more transparent.
\begin{lemma}\label{halves}
For $G \in \mathcal{G}$, $G.\mathbf{1} \ge * \iff G \ge 1/2$ and similarly $G.\mathbf{1} \le * \iff G \le -1/2$.
\end{lemma}
\begin{proof}
We prove the first claim, noting that the other follows by symmetry.
If $G$ is an integer, then $G.\mathbf{1} = G$, so
$G \ge * \iff G > 0 \iff G \ge \frac{1}{2}$.  Otherwise, by definition of Norton multiplication,
\[ G.\mathbf{1} = \{G^L.\mathbf{1} + \frac{1}{2}|G^R.\mathbf{1} - \frac{1}{2}\}.\]
So $* \le G.\mathbf{1}$ unless and only unless $G.\mathbf{1} \le 0$ or some $G^R.\mathbf{1} - \frac{1}{2} \le *$.  But $\frac{1}{2}* = \frac{1}{2}.\mathbf{1}$,
so $* \le G.\mathbf{1}$ unless and only unless
\[ G.\mathbf{1} \le 0 \text{ or some } G^R.\mathbf{1} \le \frac{1}{2}* = \frac{1}{2}.\mathbf{1}.\]
By basic properties of Norton multiplication, these happen if and only if
\[ G \le 0 \text{ or some } G^R \le \frac{1}{2},\]
which happen if and only if $\frac{1}{2} = \{0|1\} \not \le G$, by Theorem~\ref{betwixt}.  So $* \not \le G \iff \frac{1}{2} \not \le G$.
\end{proof}
Using this we can determine the outcome of every $2$-valued game:
\begin{theorem}
Let $G$ be a $2$-valued game, and let
\[ U = \{0,\frac{1}{4},\frac{3}{8},\frac{1}{2},\frac{1}{2}*,\frac{5}{8},\frac{3}{4},1\}\]
as above.  If $G$ is even-tempered, let $\psi^+(G) = u^+.\mathbf{1}$ and $\psi^-(G) = u^-.\mathbf{1}$,
where $u^+, u^- \in U$.  Then $\routcome(G)$ is the greatest integer $\le u^+$
and $\loutcome(G)$ is the least integer $\ge u^-$.
Similarly, if $G$ is odd-tempered, and $\psi^\pm(G) = u^\pm.\mathbf{1} + *$, where $u^+, u^- \in U$, then
$\routcome(G)$ is the greatest integer $\le u^- + 1/2$ and $\loutcome(G)$ is the
least integer $\ge u^+ - 1/2$.
\end{theorem}
\begin{proof}
When $G$ is even-tempered, Theorem~\ref{plethora}(h) tells us that $\loutcome(G) = \loutcome(G^-)$
and $\routcome(G) = \routcome(G^+)$.  So by Theorem~\ref{subtler},
\[ n \le \routcome(G) \iff n \le \routcome(G^+) \iff n \le \psi(G^+).\]
But by Theorem~\ref{psipm}, $\psi(G^+) = \psi^+(G) = u^+.\mathbf{1}$.  So since $n.\mathbf{1} = n$,
\[ n \le \routcome(G) \iff n \le u^+.\mathbf{1} \iff (n - u^+).\mathbf{1} \le 0 \iff n \le u^+.\]
So $\routcome(G)$ is as stated.  The case of $\loutcome(G)$ is similar.

When $G$ is odd-tempered, Theorem~\ref{plethora}(i) tells us that $\loutcome(G) = \loutcome(G^+)$
and $\routcome(G) = \routcome(G^-)$.  So by Theorem~\ref{subtler},
\[ n < \routcome(G) \iff n < \routcome(G^-) \iff n \le \psi(G^-).\]
But by Theorem~\ref{psipm}, $\psi(G^-) = \psi^-(G) = u^-.\mathbf{1} + *$.  So since $n.\mathbf{1} = n$,
\[ n < \routcome(G) \iff n \le u^-.\mathbf{1} + * \iff (n - u^-).\mathbf{1} \le * \iff n \le u^- - \frac{1}{2}\]
using Lemma~\ref{halves}.  Letting $m = n + 1$ and using the fact that $\routcome(G)$ is an integer,
we see that
\[ m \le \routcome(G) \iff n \le \routcome(G) - 1 \iff n < \routcome(G) \iff m \le u^- + \frac{1}{2}.\]
So $\routcome(G)$ is as stated.  The case of $\loutcome(G)$ is similar.
\end{proof}
In particular then, if $G$ is even-tempered then $\loutcome(G) = 1$ unless $u^- = 0$ and $\routcome(G) = 0$ unless
$u^+ = 1$.  When $G$ is odd-tempered, $\loutcome(G) = 0$ iff $u^+ \le 1/2$, and $\routcome(G) = 1$ iff $u^- \ge 1/2$.

Next, we show how $\wedge$ and $\vee$ act on 2-valued games.
\begin{lemma}
If $x,y \in U$, then there is a maximum element $z \in U$
such that $z \le x + y$.
\end{lemma}
\begin{proof}
The set $U$ is almost totally ordered, with $1/2$ and $1/2*$ its only pair of incomparable elements.
So the only possible problem would occur
if $1/2$ and $1/2*$ are both $\le x + y$, but $5/8$ is not.  However, every number of the form
$x + y$ must be of the form $n.\frac{1}{8}$ or $n.\frac{1}{8} + *$ for some integer $n$.
Then $n.\frac{1}{8} \ge 1/2*$ implies that $n > 4$, so that $5/8$ is indeed $\le n.\frac{1}{8}$.
Similarly, $n.\frac{1}{8} + * \ge 1/2$ implies that $n > 4$, so again $5/8 \le n.\frac{1}{8}$.
\end{proof}
\begin{theorem}
If $G_1$ and $G_2$ are even-tempered $2$-valued i-games, with $\psi(G_i) = u_i.\mathbf{1}$,
then $\psi(G_1 \vee G_2) = v.\mathbf{1}$, where $v$ is the greatest element of $U$ that
is less than or equal to $u_1 + u_2$.
\end{theorem}
In other words, to $\vee$ two games together, we add their $u$ values and round down.
\begin{proof}
By Theorem~\ref{generaldistortions}, $G_1 \vee G_2$ is another i-game, clearly even-tempered.
So $\psi(G_1 \vee G_2) = u_3.\mathbf{1}$ for some $u_3 \in U$.  Let $H$ be an even-tempered $2$-valued i-game
with $\psi(H) = v.\mathbf{1}$, with $v$ as in the theorem statement.   Then clearly
\[ \psi(G_1 + G_2) = \psi(G_1) + \psi(G_2) = (u_1 + u_2).\mathbf{1} \ge v.\mathbf{1} = \psi(H),\]
so that $H \lesssim G_1 + G_2$.

Now let $\mu~:~\mathbb{Z} \to \mathbb{Z}$ be the function
$n \to \min(n,1)$.  Then $\tilde{\mu}(G_1 + G_2) = G_1 \vee G_2$.  So by Theorem~\ref{fineenough}
\[ H = \tilde{\mu}(H) \lesssim \tilde{\mu}(G_1 + G_2) = G_1 \vee G_2,\]
so that $H \lesssim G_1 \vee G_2$.  Therefore
\[ v.\mathbf{1} = \psi(H) \le \psi(G_1 \vee G_2) = u_3.\mathbf{1} ,\]
so $v \le u_3$.

On the other hand, $G_1 \vee G_2 \lesssim G_1 + G_2$ by Lemma~\ref{boundingfunctions},
so
\[ u_3.\mathbf{1} = \psi(G_1 \vee G_2) \le \psi(G_1 + G_2) = (u_1 + u_2).\mathbf{1}\]
and thus $u_3 \le u_1 + u_2$.  By choice of $v$, it follows that $u_3 \le v$,
so $u_3 = v$, and $\psi(G_1 \vee G_2) = u_3.\mathbf{1} = v.\mathbf{1}$.
\end{proof}

So we can describe the general structure of $2$-valued games under $\vee$ as follows:
\begin{definition}
If $G$ is a $2$-valued game, let $u^+(G)$ and $u^-(G)$ be the values $u^+$ and
$u^-$ such that
\[ \psi^+(G) = u^+.\mathbf{1} \text{ and } \psi^-(G) = u^-.\mathbf{1}\] if $G$ is even-tempered,
and \[\psi^+(G) = u^+.\mathbf{1} + * \text{ and }\psi^-(G) = u^-.\mathbf{1} + *\] if $G$ is odd-tempered.

If $x, y$ are elements of $U$, we let $x \cup y$ be the greatest element
of $U$ that is less than or equal to $x + y$, and we let
$x \cap y$ be the least element of $U$ that is greater than or equal to $x + y - 1$
(which exists by symmetry).

If $x$ is an element of $\mathcal{G}$, we let $\lceil x \rceil$ be the least integer $n$
with $n \ge x$ and $\lfloor x \rfloor$ be the greatest integer $n$ with $n \le x$.
\end{definition}
We now summarize our results for two-valued games, mixing in the results of Section~\ref{sec:presigame}.
\begin{corollary}\label{booleansummary}
If $G$ and $H$ are $2$-valued games, then $G \approx H$ iff
$u^+(G) = u^+(H)$, $u^-(G) = u^-(H)$, and $G$ and $H$ have the same parity.
For any $G$, $u^-(G) \le u^+(G)$, and all such pairs $(u_1,u_2) \in U^2$
with $u_1 \le u_2$ occur, in both parities.

When $G$ is even-tempered, $\loutcome(G) = \lceil u^-(G) \rceil$ and
$\routcome(G) = \lfloor u^+(G) \rfloor$.  Similarly, if $G$ is odd-tempered,
then
\[ \loutcome(G) = \left\lceil u^+(G) - \frac{1}{2} \right\rceil\]
\[ \routcome(G) = \left\lfloor u^-(G) + \frac{1}{2} \right\rfloor.\]
Moreover,
\[ u^+(G \vee H) = u^+(G) \cup u^+(H)\]
\[ u^-(G \vee H) = u^-(G) \cup u^-(H)\]
\[ u^+(G \wedge H) = u^+(G) \cap u^+(H)\]
\[ u^-(G \wedge H) = u^-(G) \cap u^-(H).\]
\end{corollary}

\begin{figure}[H]
\begin{center}
\begin{tabular}{|c||c|c|c|c|c|c|c|c|}
\hline
$\cup$ & 0 & 1/4 & 3/8 & 1/2 & $1/2*$ & 5/8 & 3/4 & 1 \\
\hline
\hline
0 & 0 & 1/4 & 3/8 & 1/2 & $1/2*$ & 5/8 & 3/4 & 1 \\
\hline
1/4 & 1/4 & 1/2 & 5/8 & 3/4 & 5/8 & 3/4 & 1 & 1\\
\hline
3/8 & 3/8 & 5/8 & 3/4 & 3/4 & 3/4 & 1 & 1 & 1\\
\hline
1/2 & 1/2 & 3/4 & 3/4 & 1 & 3/4 & 1 & 1 & 1\\
\hline
$1/2*$ & $1/2*$ & 5/8 & 3/4 & 3/4 & 1 & 1 & 1 & 1\\
\hline
5/8 & 5/8 & 3/4 & 1 & 1 & 1 & 1 & 1 & 1\\
\hline
3/4 & 3/4 & 1 & 1 & 1 & 1 & 1 & 1 & 1\\
\hline
1 & 1 & 1 & 1 & 1 & 1 & 1 & 1 & 1\\
\hline
\end{tabular}
\\
$~$
\\
$~$
\\
\begin{tabular}{|c||c|c|c|c|c|c|c|c|}
\hline
$\cap$ & 0 & 1/4 & 3/8 & 1/2 & $1/2*$ & 5/8 & 3/4 & 1 \\
\hline
\hline
0 & 0 & 0 & 0 & 0 & 0 & 0 & 0 & 0 \\
\hline
1/4 & 0 & 0 & 0 & 0 & 0 & 0 & 0 & 1/4\\
\hline
3/8 & 0 & 0 & 0 & 0 & 0 & 0 & 1/4 & 3/8\\
\hline
1/2 & 0 & 0 & 0 & 0 & 1/4 & 1/4 & 1/4 & 1/2\\
\hline
$1/2*$ & 0 & 0 & 0 & 1/4 & 0 & 1/4 & 3/8 & $1/2*$\\
\hline
5/8 & 0 & 0 & 0 & 1/4 & 1/4 & 1/4 & 3/8 & 5/8\\
\hline
3/4 & 0 & 0 & 1/4 & 1/4 & 3/8 & 3/8 & 1/2 & 3/4\\
\hline
1 & 0 & 1/4 & 3/8 & 1/2 & $1/2*$ & 5/8 & 3/4 & 1\\
\hline
\end{tabular}
\caption{the $\cup$ and $\cap$ operations. Compare the table for $\cup$ with Figure~\ref{old-school}}
\label{cap-cup-tables}
\end{center}
\end{figure}

\section{Three-valued games}
Unlike two-valued games, there are infinitely many $3$-valued games, modulo $\approx$.
In fact, there is a complete copy of $\mathcal{G}$ in $\mathcal{W}_3$ modulo $\approx$.
\begin{lemma}
If $\epsilon$ is an all-small partizan game, then $1 + \epsilon = \psi(G)$ for
some $3$-valued i-game $G$.
\end{lemma}
\begin{proof}
By Theorem~\ref{psipm} and part (g) of Theorem~\ref{plethora}, it suffices to show that there is some $3$-valued game
$G$ with $\psi^-(G) = 1 + \epsilon$.  In fact we show that $G$ can be taken to be
both odd-tempered or even-tempered, by induction on $\epsilon$.  We take $\epsilon$ to be
all-small in form, meaning that every one of its positions $\epsilon'$ has options for both players or for neither.

If $\epsilon = 0$, then we can take $G$ to be either the even-tempered game $1$
or the odd-tempered game $\langle 0 | 2 \rangle$, since $\psi^-(1) = 1$ and
\[ \psi^-(\langle 0 | 2 \rangle) = \{0|2\} = 1.\]

Otherwise, $\epsilon = \{\epsilon^L|\epsilon^R\}$ and at least one
$\epsilon^L$ and at least one $\epsilon^R$ exist.  By number avoidance, $1 + \epsilon = \{1 + \epsilon^L| 1 + \epsilon^R\}$.
By induction, there are
odd-tempered $3$-valued games $G^L$ and $G^R$ with $\psi^-(G^L) = 1 + \epsilon^L$
and $\psi^-(G^R) = 1 + \epsilon^R$.  So $G = \langle G^L | G^R \rangle$ is an even-tempered $3$-valued game,
and has
\[ \psi^-(G) = \{\psi^-(G^L)|\psi^-(G^R)\} = \{1+\epsilon^L | 1+\epsilon^R \} = \epsilon.\]
Similarly, there are
even-tempered $3$-valued games $H^L$ and $H^R$ with $\psi^-(H^L) = 1+\epsilon^L$
and $\psi^-(H^R) = 1+\epsilon^R$.  So $H = \langle H^L | H^R \rangle$ is an odd-tempered $3$-valued game,
and has
\[ \psi^-(H) = \{\psi^-(H^L)|\psi^-(H^R)\} = \{1+\epsilon^L | 1+\epsilon^R \} =1+ \epsilon.\]
\end{proof}
By Corollary~\ref{hats}, $G.\uparrow$ is an all-small game
for every $G \in \mathcal{G}$, so the following definition makes sense:
\begin{definition}
For every $G \in \mathcal{G}$, let $\phi(G)$ be a $3$-valued even-tempered i-game $H$
satisfying $\psi(H) = 1 + G.\uparrow$.
\end{definition}
Note that $\phi(G)$ is only defined up to $\approx$.

The following result shows how much more complicated $3$-valued games are than $2$-valued games.
\begin{theorem}\label{triplecomplication}
For any $G \in \mathcal{G}$,
\[ \routcome(\phi(G)) \ge 1 \iff G \ge 0\]
and
\[ \loutcome(\phi(G)) \le 1 \iff G \le 0.\]
Moreover, if $G$ and $H$ are in $\mathcal{G}$, then
\begin{equation} \phi(G + H) \approx \phi(G) + \phi(H) - 1.\label{hom1}\end{equation}
Let $\star~:~\mathcal{W}_3 \times \mathcal{W}_3 \to \mathcal{W}_3$ be the extension of the operation
$(x,y) \to \max(0,\min(2,x+y-1))$ (see Figure~\ref{triplestar}).  Then we also have
\begin{equation} \phi(G + H) \approx \phi(G) \star \phi(H)\label{hom2}\end{equation}
\end{theorem}
This shows that if we look at $\mathcal{W}_3$ modulo $\star$-indistinguishability, it contains a complete
copy of $\mathcal{G}$.
\begin{proof}
Since $\phi(G)$ is even-tempered, Theorem~\ref{subtler} implies that
\[ 1 \le \routcome(\phi(G)) \iff 1 \le \psi(\phi(G)) = 1 + G.\uparrow \iff G \ge 0\]
and similarly,
\[ 1 \ge \loutcome(\phi(G)) \iff 1 \ge \psi(\phi(G)) = 1 + G.\uparrow \iff G \le 0,\]
where in both cases we use the fact that $G.\uparrow$ has the same sign as $G$.

To see (\ref{hom1}), note that
\[ \psi(\phi(G + H)) = 1 + (G + H).\uparrow = 1 + G.\uparrow + 1 + H.\uparrow - 1 =\]\[ \psi(\phi(G)) + \psi(\phi(H)) + \psi(-1)
= \psi(\phi(G) + \phi(H) - 1),\]
so
\[ \phi(G + H) \approx \phi(G) + \phi(H) - 1\] because both sides are even-tempered.  Finally, to see (\ref{hom2}), let $q:\mathbb{Z} \to \{0,1,2\}$ be the map
$n \to \max(0,\min(2,n))$.  Then by Theorem~\ref{fineenough},
\[ \phi(G) \star \phi(H) = \tilde{q}(\phi(G) + \phi(H) - 1) \approx \tilde{q}(\phi(G + H)).\]
But since $q$ acts as the identity on $\{0,1,2\}$ and $\phi(G + H) \in \mathcal{W}_3$, $\tilde{q}(\phi(G + H)) = \phi(G + H)$,
establishing (\ref{hom2}).
\end{proof}

\begin{figure}[htb]
\begin{center}
\begin{tabular}{|c||c|c|c|}
\hline
$\oplus_3$ & 0 & 1 & 2 \\
\hline
\hline
0 & 0 & 0 & 1 \\
\hline
1 & 0 & 1 & 2 \\
\hline
2 & 1 & 2 & 2 \\
\hline
\end{tabular}
\caption{The operation $\star$ of Theorem~\ref{triplecomplication}.}
\label{triplestar}
\end{center}
\end{figure}

Since we can embed $3$-valued games in $n$-valued games in an obvious way, these results
also show that $n$-valued games modulo $\approx$ are complicated.

\section{Indistinguishability for rounded sums}\label{sect:roundsum}
In this section and the next, we examine the structure of $n$-valued games modulo certain types of indistinguishability.
We specifically consider the following kinds of indistinguishability:
\begin{itemize}
\item $\{\oplus_n,\odot_n\}$-indistinguishability, which we show is merely $\approx$.
\item $\{\oplus_n\}$-indistinguishability (and similarly $\{\odot_n\}$ indistinguishability)
which turns out to be slightly coarser.
\item $\{\wedge,\vee\}$- and $\{\vee\}$-indistinguishability, which turn out to have only finitely many
equivalence classes for every $n$, coming from the finitely many classes of $2$-valued games.
\end{itemize}
In a previous section we showed that for all these operations, indistinguishability is as coarse as $\approx$,
in the sense that whenever $G \approx H$, then
$G$ and $H$ are indistinguishable with respect to all these operations.  We begin
by showing that for $\{\oplus_n, \odot_n\}$, indistinguishability is $\approx$ exactly.

\begin{theorem}\label{doubledisting}
Suppose $n > 1$.
Let $G$ and $H$ be $n$-valued games, and $G \not \approx H$.  Then there is some
$n$-valued game $X$ such that $\outcome(G \odot_n X) \ne \outcome(H \odot_n X)$
or $\outcome(G \oplus_n X) \ne \outcome(H \oplus_n X)$.
\end{theorem}
\begin{proof}
We break into cases according to whether $G$ and $H$ have the same or opposite parity.
First of all suppose that $G$ and $H$ have opposite parity.  Say $G$ is odd-tempered
and $H$ is even-tempered.  Let $\mu$ be the map $\mu(x) = \min(x,n-1)$ and
$\nu(x) = \max(x - (n-1), 0)$, and let $Q$ be the even-tempered $n$-valued game $\langle*|(n-1)*\rangle$,
which has $Q^+ \approx n - 1$ and $Q^- \approx 0$.  Then using Theorem~\ref{plethora}(h-i)
and Lemma~\ref{outcomeslide}, we have
\[ \loutcome(G + Q) = \loutcome(G^+ + Q^+) = \loutcome(G^+ + (n-1)) = \loutcome(G) + (n - 1) \ge n - 1.\]
Thus
\[ \loutcome(G \oplus_n Q) = \loutcome(\tilde{\mu}(G + Q)) = \mu(\loutcome(G + Q)) = n - 1,\]
and
\[ \loutcome(G \odot_n Q) = \loutcome(\tilde{\nu}(G + Q)) = \nu(\loutcome(G + Q)) = \loutcome(G) + (n- 1) - (n - 1) = \loutcome(G).\]

Similarly,
\[ \loutcome(H + Q) = \loutcome(H^- + Q^-) = \loutcome(H^-) = \loutcome(H) \le n - 1,\]
so that
\[ \loutcome(H \oplus_n Q) = \loutcome(\tilde{\mu}(H + Q)) = \mu(\loutcome(H + Q)) = \loutcome(H),\]
and
\[ \loutcome(H \odot_n Q) = \loutcome(\tilde{\nu}(H + Q)) = \nu(\loutcome(H + Q)) = 0.\]

Then taking $X = Q$, we are done unless
\[ \loutcome(H) = \loutcome(H \oplus_n Q) = \loutcome(G \oplus_n Q) = n - 1\]
\[ \loutcome(G) = \loutcome(G \odot_n Q) = \loutcome(H \odot_n Q) = 0.\]
But then,
\[ \loutcome(H \oplus_n 0) = \loutcome(H) \ne \loutcome(G) = \loutcome(G \oplus_n 0),\]
so we can take $X = 0$ and be done.

Now suppose that $G$ and $H$ have the same parity.
Since $G \not \approx H$, it must be the case that $G^- \not \approx H^-$ or
$G^+ \not \approx H^+$.  Suppose that $G^- \not \approx H^-$.  Without loss of generality, $G^- \not \lesssim H^-$.
By Theorem~\ref{plethora}(g) we can assume that $G^-$ and $H^-$
are also $n$-valued games.  Because they are i-games, it follows from
Corollary~\ref{restatement} that $\loutcome(G^- - H^-) > 0$.  Then by Theorem~\ref{plethora}(h),
\[ \loutcome(G - H^-) = \loutcome((G - H^-)^-) = \loutcome(G^- - H^-) > 0,\]
since $G$, $H$, $G^-$, and $H^-$ all have the same parity.  (Note that $(H^-)^+ \approx H^-$.)
On the other hand,
\[\loutcome(H - H^-) = \loutcome((H - H^-)^-) = \loutcome(H^- - H^-) = \loutcome(0) = 0.\]
Now let $X$ be the game $n - 1 - H^-$.  It follows that $\loutcome(G + X) > n - 1$
and $\loutcome(H + X) = n - 1$.  Letting $\delta$ be the map $x \to \max(x - (n-1),0)$,
we see that
\[ \loutcome(G \odot_n X) = \loutcome(\tilde{\delta}(G + X)) = \delta(\loutcome(G + X)) = \loutcome(G + X) - (n - 1) > 0,\]
while
\[ \loutcome(H \odot_n X) = \loutcome(\tilde{\delta}(H + X)) = \delta(\loutcome(H + X)) = \delta(n-1) = 0.\]
So $\outcome(G \odot_n X) \ne \outcome(H \odot_n X)$.

If we had $G^+ \not\approx H^+$ instead, a similar argument would produce $X$ such that
$\outcome(G \oplus_n X) \ne \outcome(H \oplus_n X)$.
\end{proof}
\begin{corollary}
Indistinguishability with respect to $\{\oplus_n, \odot_n\}$ is exactly $\approx$.
\end{corollary}
\begin{proof}
Let $\sim$ be $\{\oplus_n,\odot_n\}$-indistinguishability.  Then we already know
that $G \approx H \implies G \sim H$.  Conversely, suppose $G \sim H$.  Then by definition
of indistinguishability,
\[ G \odot_n X \sim H \odot_n X \text{ and so } \outcome(G \odot_n X) = \outcome(H \odot_n X)\]
\[ G \oplus_n X \sim H \oplus_n X \text{ and so } \outcome(G \oplus_n X) = \outcome(H \oplus_n X)\]
so that by the theorem, $G \approx H$.
\end{proof}

So if we look at $2$-valued games modulo $\{\oplus_2,\odot_2\}$-indistinguishability,
there are exactly 70 of them, but if we looked at $3$-valued games instead, there are infinitely
many, in a complicated structure.

The situation for $\{\oplus_n\}$-indistinguishability of $n$-valued games is a little bit more
complicated than $\{\oplus_n,\odot_n\}$-indistinguishability, because indistinguishability
turns out to be a little coarser.  But at least we have a simpler criterion:
\begin{lemma}\label{associndist2}
If $G$ and $H$ are $n$-valued games, and $\sim$ denotes $\{\oplus_n\}$-indistinguishability,
then $G \sim H$ iff $\forall X \in \mathcal{W}_n:\outcome(G \oplus_n X) = \outcome(H \oplus_n X)$.
\end{lemma}
\begin{proof}
This was Theorem~\ref{associndist}.%
%First of all, if $G \sim H$, then $G \oplus_n X \sim H \oplus_N X$, so certainly
%$\outcome(G \oplus_n X) = \outcome(H \oplus_n X)$.
%
%Conversely, let $\sim'$ be the relationship $\forall X \in \mathcal{W}_n:\outcome(G \oplus_n X) = \outcome(H \oplus_n X)$.
%Then since $G \oplus_n 0 = G$ for all $G$, taking $X = 0$ we see that $G \sim' H \implies \outcome(G) = \outcome(H)$.
%And similarly, if $G_1 \sim' G_2$ and $H_1 \sim' H_2$, then for any $X$, we have
%\[ \outcome((G_1 \oplus_n H_1) \oplus_n X) = \outcome(G_1 \oplus_n (H_1 \oplus_n X)) = 
%\outcome(G_2 \oplus_n (H_1 \oplus_n X)) =\]\[ \outcome(H_1 \oplus_n (G_2 \oplus_n X))
%= \outcome(H_2 \oplus_n (G_2 \oplus_n X)) = \outcome((G_2 \oplus_n H_2) \oplus_n X),\]
%so that $G_1 + H_1 \sim' G_2 + H_2$.  Then $\sim'$ satisfies requirements (a) and (b) of Theorem~\ref{indist},
%so $\sim' \subseteq \sim$.  Thus whenever $\outcome(G \oplus_n X) = \outcome(H \oplus_n X)$ for all $X \in \mathcal{W}_n$,
%$G \sim  H$.
\end{proof}
The same proof works if we replaced $\oplus_n$ with any commutative and associative operation with an identity.  We'll use this
same fact later for $\wedge$ and $\vee$.

To determine $\{\oplus_n\}$-indistinguishability, we'll need a few more lemmas:
\begin{lemma}\label{eventest}
Let $\mathbb{N}$ denote the nonnegative integers.  Then for any $\mathbb{N}$-valued even-tempered game $G$,
\[ m \le \loutcome(G) \iff \langle \langle 0 | m \rangle | * \rangle \lesssim G.\]
\end{lemma}
\begin{proof}
One direction is obvious: if $\langle \langle 0 | m \rangle | * \rangle \lesssim G$,
then 
\[ m = \loutcome(\langle \langle 0 | m \rangle | * \rangle) \le \loutcome(G).\]
Conversely, suppose that $m \le \loutcome(G) = \loutcome(G^-)$, where we can take $G^-$ to be $\mathbb{N}$-valued.
I claim that
\[ \routcome(\langle * | \langle -m | 0 \rangle \rangle + G^-) \ge 0.\]
Since we took $G^-$ to be $\mathbb{N}$-valued, the only way that the outcome
can fail to be $\ge 0$ is if the outcome of the $\langle * | \langle -m | 0 \rangle \rangle$ component is
$-m$.  So in the sum $\langle * | \langle -m | 0 \rangle \rangle + G^-$, with Right moving first,
Left can move to $*$ at the first available moment and guarantee an outcome of at least 0,
unless Right moves to $\langle -m | 0 \rangle$ on his first turn.  But if Right moves
to $\langle -m | 0 \rangle$ on the first move, then Left can use her first-player strategy in $G^-$ to ensure
that the final outcome of $G^-$ is at least $m$, guaranteeing a final score for the sum of at least $0$.
This works as long as Right doesn't ever move in the $\langle -m | 0 \rangle$ component to $0$.  But if he did that,
then the final score would automatically be at least $0$, because $G^-$ is $\mathbb{N}$-valued.

So $\routcome(\langle * | \langle -m | 0 \rangle \rangle + G^-) \ge 0$.  But
note that $Q = \langle \langle m | 0 \rangle | *\rangle$ is an even-tempered i-game, and we just showed that $\routcome(-Q + G^-) \ge 0$.
By Theorem~\ref{plethora}, it follows that
\[ 0 \lesssim -Q + G^-,\]
so that $Q \lesssim G^-$, because i-games are invertible.  But then $Q \lesssim G^- \lesssim G$, so we are done.
\end{proof}

Similarly, we have
\begin{lemma}\label{oddtest}
For any $\mathbb{N}$-valued odd-tempered game $G$,
\[ m \le \routcome(G) \iff \langle 0 | m \rangle \lesssim G.\]
\end{lemma}
\begin{proof}
Again, one direction is easy: if $\langle 0 | m \rangle \lesssim G$, then
\[ m = \routcome(\langle 0 | m \rangle) \lesssim \routcome(G).\]

Conversely, suppose that $m \le \routcome(G) = \routcome(G^-)$.
Take a $G^-$ which is $\mathbb{N}$-valued (possible by Theorem~\ref{plethora}(g)).
I claim that
\[ \routcome(G^- + \langle -m | 0 \rangle) \ge 0\]
By the same argument as in the previous lemma, Left can use her strategy in $G^-$ to ensure that the final score of $G^-$
is at least $m$, unless Right moves prematurely in $\langle -m | 0 \rangle$ to $0$, in which case Left automatically gets
a final score of at least 0, becaues $G^-$ is $\mathbb{N}$-valued.

Again, if $Q = \langle 0 | m \rangle$, then $Q$ is an odd-tempered i-game and we just showed that $\routcome(G^- - Q) \ge 0$.
So using Theorem~\ref{plethora}, and the fact that $G^- - Q$ is an even-tempered i-game,
\[ 0 \lesssim G^- - Q,\]
so that
\[ Q \lesssim G^- \lesssim G.\]
\end{proof}

\begin{lemma}\label{calculations}
For $m > 0$, let $Q_m = \langle \langle 0 | m \rangle | * \rangle$.  Then
\[ Q_m + Q_m \approx \langle \langle 0 | m \rangle | \langle 0 | m \rangle \rangle \approx \langle 0 | m \rangle + *\]
\[ Q_m + Q_m + Q_m \approx \langle m* | \langle 0 | m \rangle \rangle \]
\[ Q_m + Q_m + Q_m + Q_m \approx m\]
\end{lemma}
\begin{proof}
If $m = 1$, all these results follow by direct computation, using the map $\psi$ and basic properties
of Norton multiplication
\[ \psi(\langle \langle 0 | 1 \rangle | * \rangle) = \{\{0|1\}|*\} = \{\frac{1}{2}|*\} = \frac{1}{4}.\{\frac{1}{2}|\}\]
\[\psi(\langle \langle 0 | 1 \rangle | \langle 0 | 1 \rangle \rangle)
= \{\{0|1\}|\{0|1\}\} = \{\frac{1}{2}|\frac{1}{2}\} = \frac{1}{2}* = \frac{1}{2}.\{\frac{1}{2}|\}\]
\[ \psi(\langle 0 | 1 \rangle + *) = \psi(\langle 0 | 1 \rangle) + \psi(*) = \{0|1\} + * = \frac{1}{2}* = \frac{1}{2}.\{\frac{1}{2}|\}\]
\[ \psi(\langle 1* | \langle 0 | 1 \rangle \rangle) = \{1*|\{0|1\}\} = \{1*|\frac{1}{2}\} = \frac{3}{4}.\{\frac{1}{2}|\}\]
\[ \psi(1) = 1 = \frac{4}{4}.\{\frac{1}{2}|\}\]

For $m > 1$, let $\mu$ be the order-preserving map of multiplication by $m$.  Then $Q_m = \tilde{\mu}(Q_1)$, and the fact
that $\mu(x + y) = \mu(x) + \mu(y)$ for $x,y \in \mathbb{Z}$ implies that $\tilde{\mu}(G + H) = \tilde{\mu}(G) + \tilde{\mu}(H)$
for $\mathbb{Z}$-valued games $G$ and $H$.  So
\[ Q_m + Q_m = \tilde{\mu}(Q_1) + \tilde{\mu}(Q_1) = \tilde{\mu}(Q_1 + Q_1) \approx \tilde{\mu}(\langle \langle 0 | 1 \rangle |
\langle 0 | 1 \rangle \rangle) = \langle \langle 0 | m \rangle | \langle 0 | m \rangle \rangle\]
and
\[ \tilde{\mu}(\langle \langle 0 | 1 \rangle | \langle 0 | 1 \rangle \rangle)
\approx \tilde{\mu}(\langle 0 | 1 \rangle + *) = \tilde{\mu}(\langle 0 | 1\rangle) + \tilde{\mu}(*) = \langle 0 | m \rangle + *.\]
The other cases are handled analogously.
\end{proof}

\begin{lemma}\label{minimalcompare}
Let $\mu\,:\, \mathbb{Z} \to \mathbb{N}$ be the map $\mu(x) = \max(0,x)$.  Then for
any $\mathbb{Z}$-valued game $X$ and any $\mathbb{N}$-valued game $Y$,
\[ X \lesssim Y \iff \tilde{\mu}(X) \lesssim Y.\]
\end{lemma}
\begin{proof}
By Lemma~\ref{boundingfunctions} (applied to the fact that $x \le \mu(x)$ for all $x$), $X \lesssim \tilde{\mu}(X)$, so the $\Leftarrow$ direction is obvious.
Conversely, suppose that $X \lesssim Y$.  Then by Theorem~\ref{fineenough},
\[ \tilde{\mu}(X) \lesssim \tilde{\mu}(Y) = Y.\]
\end{proof}

\begin{lemma}\label{zerosmallest}
If $G$ is an $n$-valued even-tempered game, then $0 \lesssim G$.
\end{lemma}
\begin{proof}
By Theorem~\ref{plethora}, we can take $G^+$ and $G^-$ to be $n$-valued games.  Then
$\routcome(G+), \routcome(G^-) \in n = \{0,\ldots,n-1\}$, so that $0 \le \routcome(G^+)$ and
$0 \le \routcome(G^-)$.  By another part of Theorem~\ref{plethora}, it follows
that $0 \lesssim G^+$ and $0 \lesssim G^-$, so therefore $0 \lesssim G$.
\end{proof}

\begin{theorem}\label{sameparrf}
Let $G$ and $H$ be $n$-valued games of the same parity.  Let $Q = \langle \langle 0 | n - 1 \rangle | * \rangle$,
and let $\mu$ be the map $\mu(x) = \max(0,x)$ from Lemma~\ref{minimalcompare}.  Then the following statements are equivalent:
\begin{description}
\item[(a)] For every $n$-valued game $X$,
\[ \rfout(G \oplus_n X) \le \rfout(H \oplus_n X)\]
\item[(b)]
For every $n$-valued game $X$,
\[ \rfout(G + X) \ge n - 1 \implies \rfout(H + X) \ge n - 1.\]
\item[(c)] For every $n$-valued i-game $Y$, if $G + Y$ is even-tempered then
\[ \loutcome(G^- + Y) \ge n - 1 \implies \loutcome(H^- + Y) \ge n - 1,\]
and if $G + Y$ is odd-tempered, then
\[ \routcome(G^- + Y) \ge n - 1 \implies \routcome(H^- + Y) \ge n - 1.\]
\item[(d)] For every $n$-valued i-game $Y$,
\[ \langle 0 | n - 1 \rangle \lesssim G^- + Y \implies \langle 0 | n - 1 \rangle \lesssim H^- + Y\]
and
\[ Q \lesssim G^- + Y \implies Q \lesssim H^- + Y.\]
\item[(e)]
\[ \tilde{\mu}(\langle 0 | n - 1 \rangle - H^-) \lesssim \tilde{\mu}(\langle 0 | n- 1 \rangle - G^-)\]
and
\[ \tilde{\mu}(Q - H^-) \lesssim \tilde{\mu}(Q - G^-).\]
\item[(f)] $G^- \oplus_n \langle 0 | n - 1\rangle \lesssim H^- \oplus_n \langle 0 | n - 1 \rangle$.
\end{description}
\end{theorem}
\begin{proof}
Let $\nu$ be the order-preserving map $\nu(x) = \min(x,n-1)$.
\begin{description}
\item[(a) $\Rightarrow$ (b)] Suppose that $(a)$ is true, and $\rfout(G + X) \ge n - 1$.
Then $G \oplus_n X = \tilde{\nu}(G + X)$, so that by Lemma~\ref{outcomeslide},
\[ \rfout(G \oplus_n X) = \nu(\rfout(G + X)) = n - 1.\]
Then by truth of (a), it follows that $\rfout(H \oplus_n X) \ge \rfout(G \oplus_n X) = n - 1$.
So since
\[ n - 1 \le \rfout(H \oplus_n X) = \nu(\rfout(H + X)) = \min(\rfout(H + X),n-1),\]
it must be the case that $\rfout(H + X) \ge n - 1$ too.
\item[(b) $\Rightarrow$ (a)] Suppose that (a) is false, so that
\[ \rfout(G \oplus_n Y) > \rfout(H \oplus_n Y)\]
for some $Y$.  Let $k = (n - 1) - \rfout(G \oplus_n Y)$, so that
\[ \rfout(G \oplus_n Y) + k = n - 1\]
\[ \rfout(H \oplus_n Y) + k < n - 1.\]
Since $G \oplus_n Y$ is an $n$-valued game,
$k \ge 0$.  Then
\[ \rfout(G \oplus_n Y \oplus_n k) = \min(\rfout(G \oplus_n Y) + k, n - 1) = \min(n - 1, n - 1) = n - 1,\]
while
\[ \rfout(H \oplus_n Y \oplus_n k) = \min(\rfout(H \oplus_n Y) + k, n - 1) = \rfout(H \oplus_n Y) + k < n - 1.\]
So letting $X = Y \oplus_n k$, we have
\[ \min(\rfout(G + X), n - 1) = \rfout(G \oplus_n X) = n - 1 >\]\[ \rfout(H \oplus_n X) = \min(\rfout(H + X), n - 1),\]
implying that $\rfout(G + X) \ge n - 1$ and $\rfout(H + X) < n - 1$, so that (b) is false.
\item[(b) $\Leftrightarrow$ (c)] An easy exercise using Theorem~\ref{plethora}(b,d,g,h,i).  Apply part (b) to $Y$,
part (g) to see that $Y$ ranges over the same things as $X^-$, and parts (d,h,i) to see that
\[ \rfout(G + X) = \loutcome(G^- + X^-)\]
when $G + X$ is even-tempered and
\[ \rfout(G + X) = \routcome(G^- + X^-)\]
when $G + X$ is odd-tempered.  And similarly for $\rfout(H + X)$.
\item[(c) $\Leftrightarrow$ (d)] An easy exercise using Lemma~\ref{eventest}, Lemma~\ref{oddtest}, and Theorem~\ref{sameparity}.
\item[(d) $\Leftrightarrow$ (e)] For any $n$-valued game $Y$, by Lemma~\ref{minimalcompare} we have
\[ Q  \lesssim G^- + Y \iff Q - G^- \lesssim Y
\iff \tilde{\mu}(Q - G^-) \lesssim Y,\]
and similarly
\[ Q \lesssim H^- + Y \iff \tilde{\mu}(Q - G^-) \lesssim Y\]
\[ \langle 0 | n - 1 \rangle \lesssim G^- + Y \iff \tilde{\mu}(\langle 0 | n - 1 \rangle - G^-) \lesssim Y\]
\[ \langle 0 | n - 1 \rangle \lesssim H^- + Y \iff \tilde{\mu}(\langle 0 | n - 1 \rangle - H^-) \lesssim Y.\]

Using these, (d) is equivalent to the claim that for every $n$-valued i-game $Y$,
\begin{equation} \tilde{\mu}(Q - G^-) \lesssim Y \Rightarrow \tilde{\mu}(Q - H^-) \lesssim Y\label{s1}\end{equation}
and
\begin{equation} \tilde{\mu}(\langle 0 | n - 1 \rangle - G^-) \lesssim Y \Rightarrow \tilde{\mu}(\langle 0 | n - 1 \rangle - H^-) \lesssim Y.\label{s2}\end{equation}
Then (e) $\Rightarrow$ (d) is obvious.  For the converse, let $Z_1 = \tilde{\mu}(Q - G^-)$ and $Z_2 = \tilde{\mu}(\langle 0 | n - 1 \rangle - G^-)$.
Then $Z_1$ and $Z_2$ are i-games, by Theorem~\ref{generaldistortions} and the fact that $Q$, $G^-$, and $\langle 0 | n - 1 \rangle$
are i-games.  Additionally, $Z_1$ and $Z_2$ are $n$-valued games because $\mu(x - y) \in n$ whenever $x, y \in n$,
and all of $Q$, $G^-$, and $\langle 0 | n - 1 \rangle$ are $n$-valued games.  So if (d) is true, we can substitute
$Z_1$ into (\ref{s1}) and $Z_2$ into (\ref{s2}), yielding (e).
\item[(e) $\Leftrightarrow$ (f)] It is easy to verify that
\[ (n - 1) - \mu(i - j) = \nu((n-1-i) + j) = (n - 1 - i) \oplus_n j\]
for $i,j \in n$.  Consequently
\[ (n- 1) - \tilde{\mu}(Q - H^-) = \tilde{\nu}((n-1 - Q) + H^-) = (n-1 - Q) \oplus_n H^-\]
and similarly
\[ (n - 1) - \tilde{\mu}(Q - G^-) = \tilde{\nu}((n- 1 - Q) + G^-) = (n - 1 - Q) \oplus_n G^-\]
\[ (n - 1) - \tilde{\mu}(\langle 0 | n - 1\rangle - H^-) = (n - 1 - \langle 0 | n - 1 \rangle) \oplus_n H^-\]
\[ (n - 1) - \tilde{\mu}(\langle 0 | n - 1\rangle - G^-) = (n - 1 - \langle 0 | n - 1 \rangle) \oplus_n G^-.\]
So (e) is equivalent to
\[ G^- \oplus_n (n - 1 - Q) \lesssim H^- \oplus_n (n - 1 - Q)\]
and
\[ G^- \oplus_n (n - 1 - \langle 0 | n- 1 \rangle) \lesssim H^- \oplus_n (n - 1 - \langle 0 | n-1\rangle).\]

But by Lemma~\ref{calculations}
\[ n - 1 - \langle 0 | n- 1 \rangle = \langle 0 | n - 1 \rangle,\] and
\[ n - 1 - Q = \langle n * | \langle 0 | n \rangle \rangle \approx Q + \langle 0 | n - 1 \rangle.\]
So then
\[ n - 1 - Q = \tilde{\nu}(n - 1 - Q) \approx \tilde{\nu}(Q + \langle 0 | n - 1 \rangle) = Q \oplus_n \langle 0 | n - 1 \rangle.\]
So (e) is even equivalent to
\begin{equation} G^- \oplus_n \langle 0 | n - 1 \rangle \lesssim H^- \oplus_n \langle 0 | n - 1 \rangle\label{st1}\end{equation}
and
\begin{equation} G^- \oplus_n \langle 0 | n - 1 \rangle \oplus_n Q \lesssim H^- \oplus_n \langle 0 | n - 1 \rangle \oplus_n Q.\label{st2}\end{equation}
However (\ref{st2}) obviously follows from (\ref{st1}), so (e) is equivalent to (\ref{st1}), which is (f).
\end{description}
\end{proof}

On a simpler note, we can reuse the proof of the same-parity case of Theorem~\ref{doubledisting} to prove the following:
\begin{theorem}\label{sameparlf}
If $G$ and $H$ are $n$-valued games with the same parity, then $G^+ \lesssim H^+$ if and only if
\[ \lfout(G \oplus_n X) \le \lfout(H \oplus_n X)\]
for all $n$-valued games $X$.
\end{theorem}
\begin{proof}
Clearly if $G^+ \lesssim H^+$, then $G^+ + X^+ \lesssim H^+ + X^+$, so that
\[ \lfout(G + X) = \lfout(G^+ + X^+) \le \lfout(H^+ + X^+) = \lfout(H + X),\]
using the fact from Theorem~\ref{plethora} that $\lfout(K) = \lfout(K^+)$ for any game $K$.
But then
\[ \lfout(G \oplus_n X) = \min(\lfout(G + X),n-1) \le \min(\lfout(H + X),n-1) = \lfout(H \oplus_n X).\]
 
Conversely, suppose that $G^+ \not\lesssim H^+$.  By Theorem~\ref{plethora}(g) we can assume
that $G^+$ and $H^+$ are $n$-valued games too.  Because they are i-games, it follows
from Corollary~\ref{restatement} that $\routcome(H^+ - G^+) < 0$.  Then by Theorem~\ref{plethora}(h),
\[ \routcome(H - G^+) = \routcome((H - G^+)^+) = \routcome(H^+ - G^+) < 0,\]
since $G, H, G^+$, and $H^+$ all have the same parity.  On the other hand,
\[ \routcome(G - G^+) = \routcome(G^+ - G^+) = \routcome(0) = 0.\]
Now let $X$ be the game $n - 1 - G^+$, so that $\routcome(H + X) < n- 1$ and $\routcome(G + X) = n - 1$.
Let $\nu$ be the map $\nu(x) = \min(x,n-1)$.  Then
\[ \routcome(H \oplus_n X) = \routcome(\tilde{\nu}(H + X)) = \nu(\routcome(H + X)) < n - 1,\]
while
\[ \routcome(G \oplus_n X) = \nu(\routcome(G + X)) = \nu(n - 1) = n - 1.\]
So then
\[ \lfout(G \oplus_n X) = \routcome(G \oplus_n X) = n - 1 > \routcome(H \oplus_n X) = \lfout(H \oplus_n X),\]
where $\lfout$ is $\routcome$ because $X$ has the same parity as $G$ and $H$.  Then we are done, because $X$ is clearly an $n$-valued game.
\end{proof}

Finally, in the case where $G$ and $H$ have different parities, it is actually possible for
$G$ and $H$ to be $\{\oplus_n\}$-indistinguishable, surprisingly:
\begin{theorem}\label{diffpar}
If $G$ and $H$ are $n$-valued games of different parities, then $G$ and $H$ are $\oplus_n$
indistinguishable if and only if
\begin{equation} \outcome(G) = \outcome(G + *) = \outcome(H) = \outcome(H + *) = (n-1,n-1).\label{unlikely}\end{equation}
\end{theorem}
\begin{proof}
First of all, suppose that (\ref{unlikely}) is true.  I claim that for every game $X$,
\[ \outcome(G \oplus_n X) = \outcome(H \oplus_n X) = (n - 1, n - 1).\]
If $X$ is even-tempered, then $0 \lesssim X$ by Lemma~\ref{zerosmallest}, so that
\[ G \lesssim G \oplus_n X \text{ and } H \lesssim H \oplus_n X,\]
and therefore $\outcome(G \oplus_n X)$ and $\outcome(H \oplus_n X)$ must be at least
as high as $\outcome(G)$ and $\outcome(H)$.  But $\outcome(G)$ and $\outcome(H)$
are already the maximum values, so $\outcome(G \oplus_n X)$ and $\outcome(H \oplus_n X)$ must
also be $(n- 1, n - 1)$.

On the other hand, if $X$ is odd-tempered, then $X \oplus_n * = X + *$ is even-tempered, and the same argument
applied to $X + *$, $G + *$, and $H + *$ shows that
\[ \outcome(G \oplus_n X) = \outcome((G + *) \oplus_n (X + *)) = (n-1,n-1)\]
and
\[ \outcome(H \oplus_n X) = \outcome((H + *) \oplus_n (X + *)) = (n-1,n-1).\]
(Note that for any $n$-valued game $K$, $K \oplus_n * = K + *$, by Lemma~\ref{staradd}.)

Now for the converse, suppose that $G$ and $H$ are indistinguishable.  Without loss of generality,
$G$ is odd-tempered and $H$ is even-tempered.  Let $Q = \langle*|(n-1)*\rangle$, so that $Q^- \approx 0$, $Q^+ \approx n - 1$, and $Q$ is even-tempered.
Then
\[ n - 1 = \loutcome(G^+ \oplus_n (n-1)) = \loutcome(G^+ \oplus_n Q^+) = \loutcome((G \oplus_n Q)^+) = \]\[\loutcome(G \oplus_n Q) = \loutcome(H \oplus_n Q) =
\loutcome((H \oplus_n Q)^-) = \loutcome(H^- \oplus_n Q^-) = \]\[\loutcome(H^- \oplus_n 0) = \loutcome(H^-) = \loutcome(H)\]
and
\[ \routcome(G) = \routcome(G^-) = \routcome(G^- \oplus_n 0) = \routcome(G^- \oplus_n Q^-) =\]\[ \routcome((G \oplus_n Q)^-) = \routcome(G \oplus_n Q) = \routcome(H \oplus_n Q) = \routcome((H \oplus_n Q)^+) = \]\[\routcome(H^+ \oplus_n Q^+) = \routcome(H^+ \oplus_n (n - 1)) = n - 1.\]
So
\[ \loutcome(H) = n - 1 = \routcome(G).\]
But if $G$ and $H$ are indistinguishable, then
\[ \outcome(G) = \outcome(G \oplus_n 0) = \outcome(H \oplus_n 0) = \outcome(H),\]
so it must be the case that
\[ \loutcome(G) = \loutcome(H) = n - 1\]
and
\[ \routcome(H) = \routcome(G) = n - 1.\]
So every outcome of $G$ or $H$ is $n - 1$.  And by the same token, $G + *$ and $H + *$ are also indistinguishable $n$-valued games
of opposite parity, so every outcome of $G + *$ and of $H + *$ must also be $n - 1$.
\end{proof}

Combining Theorems~\ref{sameparrf}, \ref{sameparlf}, and \ref{diffpar}, we get a more explicit description
of $\{\oplus_n\}$-indistinguishability:
\begin{theorem}\label{oplusindist}
Let $\sim$ be $\{\oplus_n\}$-indistinguishability on $n$-valued games.  Then when $G$ and $H$ have the same parity,
$G \sim H$ iff
$G^+ \approx H^+$ and
\begin{equation}G^- \oplus_n \langle 0 | n - 1 \rangle \approx H^- \oplus_n \langle 0 | n - 1 \rangle.\label{complicon}\end{equation}
When $G$ is odd-tempered and $H$ is even-tempered, $G \sim H$ if and only if
\begin{equation} G^+ \approx G^- \oplus_n \langle 0 | n - 1 \rangle \approx n - 1\label{gcon}\end{equation}
and
\begin{equation} H^+ \approx H^- \oplus_n \langle 0 | n - 1 \rangle \approx (n - 1)*\label{hcon}\end{equation}
\end{theorem}
\begin{proof}
By Lemma~\ref{associndist2}, $G \sim H$ if and only if
\[ \forall X \in \mathcal{W}_n : \outcome(G \oplus_n X) = \outcome(H \oplus_n X).\]
If $G$ and $H$ have the same parity, then this is equivalent to
\[ \rfout(G \oplus_n X) = \rfout(H \oplus_n X)\]
and
\[ \lfout(G \oplus_n X) = \lfout(H \oplus_n X)\]
for all $n$-valued games $X$.  By Theorems~\ref{sameparrf} and \ref{sameparlf}, respectively,
these are equivalent to $G^+ \sim H^+$ and to (\ref{complicon}) above.  This handles the case
when $G$ and $H$ have the same parity.

Let $\mathcal{A}$ be the set of all even-tempered $n$-valued games $X$ with $\outcome(X) = \outcome(X + *) = (n-1,n-1)$,
and $\mathcal{B}$ be the set of all odd-tempered $n$-valued games with the same property.  Then Theorem~\ref{diffpar}
says that when $G$ is odd-tempered and $H$ is even-tempered, $G \sim H$ iff $G \in \mathcal{A}$ and $H \in \mathcal{B}$.
Now both $\mathcal{A}$ and $\mathcal{B}$ are nonempty, since $(n- 1) \in \mathcal{A}$ and
$(n - 1)* \in \mathcal{B}$, easily.  So by transitivity, $\mathcal{A}$ must be an equivalence class,
specifically the equivalence class of $(n - 1)$, and similarly $\mathcal{B}$ must be the equivalence
class of $(n - 1)*$.  Then (\ref{gcon}) and (\ref{hcon}) are just the conditions we just determined
for comparing games of the same parity, because of the easily-checked facts that
\[ (n-1) \oplus_n \langle 0 | n \rangle = (n-1)*\]
\[ (n-1)* \oplus_n \langle 0 | n \rangle = (n-1) + * + * \approx (n - 1).\]
\end{proof}

So as far as $\{\oplus_n\}$ indistinguishability is concerned, a game $G$ is determined by $G^+$,
$G^- \oplus_n \langle 0 | n - 1 \rangle$, and its parity - except that in one case one of the
even-tempered equivalence classes gets merged with one of the odd-tempered equivalence classes.

In the case that $n = 2$, $\langle 0 | 1 \rangle = \frac{1}{2}.\{\frac{1}{2}|\} + *$, so
the thirty five possible pairs of $(u^-,u^+)$ give rise
to the following nineteen pairs of $(u^- \cup \frac{1}{2}, u^+)$:
\[
\left(\frac{1}{2},0\right), \left(\frac{1}{2},\frac{1}{4}\right),
\left(\frac{1}{2},\frac{3}{8}\right), \left(\frac{1}{2},\frac{1}{2}\right),\]\[
\left(\frac{1}{2},\frac{1}{2}*\right), \left(\frac{1}{2},\frac{5}{8}\right),
\left(\frac{1}{2},\frac{3}{4}\right), \left(\frac{1}{2},1\right),\]\[
\left(\frac{3}{4},\frac{1}{4}\right), \left(\frac{3}{4},\frac{3}{8}\right),
\left(\frac{3}{4},\frac{1}{2}\right), \left(\frac{3}{4},\frac{1}{2}*\right),\]\[
\left(\frac{3}{4},\frac{5}{8}\right), \left(\frac{3}{4},\frac{3}{4}\right),
\left(\frac{3}{4},1\right),
\left(1,\frac{1}{2}\right),\]\[ \left(1,\frac{5}{8}\right),
\left(1,\frac{3}{4}\right), \left(1,1\right)
\]
So there are $2 \cdot 19 - 1 = 37$ equivalence classes of $2$-valued games, modulo
$\{\oplus_2\}$ indistinguishability.

We leave as an exercise to the reader the analogue of Theorem~\ref{oplusindist} for $\{\odot_n\}$-indistinguishability.

\section{Indistinguishability for min and max}
Unlike the case of $\{\oplus_n,\odot_n\}$- and $\{\oplus_n\}$- indistinguishability, $\{\wedge,\vee\}$-
and $\{\vee\}$- indinstinguishability are much simpler to understand, because they reduce in a simple way to the $n = 2$ case
(where $\oplus_2 = \vee$ and $\odot_2 = \wedge$).  In particular, there will be only finitely many equivalence
classes of $n$-valued games modulo $\{\wedge,\vee\}$-indistinguishability (and therefore modulo $\{\vee\}$-indistinguishability too,
because $\{\vee\}$-indistinguishability is a coarser relation).  The biggest issue will be showing that all the expected equivalence
classes are nonempty.

For any $n$, let $\delta_n \,:\, \mathbb{Z} \to \{0,1\}$ be given by $\delta_n(x) = 0$ if $x < n$, and $\delta_n(x) = 1$
if $x \ge n$.  Then for $m = 1, 2, \ldots, n - 1$, $\tilde{\delta_m}$ produces a map from $n$-valued games
to $2$-valued games.  By Lemma~\ref{boundingfunctions}, $\tilde{\delta_m}(G) \le \tilde{\delta_{m'}}(G)$
when $m \ge m'$.  We will see that a game is determined up to $\{\wedge,\vee\}$-indistinguishability
by the sequence $(\tilde{\delta_1}(G),\tilde{\delta_2}(G),\ldots,\tilde{\delta}_{n-1}(G))$.

\begin{theorem}\label{minmaxfirst}
If $G$ is an $n$-valued game, then $\loutcome(G)$ is the maximum $m$ between $1$ and $n - 1$ such that
$\loutcome(\tilde{\delta_m}(G)) = 1$, or 0 if no such $m$ exists.  Similarly, $\routcome(G)$ is the maximum
$m$ between $1$ and $n - 1$ such that $\routcome(\tilde{\delta_m}(G)) = 1$, or 0 if no such $m$ exists.
\end{theorem}In particular, then,
the outcome of a game is determined by the values of $\tilde{\delta_m}(G)$, so that if
$\tilde{\delta_m}(G) \approx \tilde{\delta_m}(H)$ for every $m$ for some game $H$, then $\outcome(G) = \outcome(H)$.
\begin{proof}
Note that by Lemma~\ref{outcomeslide},
\[ \delta_m(\loutcome(G)) = \loutcome(\tilde{\delta_m}(G)) \text{ and } \delta_m(\routcome(G)) = \routcome(\tilde{\delta_m}(G)).\]
Then by definition of $\delta_m$, we see that $\loutcome(\tilde{\delta_m}(G)) = 1$ iff $m \le \loutcome(G)$,
and similarly for $\routcome(\tilde{\delta_m}(G))$ and $\routcome(G)$. So since $\loutcome(G)$ and $\routcome(G)$
are integers between $0$ and $n - 1$, the desired result follows.
\end{proof}

\begin{theorem}\label{minmaxsecond}
If $G$ and $H$ are $n$-valued games, then
\[ \tilde{\delta_m}(G \wedge H) = \tilde{\delta_m}(G) \wedge \tilde{\delta_m}(H) = \tilde{\delta_m}(G) \odot_2 \tilde{\delta_m}(H)\]
and
\[ \tilde{\delta_m}(G \vee H) = \tilde{\delta_m}(G) \vee \tilde{\delta_m}(H) = \tilde{\delta_m}(G) \oplus_2 \tilde{\delta_m}(H).\]
\end{theorem}
\begin{proof}
This follows immediately from the fact that when restricted to $2$-valued games, $\wedge = \odot_2$ and $\vee = \oplus_2$, together with the
obvious equations
\[ \delta_m(\min(x,y)) = \min(\delta_m(x),\delta_m(y)) \text{ and } \delta_m(\max(x,y)) = \max(\delta_m(x),\delta_m(y)) \]
\end{proof}
Let $\sim$ be the equivalence relation on $n$-valued games given by
\[ G \sim H \iff \forall 1 \le m \le n  - 1: \tilde{\delta_m}(G) \approx \tilde{\delta_m}(H)\]
Then Theorem~\ref{minmaxfirst} implies that $\outcome(G) = \outcome(H)$ when $G \sim H$, and Theorem~\ref{minmaxsecond} implies
that $G \wedge H \sim G' \wedge H'$ and $G \vee H \sim G' \vee H'$, when $G \sim G'$ and $H \sim H'$.  So by definition of indistinguishability, $G \sim H$ implies
that $G$ and $H$ are $\{\wedge,\vee\}$-indistinguishable (and $\{\vee\}$-indistinguishable too, of course).

\begin{theorem}\label{minmaxthird}
If $G$ and $H$ are $\{\wedge,\vee\}$-indistinguishable, then $G \sim H$.  In particular then, $\sim$ is $\{\wedge,\vee\}$-indistinguishability.
\end{theorem}
\begin{proof}
Suppose that $G \not\sim H$, so that $\tilde{\delta_m}(G) \not\approx \tilde{\delta_m}(H)$ for some $1 \le m \le n - 1$.
Then by Theorem~\ref{doubledisting}, there is a two-valued game $Y$ such that
\begin{equation} \outcome(\tilde{\delta_m}(G) \odot_2 Y) \ne \outcome(\tilde{\delta_m}(H) \odot_2 Y) \text{ or }
\outcome(\tilde{\delta_m}(G) \oplus_2 Y) \ne \outcome(\tilde{\delta_m}(H) \oplus_2 Y).\label{minmaxthirdeq1}\end{equation}
Let $X = (m - 1) + Y$, which will be a $n$-valued game because $1 \le m \le n - 1$ and $Y$ is $\{0,1\}$-valued.
Then since $\delta_m((m-1) + y) = y$ for $y \in \{0,1\}$, it follows that $\tilde{\delta_m}(X) = Y$.
So by Theorem~\ref{minmaxsecond},
\[ \tilde{\delta_m}(G \wedge X) \approx \tilde{\delta_m}(G) \odot_2 Y\]
\[ \tilde{\delta_m}(H \wedge X) \approx \tilde{\delta_m}(H) \odot_2 Y\]
\[ \tilde{\delta_m}(G \vee X) \approx \tilde{\delta_m}(G) \oplus_2 Y\]
\[ \tilde{\delta_m}(G \vee Y) \approx \tilde{\delta_m}(H) \oplus_2 Y\]
Combining this with (\ref{minmaxthirdeq1}), we see that
\[ \outcome(\tilde{\delta_m}(G \wedge X)) \ne \outcome(\tilde{\delta_m}(H \wedge X)) \text{ or } \outcome(\tilde{\delta_m}(G \vee X))
\ne \outcome(\tilde{\delta_m}(H \vee X)).\]
By Lemma~\ref{outcomeslide}, this implies that either
\[ \outcome(G \wedge X) \ne \outcome(H \wedge X) \text{ or } \outcome(G \vee X) \ne \outcome(H \vee X),\]
so that $G$ and $H$ are not indistinguishable.

The converse direction, that $G \sim H$ implies that $G$ and $H$ are $\{\wedge,\vee\}$-indistinguishable, follows by the remarks before this theorem.
\end{proof}

By a completely analogous argument we see that
\begin{theorem}
If $G$ and $H$ are $n$-valued games, then $G$ and $H$ are $\{\vee\}$-indistinguishable if and only if
$\tilde{\delta_m}(G)$ and $\tilde{\delta_m}(H)$ are $\{\oplus_2\}$-indistinguishable for all $1 \le m \le n - 1$.
\end{theorem}

Now since there are only finitely many classes of $2$-valued games modulo $\approx$, it follows that there
are only finitely many $n$-valued games modulo $\{\wedge,\vee\}$-indistinguishability, and a game's class is determined
entirely by its parity and the values of $u^+(\tilde{\delta_m}(G))$ and $u^-(\tilde{\delta_m}(G))$ for $1 \le m \le n - 1$.
We can see that these sequences are weakly decreasing, by Lemma~\ref{boundingfunctions}, and it is also clear that $u^-(\tilde{\delta_m}(G)) \le u^+(\tilde{\delta_m}(G))$,
but are there any other restrictions?

It turns out that there are none: given any weakly decreasing sequence of $2$-valued games modulo $\approx$, some $n$-valued game
has them as its sequence.  Unfortunately the proof is fairly complicated.  We begin with a technical lemma.
% apricot do it for a general poset
\begin{lemma}\label{technical}
Let $A$ be the subgroup of $\mathcal{G}$ generated by short numbers and $*$, and let
$B$ be the group of $\mathbb{Z}$-valued even-tempered i-games $G$ in $I_{-2}$ such that
$\psi(G) \in A$, modulo $\approx$.  Let $P$ be either $A$ or $B$.  Suppose we have sequences $a_1, \ldots, a_n$ and $b_1, \ldots, b_n$
of elements of $P$ such that $a_i \ge b_i, a_i \ge a_{i+1}$, and $b_i \ge b_{i+1}$ for all appropriate $i$.  Then
for $0 \le j \le i \le n$, we can choose $c_{ij} \in P$, such that $c_{ij} \ge 0$ for $(i,j) \ne (n,n)$, and
\begin{equation} a_k = \sum_{0 \le j \le i \le n,\, k \le i}c_{ij}\label{aeq}\end{equation}
and
\begin{equation} b_k = \sum_{0 \le j \le i \le n,\, k \le j}c_{ij}\label{beq}\end{equation}
for all $1 \le k \le n$.
\end{lemma}
So for instance, in the $n = 2$ case, this is saying that if the rows and columns of
\[ \begin{pmatrix} a_1  & a_2 \\ b_1 & b_2 \end{pmatrix}\]
are weakly decreasing, then we can find $c_{ij} \in P$ such that
\[ \begin{pmatrix} a_1 & a_2 \\ b_1 & b_2 \end{pmatrix}
= \begin{pmatrix} c_{10} & 0 \\ 0 & 0 \end{pmatrix}
+ \begin{pmatrix} c_{11} & 0 \\ c_{11} & 0 \end{pmatrix}
+ \begin{pmatrix} c_{20} & c_{20} \\ 0 & 0 \end{pmatrix}
+ \begin{pmatrix} c_{21} & c_{21} \\ c_{21} & 0 \end{pmatrix}
+ \begin{pmatrix} c_{22} & c_{22} \\ c_{22} & c_{22} \end{pmatrix},\]
where $c_{10},$ $c_{11}$, $c_{20}$, and $c_{21} \ge 0$.
This is not trivial - if $P$ was instead the group of partizan games generated by the integers and $*$, then no
such $c_{ij} \in P$ could be found for the following matrix:
\[ \begin{pmatrix} 2 & 1* \\ 1 & 0 \end{pmatrix}.\]
\begin{proof}[Proof (of Lemma~\ref{technical})]
Since $A$ and $B$ are isomorphic as partially-ordered abelian groups (by Theorem~\ref{psisurj}), we only consider
the $P = A$ case.  The elements of $A$ are all of the form $x$ or $x*$, for $x$ a dyadic rational,
and are compared as follows:
\[ x* \ge y* \iff x \ge y\]
\[ x* \ge y \iff x \ge y* \iff x > y\]

We specify an algorithm for finding the $c_{ij}$ as follows.  First, take $c_{nn} = b_n$, because $c_{nn}$ is always
the only $c_{ij}$ that appears in the sum (\ref{beq}) for $b_n$.  Then subtract off $c_{nn}$ from every $a_k$
and $b_k$.  This clear the bottom right corner of the matrix
\[ \begin{pmatrix} a_1 & a_2 & \cdots & a_n \\
b_1 & b_2 & \cdots & b_n \end{pmatrix}\]
and preserves the weakly-decreasing property of rows and columns, leaving every element $\ge 0$.

Now, we find ways to clear more and more entries of this matrix by subtracting off matrices of the form
\[ \begin{pmatrix}
x & x & \cdots & x & x & \cdots & x & 0 & \cdots & 0 \\
x & x & \cdots & x & 0 & \cdots & 0 & 0 & \cdots & 0
\end{pmatrix}
\]
or
\[ \begin{pmatrix}
x & \cdots & x & 0 & \cdots & 0 \\
x & \cdots & x & 0 & \cdots & 0
\end{pmatrix} \]
for $x \in P, x \ge 0$.
Once the matrix is cleared, we are done.  At every step, the rows and columns of the matrix will
be weakly decreasing and there will be a zero in the bottom right corner.  Each step increases the number
of vanishing entries, so the algorithm eventually terminates.

Let the current state of the matrix be
\begin{equation} \begin{pmatrix}
a'_1 & \cdots & a'_{n-1} & a'_n \\
b'_1 & \cdots & b'_{n-1} & 0
\end{pmatrix},\label{currentstate}\end{equation}
and find the biggest $i$ and $j$ such that $a'_i$ and $b'_j$ are nonzero.  Since the rows and columns are weakly decreasing,
$j \le i$.  Also $a'_i$ and $b'_j$ are both $> 0$, so they must each be of the form $x$ or $x*$ for some number $x > 0$.

First of all suppose that $a'_i$ and $b'_j$ are comparable.  Let $k$ be $\min(a'_i,b'_j)$.  Then every nonzero element
of the matrix (\ref{currentstate}) is at least $k$, so subtracting off a matrix of type $c_{ij}$ having value $k$ in the appropriate places,
we clear either $a'_i$ or $b'_j$ (or both) and do not break the weakly-decreasing rows and columns requirement.

Otherwise, $a'_i$ and $b'_j$ are incomparable.  
I claim that we can subtract a small $\epsilon > 0$ from all the entries
which are $\ge a'_i$ and preserve the weakly-decreasing rows and columns condition.  For this to work, we need
\begin{equation} a'_{k} - a'_{k+1} \ge \epsilon\label{firstbound}\end{equation}
whenever $a'_{k+1} \not \ge a'_i$ but $a'_k \ge a'_i$,
\begin{equation} b'_k - b'_{k+1} \ge \epsilon\label{secondboud}\end{equation}
whenever $b'_{k+1} \not \ge a'_i$ but $b'_k \ge a'_i$,
and
\begin{equation} a'_k - b'_k \ge \epsilon\label{thirdbound}\end{equation}
whenever $b'_k \not \ge a'_i$ but $a'_k \ge a'_i$.  Now there are only finitely many positions in the
matrix, the dyadic rational numbers are dense, and $*$ is infinitesimal.  Consequently, it suffices to show that all of the upper bounds (\ref{firstbound}-\ref{thirdbound}) on $\epsilon$
are greater than zero.  In other words,
\[ a'_k > a'_{k+1}\]
whenever $a'_{k+1} \not \ge a'_i$ but $a'_k \ge a'_i$,
\[ b'_k > b'_{k+1}\]
whenever $b'_{k+1} \not \ge a'_i$ but $b'_k \ge a'_i$,
and
\[ a'_k > b'_k \]
whenever $b'_k \not \ge a'_i$ but $a'_k \ge a'_i$.  But all of these follow from the obvious fact that if $x, y \in P$, $x \ge y$,
$y \not \ge a'_i$, and $x \ge a'_i$, then $x > y$.

So such an $\epsilon > 0$ exists.  Because rows and columns are weakly decreasing, the set of positions in the matrix whose
values are $\ge a'_i$ is the set of nonzero positions in a $c_{ij}$ matrix for some $i,j$.  So we are allowed to subtract
off $\epsilon$ from each of those entries.  After doing so, $a'_i - \epsilon$ is no longer incomparable with $b'_i$, so we can clear one
or the other in the manner described above.
\end{proof}

To prove that all possible sequences $u^+$ and $u^-$ values occur, we use some specific functions, in a proof that generalizes
the technique of Theorems~\ref{inorder2} and \ref{allpairs}.  Fix $n$, the number of values that the $n$-valued games can take.
Here are the functions we will use:
\begin{itemize}
\item $\delta_m(x)$, as above, will be 1 if $x \ge m$ and 0 otherwise.
\item $\mu(x)$ will be $\min(0,x)$.
\item For $1 \le k \le n - 1$, $f_k\,:\, \mathbb{Z}^{n(n+1)/2 - 1} \to \mathbb{Z}$ will be
\[ f_k(x_{10},x_{11},x_{20},x_{21},x_{22},x_{30},\ldots)
= \]\[\sum_{0 \le j \le i \le n - 1, k \le i} x_{ij} + \sum_{0 \le j \le i \le n - 1, k > i} \mu(x_{ij}).\]
In other words, $f_k$ is the sum of all its arguments except for its positive arguments $x_{ij}$
where $i \ge k$.
\item For $1 \le k \le n - 1$, $g_k\,:\, \mathbb{Z}^{n(n+1)/2 - 1} \to \mathbb{Z}$ will be
\[ g_k(x_{10},x_{11},x_{20},x_{21},x_{22},x_{30},\ldots)
= \]\[\sum_{0 \le j \le i \le n - 1, k \le j} x_{ij} + \sum_{0 \le j \le i \le n - 1, k > j} \mu(x_{ij}).\]
In other words, $g_k$ is the sum of all its arguments except for its positive arguments $x_{ij}$
where $j \ge k$.
\item $h^+\,:\, \mathbb{Z}^{n(n+1)/2 - 1} \to \{0,1\}$ will be given by
\[ h^+(x_{10},x_{11},x_{20},\ldots) = a,\]
where $a$ is the unique number in $n = \{0,1,\ldots,n-1\}$ such that
\[ \delta_m(a) = \delta_1(f_m(x_{10},\ldots))\]
for every $1 \le m \le n - 1$.  Such a number exists because $f_m(x_{10},\ldots)$ is decreasing as a function of $m$.
\item $h^-(x_{10},x_{11},x_{20},\ldots)$ will be the unique $a$ such that
\[ \delta_m(a) = \delta_1(g_m(x_{10},\ldots))\]
for every $1 \le m \le n - 1$.
\end{itemize}
It is not difficult to show that all of these functions are order-preserving, and that $f_m \ge g_m$ for every $m$.
Note that $h^+$ and $h^-$ can alternatively be described as
\[ \sum_{m = 1}^{n-1} \delta_1(f_m(x_{10},\ldots)) \text{ and } \sum_{m = 1}^{n-1} \delta_1(g_m(x_{10},\ldots)),\]
respectively.  So they are order-preserving, and $h^- \le h^+$.

Repeating an argument we used in Theorem~\ref{allpairs}, we have
\begin{lemma}\label{muzero}
Let $G$ be an $\mathbb{Z}$-valued i-game $G \gtrsim 0$.  Then $\tilde{\mu}(G) \approx 0$.

Similarly, for $0 \le j \le i \le n - 1$, let $G_{ij}$ be $\mathbb{Z}$-valued i-games $G_{ij} \gtrsim 0$.  Then
for $1 \le k \le n - 1$,
\[ \tilde{f_k}(G_{10},G_{11},G_{20},\ldots) \approx \sum_{0 \le j \le i \le n - 1, k \le i} G_{ij}\]
and
\[ \tilde{g_k}(G_{10},G_{11},G_{20},\ldots) \approx \sum_{0 \le j \le i \le n - 1, k \le j} G_{ij}\]
\end{lemma}
\begin{proof}
For the first claim, notice that $G \gtrsim 0$ implies that $\loutcome(G) \ge 0$
and $\routcome(G) \ge 0$.  Thus $\loutcome(\tilde{\mu}(G)) = 0 = \routcome(\tilde{\mu}(G))$.
But since $\tilde{\mu}(G)$ is an i-game (by Lemma~\ref{distortions} or Theorem~\ref{generaldistortions}),
it follows from Corollary~\ref{restatement} that $\tilde{\mu}(G) \approx 0$.

For the second claim, the definition of $f_k$ implies that
\[ \tilde{f_k}(G_{10},\ldots) = \sum_{0 \le j \le i \le n - 1, k \le i}G_{ij} + \sum_{0 \le j \le i \le n - 1, k > i} \tilde{\mu}(G_{ij})
\approx \sum_{0 \le j \le i \le n - 1, k \le i}G_{ij},\]
and $g_k$ is handled similarly.
\end{proof}

Using these we can find all the equivalence classes of $n$-valued games modulo $\{\wedge,\vee\}$-indistinguishability.
\begin{theorem}
Let \[U = \{0,1/4,3/8,1/2,1/2*,5/8,3/4,1\}.\]  Let $a_1,\ldots,a_{n-1}$ and $b_1,\ldots,b_{n-1}$ be sequences of elements of $U$
such that $a_j \ge a_k$ and $b_j \ge b_k$ for $j \le k$, and $a_i \ge b_i$ for all $i$.  Then there is at least one
even-tempered $n$-valued game $G$ such that
\[ u^+(\tilde{\delta_i}(G)) = a_i\text{ and } u^-(\tilde{\delta_i}(G)) = b_i\]
for all $1 \le i \le n - 1$.
\end{theorem}
\begin{proof}
Let $A_i$ and $B_i$ be $2$-valued even-tempered i-games with $u^+(A_i) = a_i$ and $u^+(B_i) = b_i$ (note that
if $A$ is a $2$-valued i-game, then $u^+(A) = u^-(A)$). 
By Lemma~\ref{technical} (taking $P$ to be the group of $\mathbb{Z}$-valued even-tempered i-games in the domain of $\psi$ generated
by numbers and $*$), we can find $\mathbb{Z}$-valued even-tempered i-games $G_{ij}$ for $0 \le j \le i \le n - 1$ such that $G_{ij} \gtrsim 0$
for $(i,j) \ne (n-1,n-1)$, and
\[ A_k = \sum_{0 \le j \le i \le n - 1, k \le i}G_{ij}\]
and
\[ B_k = \sum_{0 \le j \le i \le n - 1, k \le j}G_{ij}.\]
But then since $B_{n-1} = G_{(n-1)(n-1)}$, and all $2$-valued even-tempered games are $\gtrsim 0$ (by Theorem~\ref{plethora}(g,h,i)),
it follows that even $G_{(n-1)(n-1)}$ is $\gtrsim 0$.

Then by Lemma~\ref{muzero}, we have
\[ \tilde{f_k}(G_{10},G_{11},\ldots) \approx A_k\]
and
\[ \tilde{g_k}(G_{10},G_{11},\ldots) \approx B_k\]
for all $k$.
Now by definition above, the functions $h^+$ and $h^-$ have the property that
\[ \delta_m(h^+(x_{10},\ldots)) = \delta_1(f_m(x_{10},\ldots))\]
and
\[ \delta_m(h^-(x_{10},\ldots)) = \delta_1(g_m(x_{10},\ldots))\]
for all $m$.  It then follows that letting
\[ H_\pm = \tilde{h^\pm}(G_{10},\ldots),\]
we have
\[ \tilde{\delta_m}(H_+) = \tilde{\delta_1}(\tilde{f_m}(G_{10},\ldots))
\approx \tilde{\delta_1}(A_m) = A_m\]
and
\[ \tilde{\delta_m}(H_-) = \tilde{\delta_1}(\tilde{g_m}(G_{10},\ldots))
\approx \tilde{\delta_1}(B_m) = B_m\]
for all $m$.  Moreover, $H_- \lesssim H_+$ because of Lemma~\ref{boundingfunctions} and the fact that $h^- \le h^+$.
Also, $H_\pm$ are both $n$-valued games because they are in the image
of $\tilde{h^\pm}$.  So by Theorem~\ref{allpairs}, there is a $n$-valued game $G$ for which $G^\pm = H_\pm$.  Thus
\[ \tilde{\delta_m}(G)^+ \approx \tilde{\delta_m}(G^+) \approx A_m\]
and
\[ \tilde{\delta_m}(G)^- \approx \tilde{\delta_m}(G^-) \approx B_m\]
for all $m$, so that $u^+(\tilde{\delta_m}(G)) = a_m$ and $u^-(\tilde{\delta_m}(G)) = b_m$ for all $m$.
\end{proof}

\begin{corollary}
The class of $n$-valued games modulo $\{\wedge,\vee\}$-indistinguishability is in one-to-one correspondence
with weakly-decreasing length-$(n-1)$ sequences of $2$-valued games modulo $\approx$.
\end{corollary}

As an exercise, it is also easy to show the following
\begin{corollary}
The class of $n$-valued games modulo $\{\vee\}$-indistinguishability is in one-to-one correspondence
with weakly-decreasing length-$(n-1)$ sequences of $2$-valued games modulo $\{\vee\}$-indistinguishability.
\end{corollary}
The only trick here is to let $u^-$ be $0$, $1/4$, or $1/2$, when $u^- \cup 1/2$ needs to be $1/2$, $3/4$, or $1$, respectively.

Since there are only finitely many $2$-valued games modulo $\approx$ or modulo
$\{\vee\}$-indistinguishability, one could in principle write down a formula for the number of
$n$-valued games modulo $\{\wedge,\vee\}$-indistinguishability or $\{\vee\}$-indistinguishability,
but we do not pursue the matter further here.

\part{Knots}
\chapter{To Knot or Not to Knot}
In Chapter~\ref{chap:zero} we defined the game \textsc{To Knot or Not to Knot}, in which two players,
King Lear, and Ursula take turns resolving crossings in a knot pseudodiagram,
until all crossings are resolved and a genuine knot diagram is determined.  Then \textbf{U}rsula wins
if the knot is equivalent to the \textbf{u}nknot, and \textbf{K}ing Lear wins if it is \textbf{k}notted.
We will identify King \textbf{L}ear with Left, and U\textbf{r}sula with Right, and view
TKONTK as a Boolean (2-valued) well-tempered scoring game.

For instance, the following pseudodiagram has the value $* = \langle 0 | 0 \rangle$, because
the game lasts for exactly one move, and Ursula wins no matter how the crossing is resolved.
\begin{figure}[H]
\begin{center}
\includegraphics[width=1in]
					{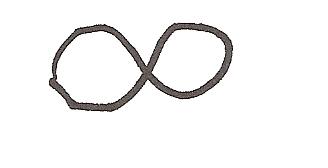}
%\caption{}
%\label{domineering-sum}
\end{center}
\end{figure}
Similarly, the following position is $1* = \langle 1 | 1 \rangle$, because it lasts one move,
but King Lear is guaranteed to win:
\begin{figure}[H]
\begin{center}
\includegraphics[width=1in]
					{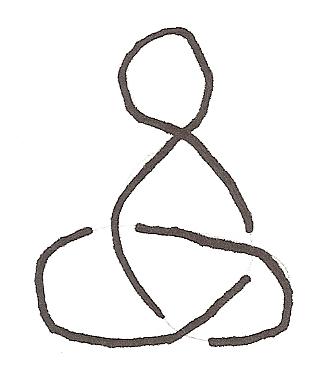}
\caption{One move remains, but King Lear has already won.}
\label{trefoilstar}
\end{center}
\end{figure}
On the other hand,
\begin{figure}[H]
\begin{center}
\includegraphics[width=2in]
					{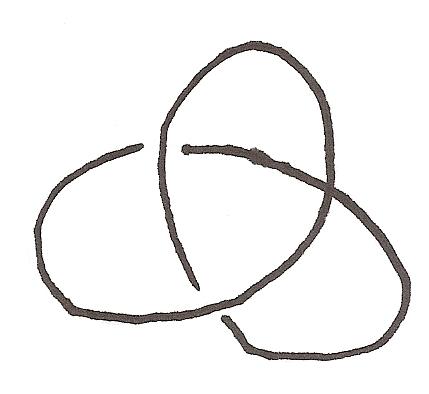}
\caption{The next move decides and ends the game.}
\label{ambiguous1}
\end{center}
\end{figure}
is $\langle 0, 1 | 0, 1 \rangle \approx \langle 1 | 0 \rangle$, because the remaining
crossing decides whether the resulting knot will be a knotted trefoil or an unknot.
\begin{figure}[H]
\begin{center}
\includegraphics[width=3in]
					{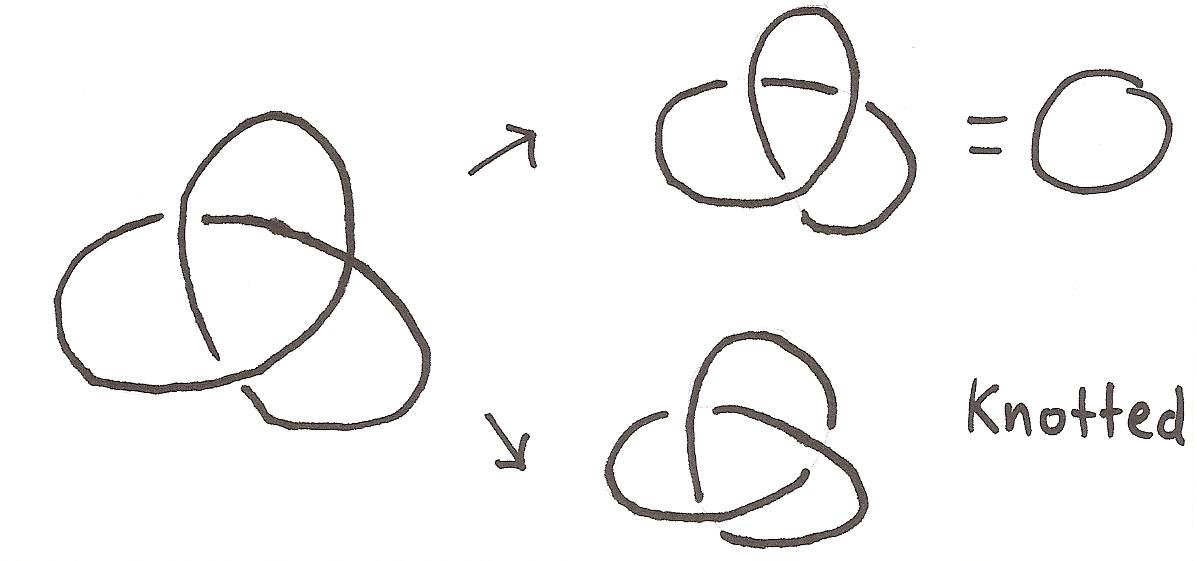}
%\caption{}
%\label{domineering-sum}
\end{center}
\end{figure}

The natural way to add TKONTK positions is the $\oplus_2 = \vee$ operation of Section~\ref{sect:roundsum}.
For example, when we add Figures~\ref{trefoilstar} and \ref{ambiguous1} ,
\begin{figure}[H]
\begin{center}
\includegraphics[width=3in]
					{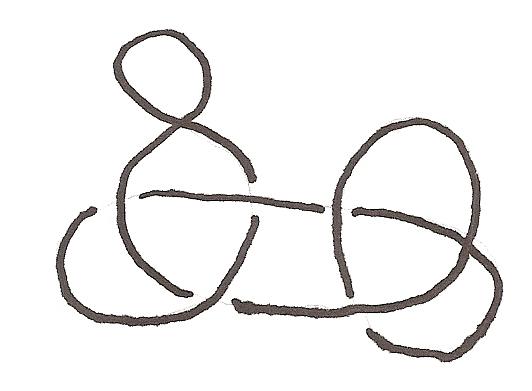}
%\caption{}
%\label{domineering-sum}
\end{center}
\end{figure}
we get a position with value $\{1|0\} \vee 1* \approx \{1|1\} + * \approx 1.$  So the resulting
position is equivalent to, say
\begin{figure}[H]
\begin{center}
\includegraphics[width=1in]
					{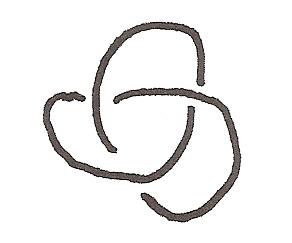}
%\caption{}
%\label{domineering-sum}
\end{center}
\end{figure}

\section{Phony Reidemeister Moves}

\begin{definition}
A pseudodiagram $S$ is obtained by a \emph{phony Reidemeister I move} from a pseudodiagram $T$ if $S$ is obtained
from $T$ by removing a loop with an unresolved crossing from $T$, as in Figure~\ref{phony-r1}.  We denote this $T \stackrel{1}{\rightarrow} S$.
\end{definition}
\begin{figure}[H]
\begin{center}
\includegraphics[width=3in]
					{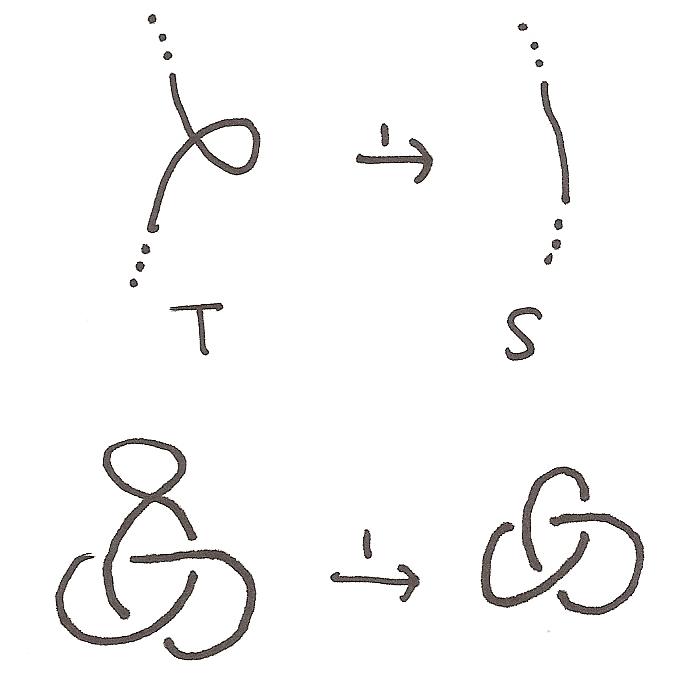}
\caption{Phony Reidemeister I move}
\label{phony-r1}
\end{center}
\end{figure}
\begin{definition}
A pseudodiagram $S$ is obtained by a \emph{phony Reidemeister II move} from a pseudodiagram $T$ if $S$ is obtained
from $T$ by uncrossing two overlapped strings, as in Figure~\ref{phony-r2}, where the two crossings eliminated are unresolved.
We denote this $T \stackrel{2}{\rightarrow} S$.
\end{definition}
\begin{figure}[H]
\begin{center}
\includegraphics[width=3in]
					{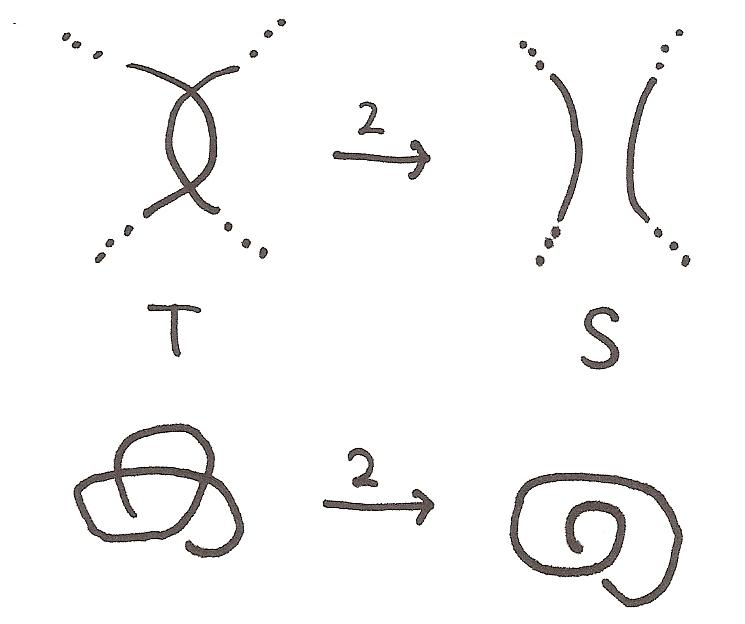}
\caption{Phony Reidemeister II move}
\label{phony-r2}
\end{center}
\end{figure}
Note that we only use these operations in one direction: $T \stackrel{i}{\rightarrow} S$ doesn't imply
$S \stackrel{i}{\rightarrow} T$.  We also use the notation $T \stackrel{i}{\Rightarrow}S$ to indicate
that $S$ is obtained from $T$ by a sequence of zero or more type $i$ moves, and $T \stackrel{*}{\Rightarrow} S$ to indicate
that $S$ is obtained by zero or more moves of either type.

If $T$ is a knot pseudodiagram, we let $\val(T)$ be the value of $T$ as a game of TKONTK, and we abuse notation
and write $u^+(T)$ and $u^-(T)$ for $u^+(\val(T))$ and $u^-(\val(T))$.

The importance of the phony Reidemeister moves is the following:
\begin{theorem}\label{redmonotone}
If $T \stackrel{1}{\rightarrow} S$ then $\val(T) = \val(S) + *$, so $u^+(T) = u^+(S)$ and $u^-(T) = u^-(S)$.

If $T \stackrel{2}{\rightarrow} S$, then $\val(T) \gtrsim_+ \val(S)$ and $\val(T) \lesssim_- \val(S)$.
So $u^+(T) \ge u^+(S)$ and $u^-(T) \le u^-(S)$.
\end{theorem}
\begin{proof}
If $S$ is obtained from $T$ by a phony Reidemeister I move, then $T$ is obtained from $S$ by adding an extra loop (with
an unresolved crossing).  In other words $T$ is $S \# K$ where $K$ is the following pseudodiagram:
\begin{figure}[H]
\begin{center}
\includegraphics[width=1in]
					{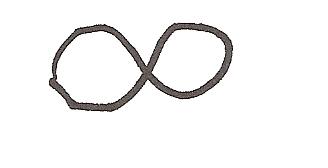}
%\caption{}
%\label{domineering-sum}
\end{center}
\end{figure}
As noted above, $\val(K) = *$, so
\[ \val(T) = \val(S) \oplus_2 * = \val(S) \oplus_2 (0 + *) = (\val(S) + *) \oplus_2 = \val(S) + *,\]
using Lemma~\ref{staradd} and the fact that $0$ is the identity element for $\oplus_2$.

For the second claim, we need to show that undoing a phony Reidemeister II move does not hurt whichever player
moves last, even if the pseudodiagram is being added to an arbitrary integer-valued game.  To see this,
suppose that $T \stackrel{2}{\rightarrow} S$ and that Alice have a strategy guaranteeing a certain score
in $\val(S) + G$ for some $G \in \mathcal{W}_\mathbb{Z}$, when Alice is the player who will make the last move.  Then Alice
can use this same strategy in $\val(T) + G$,
except that she applies a pairing strategy to manage the two new crossings.  If her opponent moves in one,
then she moves in the other, in a way that produces one of the following configurations:
\begin{figure}[H]
\begin{center}
\includegraphics[width=3in]
					{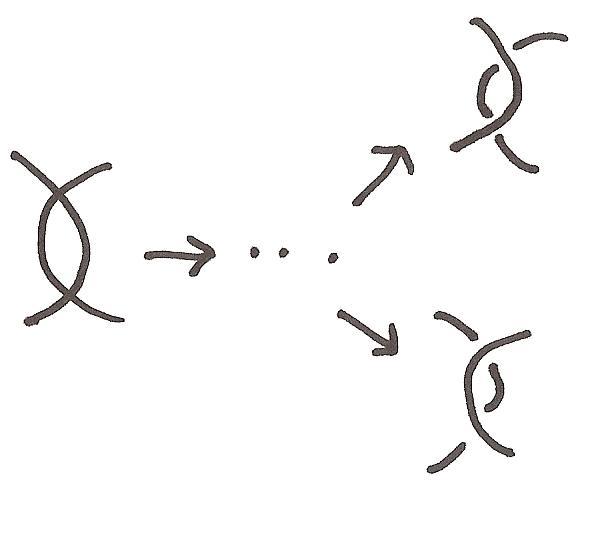}
\caption{After the first move, it is always possible to reply with a move to one of the configurations on the right.}
\label{r2-responses}
\end{center}
\end{figure}
Otherwise, she does not move in either of the two new crossings, and pretends that she is in fact playing $\val(S) + G$.
  Since she is the player who will make the last
move in the game, she is never forced to move in one of the two crossings before her opponent does, so this
pairing strategy always works.  And if the two crossings end up in one of the configurations on the right side of Figure~\ref{r2-responses},
then they can be undone by a standard Reidemeister II move, yielding a position identical to the one that Alice pretends she has reached, in $\val(S) + G$.
Since Alice had a certain guaranteed score in $\val(S) + G$, she can ensure the same score in $\val(T) + G$.  This works whether
Alice is Left or Right, so we are done.
\end{proof}

\section{Rational Pseudodiagrams and Shadows}
We use $[]$ to denote the rational tangle
\begin{figure}[H]
\begin{center}
\includegraphics[width=1in]
					{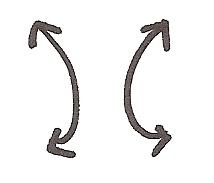}
%\caption{}
%\label{domineering-sum}
\end{center}
\end{figure}
and $[a_1,\ldots,a_n]$ to denote
the rational tangle obtained from $[a_1,\ldots,a_{n-1}]$ by reflection over a 45 degree axis and adding $a_n$ twists
to the right.
We also generalize this notation, letting $[a_1(b_1),\ldots,a_{n}(b_n)]$ denote a tangle-like pseudodiagram
in which there are $a_1$ legitimate crossings and $b_1$ unresolved crossings at each step.  See Figure~\ref{tangle-examples}
for examples.

\begin{figure}[htbp]
\begin{center}
\includegraphics[width=6in]
					{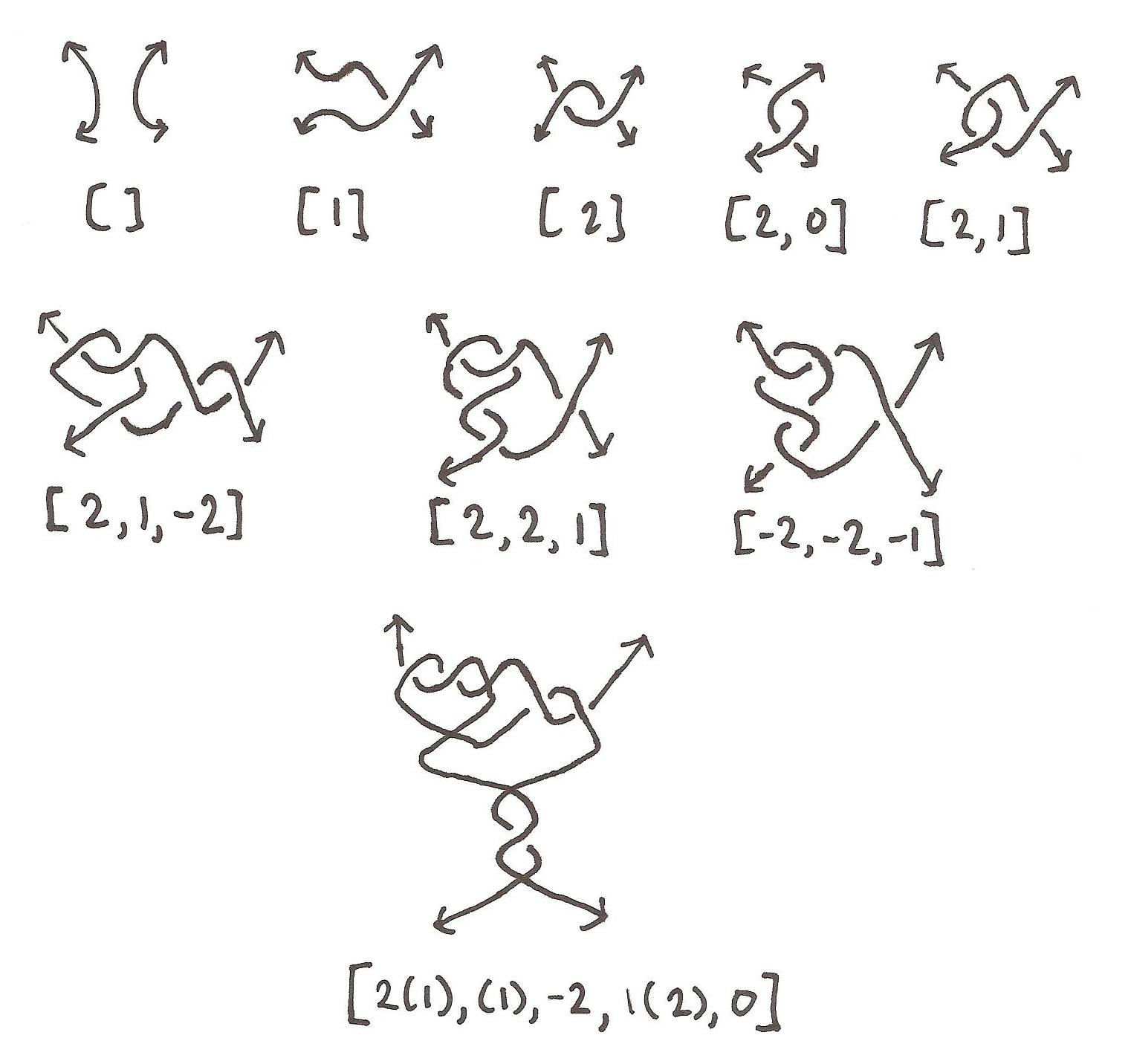}
\caption{Examples of our notation.}
\label{tangle-examples}
\end{center}
\end{figure}

So the $a_i \in \mathbb{Z}$ and the $b_i \in \mathbb{N}$, where $\mathbb{N}$ are the nonnegative integers.  If $a_i = 0$,
we write $(b_i)$ instead of $a_i(b_i)$, and similarly if $b_i = 0$, we write $a_i$ instead of $a_i(b_i)$.
A \emph{shadow} is a pseudodiagram in which all crossings are unresolved,
so a rational shadow tangle would be of the form $[(b_1),\ldots,(b_n)]$.

We abuse notation, and use the same $[a_1(b_1),\ldots,a_n(b_n)]$ notation for the pseudodiagram obtained by
connecting the top two strands of the tangle and the bottom two strands:
\begin{figure}[H]
\begin{center}
\includegraphics[width=3in]
					{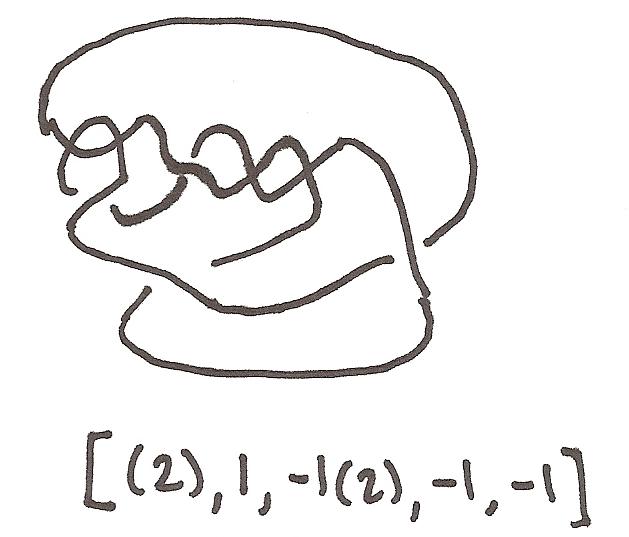}
%\caption{}
%\label{domineering-sum}
\end{center}
\end{figure}
Note that this can sometimes yield a link, rather than a knot:
\begin{figure}[H]
\begin{center}
\includegraphics[width=2in]
					{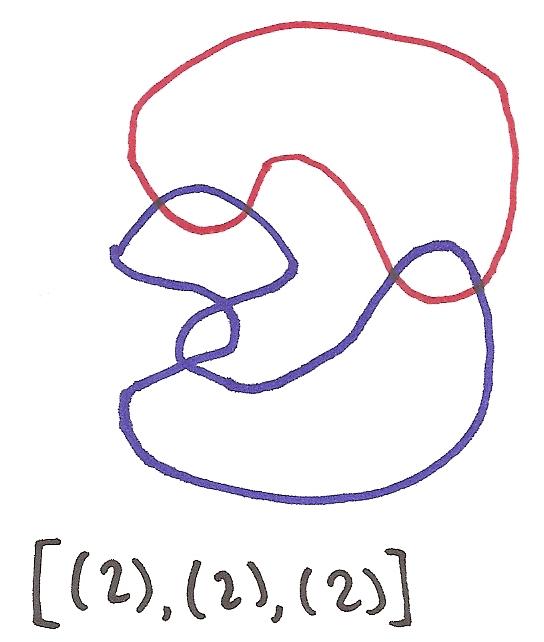}
%\caption{}
%\label{domineering-sum}
\end{center}
\end{figure}

We list some fundamental facts about rational tangles: % apricot which are proven in the appendix:
\begin{theorem}\label{fundrat}
If $[a_1,\ldots,a_m]$ and $[b_1,\ldots,b_n]$ are rational tangles, then they are equivalent if and only if
\[ a_m + \frac{1}{a_{m-1} + \frac{1}{\ddots + \frac{1}{a_1}}} = b_n + \frac{1}{b_{n-1} + \frac{1}{\ddots + \frac{1}{b_1}}}.\]
The \emph{knot or link} $[a_1,\ldots,a_m]$ is a knot (as opposed to a link) if and only if
\[ a_m + \frac{1}{a_{m-1} + \frac{1}{\ddots + \frac{1}{a_1}}} = \frac{p}{q},\]
where $p,q \in \mathbb{Z}$ and $p$ is odd.
Finally, $[a_1,\ldots,a_m]$ is the unknot if and only if $q/p$ is an integer.
\end{theorem}
Note that $[a_1(b_1),\ldots,a_n(b_n)]$ is a knot pseudodiagram (as opposed to a link pseudodiagram) if and only if
$[a_1 + b_1, \ldots, a_n + b_n]$ is a knot (as opposed to a link), since the number of components in the diagram
does not depend on how crossings are resolved.

The proofs of the following lemmas are left as an exercise to the reader.  They are easily seen by drawing pictures,
but difficult to prove rigorously without many irrelevant details.
\begin{lemma}
The following pairs of rational shadows are topologically equivalent (i.e., equivalent up to planar isotopy):
\begin{equation} \left[(1),(a_1),\ldots,(a_n)\right] = \left[(a_1 + 1),(a_2)\ldots,(a_n)\right] \label{onecombl}\end{equation}
\begin{equation} \left[(a_1),\ldots,(a_n),(1)\right] = \left[(a_1),\ldots,(a_{n-1}),(a_n + 1)\right] \label{onecombr}\end{equation}
\begin{equation} \left[(0),(0),(a_1),\ldots,(a_n)\right] = \left[(a_1),\ldots,(a_n)\right] \label{zlossl}\end{equation}
\begin{equation} \left[(a_1),\ldots,(a_i),(0),(a_{i+1}),\ldots,(a_n)\right] =
\left[(a_1),\ldots,(a_i + a_{i+1}),\ldots,(a_n)\right]\label{zloss}\end{equation}
\begin{equation} \left[(a_1),\ldots,(a_n),0,0\right] = \left[(a_1),\ldots,(a_n)\right]\label{zlossr}\end{equation}
\begin{equation} \left[(a_1),(a_2),\ldots,(a_n)\right] = \left[(a_n),\ldots,(a_2),(a_1)\right] \label{invert}\end{equation}
\end{lemma}
Only (\ref{invert}) is non-obvious.  The equivalence here follows by turning everything inside out, as in Figure~\ref{inversion}.
\begin{figure}[H]
\begin{center}
\includegraphics[width=4in]
					{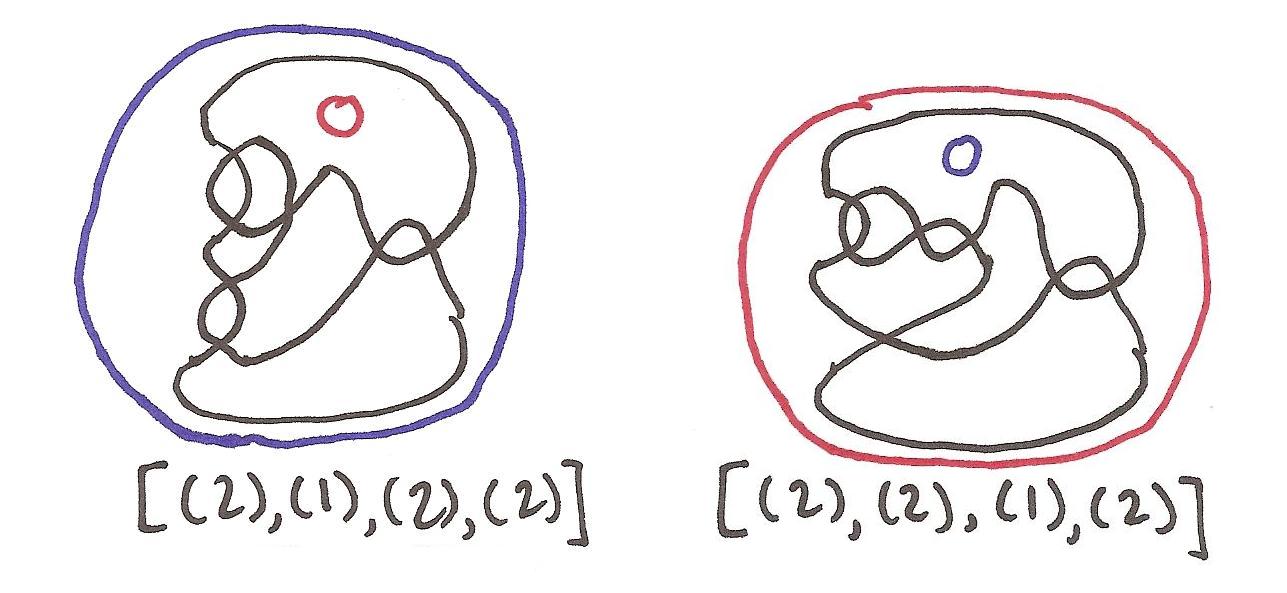}
\caption{These two knot shadows are essentially equivalent.  One is obtained from the the other by turning the diagram inside out,
exchanging the red circle on the inside and the blue circle on the outside.}
\label{inversion}
\end{center}
\end{figure}
This works
because the diagram can be thought of as living on the sphere, mainly because the following operation has no effect on a knot:
\begin{figure}[H]
\begin{center}
\includegraphics[width=4in]
					{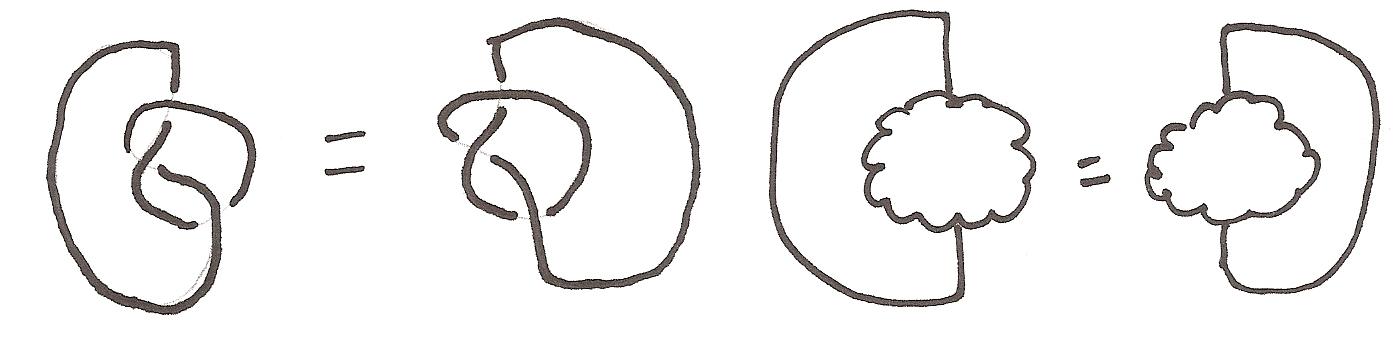}
\caption{Moving a loop from one side of the knot to the other has no effect on the knot.  So we might as well think
of knot diagrams as living on the sphere.}
\label{knots-on-a-sphere}
\end{center}
\end{figure}

Similarly, we also have
\begin{lemma}
\begin{equation} \left[(0),(a_1 + 1),(a_2),\ldots,(a_n)\right] \stackrel{1}{\rightarrow} \left[(0),(a_1),(a_2),\ldots,(a_n)\right] \label{unwindl}\end{equation}
\begin{equation} \left[(a_1),\ldots,(a_{n-1}),(a_n + 1),0\right] \stackrel{1}{\rightarrow} \left[(a_1),\ldots,(a_{n-1}),(a_n),0\right] \label{unwindr}\end{equation}
\begin{equation} \left[\ldots,(a_i + 2),\ldots\right] \stackrel{2}{\rightarrow} \left[\ldots,(a_i),\ldots\right]\label{twoloss}\end{equation}
\end{lemma}

Note that $[]$ is the unknot.
\begin{lemma}
If $T$ is a rational shadow, that resolves to be a knot (not a link), then $T \stackrel{*}{\Rightarrow} []$.
\end{lemma}
\begin{proof}
Let $T = \left[(a_1),\ldots,(a_n)\right]$ be a minimal counterexample.
Then $T$ cannot be reduced by any of the rules specified above.  Since any $a_i \ge 2$ can be reduced
by (\ref{twoloss}), all $a_i < 2$.  If $n = 0$, then $T = \left[\right]$ which turns out
to be the unknot.  If $a_0 = 0$ and $n > 1$, then either $a_1$ can be decreased
by $1$ using (\ref{unwindl}), or $a_0$ and $a_1$ can be stripped off via (\ref{zlossl}).
On the other hand, if $a_0 = 0$ and $n = 1$, then $T = \left[(0)\right]$,
which is easily seen to be a link (not a knot).  So $a_0 = 1$.  If $n > 1$, then $T$ reduces to
$\left[(a_2 + 1),\ldots,(a_n)\right]$ by (\ref{onecombl}).  So $n = 1$, and $T$ is $\left[(1)\right]$ which clearly reduces
to the unknot via a phony Reidemeister I move:
\begin{figure}[H]
\begin{center}
\includegraphics[width=1.5in]
					{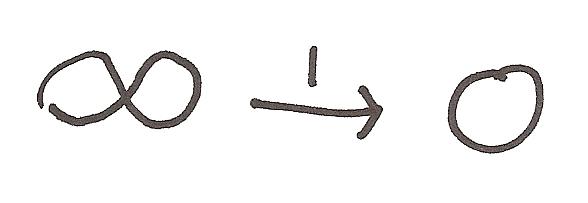}
%\caption{}
%\label{domineering-sum}
\end{center}
\end{figure}
\end{proof}

Because of this, we see that every rational shadow has downside $0$ or $*$, viewed as a TKONTK position:
\begin{theorem}\label{downrational}
Let $T$ be a rational knot shadow (not a link).  Then $\val(T) \approx_- 0$ or $\val(T) \approx_- *$.
\end{theorem}
\begin{proof}
By the lemma, $T \stackrel{*}{\Rightarrow} []$.  But then by Theorem~\ref{redmonotone}, $u^-(T) \le u^-([]) = 0$,
since $\val([]) = 0$ and $u^-(0) = 0$.  But the only possible values for $u^-(T)$ are $0, \frac{1}{4}, \frac{3}{8}, \ldots$,
so $u^-(T) \le 0$ implies that $u^-(T) = 0$.  Then $\val(T)$ is equivalent to either $0$ or $*$.
\end{proof}
Now $\lceil 0 \rceil = 0$ and $\lfloor 0 + \frac{1}{2} \rfloor = 0$, so by Corollary~\ref{booleansummary}, we see that
if $T$ is a rational shadow with an even number of crossings, then $\loutcome(\val(T)) = 0$,
and if $T$ has an odd number of crossings, then $\routcome(\val(T)) = 0$.  So in particular, there are no rational shadows
for which King Lear can win as both first and second player.  And since games with $u^-(G) = 0$ are closed under $\oplus_2$,
the same is true for any position which is a connected sum of rational knot shadows.

Rational pseudodiagrams, on the other hand,
can be guaranteed wins for King Lear.  For example, in $[0,(1),3]$, King Lear can already declare victory:
\begin{figure}[H]
\begin{center}
\includegraphics[width=3in]
					{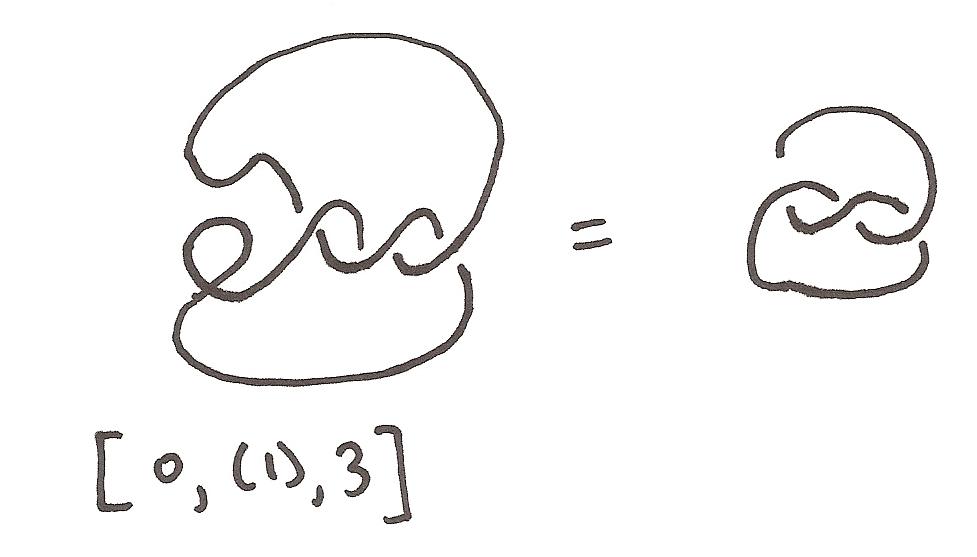}
%\caption{}
%\label{domineering-sum}
\end{center}
\end{figure}

\section{Odd-Even Shadows}
\begin{definition}
An \emph{odd-even shadow} is a shadow of the form \[\left[(a_1),(a_2),\ldots,(a_n)\right],\]
where all $a_i \ge 1$, exactly one of $a_1$ and $a_n$ is odd, and all other
$a_i$ are even.
\end{definition}
Note that these all have an odd number of crossings (so they yield odd-tempered games).
It is straightforward to verify from (\ref{onecombl}-\ref{twoloss})
that every odd-even shadow reduces by phony Reidemeister moves to the unknot.
In particular, by repeated applications of (\ref{twoloss}), we reduce to either
$[(0),\ldots,(0),(1)]$ or $[(1),(0),\ldots,(0)]$.  Then by applying (\ref{zlossl}) or
(\ref{zlossr}), we reach one of the following:
\[ [(1)], [(0),(1)], [(1),(0)].\]
Then all of these are equivalent to $[(1)]$ by (\ref{onecombl}) or (\ref{onecombr}).
So since every odd-even shadow reduces to the unknot, ever odd-even shadow is an actual knot shadow, not a link shadow.
Thus any odd-even shadow can be used as a game of TKONTK.

\begin{theorem}
If $T$ is an odd-even shadow, then $u^+(T) = 0$.
\end{theorem}
\begin{proof}
Suppose that $\loutcome(\val(T) \oplus_2 \langle 0 | 1 \rangle \oplus_2 *) = 0$.  Then by Corollary~\ref{booleansummary},
\[\left\lceil (u^+(T) \cup u^+(\langle 0 | 1 \rangle) \cup u^+(*)) - \frac{1}{2}\right\rceil = 0,\]
since $\val(T) \oplus_2 \langle 0 | 1 \rangle \oplus_2 *$ is odd-tempered.  But $u^+(*) = 0$, $u^+(\langle 0 | 1 \rangle) = \frac{1}{2}$,
and $\frac{1}{2} \cup 0 = \frac{1}{2}$.  And $\lceil x \rceil \le 0$ if and only if $x \le 0$.  So
\[ u^+(T) \cup \frac{1}{2} \le \frac{1}{2}.\]
But if $u^+(T) \ne 0$, then $u^+(T) \ge \frac{1}{4}$, so that $u^+(T) \cup \frac{1}{2} \ge \frac{1}{4} \cup \frac{1}{2} = \frac{3}{4} \not\le \frac{1}{2}$, a contradiction. So it suffices to show that
\[ \loutcome(\val(T) \oplus_2 \langle 0 | 1 \rangle \oplus_2 *) = 0,\]
i.e., that Ursula has a winning strategy as the second player in
\begin{equation} \val(T) \oplus_2 G \label{thecombo}\end{equation}
where $G = \langle 0 | 1 \rangle \oplus_2 * = \langle 0 | 1 \rangle + *$.

Let $T = [(a_1),\ldots,(a_n)]$.  Suppose first that $a_1$ is odd and the other $a_i$ are even.  Then
Ursula's strategy in (\ref{thecombo}) is to always move to a position of one of the following forms:
\begin{description}
\item[(A)] $\val([(b_1),\ldots,(b_m)]) \oplus_2 G'$, where $b_1$ is odd and the other $b_i$ are even, and $G'$ is $G$ or $0$.
\item[(B)] $\val([1(b_1),\ldots,(b_m)]) \oplus_2 G'$, where the $b_i$ are all even and $G'$ is an odd-tempered subposition of $G$.
\end{description}
Note that the initial position is of the first form, with $b_i = a_i$ and $G' = G$.  To show that this is an actual strategy,
we need to show that Ursula can always move to a position of one of the two forms.
\begin{itemize}
\item In a position of type (A), if Lear moves in one of the $b_i$, from $(b_i)$ to $\pm1 (b_i - 1)$, then Ursula
can reply with a move to $(b_i - 2)$, as in Figure~\ref{undoing-move}, \emph{unless} $b_i = 1$. 
\begin{figure}[H]
\begin{center}
\includegraphics[width=5in]
					{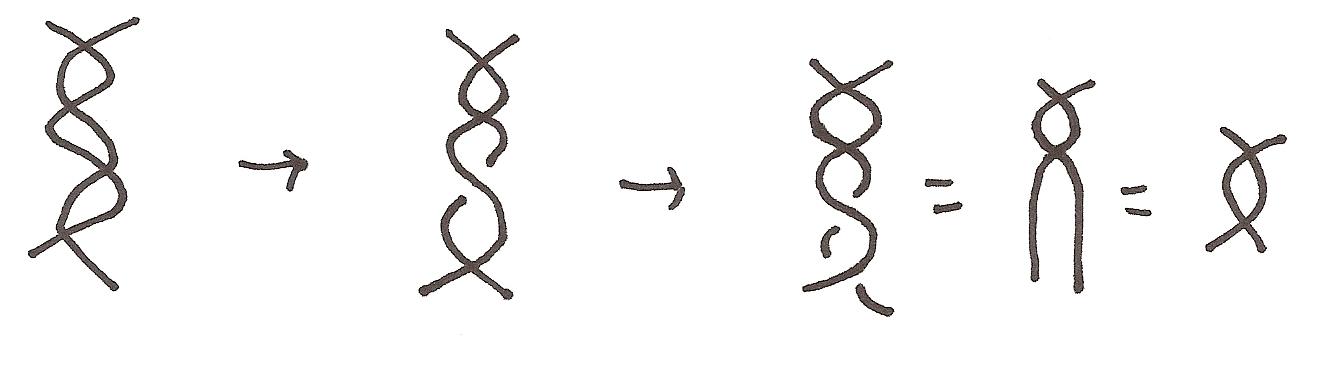}
\caption{Ursula responds to a twisting move by King Lear with a cancelling twist in the opposite direction.}
\label{undoing-move}
\end{center}
\end{figure}
But if $b_i = 1$, then King Lear has just moved to
\[ \val([\pm 1,(b_2),\ldots,(b_n)]) \oplus_2 G' = \val([\pm1(b_2),\ldots,(b_n)]) \oplus_2 G',\]
using the fact that $[\pm 1,(b_2),\ldots,(b_n)] = [\pm 1 (b_2),(b_3),\ldots,(b_n)]$ (obvious from a picture).  So now
Ursula can reply using $b_2$ instead of $b_1$, and move back to a position of type $A$, unless King Lear has just moved to
\[ \val([\pm1]) \oplus_2 G'.\]
But $[\pm 1]$ are unknots, so $\val([\pm 1]) = 0$, and $G'$ is $0$ or $G$, both of which are first player wins for Ursula, so
King Lear has made a losing move.

\item In a position of type (A), if Lear moves in $G' = G$ to $0 + *$, then Ursula replies by moving from $0 + *$ to $0 + 0$,
getting back to a position of type (A).

\item In a position of type (A), if Lear moves in $G' = G$ to $\langle 0 | 1 \rangle + 0$, then Ursula moves $(b_1) \to 1(b_1 - 1)$,
creating a position of type (B).

\item In a position of type (B), if Lear moves in $G'$ to $0$ (this is the only left option of either possibility for $G'$),
then Ursula replies with a move from $1(b_1)$ to $0(b_1 - 1)$, where $0 = 1 - 1$.  This works as long as $b_1 \ne 0$.  But
if $b_1 = 0$, then we could have rewritten $[1(b_1),(b_2),\ldots]$ as $[1(b_2),(b_3),\ldots]$ as before.  If $b_2 = 0$ too,
then we can keep on sliding over, until eventually Ursula finds a move, or it turns out that
King Lear moved to a position of the form
\[ [1] \oplus_2 0\]
which is $0$, a win for Ursula.

\item In a position of type (B), if Lear moves in any $b_i$, then Ursula makes the cancelling move
\[ (b_i) \to \pm 1 (b_i - 1) \to 0(b_i - 2),\]
this is always possible because if $b_i \ge 1$, then $b_i \ge 2$.
\end{itemize}
From the discussion above, Ursula can keep following this strategy until Lear makes a losing move.  So this is a winning
strategy for Ursula and we are done.

The other case, in which $a_n$ is odd and the other $a_i$ are even, is handled completely analogously.
\end{proof}

\section{The other games}
\begin{lemma}\label{ohnoproof}
The following rational shadows have $u^+(T) = 1$:
\[ \left[(3),(1),(3)\right], \left[(2),(1),(2),(2)\right],
\left[(2),(2),(1),(2)\right], \left[(2),(1),(1),(2)\right],\]\[
\left[(2),(2),(1),(2),(2)\right], \left[(2),(2)\right]\]
\end{lemma}
\begin{proof}
To show that $u^+(T) = 1$,
it suffices by Corollary~\ref{booleansummary} to show that $G$ is a win
for King Lear when Ursula moves first, where $G$ is $\val(T)$ if $T$ has an even number of crossings,
and $G$ is $\val(T) + *$ otherwise.  For $\routcome(G) = 1$ iff $\lfloor u^+(G) \rfloor = 1$,
which happens if and only if $1 \le u^+(G) = u^+(T)$.

Unfortunately, the only way I know to prove this criterion for all the knots listed above is by computer,
making heavy use of Theorem~\ref{fundrat}.  
\end{proof}

These are shown in Figure ~\ref{the-irreducibles}.
\begin{figure}[htb]
\begin{center}
\includegraphics[width=4.5in]
					{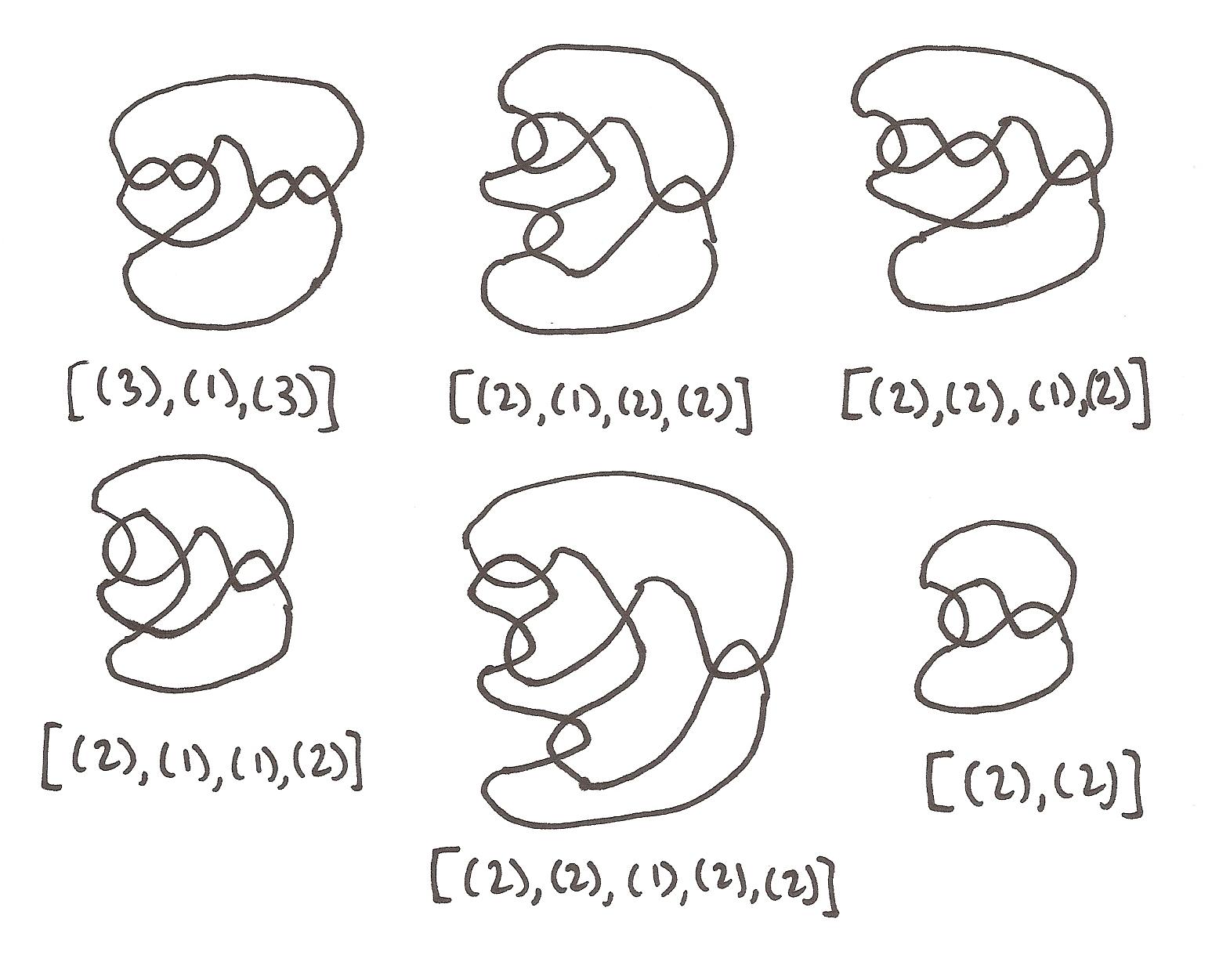}
\caption{The shadows of Lemma~\ref{ohnoproof}.}
\label{the-irreducibles}
\end{center}
\end{figure}

\begin{lemma}\label{minimal}
If $T = \left[(a_1),\ldots,(a_n)\right]$ is a rational shadow corresponding to
a knot (not a link) with at least one crossing, then either $T \stackrel{1}{\Rightarrow} O$ for some odd-even shadow $O$, or
$T \stackrel{*}{\Rightarrow} A$, where $A$ is equivalent to one of the six shadows in Lemma \ref{ohnoproof}.
\end{lemma}
\begin{proof}
Without loss of generality, $T$ is irreducible as far as phony Reidemeister I moves go.  Then we can make
the assumption that all $a_i > 0$.  If all of the $a_i$ are even, then by applying
(\ref{twoloss}) and (\ref{zlossl}-\ref{zlossr}),
we can reduce $T$ down to either $\left[(2),(2)\right]$ or $\left[(2)\right]$.  But the second of these
is easily seen to be a link, so $T \stackrel{*}{\Rightarrow} \left[(2),(2)\right]$.
Otherwise, at least one of the $a_i$ is odd.  If the only odd $a_i$ are $i = 1$ and/or $i = n$, then either
$T$ is an odd-even shadow, or $a_1$ and $a_n$ are both odd.  But if both $a_1$ and $a_n$ are odd, then by applying
(\ref{twoloss}) and (\ref{zloss}), we can reduce to one of the cases $\left[(1),(0),(1)\right]$ or $\left[(1),(1)\right]$.
By (\ref{zloss}) or (\ref{onecombl}), both of these are equivalent to $\left[(2)\right]$, which is not a knot.

This leaves the case where at least one $a_i$ is odd, $1 < i < n$.  Let $T$ be (a) not reducible by
phony Reidemeister I moves or by (\ref{onecombl}-\ref{onecombr}), and (b) as reduced as possible
by phony Reidemeister II moves, without breaking the property of having one of the $a_i$ be odd,
for $1 < i < n$.  If $a_j > 2$ for any $1 < j < n$, then we can reduce $a_j$ by two, via (\ref{twoloss}).
So for every $1 < j < n$, $a_j \le 2$.  Similarly, $a_1$ and $a_n$ must be either
$2$ or $3$.  (They cannot be $1$ or else $T$ would be reducible by (\ref{onecombl}) or (\ref{onecombr}).)

Choose $i$ for which $a_i$ is odd.
If $a_1 = 3$ and $i > 2$, then we can reduce $a_1$ by two (\ref{twoloss}) and combine it (\ref{onecombl}) into $a_2$ to yield a smaller
$T$.  So if $a_1 = 3$, then $a_2 = 1$ and $a_j \ne 1$ for $j > 2$ (or else we could have
chosen a different $i$ and reduced).  Similarly, if $a_n = 3$, then
$a_{n-1} = 1$ and $a_j \ne 1$ for $j < n - 1$.  Thus, if a sequence begins with $(3)$, the next
number must be $(1)$, and the $(1)$ must be unique.  For example, the sequence $\left[(3),(1),(1),(3)\right]$
can be reduced to $\left[(1),(1),(1),(3)\right]$ and thence to $\left[(2),(1),(3)\right]$.

On the other hand, suppose $a_1 = 2$.  If $i > 4$ then we can reduce $T$ farther by decreasing
$a_1$ by (\ref{twoloss}), and then decreasing $a_2$ one by one via (\ref{unwindl})
until both $a_1$ and $a_2$ are zero.  Then both can be removed by (\ref{zlossl}), yielding a smaller $T$.
A further application of (\ref{onecombl}) may be necessary to remove an initial 1.
Moreover, if (\ref{onecombl}) is unnecessary, because $a_3 > 1$, then this also works if $i = 4$.

Therefore, what precedes any $a_i = 1$ must be one of the following:
\begin{itemize}
\item $(3)$
\item $(2)$
\item $(2)(2)$
\item $(2)(1)$
\item $(2)(2)(1)$
\item $(2)(1)(1)$
\end{itemize}
and only the first three of these can precede the first $(1)$.
The same sequences reversed must follow any $(1)$ in sequence.  Then the only
combinations which can occur are:
\begin{itemize}
\item $\left[(3),(1),(3)\right]$
\item $\left[(3),(1),(2)\right]$ and its reverse
\item $\left[(3),(1),(2),(2)\right]$ and its reverse
\item Not $\left[(3),(1),(1),(2)\right]$ because more than just $(3)$ precedes the second $(1)$.
\item $\left[(2),(1),(2)\right]$
\item $\left[(2),(1),(2),(2)\right]$ and its reverse
\item $\left[(2),(1),(1),(2)\right]$
\item $\left[(2),(1),(1),(2),(2)\right]$ and its reverse
\item $\left[(2),(1),(1),(1),(2)\right]$
\item $\left[(2),(2),(1),(2),(2)\right]$
\item $\left[(2),(2),(1),(1),(2),(2)\right]$
\item Not $\left[(2),(2),(1),(1),(1),(2)\right]$ because too much precedes the last $(1)$.
\end{itemize}
So either $T$ is one of the combinations in Lemma \ref{ohnoproof} or one of the following
happens:
\begin{itemize}
\item $\left[(3),(1),(2)\right]$ reduces by (\ref{twoloss}) to $\left[(1),(1)
,(2)\right] = \left[(2),(2)\right]$.  So does its reverse.
\item $\left[(3),(1),(2),(2)\right]$ reduces by two phony Reidemeister II moves to
\[\left[(3),(1),(0),(0)\right] = \left[(3),(1)\right] = \left[(4)\right]\] which is a link,
not a knot.  Nor is its reverse.
\item $\left[(2),(1),(2)\right]$ reduces by a phony Reidemeister II move to $\left[(0),(1),(2)\right]$,
which in turn reduces by a phony Reidemeister I move to $\left[(0),(0),(2)\right] = \left[(2)\right]$
which is a link, not a knot.  So this case can't occur.
\item $\left[(2),(1),(1),(2),(2)\right]$ reduces by phony Reidemeister moves to
\[\left[(2),(1),(1),(0),(2)\right] = \left[(2),(1),(3)\right]\] so it isn't actually minimal.
\item $\left[(2),(1),(1),(1),(2)\right]$ likewise reduces by a phony Reidemeister II move and
a I move to \[\left[(0),(0),(1),(1),(2)\right] = \left[(1),(1),(2)\right] = \left[(2),(2)\right]\]
\item $\left[(2),(2),(1),(1),(2),(2)\right]$ reduces by a phony Reidemeister II move to
\[\left[(2),(0),(1),(1),(2),(2)\right] = \left[(3),(1),(2),(2)\right],\] so it isn't actually
minimal.
\end{itemize}
In summary then, every $T$ that does not reduce by phony Reidemeister I moves to
an odd-even shadow reduces down to a finite set of minimal cases.  Each of these minimal cases
is either reducible to one of the six shadows in Lemma \ref{ohnoproof}, or is not actually
a knot.
\end{proof}

\section{Sums of Rational Knot Shadows}
Putting everything together we have
\begin{theorem}\label{firstmain}
Let $T$ be a rational knot shadow, and let $T' = \left[a_1,a_2,\ldots,a_n\right]$ be
the smallest $T'$ such that $T \to_{1}^* T'$.  Then if $T'$ is an odd-even shadow,
$T$, $u^+(T) = u^-(T) = 0$, and otherwise, $u^+(T) = 1$, $u^-(T) = 0$.
\end{theorem}
\begin{proof}
We know that $\val(T')$ and $\val(T)$ differ by $0$ or $*$,
so $u^+(T) = u^+(T')$ and $u^-(T) = u^-(T')$.  We already know that if
$T'$ is an odd-even shadow, then $u^-(T') = 0$.  Otherwise,
by Lemma~\ref{minimal}, $T'$ must reduce by phony Reidemeister I and II moves
to a rational shadow $T''$ that is one of the six shadows in Lemma~\ref{ohnoproof}.
By Lemma~\ref{ohnoproof}, $u^+(T'') = 1$.  Then by Theorem~\ref{redmonotone}, $u^+(T) \ge u^+(T') \ge u^+(T'')$.
But $u^+(T'')$ is already the maximum value $1$, so $u^+(T) = 1$ too.  On the other hand,
we know that $u^-(T) = 0$ by Theorem~\ref{downrational}, regardless of what $T$ is.
\end{proof}

\begin{definition} A rational knot shadow \emph{reduces to an odd-even shadow}
if it reduces to an odd-even shadow via phony Reidemeister I moves.
\end{definition}
The previous theorem can be restated to say that a rational knot shadow has $u^+ = 0$ if it reduces to an odd-even shadow, and
$u^+ = 1$ otherwise.

\begin{theorem}
If $T_1, T_2, \ldots T_n$ are rational knot shadows, and $T = T_1 + T_2 + \ldots + T_n$ is
their connected sum, then $T$ is a win for Ursula if all of the
$T_i$ reduce to odd-even shadows.  Otherwise, if $T$ has an odd number of crossings,
then $T$ is a win for whichever player goes first, and if $T$ has an even number of crossings, then
$T$ is a win for whichever player goes second.
\end{theorem}
\begin{proof}
Note that $0 \cup 0 = 0$ and $1 \cup 0 = 1 \cup 1 = 1$.  So by Theorem~\ref{firstmain},
if every $T_i$ reduces to an odd-even shadow, then $u^\pm(T_1 + \cdots + T_n) = 0$.
So then $\val(T_1 + \cdots + T_n) \approx 0$, and so $T_1 + \cdots + T_n$ is a win
for Right (Ursula) no matter who goes first.

Otherwise, it follows by Theorem~\ref{firstmain} that $u^-(T_1 + \cdots + T_n) = 0$
and $u^+(T_1 + \cdots + T_n) = 1$.  So by Corollary~\ref{booleansummary},
if $\val(T_1 + \cdots + T_n)$ is even-tempered, then
\[ \loutcome(\val(T_1 + \cdots + T_n)) = \lceil 0 \rceil = 0\]
so that Ursula wins if King Lear goes first, and
\[ \routcome(\val(T_1 + \cdots + T_n)) = \lfloor 1 \rfloor = 1\]
so that King Lear wins if Ursula goes first.

On the other hand, if $\val(T_1 + \cdots + T_n)$ is odd-tempered, then
\[ \loutcome(\val(T_1 + \cdots + T_n)) = \lceil 1 - \frac{1}{2} \rceil = 1\]
so that King Lear wins when he goes first, and similarly
\[ \routcome(\val(T_1 + \cdots + T_n)) = \lfloor 0 - \frac{1}{2} \rfloor = 0\]
so that Ursula wins when she goes first.
\end{proof}

\section{Computer Experiments and Additional Thoughts}
% apricot TODO: look for more rational pseudodiagrams with interesting values.

So far, we have determined the value of rational \emph{shadows}, that is,
rational pseudodiagrams in which no crossings are resolved.  Although
we have ``solved'' the instances of \textsc{To Knot or Not to Knot} that correspond
to rational knots, we have not \emph{strongly solved} them, by finding
winning strategies in all their subpositions.  This would amount
to determining the values of all rational pseudodiagrams.

Since rational pseudodiagrams resolve to rational knots, a computer can check
whether the outcome of a game is knotted or not.
I wrote a program to determine the values of small rational pseudodiagrams.  Interestingly,
the only values of $u^+$ and $u^-$ which appeared were $0$, $1$, and $\frac{1}{2}*$.
This also appeared when I analyzed the positions of the following shadow:
\begin{figure}[H]
\begin{center}
\includegraphics[width=1in]
					{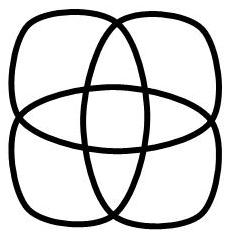}
\caption{The simplest shadow which does not completely reduce via phony Reidemeister I
and II moves.  %Consequently, Theorem~\ref{xymonotone} does not place this shadow
%in $X_0$.  In fact, it is in $X_2 \cap Y_2$, and is thus the simplest shadow for which
%the Knotter wins playing as both first and second.
}
\label{petals}
\end{center}
\end{figure}
which is the simplest shadow which does not reduce via phony Reidemeister I and II moves
to the unknot.  So Theorem~\ref{downrational} does not apply, and in fact, by a computer,
I verified that this knot does not have $u^- = 0$.

Since Figure~\ref{petals} is not a rational shadow or sum of rational shadows, I used another invariant
called the \emph{knot determinant} to check whether the final resolution was an unknot.

\begin{definition}
Let $K$ be a knot diagram with $n$ crossings and $n$ strands.  Create a matrix
$M$ such that $M_{ij}$ is -1 if the $i$th strand terminates at the $j$th crossing,
2 if the $i$th strand passes over the $j$th crossing, and 0 otherwise.  The
\emph{knot determinant} is defined as $|\det(M')|$, where $M'$ is any
$(n-1)\times(n-1)$ submatrix of $M$.
\end{definition}
It turns out that the knot determinant is well defined, and is even a knot invariant.
In fact, if $\Delta(z)$
is the Alexander polynomial, then the knot determinant is just $|\Delta(-1)|$.
The knot determinant of the unknot equals 1.

\begin{lemma}
If the knot shadow in Figure~\ref{petals}
 is resolved into a knot $K$, then $K$
is the unknot iff the knot determinant of $K$ equals 1.
\end{lemma}
\begin{proof}
We can use a computer to check whether a resolution of the diagram has knot determinant
1.  There are only 256 resolutions, so it is straightforward to iterate over all resolutions.
Up to symmetry, it turns out that the only resolutions with knot determinant 1 are those
shown in Figure~\ref{unknots}.
It is straightforward to check that all of these are the
unknot.  Conversely, any knot whose determinant is not 1 cannot be the unknot, since
the knot determinant is a knot invariant.
\end{proof}

\begin{figure}[h]
\begin{center}
\includegraphics[width=4.5in]
					{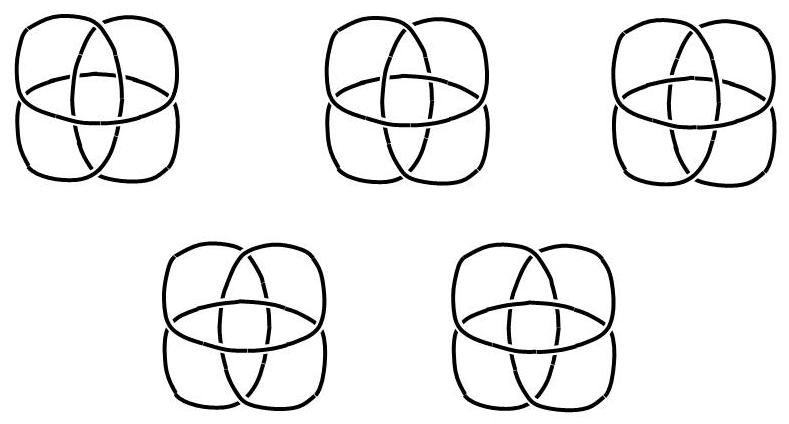}
\caption{Up to symmetry, these are the only ways to resolve Figure~\ref{petals} and have the knot determinant equal 1.
They are all clearly the unknot.}
\label{unknots}
\end{center}
\end{figure}

Because of this, we can use a computer to determine the value of the game played
on the diagram in Figure~\ref{petals}.
The value turned out to be $u^+ = 1$.  Then by Corollary~\ref{booleansummary},
this game is a win for King Lear, no matter who goes first.  This answers a question posed in \emph{A Midsummer Knot's Dream}.

The program used to analyze the shadow of Figure~\ref{petals}
also determined the
values that occur in subpositions of this game.  Again, only the values
$0$, $1$, and $\frac{1}{2}*$ were seen for $u^\pm$ values.

This led me to conjecture that these were the only possible values for any knot pseudodiagrams.
However, it seems very unlikely that this could be due to some special property of knots.  In fact, it seems
like it might be true of a larger class of games, in which Left and Right take turns setting the arguments
of a fixed function $f:\{0,1\}^n \to \{0,1\}$, and then the value of $f$ determines who wins.

I tried for a long time to prove that for such games, the only possible $u^\pm$ values were $0$,
$1$, and $\frac{1}{2}*$.  This was unlikely, for the following reason:
\begin{theorem}
If $G$ is an odd-tempered Boolean game, and $u^\pm(G) \in \{0,1,\frac{1}{2}*\}$, then $G$ is not a second-player win.
\end{theorem}
\begin{proof}
If $G$ is a first-player win, then $\loutcome(G) = 0$ and $\routcome(G) = 1$.  By Corollary~\ref{booleansummary},
this means that
\[ 0 = \lceil u^+(G) - \frac{1}{2} \rceil,\]
so that $u^+(G) \le \frac{1}{2}$.  Similarly,
\[ 1 = \lfloor u^-(G) + \frac{1}{2} \rfloor,\]
so that $u^-(G) \ge \frac{1}{2}$.  Then we have
\[ \frac{1}{2} \le u^-(G) \le u^+(G) \le \frac{1}{2},\]
so that $u^-(G) = u^+(G) = \frac{1}{2}$, a contradiction.
\end{proof}

\emph{Projective Hex}\footnote{Invented by Bill Taylor and Dan Hoey, according to the Internet.} is a positional game like Hex, in which the two players take turns placing pieces of their own colors
on a board until somebody creates a path having a certain property.  In Hex, the path needs to connect your two sides of the board,
but in projective Hex, played on a projective plane, the path needs to wrap around the world an odd number of times:
\begin{figure}[H]
\begin{center}
\includegraphics[width=4in]
					{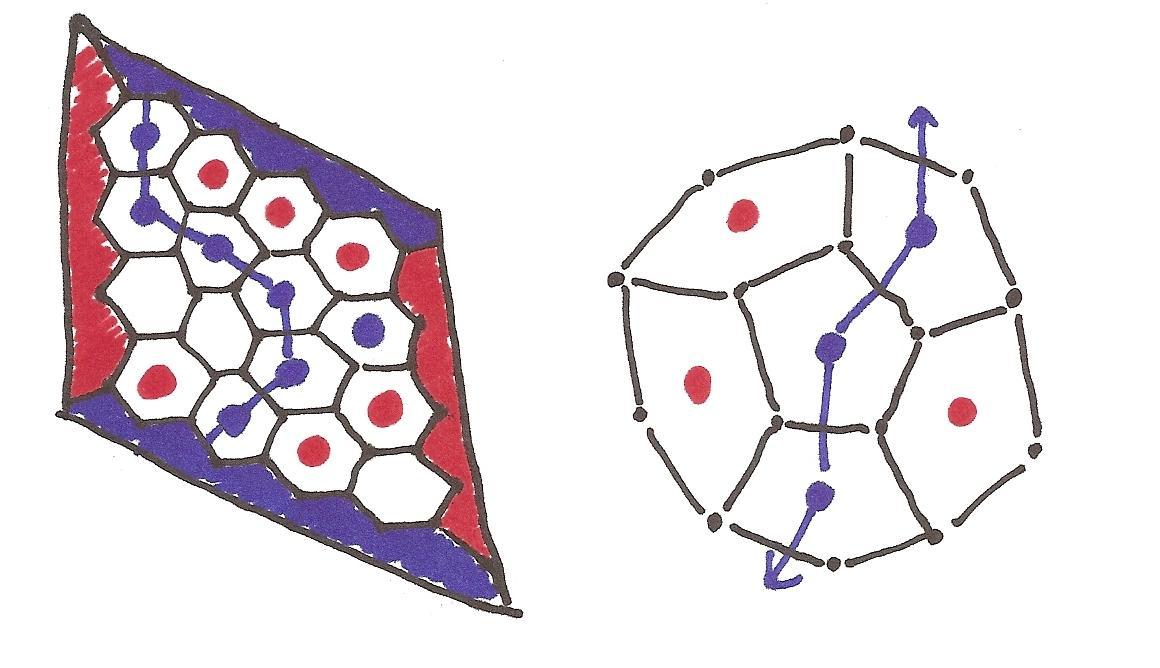}
\caption{Hex (left) and Projective Hex (right, here played on the faces of a dodecahedron).  In both games,
Blue has won.  In the Hex game, she connected her two sides, and in the Projective Hex game, she created a path which wrapped
around the world an odd number of times.}
\label{hex-and-projective-hex}
\end{center}
\end{figure}
By a standard strategy-stealing argument, Hex and Projective Hex are necessarily wins for the first player.  When Projective Hex
is played on the faces of a dodecahedron (or rather on the pairs of opposite faces)
it has the property that every opening move is a winning move, by symmetry.

Now modify dodecahedral projective hex by adding another position where the players can play (making seven positions total).
If a white piece ends up in the extra position, then the outcome is reversed, and otherwise the outcome is as before.  Also,
let players place pieces of either color.

Effectively, the players are playing dodecahedral projective Hex, but XOR'ing the outcome with the color of the piece in the extra
position.
\begin{figure}[H]
\begin{center}
\includegraphics[width=3in]
					{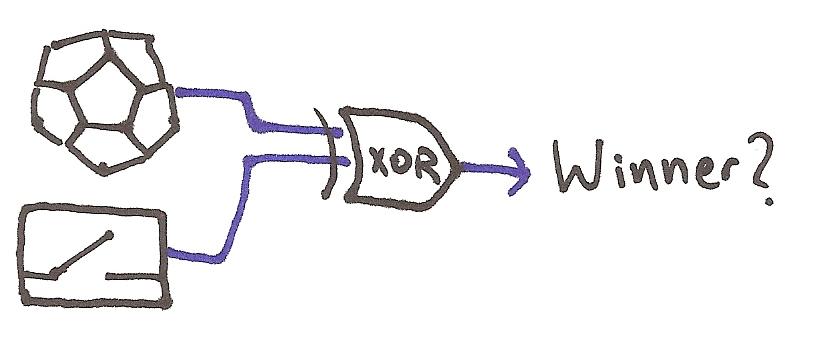}
%\caption{}
%\label{hex-and-projective-hex}
\end{center}
\end{figure}
The resulting game comes from a function $\{0,1\}^7 \to \{0,1\}$, and I claim that it is a second-player win.  If the first player
places a white piece on the dodecahedron, the second player can choose to be the white player in dodecahedral projective Hex, by making
an appropriate move in the special location.  The fact that players can place pieces of the wrong color then becomes immaterial,
because playing pieces of the wrong color is never to your advantage in projective Hex or ordinary Hex.

On the other hand, if the first player tries playing the special location, then he has just selected what color he will be,
and given his opponent the first move in the resulting game of projective Hex, so his opponent will win.

Therefore, the resulting modified projective Hex is a counterexample to the idea that only $\{0,1,\frac{1}{2}*\}$ can occur
as $u^\pm$ values for games coming from Boolean functions $\{0,1\}^n \to \{0,1\}$.
For all I know, it might be possible to embed this example within a game of \textsc{To Knot or Not to Knot}.  Consequently, I now conjecture that
all of the possible values occur in positions of TKONTK, though I don't know how to prove this.

\appendix
\chapter{Bibliography}
I would like to thank my advisor Jim Morrow, who reviewed this thesis and provided valuable feedback.

I used the following sources and papers:
\begin{itemize}
\item Adams, Colin. \emph{The Knot Book: an Elementary Introduction to the Mathematical Theory of Knots}.
\item Albert, Michael and Richard Nowakowski, eds. \emph{Games of No Chance 3}
\item Berlekamp, Elwyn, John Conway, and Richard Guy. \emph{Winning Ways for your Mathematical Plays}.  Second edition, 4 volumes.
\item Berlekamp, Elwyn and David Wolfe. \emph{Mathematical Go: Chilling Gets the Last Point.}
\item Conway, John. \emph{On Numbers and Games}.  Second Edition.
\item \href{http://citeseerx.ist.psu.edu/viewdoc/summary?doi=10.1.1.37.5278}{Ettinger, J. Mark.  ``On the Semigroup of Positional Games.''}
\item \href{http://citeseerx.ist.psu.edu/viewdoc/summary?doi=10.1.1.37.7498}{Ettinger, J. Mark.  ``A Metric for Positional Games.''}
\item \href{http://www.math.washington.edu/~reu/papers/current/allison/knotpaper1.pdf}{Henrich, A., N. MacNaughton, S. Narayan, O. Pechenik, R. Silversmith, and J. Townsend. ``A Midsummer Knot's Dream.''}
\item Milnor, J. W. ``Sums of Positional Games'' in \emph{Contributions to the Theory of Games, Annals of Mathematics} edited by H. W. Kuhn and A. W. Tucker.
\item \href{http://www.mathstat.dal.ca/~rjn/papers/HistoryCGT.pdf}{Nowakowski, Richard. ``The History of Combinatorial Game Theory.''}
\item Nowakowski, Richard, ed. \emph{Games of No Chance}.
\item Nowakowski, Richard, ed. \emph{More Games of No Chance}.
\item \href{http://www.plambeck.org/archives/integerssemigrouppaperORIGINAL.pdf}{Plambeck, Thane. \emph{Taming the Wild in Impartial Combinatorial Games}}
\end{itemize}
Preliminary versions of my own results are in the following papers:
\begin{itemize}
\item \href{http://www.math.washington.edu/~reu/papers/current/will/knot-sum-games.pdf}{\emph{A Framework for the Addition of Knot-Type Combinatorial Games}.}
\item \href{http://www.math.washington.edu/~reu/papers/current/will/rational-knot-sums.pdf}{\emph{Who wins in \emph{To Knot or Not to Knot} played
on Sums of Rational Shadows}.}
\end{itemize}
%\appendix apricot
%\chapter{Some Knot Theory}
%Only write if you have time.
\end{document}